 \documentclass[final,1p,times,toc=flat]{elsarticle}

\usepackage{amssymb}
\usepackage{amsthm}
\usepackage{amsmath}
\usepackage{MnSymbol}
 \usepackage{titlesec}
 \usepackage{titletoc}
\usepackage{appendix}
\usepackage{xhfill}
\usepackage{graphicx}
\usepackage{lineno}
\newtheorem{theorem}{Theorem}[section]
\newtheorem{lemma}[theorem]{Lemma}

\theoremstyle{definition}
\newtheorem{definition}[theorem]{Definition}

\theoremstyle{proposition}
\newtheorem{proposition}[theorem]{Proposition}

\theoremstyle{remark}
\newtheorem{remark}[theorem]{Remark}

\theoremstyle{notation}
\newtheorem{notation}[theorem]{Notation}

\theoremstyle{construction}
\newtheorem{construction}[theorem]{Construction}
\renewcommand\appendix{\setcounter{secnumdepth}{-2}}
\numberwithin{equation}{section}
\journal{Journal of Mathematical Logic}

\begin{document}

\begin{frontmatter}



\title{Density Elimination for Semilinear Substructural Logics\tnoteref{mytitlenote}}
\tnotetext[mytitlenote]{This work is supported by the National Foundation of Natural Sciences of China (Grant No: 61379018 \&61662044\& 11571013)}

\author{SanMin Wang}

\address{Faculty of Science, Zhejiang Sci-Tech University, Hangzhou 310018, P.R. China

Email:wangsanmin@hotmail.com}

\begin{abstract}
We present a uniform method of density elimination for several semilinear substructural logics. Especially, the density elimination for the
involutive uninorm logic \textbf{IUL} is proved. Then the standard
completeness of \textbf{IUL} follows as a lemma by virtue of previous work by
Metcalfe and Montagna.
\end{abstract}


\end{frontmatter}
\setlength\linenumbersep{1cm}
\renewcommand\linenumberfont{\normalfont}

\begin{center} \bf Contents  \end{center}
\rm 1.  Introduction  \dotfill   2\\
2. Preliminaries   \dotfill  5\\
3. Proof of the main  theorem: A computational example   \dotfill  7\\
4. Preprocessing of Proof Tree   \dotfill  15\\
5. The generalized density rule $(\mathcal {D})$ for $ {\bf GL_{\Omega}}$   \dotfill  22\\
6. Extraction of Elimination Rules   \dotfill  26\\
7. Separation of one branch   \dotfill  31\\
8. Separation algorithm of multiple branches    \dotfill  35\\
9. The proof of Main  theorem    \dotfill  45\\
10. Final remarks and open problems   \dotfill  46\\
\rm References   \dotfill  47\\
Appendices   \dotfill  47\\
A.1.  Why do we adopt Avron-style hypersequent calculi?   \dotfill  47\\
A.2. Why do we need the constrained external contraction rule?   \dotfill  48\\
A.3. Why do we need the separation of branches?   \dotfill  49\\
A.4. Some questions about Theorem 8.2   \dotfill  49\\
A.5. Illustrations  of   notations and algorithms   \dotfill  50\\
A.5.1. Illustration of two cases of $(COM)$ in the proof of Lemma 5.6   \dotfill  50\\
A.5.2 Illustration of  Construction 6.1\dotfill  51\\
A.5.3. Illustration of Notation 6.10 and Construction 6.11   \dotfill  51\\
A.5.4. Illustration of Theorem 8.2   \dotfill  52
\begin{center}
\bf {Notation}
\end{center}
\rm
$ G_1 \equiv G_2$   \dotfill  The symbol $ G_1  $  denotes  a complex hypersequent $G_2$  temporarily for convenience.\\
$X\coloneq Y $   \dotfill  Define  $X$  as  $ Y$  for two hypersequents (sets or derivations)  $X $  and  $Y$.\\
$G_0$  \dotfill The upper hypersequent of  strong density rule in Theorem 1.1, Page 4\\
$\tau $ \dotfill A cut-free proof  of  $ G_0  $  in  $ {\rm {\bf GL}}$,  in Theorem 1.1,  Page 4 \\
$\mathcal{P}(H)$ \dotfill  The position of  $H\in\tau$, Def. 2.13, Construction 6.1,  Page 7,26\\
$\left\langle {H_k } \right\rangle _{H:H'} $ and $\tau _{H:H'}^{2}
(\left\langle {H_k } \right\rangle _{H:H'} )$ \dotfill Construction 4.7,  Page 17\\
$G_{H:H'}^{2} $ and $\tau _{H:H'}^{2} $  \dotfill Notation 4.10,  Page 18\\
$\tau ^ *$   \dotfill   The proof of  $ G\vert G^ *  $  in  $ {\rm {\bf GL}}_\Omega  $
resulting from preprocessing of  $ \tau$, Notation 4.13,  Page 20\\
$G\vert G^ * $   \dotfill  The root of  $\tau ^ *$ corresponding to the root $G_0$ of $\tau$, Notation 4.13,  Page 20\\
$ H_i^c  $   \dotfill  The $i$-th $(pEC)$-node in  $ \tau^*$,  the superscript $'c'$  means contraction, Notation 4.14,  Page 20\\
$ S_{i1}^c $    \dotfill  The focus sequent of $ H_i^c$, Notation 4.14,  Page 20\\
$ S_{i}^c$ or $ S_{iu}^c $    \dotfill   $ S_{i1}^c$ or  one copy of $ S_{i1}^c$, Notation 4.14,  Page 20\\
$\{H_{1}^{c},\cdots, H_{N}^{c} \} $  \dotfill   The set of  all  $(pEC)$-nodes in  $ \tau^*$, Notation 4.14,  Page 20\\
${\rm {\bf GL}}_\Omega  $   \dotfill  A restricted subsystem of ${\rm {\bf GL}}$, Definition 4.16,  Page 20\\
$\left[ S \right]_G ,\left[ {G'} \right]_G $   \dotfill  The minimal closed unit of $S$ and $G'$ in $G$, respectively, Definition 5.1,  Page 22\\
$(\mathcal{D})$  \dotfill  The generalized density rule of ${\rm {\bf GL}}_\Omega $,
Definition 5.4,  Page 22\\
$\tau _{S_{i1}^c }^\ast $ and $G_{S_{i1}^c }^\ast $   \dotfill  Notation 6.5. ,  Page 28\\
$H_i^c \rightsquigarrow H_j^c, H_i^c \leftrightsquigarrow H_j^c$   \dotfill  Definition 6.8,  Page 28\\
$I=  \{ {H_{i_1 }^c, \cdots,H_{i_m }^c } \}  $   \dotfill   A subset of $\{H_{1}^{c},\cdots, H_{N}^{c} \}$, Notation  6.10,  Page 29\\
  $ H_{I}^V$, $ H_{ij}^V$    \dotfill  The intersection nodes of  $I$ and,  that of  $ H_{i }^c$ and $H_{j}^c$,  Notation  6.10,  Page 29\\
$ \mathcal{I'}=  \{ {S_{i_1 u_1}^c, \cdots, S_{i_m u_m}^c } \} $   \dotfill   A subset of $(pEC)$-sequents to $I$, Definition   6.14,  Page 30\\
 $ {\rm {\bf I'}}= \{ {G_{b_1} \vert S_{i_1 u_1}^c, \cdots, G_{b_m} \vert S_{i_m u_m}^c } \} $   \dotfill  A set of closed hypersequents  to  $I$, Def. 6.14,  Page 30 \\
$\left\langle H \right\rangle _\mathcal{I}, \tau _\mathcal{I}^\ast $ and $G_\mathcal{I}^\ast $   \dotfill The elimination derivation,  Construction 6.11,  Lemma 6.13,  Page 29,  30\\
$\tau _{\rm {\bf I}'}^\ast$  \dotfill  The elimination rule, Definition 6.14, ,  Page 30\\
$\left\lceil {S_{i_k }^c } \right\rceil _I $  \dotfill  A branch of $H_{i_k }^c $
to $I$, Definition 7.2,  Page 31\\
$G_{\rm {\bf I}}^{\medstar (q)} ,G_{\rm {\bf I}}^{\medstar( J_{\rm {\bf I}} )} , \tau _{\rm {\bf I}}^{\medstar (q)}  $   \dotfill Construction 7.3,  Page 31\\
$G_{H:H_1 }^{\medstar (J)} ,\tau _{H:H_1 }^{\medstar (J)} $  \dotfill  Construction 7.5,  Page 33\\
$G_{\rm {\bf I}}^\medstar,\tau _{\rm {\bf I}}^\medstar$  \dotfill  Construction 7.7,  Theorem 8.2,  Page 34, 36\\
$\overline \tau _{\rm {\bf I}}^\medstar$  \dotfill  The skeleton of $\tau _{\rm {\bf I}}^\medstar$, Definition 7.13, Page 35\\
$\partial _{\tau _{\rm {\bf I}}^\medstar} (H)$  \dotfill  Theorem 8.2 (ii),  Page 36\\
$\tau _{{\rm {\bf I}}:G_2 }^\medstar$  \dotfill The module of $\tau _{\rm {\bf I}}^\medstar$ at $G_2$, Definition 8.7,  Page 38\\

\section {Introduction}
The problem of the completeness of {\L}ukasiewicz infinite-valued logic
(\textbf{\L} for short)  was posed by {\L}ukasiewicz and Tarski in the 1930s. It
was twenty-eight years later that it was syntactically solved by Rose and
Rosser [20].  Chang [4] developed at almost the same time a theory of
algebraic systems for \textbf{\L}, which is called \textbf{MV}-algebras,
with an attempt to make \textbf{MV}-algebras correspond to
\textbf{\L} as
Boolean algebras to the classical two-valued logic.  Chang [5] subsequently
finished another proof for the completeness of \textbf{\L} by virtue of his
\textbf{MV}-algebras.

It was Chang who observed that the key role in the structure theory of
\textbf{MV}-algebras is not locally finite \textbf{MV}-algebras but
 linearly ordered ones.  The observation was formalized  by H\'{a}jek [12] who showing the completeness for his basic fuzzy logic (\textbf{BL} for short) with respect to linearly ordered \textbf{BL}-algebras. Starting with the structure
of \textbf{BL}-algebras, H\'{a}jek [13] reduced the problem of the standard
completeness of \textbf{BL} to two formulas to be provable in
\textbf{BL}. Here and thereafter, by the standard completeness we mean that logics are complete with respect to algebras with lattice reduct [0, 1].  Cignoli et al. [6]  subsequently  proved the standard
completeness of \textbf{BL}, i.e., \textbf{BL} is the logic of continuous
t-norms and their residua.

Hajek's approach toward fuzzy logic has been extended
by Esteva and Godo in [9], where the authors introduced the logic
\textbf{MTL} which aims at capturing the tautologies of left-continuous
t-norms and their residua. The standard completeness of \textbf{MTL} was
proved by Jenei and Montagna in [15], where the major step is to embed
linearly ordered \textbf{MTL}-algebras into the dense ones under the
situation that the structure of \textbf{MTL}-algebras has been unknown as
yet.

Esteva and Godo's work  was  further promoted by Metcalfe and Montagna [16]  who
introduced the uninorm logic \textbf{UL} and involutive uninorm logic
\textbf{IUL} which aim at capturing tautologies of left-continuous uninorms and their residua and those of involutive left-continuous ones, respectively.  Recently, Cintula and Noguera [8] introduced semilinear substructural logics which are substructural logics complete with respect to linearly ordered models.
Almost all well-known families of fuzzy logics such as
 \textbf{{\L}}, \textbf{BL}, \textbf{MTL}, \textbf{UL} and \textbf{IUL} belong to the class of semilinear substructural logics.

Metcalfe and Montagna's method to prove standard completeness for \textbf{UL
}and its extensions is of proof theory in nature and  consists of two key
steps.  Firstly, they extended\textbf{ UL } with the
density rule of Takeuti and Titani [21]:
\[
\cfrac{\Gamma \vdash (A \to p) \vee (p \to B) \vee C}{ \Gamma
\vdash (A \to B) \vee C}   (D)
\]
\noindent
where  $ p $  does not occur in  $ \Gamma,A, B $  or  $ C $,   and then prove the logics
with $(D)$  are complete with respect to algebras with lattice reduct [0,1].
Secondly,  they give a syntactic elimination of $(D)$  that was formulated as a
rule of the corresponding hypersequent calculus.

 Hypersequents are a natural generalization of sequents which were
introduced independently by Avron [1] and Pottinger [19] and   have  proved to be particularly
suitable for logics with prelinearity [2, 16].  Following the spirit of
Gentzen's cut elimination, Metcalfe and Montagna succeeded to eliminate the density rule for \textbf{GUL} and several extensions of \textbf{GUL} by induction on the height
of a derivation of the premise and shifting applications of the rule
upwards, but failed for\textbf{ GIUL} and therefore left it as an open problem.

There are several relevant works about the standard completeness of
\textbf{IUL} as follows. With an attempt to prove the standard completeness
of \textbf{IUL}, we generalized Jenei and Montagna's method for
\textbf{IMTL} in [22], but our effort was only partially successful. It seems
that the subtle reason why it does not work for \textbf{UL }and \textbf{IUL}
is the failure of FMP of these systems [23].  Jenei [14] constructed several classes
of involutive FL$ _{e}$-algebras, as he said, in order to gain a better
insight into the algebraic semantic of the substructural logic \textbf{IUL}, and also to the long-standing open problem about its standard completeness.  Ciabattoni and Metcalfe [7] introduced the method of density
elimination by substitutions which is applicable to a general classes of
(first-order) hypersequent calculi but fails to the case of \textbf{GIUL}.

We reconsidered Metcalfe and Montagna's proof-theoretic method to
investigate the standard completeness of \textbf{IUL}, because they have proved the standard
completeness of \textbf{UL} by their method and we can't prove such a result
by the Jenei and Montagna's model-theoretic method. In order to prove the
density elimination for \textbf{GUL}, they prove that the following
generalized density rule $(\mathcal{D}_1 )$:

$$\cfrac{G_0 \equiv \{\Gamma _i,\lambda _i p\Rightarrow \Delta _i
\}_{i=1\cdots n} \vert \{\Sigma _k,(\mu _k \mbox{+}1)p\Rightarrow
p\}_{k=1\cdots o} \vert \{\Pi _j \Rightarrow p\}_{j=1\cdots m}
}{\mathcal{D}_1 (G_0 )\equiv \{\Gamma _i,\lambda _i \Pi _j \Rightarrow
\Delta _i \}_{i=1\cdots n}^{j=1\cdots m} \vert \{\Sigma _k,\mu _k \Pi _j
\Rightarrow t\}_{k=1\cdots o}^{j=1\cdots m} }(\mathcal{D}_1 )$$
is admissible for \textbf{GUL}, where they set two constraints to the form of $G_0 $:
(i) $n,m\geqslant 1$ and $\lambda _i {\kern 1pt}\geqslant 1$ for some $1\leqslant i\leqslant n$;
(ii) $p$ does not occur in $\Gamma _i $, $\Delta _i $, $\Pi _j $, $\Sigma _k $ for $i=1\cdots n$, $j=1\cdots m$, $k=1\cdots o$.

  We may regard  $ (\mathcal{D}_1)$  as a procedure whose input and output are the
premise and conclusion of  $ ( {\mathcal{D}_1} ) $,   respectively.
We denote the conclusion of  $ ( {\mathcal{D}_1} ) $  by $\mathcal{D}_1( {G_0 })$  when its premise is $G_0$.
Observe that Metcalfe and Montagna had succeeded to define the suitable
conclusion for almost arbitrary premise in $(\mathcal{D}_1 )$, but it seems
impossible for \textbf{GIUL} (See Section 3 for an example).  We then define the following generalized density rule
$(\mathcal{D}_0 )$ for $$ {\rm {\bf GL}} \in \{{\rm {\bf GUL}},{\rm {\bf
GIUL}},{\rm {\bf GMTL}},{\rm {\bf GIMTL}}\}$$ and prove its admissibility in Section
9.

\begin{theorem}[\textbf{Main theorem}]
 Let $n,m\geqslant 1$, $p$ does not occur in $G',\Gamma _i
,\Delta _i,\Pi _j $ or ${\kern 1pt}\Sigma _j $ for all $1\leqslant i\leqslant
n$,$1\leqslant j\leqslant m$.  Then the strong density rule
$$\cfrac{G_0 \equiv
G'\vert \left\{ {\Gamma _i,p\Rightarrow \Delta _i } \right\}_{i=1\cdots n}
\vert \left\{ {\Pi _j \Rightarrow p,\Sigma _j {\kern 1pt}}
\right\}_{j=1\cdots m} }{\mathcal{D}_0 \left( {G_0 } \right)\equiv G'\vert
\{\Gamma _i,\Pi _j \Rightarrow \Delta _i,\Sigma _j \}{\kern
1pt}_{i=1\cdots n;j=1\cdots m} }(\mathcal{D}_0 )$$
is admissible in ${\rm {\bf GL}}$.
\end{theorem}

In proving the admissibility of $(\mathcal{D}_1)$, Metcalfe and Montagna made
some restriction on the proof $\tau $ of $G_0 $, i.e., converted $\tau $
into an r-proof.  The reason why they need an r-proof is that
they set the constraint (i) to $G_0 $. We may imagine the restriction on
$\tau $ and the constraints to $G_0 $ as two pallets of a balance, i.e., one
is strong if another is weak and vice versa. Observe that we select the
weakest form of $G_0 $ in $(\mathcal{D}_0 )$ that guarantees the validity of $(D)$. Then it is natural that we need make the strongest
restriction on the proof $\tau $ of $G_0 $. But it seems extremely
difficult to follow such a way to prove the admissibility of $(\mathcal{D}_0
)$.

In order to overcome  such a difficulty,  we first of all choose Avron-style
hypersequent calculi as the underlying systems (See A.1).
Let $\tau $ be a cut-free proof of $G_0 $ in ${\rm {\bf GL}}$. Starting with $\tau $,  we construct a proof $\tau ^\ast $ of
$G\vert G^\ast $ in a restricted subsystem  $ {\rm {\bf GL}}_\Omega  $  of
 $ {\rm {\bf GL}} $   by a systematic novel manipulations in Section 4. Roughly speaking,  each sequent of $G$ is a copy of some sequent of $G_0 $,  and each sequent of $G^\ast $ is a copy of some contraction sequent in $\tau$.
  In Section 5,  we define the generalized density rule $(\mathcal{D})$
in ${\rm {\bf GL}}_\Omega $ and prove that it is admissible.

Now,  starting with $G\vert G^\ast $ and its proof $\tau ^\ast $,  we
construct a proof $\tau ^{\medstar}$ of $G^{\medstar}$  in ${\rm {\bf GL}}_\Omega $ such that each sequent of $G^{\medstar}$  is a copy of some sequent of $G$. Then $\vdash_{{\rm {\bf GL}}_\Omega } \mathcal{D}(G^{\medstar})$ by the admissibility of  $\mathcal{(D)}$.  Then $\vdash _{{\rm {\bf GL}} } \mathcal{D}_0
(G_0 )$ by Lemma 9.1.  Hence the density elimination theorem holds in ${\rm {\bf GL}}$.  Especially,  the standard completeness of \textbf{IUL} follows from Theorem 62 of [16].

$G^{\medstar}$  is constructed by eliminating $(pEC)$-sequents in $G\vert G^\ast $ one by one. In order to control the process,  we introduce the set $I=\{H_{i_1 }^c, \cdots, H_{i_m }^c \}$ of  $(pEC)$-nodes of $\tau^\ast $  and the set ${\rm {\bf I}}$ of the branches relative to $I$ and construct $G_{\rm {\bf I}}^{\medstar}$ such that $G_{\rm {\bf I}}^{\medstar}$ doesn't contain $(pEC)$-sequents lower than any node in $I$, i.e.,  $S_j^c \in G_{\rm {\bf I}}^{\medstar}$ implies $H_j^c \vert \vert H_i^c $ for all $H_i^c \in I$. The procedure is called the separation algorithm of
branches  in  which we  introduce another novel  manipulation and call it derivation-grafting operation in Section 7, 8.
\section{Preliminaries}

In this section, we recall the basic definitions and results involved, which are mainly from [16].
Substructural fuzzy logics are based on a countable propositional language
with formulas FOR built inductively as usual from a set of propositional
variables VAR, binary connectives  $  \odot, \to, \wedge,   \vee,  $  and
constants  $ \bot,  \top,  t,  f  $  with definable connective $\neg A\coloneq A \to f.$

\begin{definition}([1, 16]) A sequent is an ordered pair  $ (\Gamma,\Delta ) $
of finite multisets (possibly empty) of formulas, which we denote by  $ \Gamma
\Rightarrow \Delta$.   $ \Gamma  $  and  $ \Delta  $  are called the antecedent and
succedents, respectively, of the sequent and each formula in  $ \Gamma  $  and
 $ \Delta  $  is called a sequent-formula. A hypersequent $G$ is a finite multiset
of the form  $ \Gamma _1 \Rightarrow \Delta _1 \vert \cdots \vert \Gamma _n
\Rightarrow \Delta _n  $,   where each $ \Gamma _i \Rightarrow \Delta _i  $  is a
sequent and is called a component of $G$ for each  $ 1 \leqslant i \leqslant n$.
 If  $ \Delta _i  $  contains at
most one formula for  $ i = 1 \cdots n $,   then the hypersequent is
single-conclusion, otherwise it is multiple-conclusion.
\end{definition}

\begin{definition}
 Let  $ S $  be a sequent and  $ G=S_1 \vert \cdots \vert S_m$  a hypersequent. We say that  $ S
\in G $  if  $ S $  is  one of  $ S_1, \cdots,  S_m $.
\end{definition}

\begin{notation}
 Let  $ G_1  $  and  $ G_2  $  be two hypersequents.  We will assume from now on that all set terminology refers to multisets, adopting the conventions of writing
$\Gamma, \Delta$ for the multiset union of $\Gamma$ and $\Delta$, $A$ for the singleton multiset $\{A\}$, and $\lambda\Gamma$ for the multiset union of $\lambda$ copies of $\Gamma$ for $\lambda \in \mathbf{N}$.   By  $ G_1
\subseteq G_2  $  we mean that  $ S \in G_2  $  for all  $ S \in G_1  $  and the
multiplicity of  $ S $  in  $ G_1  $  is not~more~than~that of  $ S $  in  $ G_2$.  We
will use $ G_1 = G_2  $,    $ G_1 \bigcap G_2  $,   $ G_1 \bigcup G_2  $,    $ G_1 \backslash G_2  $
by their standard meaning for multisets by default  and we will declare when we use them for sets. We sometimes write  $ S_1 \vert \cdots \vert S_m$ and $G\vert  \overbrace {S\vert  \cdots \vert S}^{n \,\,\, copies}$ as  $ \{S_1, \cdots,S_m \} $,  $G\vert  S^n$(or $G\vert  \{S\}^n$), respectively.
\end{notation}

\begin{definition}([1])
 A hypersequent rule is an ordered pair consisting of
a sequence of hypersequents  $ G_1,\cdots,G_n  $
called the premises (upper hypersequents) of the rule, and a hypersequent  $ G $  called the
conclusion (lower  hypersequent), written by $ \cfrac{G_1 \cdots G_n
}{G}. $  If  $ n=0, $  then the rule has no premise and is called an initial sequent. The single-conclusion version of a rule adds the
restriction that both the premises and conclusion must be single-conclusion;
otherwise the rule is multiple-conclusion.
\end{definition}

\begin{definition} ([16])
\textbf{GUL} and  \textbf{GIUL} consist of the
single-conclusion and multiple-conclusion versions of the following
initial sequents and rules, respectively:

\noindent\textbf{Initial Sequents }

\[
\cfrac{}{A \Rightarrow A}
(ID)\quad
\cfrac{}{\Gamma \Rightarrow \top,\Delta  }
    ( \top _r )\quad
\cfrac{}{\Gamma, \bot \Rightarrow \Delta}( \bot _l )\quad
\cfrac{}{ \Rightarrow t} (t_r )\quad
\cfrac{}{f \Rightarrow }(f_l )
\]

\noindent\textbf{Structural Rules}

\[
\cfrac{G\vert \Gamma \Rightarrow A\vert \Gamma \Rightarrow A}{G\vert \Gamma
\Rightarrow A}(EC)\quad
\cfrac{G}{G\vert \Gamma \Rightarrow A}(EW)
\]
\[
 \cfrac{G_1\vert \Gamma _1,\Pi _1 \Rightarrow \Sigma _1,\Delta _1 \quad G_2\vert\Gamma _2
,\Pi _2 \Rightarrow \Sigma _2,\Delta _2 }{G_1\vert G_2\vert \Gamma _1,\Gamma _2
\Rightarrow \Delta _1,\Delta _2 \vert \Pi _1,\Pi _2 \Rightarrow \Sigma
_1,\Sigma_2}(COM)
\]
\noindent\textbf{Logical Rules}\\
\begin{minipage}[l]{0.49\linewidth}
\begin{align*}
\cfrac{G\vert \Gamma \Rightarrow \Delta }{G\vert \Gamma,t, \Rightarrow
\Delta } (t_l)&\\
\cfrac{G_1\vert \Gamma _1 \Rightarrow A,\Delta _1 \quad G_2\vert \Gamma _2,B
\Rightarrow \Delta _2 }{G_1\vert G_2\vert\Gamma _1,\Gamma _2,A \to B \Rightarrow
\Delta _1,\Delta _2 }(\to _l)&\\
\cfrac{G\vert \Gamma,A,B \Rightarrow \Delta }{G\vert \Gamma,A\odot  B
\Rightarrow \Delta }(\odot _l )&\\
\cfrac{G\vert \Gamma,A \Rightarrow \Delta }{G\vert \Gamma,A \wedge B
\Rightarrow \Delta }(\wedge_{lr})&\\
\cfrac{G_1\vert\Gamma \Rightarrow A,\Delta \quad G_2\vert \Gamma \Rightarrow
B,\Delta }{G_1\vert G_2\vert\Gamma \Rightarrow A \wedge B,\Delta }(\wedge_r)\\
\cfrac{G\vert \Gamma \Rightarrow B,\Delta }{G\vert \Gamma \Rightarrow A \vee
B,\Delta }(\vee_{rl})&\\
\end{align*}
\end{minipage}
\begin{minipage}[r]{0.49\linewidth}
\begin{align*}
\cfrac{G\vert \Gamma \Rightarrow \Delta }{G\vert \Gamma \Rightarrow f,\Delta}
 (\mathop f\nolimits_r)&\\
\cfrac{G\vert \Gamma,A \Rightarrow B,\Delta }{G\vert \Gamma \Rightarrow A
\to B,\Delta }( \to _r )&\\
\cfrac{G_1\vert \Gamma _1 \Rightarrow A,\Delta _1 \,\,\, G_2\vert\Gamma _2
\Rightarrow B,\Delta _2 }{G_1\vert G_2\vert \Gamma _1,\Gamma _2 \Rightarrow A\odot
B,\Delta _1,\Delta _2 }(\odot _r )&\\
\cfrac{G\vert \Gamma,B \Rightarrow \Delta }{G\vert \Gamma,A \wedge B
\Rightarrow \Delta }( \wedge_{ll})&\\
  \cfrac{G\vert \Gamma \Rightarrow A,\Delta }{G\vert
\Gamma \Rightarrow A \vee B,\Delta }(\vee _{rr})&\\
\cfrac{G_1\vert \Gamma,A \Rightarrow \Delta \quad G_2\vert \Gamma,B \Rightarrow
\Delta }{G_1\vert G_2\vert\Gamma,A \vee B \Rightarrow \Delta }(\vee_l)
\end{align*}
\end{minipage}

\noindent\textbf{Cut Rule }

\[\cfrac{G_1\vert\Gamma _1,A \Rightarrow \Delta _1 \quad
G_2\vert\Gamma _2 \Rightarrow A,\Delta _2 }{G_1\vert G_2\vert\Gamma _1,\Gamma _2
\Rightarrow \Delta _1,\Delta _2 }(CUT)\]

\end{definition}

\begin{definition} ([16])
\textbf{ GMTL} and \textbf{GIMTL} are \textbf{GUL} and \textbf{GIUL} plus the single conclusion and multiple-conclusion versions, respectively, of:
\[\cfrac{G\vert\Gamma\Rightarrow \Delta}{G\vert\Gamma, A\Rightarrow \Delta}(WL),
\,\,\, \cfrac{G\vert\Gamma\Rightarrow \Delta}{G\vert\Gamma\Rightarrow A, \Delta}(WR).
\]
\end{definition}

\begin{definition}
(i)  $ (I) \in \{(t_l ),(f_r ),( \to _r ),(\odot _l ),(
\wedge_{lr} ),( \wedge_{ll} ), ( \vee _{rr} ),( \vee _{rl}), (WL), (WR)\} $ and \\
$ (II) \in \{( \to_l ),(\odot  _r ),( \wedge_r ),( \vee_l ),
 (COM)\}$;

(ii) By  $ \cfrac{G'\vert S'  \quad G''\vert S'' }{G'\vert G''\vert  H'}(II) $  (or  $ \cfrac{G'\vert S'}{G'\vert H'}(I)) $  we denote
 an instance of a  two-premise  rule $ (II) $ (or one-premise rule
 $ (I) $) of  $ {\rm {\bf GL}} $,   where  $ S' $  and  $ S'' $  are its focus sequents and  $ H' $  is its principle sequent (for  $ (\to_l) $,    $ (\odot _r) $,   $ (\wedge_r) $  and  $ (\vee_l )) $  or hypersequent (for  $ (COM) $,    $ (\wedge_{rw} ) $  and  $ ( \vee_{lw} ) $,   see Definition 4.2).
\end{definition}

\begin{definition} ([16])  $ \mathbf{GL}^{\mathrm{D}} $  is \textbf{GL} extended with
the following density rule:

\[\cfrac{G\vert \Gamma _1,p \Rightarrow \Delta _1 \vert \Gamma _2 \Rightarrow
p,\Delta _2 }{ G\vert \Gamma _1,\Gamma _2 \Rightarrow \Delta _1
,\Delta _2 }(D)
\] where  $ p $  does not occur in  $ G, \Gamma _1, \Gamma _2, \Delta _1
 $  or  $ \Delta _2$.
\end{definition}
\begin{definition}([1]) A derivation  $ \tau  $  of a hypersequent  $ G $  from
hypersequents  $ G_1, \cdots,G_n  $  in a hypersequent calculus  $ {\rm
{\bf GL}} $  is a labeled tree with the root labeled by  $ G $,  leaves labeled initial sequents or some $ G_1, \cdots,G_n  $,   and for each
node labeled  $ G_0 ' $  with parent nodes labeled  $ G_1 ', \cdots,G_m
' $  (where possibly  $ m = 0) $,    $ \cfrac{G_1'  \cdots  G_m '}{G_0'} $  is an instance of a rule of  $ {\rm {\bf GL}}$.
\end{definition}

\begin{notation}
(i)$ \cfrac{\underline{G_1  \cdots G_n }}{G_{0}}\left\langle \tau
\right\rangle  $ denotes that $ \tau  $  is a derivation of  $G_{0}$  from  $ G_1, \cdots, G_n$;

(ii) Let  $ H $  be a hypersequent.  $ H \in \tau  $  denotes that  $ H $  is a node of
 $ \tau  $.  We call  $ H $  a leaf hypersequent if  $ H $  is a leaf of  $ \tau  $,   the root hypersequent if it is the root of  $ \tau $. $ \cfrac{G_1 ' \cdots G_m '}{G_0'}\in \tau  $  denotes that   $ G_0' \in \tau  $
and its parent nodes are  $ G_1', \cdots, G_m' $;

(iii) Let $ H \in \tau  $ then $ \tau (H) $ denotes  the subtree of  $ \tau  $  rooted at $ H $;

(iv)  $ \tau  $  determines a partial order  $  \leqslant _\tau  $  with the root  as the least
element.  $ H_1  \| H_2  $  denotes  $ H_1 \nleqslant_\tau H_2  $  and  $ H_2
\nleqslant_\tau H_1  $ for any $ H_1,H_2 \in \tau$.  By  $ H_1 = _\tau H_2
 $  we mean that  $ H_1  $  is the same node as  $ H_2  $  in  $ \tau  $.   We sometimes write  $  \leqslant _\tau  $   as  $\leqslant$;

(v) An inference of the form $\cfrac{G'\vert S^n}{G'\vert S} \in \tau $ is called the full external contraction and denoted by $(EC^{ \ast})$,  if  $n\geqslant2$,  $G'\vert S^n$ is not a lower hypersequent of an application of  $ (EC) $ whose contraction sequent is $S$,  and $G'\vert S$  not an upper one in $ \tau $.
\end{notation}

\begin{definition}
Let  $ \tau  $  be a derivation of  $ G $ and   $ H \in \tau$.
The thread  $ Th_\tau (H) $  of  $ \tau  $  at  $ H $  is a sequence  $ H_0, \cdots, H_n $  of node hypersequents of  $ \tau  $  such that  $ H_0 = _\tau H $,   $ H_n = _\tau
G $,    $ \cfrac{H_k }{H_{k + 1} } \in \tau  $  or there exists  $ G' \in \tau  $
such that  $ \cfrac{H_k  \quad G'}{H_{k + 1} } $  or   $ \cfrac{ G'\quad H_k }{H_{k + 1} } $ in $\tau$ for all  $ 0 \leqslant k \leqslant n - 1$.
\end{definition}

\begin{proposition}
Let $ H_1,H_2 \in \tau$.  Then \\
 (i) $ H_1 \leqslant  H_2  $  if and only if
  $  H_1 \in Th_\tau ( H_2)$;\\
  (ii) $ H_1  \| H_2  $ and $ H_1 \leqslant  H_3  $ imply $ H_2  \| H_3  $;\\
 (iii) $ H_1 \leqslant  H_3  $ and $ H_2 \leqslant  H_3  $ imply $ H_1 \nparallel H_2 $.

\end{proposition}

We need the following definition to give each node of   $\tau$  an identification number, which is used in Construction 6.1 to differentiate  sequents in a hypersequent in a proof.

\begin{definition}  ([A.5.2])
Let $H\in \tau$ and $Th(H)=(H_{0},\cdots,H_{n})$.   Let $b_{n}\coloneq 1$,
\[
b_k  \coloneq \left\{ \begin{array}{l}
 1 \qquad \qquad\mathrm{ if}\,\,\,\cfrac{G'\,\,\,H_{k}}{H_{k+1}}\in \tau,  \\
 0 \qquad \qquad \mathrm{if}  \,\,\,\cfrac{H_{k}}{H_{k+1}}\in \tau\,\,\, \mathrm{or} \,\,\, \cfrac{H_{k}\,\,\,G'}{H_{k+1}}\in \tau\\
 \end{array} \right.
\]
 for all $0\leqslant  k\leqslant  n-1$. Then  $\mathcal{P}(H)\coloneq \sum_{k=0}^{k=n}2^{k}b_{k}$ and  call it the position of  $H$ in $ \tau$.
\end{definition}

\begin{definition} A rule is admissible for a calculus \textbf{GL} if
whenever its premises are derivable in \textbf{GL}, then so is its
conclusion.
\end{definition}

\begin{lemma} $\mathrm{([16]) }$
Cut-elimination holds for  $ {\rm {\bf GL}} $,   i.e., proofs using  $ (CUT) $  can be transformed syntactically
into proofs not using  $ (CUT)$.
\end{lemma}

\section{Proof of the main  theorem: A computational example}

In this section, we present an example to illustrate the proof of Main theorem.

Let  $ G_0 \equiv  \Rightarrow p,B\vert B \Rightarrow p,\neg A\odot  \neg
A \vert p \Rightarrow C\vert C,p \Rightarrow A\odot  A$.  $ G_0
 $  is a theorem of  $ {\rm {\bf IUL}} $  and a cut-free proof  $\tau$  of  $ G_0
 $  is shown in Figure 1, where we use an additional rule  $
\cfrac{{\Gamma ,A \Rightarrow \Delta }}
{{\Gamma  \Rightarrow \neg A,\Delta }}(\neg _r )
$ for simplicity. Note that we denote  three applications of $(EC)$ in $\tau$ respectively  by  $(EC)_1,(EC)_2,(EC)_3$  and three $(\odot  _r )$
 by $(\odot  _r )_1,(\odot  _r )_2$ and $(\odot  _r )_3$.
 \[
 \boxed {
\cfrac{\cfrac{\cfrac{\cfrac{p \Rightarrow p \quad A\Rightarrow A}{A \Rightarrow p \vert p \Rightarrow A}(COM)  \cfrac{p \Rightarrow p\quad A \Rightarrow A}{A \Rightarrow p \vert p \Rightarrow
A}(COM)}{ A \Rightarrow p \vert  A \Rightarrow p \vert   p,p \Rightarrow A\odot  A}(\odot  _r )_1}{A \Rightarrow p \vert p,p\Rightarrow A\odot  A}(EC)_1}{ \Rightarrow p,\neg A \vert p,p \Rightarrow A\odot
A}(\neg _{r})}
\]

\[\boxed {
\cfrac{\cfrac{\cfrac{\cfrac{p \Rightarrow p \quad A\Rightarrow A}{A \Rightarrow p \vert p \Rightarrow A}(COM)  \cfrac{p \Rightarrow p\quad A \Rightarrow A}{A \Rightarrow p \vert p \Rightarrow
A}(COM)}{A \Rightarrow p \vert A \Rightarrow p \vert p,p \Rightarrow A\odot  A}(\odot _r )_2}{A \Rightarrow p \vert p,p\Rightarrow A\odot  A}(EC)_2}{ \Rightarrow p,\neg A \vert p,p \Rightarrow A\odot
A}(\neg _{r})}
\]
(continued)
\small
\[
\boxed {
\cfrac{\cfrac{}{C \Rightarrow C}\cfrac{\cfrac{}{B \Rightarrow B}\cfrac{\cfrac{
\begin{array}{*{20}c}
   \ddots  \vdots  {\mathinner{\mkern0.5mu\raise0.5pt\hbox{.}\mkern0.0mu
 \raise1.9pt\hbox{.}\mkern0.0mu\raise3.9pt\hbox{.}}}\\
 \Rightarrow p,\neg A \vert p,p
\Rightarrow A\odot  A \\
 \end{array}
\quad
\begin{array}{*{20}c}
   \ddots  \vdots  {\mathinner{\mkern0.5mu\raise0.5pt\hbox{.}\mkern0.0mu
 \raise1.9pt\hbox{.}\mkern0.0mu\raise3.9pt\hbox{.}}}\\
 \Rightarrow p,\neg A \vert p,p
\Rightarrow A\odot  A \\
 \end{array} }{
H_{\times\times}\equiv \Rightarrow p,p,\neg A\odot  \neg A \vert p,p \Rightarrow A\odot
A \vert p,p \Rightarrow A\odot  A}(\odot _r )_3}{ H_{\times}\equiv \Rightarrow p,p,\neg
A\odot  \neg A \vert p,p \Rightarrow A\odot  A}(EC)_3 }{
\Rightarrow p,B\vert B \Rightarrow p,\neg A\odot  \neg A \vert p,p
\Rightarrow A\odot  A}(COM)}{\Rightarrow
p,B\vert B \Rightarrow p,\neg A\odot  \neg A \vert p \Rightarrow C\vert C,p \Rightarrow A\odot  A}(COM)}
\]
\normalsize
\begin{center}
 \footnotesize FIGURE 1\normalsize\quad A proof $\tau$ of $G_0$
\end{center}

By applying (D) to free combinations of all sequents in $\Rightarrow
p,B\vert B\Rightarrow p,\neg A\odot \neg A$ and in $p\Rightarrow C\vert
C,p\Rightarrow A\odot A$, we get that $H_0 \equiv  \Rightarrow B,C\vert
C\Rightarrow A\odot A,B\vert B\Rightarrow C,\neg A\odot \neg A\vert
C,B\Rightarrow A\odot A,\neg A\odot \neg A$.  $H_0 $ is a theorem of
${\rm {\bf IUL}}$ and a cut-free proof  $\rho$  of $H_0 $ is shown in
Figure 2.  It supports the validity of  the generalized density rule $( \mathcal{D}_0)$ in Section 1,  as an  instance of $( \mathcal{D}_0)$.

\[
\boxed {\cfrac{\cfrac{}{C\Rightarrow C}\cfrac{\cfrac{}{B\Rightarrow
B}\cfrac{\cfrac{\cfrac{A\Rightarrow A{\kern 1pt}{\kern 1pt}{\kern 1pt}{\kern
1pt}{\kern 1pt}{\kern 1pt}{\kern 1pt}{\kern 1pt}{\kern 1pt}{\kern 1pt}{\kern
1pt}{\kern 1pt}{\kern 1pt}{\kern 1pt}{\kern 1pt}{\kern 1pt}{\kern 1pt}{\kern
1pt}{\kern 1pt}{\kern 1pt}{\kern 1pt}{\kern 1pt}{\kern 1pt}{\kern 1pt}{\kern
1pt}A\Rightarrow A}{A,A\Rightarrow A\odot A}{\kern 1pt}}{A\Rightarrow \neg
A,A\odot A}{\kern 1pt}{\kern 1pt}{\kern 1pt}{\kern 1pt}{\kern 1pt}{\kern
1pt}{\kern 1pt}{\kern 1pt}{\kern 1pt}{\kern 1pt}{\kern
1pt}\cfrac{\cfrac{A\Rightarrow A{\kern 1pt}{\kern 1pt}{\kern 1pt}{\kern
1pt}{\kern 1pt}{\kern 1pt}{\kern 1pt}{\kern 1pt}{\kern 1pt}{\kern 1pt}{\kern
1pt}{\kern 1pt}{\kern 1pt}{\kern 1pt}{\kern 1pt}{\kern 1pt}{\kern 1pt}{\kern
1pt}{\kern 1pt}{\kern 1pt}{\kern 1pt}{\kern 1pt}{\kern 1pt}{\kern 1pt}{\kern
1pt}A\Rightarrow A}{A,A\Rightarrow A\odot A}{\kern 1pt}}{A\Rightarrow \neg
A,A\odot A}}{A,A\Rightarrow \neg A\odot \neg A,A\odot A,A\odot
A}}{A,B\Rightarrow A\odot A,\neg A\odot \neg A\vert A\Rightarrow A\odot
A,B}}{H_1 \equiv  A\Rightarrow C\vert A,B\Rightarrow A\odot A,\neg
A\odot \neg A\vert C\Rightarrow A\odot A,B}}
\]
\scriptsize
\[
\boxed {\cfrac{\cfrac{\cfrac{}{C\Rightarrow C}\cfrac{\cfrac{}{C\Rightarrow
C}\cfrac{\cfrac{}{B\Rightarrow B}\cfrac{\cfrac{\cfrac{A\Rightarrow A{\kern
1pt}{\kern 1pt}{\kern 1pt}{\kern 1pt}{\kern 1pt}{\kern 1pt}{\kern 1pt}{\kern
1pt}{\kern 1pt}{\kern 1pt}{\kern 1pt}{\kern 1pt}{\kern 1pt}{\kern 1pt}{\kern
1pt}{\kern 1pt}{\kern 1pt}{\kern 1pt}{\kern 1pt}{\kern 1pt}{\kern 1pt}{\kern
1pt}{\kern 1pt}{\kern 1pt}{\kern 1pt}A\Rightarrow A}{A,A\Rightarrow A\odot
A}{\kern 1pt}}{A\Rightarrow \neg A,A\odot A}{\kern 1pt}{\kern 1pt}{\kern
1pt}{\kern 1pt}{\kern 1pt}{\kern 1pt}{\kern 1pt}{\kern 1pt}{\kern 1pt}{\kern
1pt}{\kern 1pt}\cfrac{\begin{array}{*{20}c}
   \ddots  \vdots  {\mathinner{\mkern0.5mu\raise0.5pt\hbox{.}\mkern0.0mu
 \raise1.5pt\hbox{.}\mkern0.0mu\raise2.9pt\hbox{.}}}  \\
H_1 =A\Rightarrow C\vert A,B\Rightarrow A\odot A,\neg
A\odot \neg A\vert C\Rightarrow A\odot A,B \\
 \end{array}
}{\Rightarrow \neg A,C\vert
A,B\Rightarrow A\odot A,\neg A\odot \neg A\vert C\Rightarrow A\odot
A,B}}{A\Rightarrow \neg A\odot \neg A,A\odot A,C\vert A,B\Rightarrow A\odot
A,\neg A\odot \neg A\vert C\Rightarrow A\odot A,B}}{\Rightarrow B,C\vert
A,B\Rightarrow \neg A\odot \neg A,A\odot A\vert A,B\Rightarrow A\odot A,\neg
A\odot \neg A\vert C\Rightarrow A\odot A,B}}{A\Rightarrow C\vert \Rightarrow
B,C\vert C,B\Rightarrow \neg A\odot \neg A,A\odot A\vert A,B\Rightarrow
A\odot A,\neg A\odot \neg A\vert C\Rightarrow A\odot A,B}}{A\Rightarrow
C\vert A\Rightarrow C\vert \Rightarrow B,C\vert C,B\Rightarrow \neg A\odot
\neg A,A\odot A\vert C,B\Rightarrow A\odot A,\neg A\odot \neg A\vert
C\Rightarrow A\odot A,B}}{H_2 \equiv  A\Rightarrow C\vert \Rightarrow
B,C\vert C,B\Rightarrow A\odot A,\neg A\odot \neg A\vert C\Rightarrow A\odot
A,B}}
\]

\[
\boxed {\cfrac{\cfrac{\cfrac{}{C\Rightarrow C}\cfrac{\cfrac{}{C\Rightarrow
C}\cfrac{\cfrac{}{B\Rightarrow B}\cfrac{\cfrac{\cfrac{A\Rightarrow A{\kern
1pt}{\kern 1pt}{\kern 1pt}{\kern 1pt}{\kern 1pt}{\kern 1pt}{\kern 1pt}{\kern
1pt}{\kern 1pt}{\kern 1pt}{\kern 1pt}{\kern 1pt}{\kern 1pt}{\kern 1pt}{\kern
1pt}{\kern 1pt}{\kern 1pt}{\kern 1pt}{\kern 1pt}{\kern 1pt}{\kern 1pt}{\kern
1pt}{\kern 1pt}{\kern 1pt}{\kern 1pt}A\Rightarrow A}{A,A\Rightarrow A\odot
A}{\kern 1pt}}{A\Rightarrow \neg A,A\odot A}{\kern 1pt}{\kern 1pt}{\kern
1pt}{\kern 1pt}{\kern 1pt}{\kern 1pt}{\kern 1pt}{\kern 1pt}{\kern 1pt}{\kern
1pt}{\kern 1pt}\cfrac{\begin{array}{*{20}c}
   \ddots  \vdots  {\mathinner{\mkern0.5mu\raise0.5pt\hbox{.}\mkern0.0mu
 \raise1.5pt\hbox{.}\mkern0.0mu\raise2.9pt\hbox{.}}}\\
H_1 =A\Rightarrow C\vert A,B\Rightarrow A\odot A,\neg
A\odot \neg A\vert C\Rightarrow A\odot A,B \\
 \end{array} }{A\Rightarrow C\vert B\Rightarrow
\neg A,A\odot A,\neg A\odot \neg A\vert C\Rightarrow A\odot
A,B}}{A\Rightarrow C\vert A,B\Rightarrow \neg A\odot \neg A,A\odot A,A\odot
A,\neg A\odot \neg A\vert C\Rightarrow A\odot A,B}}{B,B\Rightarrow \neg
A\odot \neg A,A\odot A,\neg A\odot \neg A\vert A\Rightarrow C\vert
A\Rightarrow A\odot A,B\vert C\Rightarrow A\odot A,B}}{A\Rightarrow C\vert
B,B\Rightarrow \neg A\odot \neg A,A\odot A,\neg A\odot \neg A\vert
A\Rightarrow C\vert C\Rightarrow A\odot A,B\vert C\Rightarrow A\odot
A,B}}{{\kern 1pt}A\Rightarrow C\vert C,B\Rightarrow A\odot A,\neg A\odot
\neg A\vert A\Rightarrow C\vert B\Rightarrow C,\neg A\odot \neg A\vert
C\Rightarrow A\odot A,B\vert C\Rightarrow A\odot A,B}}{H_3 \equiv
A\Rightarrow C\vert C\Rightarrow A\odot A,B\vert B\Rightarrow C,\neg A\odot
\neg A\vert C,B\Rightarrow A\odot A,\neg A\odot \neg A}}
\]

\[
\boxed {\cfrac{\cfrac{\cfrac{}{B\Rightarrow B}\cfrac{{\begin{array}{l}
 \left( {\cfrac{\begin{array}{*{20}c}
   \ddots  \vdots  {\mathinner{\mkern0.5mu\raise0.5pt\hbox{.}\mkern0.0mu
 \raise1.9pt\hbox{.}\mkern0.0mu\raise3.9pt\hbox{.}}}\\
 H_2 =A\Rightarrow C\vert \Rightarrow B,C\vert C,B\Rightarrow
A\odot A,\neg A\odot \neg A\vert C\Rightarrow A\odot A,B \\
 \end{array}
}{\Rightarrow \neg
A,C\vert \Rightarrow B,C\vert C,B\Rightarrow A\odot A,\neg A\odot \neg
A\vert C\Rightarrow A\odot A,B}} \right) \\
 \left( {\cfrac{\begin{array}{*{20}c}
   \ddots  \vdots  {\mathinner{\mkern0.5mu\raise0.5pt\hbox{.}\mkern0.0mu
 \raise1.9pt\hbox{.}\mkern0.0mu\raise3.9pt\hbox{.}}}\\
 H_3 =A\Rightarrow C\vert C\Rightarrow A\odot A,B\vert
B\Rightarrow C,\neg A\odot \neg A\vert C,B\Rightarrow A\odot A,\neg A\odot
\neg A \\
 \end{array}
 }{\Rightarrow \neg A,C\vert C\Rightarrow A\odot A,B\vert B\Rightarrow
C,\neg A\odot \neg A\vert C,B\Rightarrow A\odot A,\neg A\odot \neg A}}
\right) \\
 \end{array}}}{{\begin{array}{l}
 \Rightarrow \neg A\odot \neg A,C,C\vert \Rightarrow B,C\vert C,B\Rightarrow
A\odot A,\neg A\odot \neg A\vert C\Rightarrow A\odot A,B\vert \\
 C\Rightarrow A\odot A,B\vert B\Rightarrow C,\neg A\odot \neg A\vert
C,B\Rightarrow A\odot A,\neg A\odot \neg A \\
 \end{array}}}(\odot_r)}{{\begin{array}{l}
 B\Rightarrow \neg A\odot \neg A,C\vert \Rightarrow B,C\vert \Rightarrow
B,C\vert C,B\Rightarrow A\odot A,\neg A\odot \neg A\vert C\Rightarrow A\odot
A,B\vert \\
 C\Rightarrow A\odot A,B\vert B\Rightarrow C,\neg A\odot \neg A\vert
C,B\Rightarrow A\odot A,\neg A\odot \neg A \\
 \end{array}}}(COM)}{H_0 =\Rightarrow B,C\vert C\Rightarrow A\odot A,B\vert
B\Rightarrow C,\neg A\odot \neg A\vert C,B\Rightarrow A\odot A,\neg A\odot
\neg A}(EC^{\ast})}{\kern 1pt}
\]
\normalsize
\begin{center}
Figure 2 a  proof  $\rho$  of $H_0 $
\end{center}

Our task is to construct $\rho$, starting from $\tau $. The tree
structure of $\rho$  is more complicated than that of $\tau $.  Compared
with \textbf{UL}, \textbf{MTL} and\textbf{ IMTL}, there is no one-to-one
correspondence between nodes in $\tau $ and $\rho$.

Following the method given by G. Metcalfe and F. Montagna, we need to define
a generalized density rule for \textbf{IUL}.  We denote such an expected unknown rule by $(\mathcal{D}_x )$ for convenience.  Then  $\mathcal{D}_x (H)$  must be definable for all $H\in \tau $.  Naturally, $\mathcal{D}_x (p\Rightarrow p)=\Rightarrow t$;$\mathcal{D}_x
(A\Rightarrow p{\kern 1pt}\vert p\Rightarrow A)=A\Rightarrow
A$;$\mathcal{D}_x (\Rightarrow p,\neg A{\kern 1pt}\vert p,p\Rightarrow
A\odot A)=\Rightarrow \neg A,\neg A,A\odot A$;
$\mathcal{D}_x (\Rightarrow p,B\vert B\Rightarrow p,\neg A\odot \neg A{\kern
1pt}\vert p,p\Rightarrow A\odot A)=\Rightarrow B,B,A\odot A{\kern 1pt}\vert
B,B\Rightarrow A\odot A,\neg A\odot \neg A,\neg A\odot \neg A\vert
B\Rightarrow A\odot A,B,\neg A\odot \neg A;
\mathcal{D}_x (G_0 )=\mathcal{D}_0 (G_0 )=H_0.
$

However, we couldn't find  a suitable way to define $\mathcal{D}_x (
H_{\times\times})$ and $\mathcal{D}_x (H_{\times})$  for
$H_{\times\times}$¡¡and  $H_{\times}$  in $\tau$, see Figure 1.  This is the biggest difficulty we encounter  in the case of  $ {\rm {\bf IUL}}$ such that it is hard to prove density elimination for \textbf{IUL}.   A possible way is to define  $ \mathcal{D}_x (
\Rightarrow p,p,\neg A\odot  \neg A
\vert p,p \Rightarrow A\odot  A ) $  as  $  \Rightarrow t, A\odot  A,\neg A\odot  \neg A$.  Unfortunately,
it  is not a theorem of  $ {\rm {\bf IUL}}$.

Notice that  two upper hypersequents $\Rightarrow p,\neg A \vert p,p
\Rightarrow A\odot  A$ of $ (\odot _r )_3$ are permissible inputs of
$(\mathcal{D}_x)$.    Why is $ H_{\times\times} $   an  invalid input?  One reason is  that,  two applications $(EC)_1$  and $(EC)_2$  cut off  two sequents $A\Rightarrow p$ such that two $p's$ disappear in all nodes lower than upper hypersequent of $(EC)_1$  or  $(EC)_2$,
including  $H_{\times\times}$.  These make  occurrences of  $p's$ to be incomplete in $ H_{\times\times} $.  We  then perform the following operation in order to get complete occurrences of  $p's$  in $ H_{\times\times} $.

{\bf Step 1 (Preprocessing of $\tau$).}
Firstly,  we replace $H$  with $ H\vert S'$  for all $\cfrac{G'\vert S'\vert S'}{G'\vert S'}(EC)_k\in \tau$,  $ H\leqslant  G'\vert S'$
then replace $\cfrac{G'\vert S'\vert S'}{G'\vert S'\vert S'}(EC)_k$  with
$ G'\vert S'\vert S'$ for all $k=1,2,3$.  Then we construct a proof without
$(EC)$,  which we denote by  $\tau _1$, as shown in Figure 3.  We call  such
 manipulations  sequent-inserting operations,  which  eliminate applications of $(EC)$ in $\tau$.
\[
\boxed{ \cfrac{C \Rightarrow C\cfrac{B \Rightarrow B\cfrac{
  \cfrac{\cfrac{\cfrac{p \Rightarrow p  \,\,A \Rightarrow A}{A \Rightarrow p  \vert
p \Rightarrow A} \,\,
\cfrac{p \Rightarrow p\,\,
     A \Rightarrow
A}{A \Rightarrow p  \vert p \Rightarrow A}}{A \Rightarrow p
 \vert A \Rightarrow p  \vert p,p \Rightarrow
A\odot  A}}{A \Rightarrow p
\vert \Rightarrow p,\neg A \vert p,p \Rightarrow A\odot
A} \,\,   \cfrac{\cfrac{\cfrac{p\Rightarrow p\,\,A \Rightarrow A}
{A\Rightarrow p  \vert p \Rightarrow A}\,\,
    \cfrac{p \Rightarrow p\,\,A \Rightarrow A}
    {A \Rightarrow p \vert p \Rightarrow A} }
    {A \Rightarrow p \vert A \Rightarrow p  \vert p,p \Rightarrow A\odot A}}{A \Rightarrow
p  \vert \Rightarrow p,\neg A \vert p,p
\Rightarrow A\odot  A}}{ H_{\times\times}'\equiv A \Rightarrow p  \vert
\Rightarrow p,p,\neg A\odot  \neg A \vert p,p
\Rightarrow A\odot  A \vert A \Rightarrow p  \vert p,p \Rightarrow A\odot  A} }{A \Rightarrow p  \vert
\Rightarrow p,B\vert B \Rightarrow p,\neg A\odot  \neg A
\vert p,p \Rightarrow A\odot  A \vert A \Rightarrow p\vert p,p \Rightarrow A\odot  A}
}{A \Rightarrow p  \vert \Rightarrow p
,B\vert B \Rightarrow p,\neg A\odot  \neg A \vert
p\Rightarrow C\vert C,p \Rightarrow A\odot  A \vert A \Rightarrow
p  \vert p,p \Rightarrow A\odot  A}} \]
\normalsize
\begin{center}
 \footnotesize  FIGURE 3\normalsize\quad A proof  $\tau _1$
\end{center}

 However,   we also can't define $ \mathcal{D}_x(H_{\times\times}')$  for $H_{\times\times}'  \in \tau _1$  in  that $\Rightarrow p,p,\neg A\odot  \neg A \vert p,p
\Rightarrow A\odot  A\subseteq H_{\times\times}'$.  The reason is that the origins of  $p's$  in  $H_{\times\times}'$ are indistinguishable if  we regard all leaves in the form $p\Rightarrow p$ as the origins of   $p's$ which occur in the inner node.   For example,  we don't know which $p$  comes from the left subtree of  $\tau _1(H_{\times\times}')$  and which from the right subtree in  two occurrences of $p's$  in $\Rightarrow p,p,\neg A\odot  \neg A\in H_{\times\times}'$.  We  then perform the following operation in order to make all occurrences of  $p's$  in $ H_{\times\times}' $ distinguishable.

We assign the unique identification number to each leaf  in the form   $ p\Rightarrow p\in\tau _1 $ and transfer these identification numbers  from  leaves to the  root,  as shown in Figure 4.  We
denote the proof of  $G\vert G^{\ast} $ resulting from this step by  $ \tau ^ * $, where $G \equiv  \Rightarrow p_2
,B\vert B \Rightarrow p_4,\neg A\odot  \neg A \vert
p_1\Rightarrow C\vert C,p_2 \Rightarrow A\odot  A $  in which each sequent is a copy of some sequent in $G_0$ and
$G^{\ast} \equiv  A \Rightarrow p_1  \vert A \Rightarrow p_3  \vert p_3,p_4 \Rightarrow A\odot  A$   in which each sequent is a copy of some external contraction  sequent in  $(EC)$-node of $\tau$.  We call  such manipulations
eigenvariable-labeling operations, which make us  to trace eigenvariables in $\tau$.
$$\boxed { \cfrac{C \Rightarrow C\cfrac{B \Rightarrow B\cfrac{
  \cfrac{\cfrac{\cfrac{p_1 \Rightarrow p_1  \,\,A \Rightarrow A}{A \Rightarrow p_1  \vert
p_1 \Rightarrow A} \,\,
\cfrac{p_2 \Rightarrow p_2\,\,
     A \Rightarrow
A}{A \Rightarrow p_2  \vert p_2 \Rightarrow A}}
{H_{1}^{c}\equiv  A \Rightarrow p_1 \vert A \Rightarrow p_2  \vert p_1,p_2 \Rightarrow
A\odot  A}}{A \Rightarrow p_1
\vert \Rightarrow p_2,\neg A \vert p_1,p_2 \Rightarrow A\odot
A}    \cfrac{\cfrac{\cfrac{p_3\Rightarrow p _3\,\,A \Rightarrow A}
{A\Rightarrow p_3  \vert p_3 \Rightarrow A}
    \cfrac{p_4 \Rightarrow p_4\,\,A \Rightarrow A}
    {A \Rightarrow p_4 \vert p_4 \Rightarrow A} }
    {H_{2}^{c}\equiv  A \Rightarrow p_3 \vert A \Rightarrow p_4  \vert p_3,p_4 \Rightarrow A\odot A}}{A \Rightarrow
p_3  \vert \Rightarrow p_4,\neg A \vert p_3,p_4
\Rightarrow A\odot  A}}{A \Rightarrow p_1  \vert
\Rightarrow p_2,p_4,\neg A\odot  \neg A \vert p_1,p_2
\Rightarrow A\odot  A \vert A \Rightarrow p_3  \vert p_3
,p_4 \Rightarrow A\odot  A} }{H_{3}^{c}\equiv  A \Rightarrow p_1  \vert
\Rightarrow p_2,B\vert B \Rightarrow p_4,\neg A\odot  \neg A
\vert p_1,p_2 \Rightarrow A\odot  A \vert A \Rightarrow p_3\vert p_3,p_4 \Rightarrow A\odot  A}
}{A \Rightarrow p_1  \vert \Rightarrow p_2
,B\vert B \Rightarrow p_4,\neg A\odot  \neg A \vert
p_1\Rightarrow C\vert C,p_2 \Rightarrow A\odot  A \vert A \Rightarrow p_3  \vert p_3,p_4 \Rightarrow A\odot  A}} $$
\normalsize
\begin{center}
 \footnotesize  FIGURE 4\normalsize\quad A proof $\tau ^\ast$ of $G \vert G^{\ast}$
\end{center}

Then all occurrences of $ p $  in  $ \tau^\ast  $ are distinguishable and we regard them as
distinct eigenvariables (See Definition 4.16 (i)). Firstly,  by selecting $p_1$ as the eigenvariable and applying $(D)$ to $ G \vert G^{\ast}$, we get
$$G'\equiv  A \Rightarrow C  \vert \Rightarrow p_2
,B\vert B \Rightarrow p_4,\neg A\odot  \neg A \vert
C,p_2 \Rightarrow A\odot  A \vert A \Rightarrow
p_3  \vert p_3,p_4 \Rightarrow A\odot  A.$$
Secondly, by selecting $p_2$ and applying $(D)$ to $ G' $, we get
$$G''\equiv  A \Rightarrow C  \vert B \Rightarrow p_4,\neg A\odot  \neg A \vert
 C \Rightarrow B,A\odot  A \vert A \Rightarrow
p_3  \vert p_3,p_4 \Rightarrow A\odot  A.$$ Repeatedly, we get
$$ G''''\equiv A \Rightarrow C\vert A,B \Rightarrow A\odot  A,\neg A\odot  \neg A\vert C \Rightarrow A\odot
A,B.$$ We define such iterative applications of $(D)$ as $\mathcal{D}$-rule (See Definition 5.4).  Lemma 5.8  shows that $\vdash _{{\rm {\bf GIUL}}} \mathcal{D}( {G} \vert G^{\ast}) $ if $  \vdash _{\rm {\bf GIUL}}G\vert G^{\ast}$. Then we obtain  $  \vdash _{{\rm {\bf GIUL}}} \mathcal{D}
({G \vert G^{\ast}}) $,   i.e., $  \vdash _{{\rm {\bf GIUL}}} G''''$.

A miracle happens here!  The difficulty that we encountered in \textbf{GIUL} is overcome by  converting $H_{\times\times}'=A \Rightarrow p  \vert
\Rightarrow p,p,\neg A\odot  \neg A \vert p,p
\Rightarrow A\odot  A \vert A \Rightarrow p  \vert p
,p \Rightarrow A\odot  A$ into $A \Rightarrow p_1  \vert
\Rightarrow p_2,p_4,\neg A\odot  \neg A \vert p_1,p_2
\Rightarrow A\odot  A \vert A \Rightarrow p_3  \vert p_3
,p_4 \Rightarrow A\odot  A$ and  using $(\mathcal{D})$  to replace $(\mathcal{D}_{x})$.

Why do we assign the unique identification number to each  $ p\Rightarrow p\in\tau _1 $?
We would return back to the same situation as that of $\tau _1 $ if we assign the  same indices  to all  $ p\Rightarrow p\in\tau _1 $ or,
replace $p_3\Rightarrow p_3$ and $p_4\Rightarrow p_4$  by $p_2\Rightarrow p_2$ in $\tau^\ast$.

Note that $ \mathcal{D}({G \vert G^{\ast}})=H_1 $.  So we have built up a  one-one correspondence between
the proof $\tau^{\ast}$ of $G \vert G^{\ast}$ and that of $H_1$.
Observe that  each sequent in  $ G^{\ast}$  is not a copy of any sequent in  $ G_0$.  In the following  steps,  we work on eliminating these sequents in $ G^{\ast}$.

\textbf{Step 2  (Extraction of Elimination Rules).}
We select $A \Rightarrow p_2$   as the focus sequent in $H_{1}^{c}$  in $\tau^{\ast}$ and keep  $A \Rightarrow p_1$ unchanged from  $H_{1}^{c}$  downward to $G \vert G^{\ast}$ (See Figure 4).
    So we extract  a  derivation  from  $A \Rightarrow p_2$  by
pruning some sequents (or hypersequents) in  $ \tau ^ *  $,    which we denote by  $ \tau _{H_{1}^{c}:A \Rightarrow p_2}^ {*} $, as shown in Figure 5.

\[\boxed {
\cfrac{B \Rightarrow B\cfrac{
\cfrac{A \Rightarrow p_2 }{ \Rightarrow p_2
,\neg A}
\cfrac{\cfrac{\cfrac{p_3 \Rightarrow p _3  \,\,\,\,    A
\Rightarrow A}{A \Rightarrow p_3  \vert p_3 \Rightarrow A}
\,\,\,\, \cfrac{p_4 \Rightarrow p_4\,\,\,\,
 A \Rightarrow A}{A \Rightarrow p_4  \vert p_4
\Rightarrow A} }{A \Rightarrow p_3  \vert A
\Rightarrow p_4  \vert p_3,p_4 \Rightarrow A\odot  A}
}{A \Rightarrow p_3
\vert \Rightarrow p_4,\neg A \vert p_3,p_4 \Rightarrow A\odot
A} }{ \Rightarrow p_2,p_4,\neg A\odot  \neg A{\kern
1pt} \vert A \Rightarrow p_3  \vert p_3,p_4 \Rightarrow A\odot
A}  }{ \Rightarrow p_2,B\vert B \Rightarrow p_4,\neg
A\odot  \neg A \vert A \Rightarrow p_3  \vert p_3,p_4
\Rightarrow A\odot  A}}
\]

\begin{center}
 \footnotesize  FIGURE 5\normalsize\quad A derivation $ \tau _{H_{1}^{c}:A \Rightarrow p_2}^ {*}  $ from  $ A \Rightarrow p_2  $
\end{center}
A derivation  $ \tau _{H_{1}^{c}:A \Rightarrow p_1}^ {*} $  from  $ A \Rightarrow p_1
 $  is constructed by replacing  $ p_2  $  with  $ p_1  $, $ p_3 $  with  $ p_5  $  and $ p_4 $  with  $ p_6 $ in  $ \tau _{H_{1}^{c}:A \Rightarrow p_2 }^ {*}$, as shown in Figure 6.  Notice that we assign new identification numbers to new
occurrences of  $ p $ in  $ \tau _{H_{1}^{c}:A\Rightarrow p_1}^ {*}$.
\[\boxed {
\cfrac{B \Rightarrow B\cfrac{
\cfrac{A \Rightarrow p_1}{ \Rightarrow p_1
,\neg A}
\cfrac{\cfrac{\cfrac{p_5 \Rightarrow p _5  \,\,\,\,    A
\Rightarrow A}{A \Rightarrow p_5  \vert p_5 \Rightarrow A}
\,\,\,\, \cfrac{p_6 \Rightarrow p_6\,\,\,\,
 A \Rightarrow A}{A \Rightarrow p_6  \vert p_6
\Rightarrow A} }{A \Rightarrow p_5  \vert A
\Rightarrow p_6  \vert p_5,p_6 \Rightarrow A\odot  A}
}{A \Rightarrow p_5
\vert \Rightarrow p_6,\neg A \vert p_5,p_6 \Rightarrow A\odot
A} }{ \Rightarrow p_1,p_6,\neg A\odot  \neg A{\kern
1pt} \vert A \Rightarrow p_5  \vert p_5,p_6 \Rightarrow A\odot
A}  }{ \Rightarrow p_1,B\vert B \Rightarrow p_6,\neg
A\odot  \neg A \vert A \Rightarrow p_5  \vert p_5,p_6
\Rightarrow A\odot  A}}
\]

\begin{center}
 \footnotesize  FIGURE 6\normalsize \quad A derivation $ \tau _{H_{1}^{c}:A \Rightarrow p_1}^ {*}  $ from  $ A \Rightarrow p_1  $
\end{center}

Next,  we apply $ \tau _{H_{1}^{c}:A \Rightarrow p_1}^{*}$  to  $ A \Rightarrow p_1  $  in  $G\vert G^{\ast}$.  Then we construct  a proof $ \tau _{{H_{1}^{c}:G\vert G^{\ast}}}^{\medstar(1)}  $,  as
shown in Figure 7,  where $G'\equiv   G\vert G^{\ast} \backslash \{A \Rightarrow p_1\}$.

\[\boxed {
\cfrac{B \Rightarrow B\cfrac{
\cfrac{G'\vert A \Rightarrow p_1}{G'\vert
\Rightarrow p_1,\neg A}\quad
 \cfrac{\cfrac{\cfrac{p_5 \Rightarrow p _5\,\,\,\,
 A \Rightarrow A}{A \Rightarrow p_5  \vert p_5
\Rightarrow A}\,\,\,\,
    \cfrac{p_6 \Rightarrow p_6\,\,\,\,
   A \Rightarrow A}{A \Rightarrow p_6
 \vert p_6 \Rightarrow A} }{A \Rightarrow p_5
 \vert A \Rightarrow p_6  \vert p_5,p_6 \Rightarrow A\odot
A}}{A \Rightarrow p_5
 \vert \Rightarrow p_6,\neg A \vert p_5,p_6 \Rightarrow
A\odot  A}}{G' \vert
\Rightarrow p_1,p_6,\neg A\odot  \neg A \vert A \Rightarrow p_5
 \vert p_5,p_6 \Rightarrow A\odot  A} }{ G_{H_{1}^{c}:G\vert G^{\ast}}^{\medstar(1)}\equiv   G' \vert \Rightarrow
p_1,B\vert B \Rightarrow p_6,\neg A\odot  \neg A \vert A
\Rightarrow p_5  \vert p_5,p_6 \Rightarrow A\odot  A}}
\]

\begin{center}
 \footnotesize  FIGURE 7\quad \normalsize A proof $\tau _{H_{1}^{c}:G\vert G^{\ast}}^{\medstar(1)}$  of   $ G_{H_{1}^{c}:G\vert G^{\ast}}^{\medstar(1)}$
\end{center}

However,  $ G_{H_{1}^{c}:G\vert G^{\ast}}^{\medstar(1)}=
\Rightarrow p_2
,B\vert B \Rightarrow p_4,\neg A\odot  \neg A \vert
p_1\Rightarrow C\vert C,p_2 \Rightarrow A\odot  A \vert A \Rightarrow p_3  \vert p_3,p_4 \Rightarrow A\odot  A
\vert \Rightarrow
p_1,B\vert B \Rightarrow p_6,\neg A\odot  \neg A \vert A
\Rightarrow p_5  \vert p_5,p_6 \Rightarrow A\odot  A
$  contains more copies of sequents from $G^{\ast}$ and
seems more complex than
$G\vert G^{\ast}$.  We will present  a unified method to tackle with it in the following steps.   Other derivations are shown in Figures $ 8,9,10,11$.

\[\boxed {
\cfrac{C \Rightarrow C\cfrac{B \Rightarrow B\cfrac{
  A \Rightarrow p_1  \vert \Rightarrow p_2
,\neg A \vert p_1,p_2 \Rightarrow A\odot  A \,\,\,\,
    \cfrac{A \Rightarrow
p_4}{ \Rightarrow p_4,\neg A} {\kern
1pt} }{ A \Rightarrow p_1  \vert \Rightarrow p_2,p_4
,\neg A\odot  \neg A \vert p_1,p_2 \Rightarrow A\odot  A
} }{A \Rightarrow p_1  \vert \Rightarrow p_2,B\vert B
\Rightarrow p_4,\neg A\odot  \neg A \vert p_1,p_2 \Rightarrow
A\odot  A }   }{A \Rightarrow p_1
 \vert \Rightarrow p_2,B\vert B \Rightarrow p_4,\neg A\odot  \neg
A \vert p_1 \Rightarrow C\vert C,p_2 \Rightarrow A\odot  A}}
\]
\begin{center}
 \footnotesize  FIGURE 8\normalsize\quad A derivation $ \tau _{H_{2}^{c}:A \Rightarrow p_4}^ {*} $ from  $ A \Rightarrow p_4  $
\end{center}
\[\boxed {
\cfrac{C \Rightarrow C\cfrac{B \Rightarrow B\cfrac{
  A \Rightarrow p_5  \vert \Rightarrow p_6
,\neg A \vert p_5,p_6 \Rightarrow A\odot  A\,\,\,\,
    \cfrac{A \Rightarrow
p_3}{ \Rightarrow p_3,\neg A} {\kern
1pt} }{ A \Rightarrow p_5  \vert \Rightarrow p_6,p_3
,\neg A\odot  \neg A \vert p_5,p_6 \Rightarrow A\odot  A
} }{A \Rightarrow p_5  \vert \Rightarrow p_6,B\vert B
\Rightarrow p_3,\neg A\odot  \neg A \vert p_5,p_6 \Rightarrow
A\odot  A }   }{A \Rightarrow p_5
 \vert \Rightarrow p_6,B\vert B \Rightarrow p_3,\neg A\odot  \neg
A \vert p_5 \Rightarrow C\vert C,p_6 \Rightarrow A\odot  A{\kern
1pt} }}
\]
\begin{center}
 \footnotesize  FIGURE 9\quad\normalsize A derivation $ \tau _{H_{2}^{c}:A \Rightarrow p_3}^ {*}$ from  $ A \Rightarrow p_3  $
\end{center}

\[\boxed {
\cfrac{B \Rightarrow B\cfrac{
\cfrac{A \Rightarrow p_2}{ \Rightarrow p_2
,\neg A}\,\,\,\,    \cfrac{A\Rightarrow p_4 }{
\Rightarrow p_4,\neg A}}{ \Rightarrow p_2,p_4,\neg
A\odot  \neg A}}{ \Rightarrow p_2,B\vert B
\Rightarrow p_4,\neg A\odot  \neg A}\qquad
\cfrac{B \Rightarrow B\cfrac{
\cfrac{A \Rightarrow p_5}{ \Rightarrow p_5
,\neg A}\,\,\,\,    \cfrac{A\Rightarrow p_3}{
\Rightarrow p_3,\neg A}}{ \Rightarrow p_5,p_3,\neg
A\odot  \neg A}}{ \Rightarrow p_5,B\vert B
\Rightarrow p_3,\neg A\odot  \neg A}}
\]

\begin{center}
 \footnotesize  FIGURE 10\normalsize \quad $ \tau _{\{H_{1}^{c}:A
\Rightarrow p_2, H_{2}^{c}:A \Rightarrow p_4 \}}^ {*}$ and $ \tau _{\{H_{1}^{c}:A \Rightarrow p_5, H_{2}^{c}:A \Rightarrow p_3 \}}^ {*}$
\end{center}
\[\boxed {\cfrac{C \Rightarrow C \,\,\,\,  p_1,p_2 \Rightarrow A\odot  A
  }{p_1 \Rightarrow C\vert C,p_2 \Rightarrow A\odot
A}\quad
\cfrac{C \Rightarrow C \,\,\,\,  p_3,p_4 \Rightarrow A\odot  A
  }{p_3 \Rightarrow C\vert C,p_4 \Rightarrow A\odot
A}}
\]
\[\boxed {\cfrac{C \Rightarrow C \,\,\,\,  p_5,p_6 \Rightarrow A\odot  A
  }{p_5 \Rightarrow C\vert C,p_6 \Rightarrow A\odot
A}}
\]

\begin{center}
 \footnotesize  FIGURE 11\normalsize \quad $\tau _{H_{3}^{c}:p_1,p_2 \Rightarrow A\odot  A}^ {*}$, $\tau _{H_{3}^{c}:p_3,p_4 \Rightarrow A\odot  A}^ {*}$ and $\tau _{H_{3}^{c}:p_5,p_6 \Rightarrow A\odot  A}^ {*}$
\end{center}

\textbf{Step 3  (Separation of one branch).}  A proof $ \tau _{H_{1}^{c}:G\vert G^{\ast}}^{\medstar(2)}  $  is constructed by applying sequentially $$ \tau _{H_{3}^{c}:p_3,p_4 \Rightarrow A\odot  A}^{*},
\tau _{H_{3}^{c}:p_5,p_6 \Rightarrow A\odot  A}^ {*}$$  to $ p_3,p_4 \Rightarrow A\odot  A$ and $ p_5,p_6 \Rightarrow A\odot  A$   in  $ G_{H_{1}^{c}:G\vert G^{\ast}}^{\medstar(1)}$,  as
shown in Figure 12, where $G''\equiv  G_{H_{1}^{c}:G\vert G^{\ast}}^{\medstar(1)}\backslash\{p_3,p_4 \Rightarrow A\odot  A,p_5,p_6 \Rightarrow A\odot  A\}$

\[\boxed {\cfrac{C \Rightarrow C \,\,\,\, \cfrac{C \Rightarrow C \,\,\,\,  G''\vert p_3,p_4 \Rightarrow A\odot  A\vert p_5,p_6 \Rightarrow A\odot  A
  }{G''\vert p_3 \Rightarrow C\vert C,p_4 \Rightarrow A\odot
A\vert p_5,p_6 \Rightarrow A\odot  A}}{ G_{H_{1}^{c}:G\vert G^{\ast}}^{\medstar(2)}\equiv G''\vert p_3 \Rightarrow C\vert C,p_4 \Rightarrow A\odot A\vert p_5 \Rightarrow C\vert C,p_6 \Rightarrow A\odot A}}
\]

\begin{center}
 \footnotesize  FIGURE 12 \quad \normalsize A proof $\tau _{H_{1}^{c}:G\vert G^{\ast}}^{\medstar(2)}$  of   $G_{H_{1}^{c}:G\vert G^{\ast}}^{\medstar(2)}$
\end{center}

$$G_{H_{1}^{c}:G\vert G^{\ast}}^{\medstar(2)}=
\Rightarrow p_2,B\vert B \Rightarrow p_4,\neg A\odot  \neg A \vert p_1\Rightarrow C\vert C,p_2 \Rightarrow A\odot  A\vert A \Rightarrow p_3  \vert\Rightarrow p_1,B\vert $$ $$B \Rightarrow p_6,\neg A\odot  \neg A \vert A\Rightarrow p_5 \vert p_3 \Rightarrow C\vert C,p_4 \Rightarrow A\odot A\vert p_5 \Rightarrow C\vert C,p_6 \Rightarrow A\odot A.$$
Notice that
 $$ \mathcal{D} ( {B \Rightarrow p_4,\neg A\odot  \neg A \vert
A \Rightarrow p_3  \vert p_3 \Rightarrow C\vert C,p_4 \Rightarrow A\odot A} )$$$$=
\mathcal{D}( { B \Rightarrow p_6,\neg A\odot  \neg A \vert A\Rightarrow p_5 \vert p_5 \Rightarrow C\vert C,p_6 \Rightarrow A\odot A} )$$$$=
 A\Rightarrow C\vert C,B \Rightarrow A\odot A,\neg A\odot  \neg A.\,\,\,\,\, \qquad\qquad \quad \quad \qquad  $$
Then it is permissible  to cut off the
part  $$ B \Rightarrow p_6,\neg A\odot  \neg A \vert A\Rightarrow p_5 \vert p_5 \Rightarrow C\vert C,p_6 \Rightarrow A\odot A$$  of  $G_{H_{1}^{c}:G\vert G^{\ast}}^{\medstar(2)}$,   which
corresponds to applying  $ (EC) $  to  $ \mathcal{D}( G_{H_{1}^{c}:G\vert G^{\ast}}^{\medstar(2)}) $.    We
regard such a manipulation  as a constrained contraction rule applied to  $ G_{H_{1}^{c}:G\vert G^{\ast}}^{\medstar(2)}$  and denote it
by  $ ( {EC_\Omega } )$.  Define   $ G_{H_1^c :G|G^* }^{\medstar}$  to be $$\Rightarrow p_2,B\vert B \Rightarrow p_4,\neg A\odot  \neg A \vert
p_1\Rightarrow C\vert C,p_2 \Rightarrow A\odot  A\vert $$$$ A \Rightarrow p_3  \vert \Rightarrow
p_1,B\vert  p_3 \Rightarrow C\vert C,p_4 \Rightarrow A\odot A. $$  Then we construct a proof of  $G_{ H_1^c :G|G^* }^{\medstar} $  by  $ \cfrac{G_{H_{1}^{c}:G\vert G^{\ast}}^{\medstar(2)} }{G_{ H_1^c :G|G^* }^{\medstar} }({EC_\Omega})$,  which  guarantees the validity
of  $$  \vdash _{{\rm {\bf GIUL}}} \mathcal{D}(G_{H_1^c :G|G^* }^{\medstar} )  $$  under
the condition  $$  \vdash _{{\rm {\bf GIUL}}} \mathcal{D}( G_{H_{1}^{c}:G\vert G^{\ast}}^{\medstar(2)}).  $$

A change happens here!  There is only one sequent which is a copy of  a sequent in $G^{\ast}$ in $G_{H_1^c :G|G^* }^{\medstar}$.  It is simpler than $G\vert G^{\ast}$. So we are moving forward. The above procedure is called the separation of $G|G^ *$ as a branch of $H_1^c$ and reformulated as follows (See Section 7  for details).

\[\boxed {
\cfrac{{\underline {\cfrac{{\underline {\cfrac{{\underline {\qquad G|G^*\quad} }}
{{G_{H_1^c :G|G^* }^{ \medstar (1)} }}\left\langle {\tau _{H_1^c :A \Rightarrow p_1 }^{ *} } \right\rangle } }}{{G_{H_1^c :G|G^* }^{\medstar(2)} }}\left\langle {\tau _{H_3^c :p_3 ,p_4  \Rightarrow A \odot A}^{ *},\tau _{H_3^c :p_5 ,p_6  \Rightarrow A \odot A }^{ *} } \right\rangle } }}{{G_{H_1^c :G|G^* }^{ \medstar} }}\left\langle {EC_\Omega} \right\rangle}
\]

The separation of $G|G^ *$ as a branch of $H_2^c$ is constructed by a similar procedure as follows.

\[\boxed {
\cfrac{{\underline {\cfrac{{\underline {\cfrac{{\underline {\qquad G|G^ *\quad} }}{{G_{H_2^c :G|G^* }^{ \medstar (1)} }}\left\langle {\tau _{H_2^c :A \Rightarrow p_3}^{ *} } \right\rangle } }}{{G_{ H_2^c :G|G^* }^{\medstar(2)} }}\left\langle {\tau _{H_3^c :p_3 ,p_4  \Rightarrow A \odot A}^{ *}} \right\rangle } }}{{G_{ H_2^c :G|G^* }^{\medstar} }}\left\langle {EC_\Omega  } \right\rangle}
\]

Note that $ \mathcal{D}(G_{H_1^c :G|G^* }^{\medstar})=H_2 $ and $ \mathcal{D}(G_{H_2^c :G|G^* }^{\medstar})=H_3 $.  So we have built up   one-one correspondences between
proofs  of $G_{H_1^c :G|G^* }^{\medstar},
G_{H_1^c :G|G^* }^{\medstar}$ and those of $H_2, H_3$.

\textbf{Step 3  (Separation algorithm of multiple branches).}
We  will prove  $  \vdash _{{\rm {\bf GIUL}}} \mathcal{D}_0 ( {G_0 } ) $  in a  direct way,
i.e., only the major step of Theorem 8.2  is presented in the following.  (See A.5.4 for a complete illustration.)   Recall that
$$G_{H_1^c :G|G^* }^{\medstar}=\Rightarrow p_2,B\vert B \Rightarrow p_4,\neg A\odot  \neg A \vert
p_1\Rightarrow C\vert C,p_2 \Rightarrow A\odot  A\vert $$$$ A \Rightarrow p_3  \vert \Rightarrow
p_1,B\vert  p_3 \Rightarrow C\vert C,p_4 \Rightarrow A\odot A, $$
$$G_{H_2^c :G|G^*}^{\medstar}= A \Rightarrow p_1 |
   \Rightarrow p_2 ,B|B \Rightarrow p_4 ,\neg A \odot \neg A |
   p_1  \Rightarrow C|C,p_2  \Rightarrow A \odot A |$$$$
      B \Rightarrow p_3 ,\neg A \odot \neg A |p_3  \Rightarrow C|C,p_4  \Rightarrow A \odot A.
$$

By reassigning   identification numbers to occurrences of $p's$ in $G_{H_2^c :G|G^* }^{\medstar}$,
 $$G_{H_2^c :G|G^* }^{\medstar}= A \Rightarrow p_5 |
   \Rightarrow p_6 ,B|B \Rightarrow p_8 ,\neg A \odot \neg A |
   p_5  \Rightarrow C|C,p_6  \Rightarrow A \odot A |$$$$
      B \Rightarrow p_7 ,\neg A \odot \neg A{\kern 1pt} |p_7  \Rightarrow C|C,p_8
       \Rightarrow A \odot A.
$$
By applying $\tau _{\{H_{1}^{c}:A \Rightarrow p_5,H_{2}^{c}:A \Rightarrow p_3 \}}^ {*}$ to
$A \Rightarrow p_3$ in  $G_{H_1^c :G|G^* }^{\medstar}$ and  $A \Rightarrow p_5$ in $G_{H_2^c :G|G^* }^{\medstar}$, we get $  \vdash _{{\rm {\bf GIUL}}} G'  $, where
\[G'\equiv \Rightarrow p_2,B\vert B \Rightarrow p_4,\neg A\odot  \neg A \vert
p_1\Rightarrow C\vert C,p_2 \Rightarrow A\odot  A\vert
 \Rightarrow p_1,B\vert $$$$
  p_3 \Rightarrow C\vert C,p_4 \Rightarrow A\odot A\vert
   \Rightarrow p_6 ,B|B \Rightarrow p_8 ,\neg A \odot \neg A |
   p_5  \Rightarrow C|C,p_6  \Rightarrow A \odot A |$$$$
      B \Rightarrow p_7 ,\neg A \odot \neg A  |p_7  \Rightarrow C|C,p_8  \Rightarrow A \odot A \vert \Rightarrow p_5,B\vert B\Rightarrow p_3,\neg A\odot  \neg A. \]

Why do you reassign   identification numbers to occurrences of $p's$ in $G_{H_2^c :G|G^* }^{\medstar}$?  It  makes  different occurrences of  $p's$  to be assigned different identification numbers in two nodes
$G_{H_1^c :G|G^* }^{\medstar}$ and  $G_{H_2^c :G|G^* }^{\medstar}$  of  the proof of $G'$.

By applying $\left\langle {EC_\Omega^{\ast}  } \right\rangle$ to $G'$, we get $  \vdash _{{\rm {\bf GIUL}}_{\Omega}}  G_{\rm {\bf I}}^\medstar$, where
\[G_{\rm {\bf I}}^\medstar\equiv \Rightarrow p_2,B\vert B \Rightarrow p_4,\neg A\odot  \neg A \vert
p_1\Rightarrow C\vert C,p_2 \Rightarrow A\odot  A\vert
 \Rightarrow p_1,B\vert $$$$
  p_3 \Rightarrow C\vert C,p_4 \Rightarrow A\odot A \vert B\Rightarrow p_3,\neg A\odot  \neg A. \]

A great change  happens here!  We have eliminated all sequents which are copies of some sequents in $G^{\ast}$  and convert  $G\vert G^{\ast}$ into $G_{\rm {\bf I}}^\medstar$ in which each sequent is
some copy of a sequent in $G_0$.

Then $  \vdash _{{\rm {\bf GIUL}}} \mathcal{D}(G_{\rm {\bf I}}^\medstar)  $ by Lemma 5.6,  where $\mathcal{D}( G_{\rm {\bf I}}^\medstar)=H_0=$
\[\Rightarrow C,B \vert C \Rightarrow B,A\odot
A \vert B \Rightarrow C,\neg A\odot  \neg A\vert C,B \Rightarrow
A\odot  A,\neg A\odot  \neg A. \]
So we have built up   one-one correspondences between
the proof of $ G_{\rm {\bf I}}^\medstar$ and that of $H_0$, i.e., the proof of $H_0$  can be constructed by applying  $(\mathcal{D})$ to the proof of $ G_{\rm {\bf I}}^\medstar$.  The major steps of  constructing
$G_{\rm {\bf I}}^\medstar$ are shown in the following figure, where $\mathcal{D}( G\vert G^\ast)=H_1$, $\mathcal{D}(G_{H_1^c :G\vert G^\ast }^\medstar)=H_2$, $\mathcal{D}( G_{H_2^c :G\vert G^\ast }^\medstar)=H_3$ and $\mathcal{D}( G_{\rm {\bf I}}^\medstar)=H_0$.
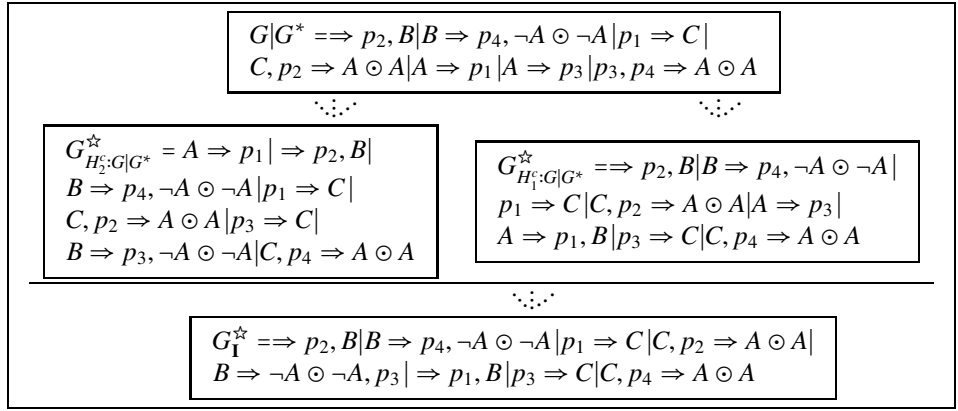
\[
\boxed {\begin{array}{l}
 \underline{{\begin{array}{l}
\qquad {\kern 1pt}{\kern 1pt}{\kern 1pt}{\kern 1pt}{\kern 1pt}{\kern 1pt}{\kern
1pt}{\kern 1pt}{\kern 1pt}{\kern 1pt}{\kern 1pt}{\kern 1pt}{\kern 1pt}{\kern
1pt}{\kern 1pt}{\kern 1pt}{\kern 1pt}{\kern 1pt}{\kern 1pt}{\kern 1pt}{\kern
1pt}{\kern 1pt}{\kern 1pt}{\kern 1pt}{\kern 1pt}{\kern 1pt}{\kern 1pt}{\kern
1pt}{\kern 1pt}{\kern 1pt}{\kern 1pt}{\kern 1pt}{\kern 1pt}{\kern 1pt}{\kern
1pt}{\kern 1pt}{\kern 1pt}{\kern 1pt}{\kern 1pt}{\kern 1pt}{\kern 1pt}{\kern
1pt}{\kern 1pt}{\kern 1pt}{\kern 1pt}{\kern 1pt}{\kern 1pt}{\kern 1pt}{\kern
1pt}\boxed {\begin{array}{l}
 G\vert G^\ast =\Rightarrow p_2 ,B\vert B\Rightarrow p_4 ,\neg A\odot \neg
A{\kern 1pt}\vert
 p_1 \Rightarrow C{\kern 1pt}\vert \\
 C,p_2 \Rightarrow A\odot A\vert
A\Rightarrow p_1 {\kern 1pt}\vert A\Rightarrow p_3 {\kern 1pt}\vert p_3 ,p_4
\Rightarrow A\odot A \\
 \end{array}}{\kern 1pt} \\
 {\kern 1pt}{\kern 1pt}{\kern 1pt}{\kern 1pt}{\kern 1pt}{\kern 1pt}{\kern
1pt}{\kern 1pt}{\kern 1pt}{\kern 1pt}{\kern 1pt}{\kern 1pt}{\kern 1pt}{\kern
1pt}{\kern 1pt}{\kern 1pt}{\kern 1pt}{\kern 1pt}{\kern 1pt}{\kern 1pt}{\kern
1pt}{\kern 1pt}{\kern 1pt}{\kern 1pt}{\kern 1pt}{\kern 1pt}{\kern 1pt}{\kern
1pt}{\kern 1pt}{\kern 1pt}{\kern 1pt}{\kern 1pt}{\kern 1pt}{\kern 1pt}{\kern
1pt}{\kern 1pt}{\kern 1pt}{\kern 1pt}{\kern 1pt}{\kern 1pt}{\kern 1pt}{\kern
1pt}{\kern 1pt}{\kern 1pt}{\kern 1pt}{\kern 1pt}{\kern 1pt}{\kern 1pt}{\kern
1pt}{\kern 1pt}{\kern 1pt}{\kern 1pt}{\kern 1pt}{\kern 1pt}{\kern 1pt}{\kern
1pt}{\kern 1pt}{\kern 1pt}{\kern 1pt}{\kern 1pt}{\kern 1pt}{\kern 1pt}{\kern
1pt}{\kern 1pt}{\kern 1pt}{\kern 1pt}{\kern 1pt}{\kern 1pt}{\kern 1pt}{\kern
1pt}{\kern 1pt}{\kern 1pt}{\kern 1pt}{\kern 1pt}{\kern 1pt}{\kern 1pt}{\kern
1pt}{\kern 1pt}{\kern 1pt}{\kern 1pt}{\kern 1pt}{\kern 1pt}{\kern 1pt}{\kern
1pt}{\kern 1pt}{\kern 1pt}{\kern 1pt}{\kern 1pt}{\kern 1pt}{\kern 1pt}{\kern
1pt}{\kern 1pt}{\kern 1pt}{\kern 1pt}{\kern 1pt}{\kern 1pt}{\kern 1pt}{\kern
1pt}{\kern 1pt}{\kern 1pt}{\kern 1pt}
   \ddots  \vdots  {\mathinner{\mkern0.5mu\raise0.2pt\hbox{.}\mkern0.0mu
 \raise1.9pt\hbox{.}\mkern0.2mu\raise4.1pt\hbox{.}}}{\kern 1pt}{\kern
1pt}{\kern 1pt}{\kern 1pt}{\kern 1pt}{\kern 1pt}{\kern 1pt}{\kern 1pt}{\kern
1pt}{\kern 1pt}{\kern 1pt}{\kern 1pt}{\kern 1pt}{\kern 1pt}{\kern 1pt}{\kern
1pt}{\kern 1pt}{\kern 1pt}{\kern 1pt}{\kern 1pt}{\kern 1pt}{\kern 1pt}{\kern
1pt}{\kern 1pt}{\kern 1pt}{\kern 1pt}{\kern 1pt}{\kern 1pt}{\kern 1pt}{\kern
1pt}{\kern 1pt}{\kern 1pt}{\kern 1pt}{\kern 1pt}{\kern 1pt}{\kern 1pt}{\kern
1pt}{\kern 1pt}{\kern 1pt}{\kern 1pt}{\kern 1pt}{\kern 1pt}{\kern 1pt}{\kern
1pt}{\kern 1pt}{\kern 1pt}{\kern 1pt}{\kern 1pt}{\kern 1pt}{\kern 1pt}{\kern
1pt}{\kern 1pt}{\kern 1pt}{\kern 1pt}{\kern 1pt}{\kern 1pt}{\kern 1pt}{\kern
1pt}{\kern 1pt}{\kern 1pt}{\kern 1pt}{\kern 1pt}{\kern 1pt}{\kern 1pt}{\kern
1pt}{\kern 1pt}{\kern 1pt}{\kern 1pt}{\kern 1pt}{\kern 1pt}{\kern 1pt}{\kern
1pt}{\kern 1pt}{\kern 1pt}{\kern 1pt}{\kern 1pt}{\kern 1pt}{\kern 1pt}{\kern
1pt}{\kern 1pt}{\kern 1pt}{\kern 1pt}{\kern 1pt}{\kern 1pt}{\kern 1pt}{\kern
1pt}{\kern 1pt}{\kern 1pt}{\kern 1pt}{\kern 1pt}{\kern 1pt}{\kern 1pt}{\kern
1pt}{\kern 1pt}{\kern 1pt}{\kern 1pt}{\kern 1pt}{\kern 1pt}{\kern 1pt}{\kern
1pt}{\kern 1pt}{\kern 1pt}{\kern 1pt}{\kern 1pt}{\kern 1pt}{\kern 1pt}{\kern
1pt}{\kern 1pt}{\kern 1pt}{\kern 1pt}{\kern 1pt}{\kern 1pt}{\kern 1pt}{\kern
1pt}{\kern 1pt}{\kern 1pt}{\kern 1pt}{\kern 1pt}{\kern 1pt}{\kern 1pt}{\kern
1pt}{\kern 1pt}{\kern 1pt}{\kern 1pt}{\kern 1pt}{\kern 1pt}{\kern 1pt}
   \ddots  \vdots  {\mathinner{\mkern0.5mu\raise0.2pt\hbox{.}\mkern0.0mu
 \raise1.9pt\hbox{.}\mkern0.2mu\raise4.1pt\hbox{.}}}{\kern 1pt}{\kern 1pt}{\kern 1pt}{\kern 1pt}{\kern
1pt}{\kern 1pt}{\kern 1pt}{\kern 1pt}{\kern 1pt}{\kern 1pt}{\kern 1pt}{\kern
1pt}{\kern 1pt}{\kern 1pt}{\kern 1pt}{\kern
1pt}{\kern 1pt}{\kern 1pt}{\kern 1pt}{\kern 1pt}{\kern 1pt}{\kern 1pt}{\kern
1pt}{\kern 1pt}{\kern 1pt}{\kern 1pt}{\kern 1pt}{\kern 1pt}{\kern 1pt}{\kern
1pt}{\kern 1pt}{\kern 1pt}{\kern 1pt}{\kern 1pt}{\kern 1pt}{\kern 1pt}{\kern
1pt}{\kern 1pt}{\kern 1pt}{\kern 1pt}{\kern 1pt}{\kern 1pt}{\kern 1pt}{\kern
1pt} \\
 \boxed {\begin{array}{l}
 G_{H_2^c :G\vert G^\ast }^\medstar=A\Rightarrow p_1 {\kern 1pt}\vert \Rightarrow
p_2 ,B\vert \\
 B\Rightarrow p_4 ,\neg A\odot \neg A{\kern 1pt}\vert p_1 \Rightarrow
C{\kern 1pt}\vert \\
 C,p_2 \Rightarrow A\odot A{\kern 1pt}\vert p_3 \Rightarrow C\vert \\
 B\Rightarrow p_3 ,\neg A\odot \neg A\vert C,p_4 \Rightarrow A\odot A \\
 \end{array}}{\kern 1pt}{\kern 1pt}{\kern 1pt}{\kern 1pt}{\kern 1pt}{\kern
1pt}{\kern 1pt}{\kern 1pt}{\kern 1pt}{\kern 1pt}{\kern 1pt}{\kern 1pt}{\kern
1pt}{\kern 1pt}\boxed {\begin{array}{l}
 G_{H_1^c :G\vert G^\ast }^\medstar=\Rightarrow p_2 ,B\vert B\Rightarrow p_4 ,\neg
A\odot \neg A{\kern 1pt}\vert \\
 p_1 \Rightarrow C{\kern 1pt}\vert C,p_2 \Rightarrow A\odot A\vert
A\Rightarrow p_3 {\kern 1pt}\vert \\
 A\Rightarrow p_1 ,B{\kern 1pt}\vert p_3 \Rightarrow C\vert C,p_4
\Rightarrow A\odot A \\
 \end{array}} \\
 \end{array}}} \\
 \qquad \qquad \qquad \qquad \qquad \qquad \qquad \qquad \qquad    \ddots  \vdots  {\mathinner{\mathinner{\mkern0.5mu\raise0.2pt\hbox{.}\mkern0.0mu
 \raise1.9pt\hbox{.}\mkern0.2mu\raise4.1pt\hbox{.}}}}{\kern 1pt}{\kern 1pt}{\kern 1pt}{\kern 1pt}{\kern
1pt}{\kern 1pt}{\kern 1pt}{\kern 1pt}{\kern 1pt}{\kern 1pt}{\kern 1pt}{\kern
1pt}{\kern 1pt}{\kern 1pt}{\kern 1pt}{\kern 1pt}{\kern 1pt}{\kern 1pt}{\kern
1pt}{\kern 1pt}{\kern 1pt}{\kern 1pt}{\kern 1pt}{\kern 1pt}{\kern 1pt}{\kern
1pt}{\kern 1pt}{\kern 1pt}{\kern 1pt}{\kern 1pt}{\kern 1pt}{\kern 1pt}{\kern
1pt}{\kern 1pt}{\kern 1pt}{\kern 1pt}{\kern 1pt}{\kern 1pt}{\kern 1pt}{\kern
1pt}{\kern 1pt}{\kern 1pt}{\kern 1pt} \\
 \qquad  \qquad \qquad\boxed {\begin{array}{l}
 G_{\rm {\bf I}}^\medstar=\Rightarrow p_2 ,B\vert B\Rightarrow p_4 ,\neg A\odot \neg A{\kern
1pt}\vert p_1 \Rightarrow C{\kern 1pt}\vert C,p_2 \Rightarrow A\odot A\vert
\\
 B\Rightarrow \neg A\odot \neg A,p_3 {\kern 1pt}\vert \Rightarrow p_1
,B{\kern 1pt}\vert p_3 \Rightarrow C\vert C,p_4 \Rightarrow A\odot A \\
 \end{array}} \\
 \end{array}}
\]
\normalsize

 In the above example,  $\mathcal{D}(G_{\rm {\bf I}}^\medstar) = \mathcal{D}_0 (G_0 )$. But it is not always the case. In general,   we can prove that  $\vdash _{{\mathbf{GL}}} \mathcal{D}_0 (G_0 )$  if $ \vdash _{{\mathbf{GL}}} \mathcal{D}(G_{\rm {\bf I}}^\medstar)$,  which is shown in the proof of Main  theorem in Page 46.  This example shows that the proof of Main  theorem essentially presents an algorithm to construct a proof of  $ \mathcal{D}_0 ( {G_0 }) $  from  $ \tau  $.

\section{ Preprocessing of Proof Tree}
Let $ \tau $ be a cut-free proof  of  $ G_0$  in Main theorem  in  $ {\rm {\bf GL}}$ by Lemma 2.15.   Starting with $ \tau  $,  we will construct a proof  $ \tau^{\ast}$ which contains no application of $(EC)$ and has some other properties in this section.

\begin{lemma}
(i) If  $  \vdash _{\rm {\bf GL}} \Gamma _1 \Rightarrow A,\Delta _1  $  and  $
\vdash _{\rm {\bf GL}} \Gamma _2 \Rightarrow B,\Delta _2  $  $$ then\,\,\, \vdash
_{\rm {\bf GL}} \Gamma _1 \Rightarrow A \wedge B,\Delta _1 \vert \Gamma _2
\Rightarrow A \wedge B,\Delta _2;$$

(ii) If  $  \vdash _{\rm {\bf GL}} \Gamma _1,A \Rightarrow \Delta _1  $  and  $
\vdash _{\rm {\bf GL}} \Gamma _2,B \Rightarrow \Delta _2  $  $$ then\,\,\, \vdash
_{\rm {\bf GL}} \Gamma _1,A \vee B \Rightarrow \Delta _1 \vert \Gamma _2,A
\vee B \Rightarrow \Delta _2.$$
\end{lemma}

\begin{proof}
(i)
\footnotesize
 $$ \cfrac{\Gamma _2 \Rightarrow B,\Delta _2 \cfrac{\Gamma _1
\Rightarrow A,\Delta _1 \cfrac{B \Rightarrow B  \cfrac{A \Rightarrow A
\cfrac{A \Rightarrow A \,\,\, B \Rightarrow B}{A \Rightarrow B  \vert B
\Rightarrow A
   }(COM)}{A \Rightarrow A \wedge B \vert B \Rightarrow A}( \wedge _r )}{A \Rightarrow A \wedge
B  \vert B \Rightarrow A \wedge B}( \wedge _r
)}{\Gamma _1 \Rightarrow A \wedge B,\Delta _1   \vert
B \Rightarrow A \wedge B}(CUT)}{\Gamma _1 \Rightarrow A \wedge B,\Delta _1
  \vert \Gamma _2 \Rightarrow A \wedge B,\Delta _2}(CUT)$$
\normalsize

(ii) is proved by a procedure similar to that of (i) and omitted.
\end{proof}

We introduce two new rules  by Lemma 4.1.

\begin{definition}
$ \cfrac{G_1\vert\Gamma _1 \Rightarrow A,\Delta_1\,\,\,\,\,\,G_2\vert\Gamma _2 \Rightarrow B,\Delta _2 }{G_1\vert G_2\vert\Gamma _1 \Rightarrow A \wedge B,\Delta _1 \vert \Gamma _2 \Rightarrow A \wedge B,\Delta _2 }( \wedge _{rw} )  $\\ and $ \,\,\,\,\cfrac{G_1\vert\Gamma _1,A \Rightarrow \Delta _1\,\,\,\,\,\,G_2\vert \Gamma _2,B \Rightarrow
\Delta _2    }{G_1\vert G_2\vert\Gamma _1,A \vee B
\Rightarrow \Delta _1 \vert \Gamma _2,A \vee B \Rightarrow \Delta _2 }(
\vee _{lw} )$ are called the generalized $(\wedge _{r})$ and $ (\vee _{l} )$ rules,  respectively.
\end{definition}

Now, we begin to process  $ \tau  $  as follows.

\textbf{Step 1 } A proof  $ \tau^{1}  $   is constructed by replacing inductively
all applications of
\[ \cfrac{G_1\vert \Gamma \Rightarrow A,\Delta \,\,\,\,\,\,G_2\vert \Gamma
\Rightarrow B,\Delta }{G_1\vert G_2\vert\Gamma \Rightarrow A \wedge B,\Delta }( \wedge
_r )\,\,\,(\mathrm{or\,\,\,} \cfrac{G_1\vert \Gamma,A \Rightarrow \Delta \,\,\,\,\,\,
G_2\vert \Gamma,B \Rightarrow \Delta }{G_1\vert G_2\vert \Gamma,A \vee B \Rightarrow
\Delta } (\vee_l) )\] in  $ \tau  $  with
\[ \cfrac{\cfrac{G_1\vert \Gamma \Rightarrow A,\Delta \,\,\,\,\,\,G_2\vert \Gamma
\Rightarrow B,\Delta }{G_1\vert G_2\vert \Gamma \Rightarrow A \wedge B,\Delta \vert
\Gamma \Rightarrow A \wedge B,\Delta }( \wedge _{rw} )}{G_1\vert G_2\vert\Gamma
\Rightarrow A \wedge B,\Delta }(EC) \]
\[ (\mathrm{accordingly\,\,\,} \cfrac{\cfrac{G_1\vert \Gamma,A \Rightarrow \Delta
\,\,\,\,\,\, G_2\vert \Gamma,B \Rightarrow \Delta }{G_1\vert G_2\vert \Gamma,A \vee B
\Rightarrow \Delta \vert \Gamma,A \vee B \Rightarrow \Delta }( \vee _{lw}
)}{G_1\vert G_2\vert \Gamma,A \vee B \Rightarrow \Delta }(EC) \,\,\,\mathrm{for}\,\,\,  (\vee_l)).\]
The replacements in Step 1 are local and the root of  $ \tau^{1}
 $  is also labeled by  $ G_0$.

\begin{definition}
We sometimes may regard  $ \cfrac{G'}{G'} $  as a structural rule of  $ {\rm {\bf
GL}} $  and denote it by  $ (ID_\Omega ) $  for convenience.  The focus sequent for $ (ID_\Omega ) $
is undefined.
\end{definition}

\begin{lemma}
Let  $ \cfrac{G'\vert S^m}{G'\vert S}(EC^ * ) \in \tau^{1}  $,   $ Th_{\tau^{1} }
(G'\vert S)=(H_0,H_1, \cdots,H_n ) $,   where  $ H_0 = G'\vert S $  and  $ H_n =
G_0$.  A tree  $ {\tau}'$  is constructed by replacing each  $ H_k  $  in  $ \tau^{1}  $
with  $ H_k \vert S^{m - 1} $  for all  $ 0 \leqslant k \leqslant n $.  Then  $ {\tau}'$  is a proof of  $ G_0\vert S^{m - 1}$.
\end{lemma}
\begin{proof}
 The proof is by induction on  $n$ .   Since $ \tau^{1}(G'\vert S^{m} )$ is a proof  and $ \cfrac{G'\vert S^m}{H_{0}\vert S^{m - 1} }(ID_{\Omega})$  is valid in $ {\rm {\bf GL}} $,  then  $ {\tau}'(H_{0}\vert S^{m - 1} )$  is a proof.
  Suppose that  $ {\tau}' (H_{n-1} \vert S^{m - 1} )$   is a proof.   Since  $ \cfrac{H_{n-1}  \quad G''}{H_{n} }(II)$ (or  $ \cfrac{H_{n-1}
}{H_{n} }(I)) $  in  $ \tau^{1}$,  then
 $ \cfrac{H_{n-1}\vert S^{m - 1}\quad
  G''}{H_{n} \vert S^{m - 1}} $  (or
 $ \cfrac{H_{n-1} \vert S^{m - 1}}{H_{n} \vert S^{m - 1}}) $ is an application of
 the same rule $(II)$ (or $(I)$). Thus
 $ {\tau}'(H_{n} \vert S^{m - 1} )$   is a proof.
\end{proof}

\begin{definition}
The manipulation described in Lemma 4.4 is called  sequent-inserting operation.
\end{definition}

Clearly, the number of  $ (EC^ * )$-applications in  $ {\tau}' $  is less than
 $ \tau^{1}$.    Next, we continue to process  $ \tau$.

\textbf{Step 2 }Let  $ \cfrac{G_1'''\vert \{S_1^c \}^{m_1' }}{G_1 '\vert S_1^c
}(EC^ * ), \cdots,\cfrac{G_N'''\vert \{S_N^c \}^{m_N' }}{G_N '\vert S_N^c }(EC^
* ) $  be all applications of  $ (EC^ * ) $  in  $ \tau^{1}  $ and  $ G_0^ * \coloneq \{S_1^c \}^{m_1' -1}\vert \cdots\vert \{S_N^c\}^{m_N' - 1} $.  By repeatedly applying sequent-inserting operation,
we construct  a proof of  $ G_0 \vert G_0^ *  $
 in  $ {\rm {\bf GL}} $  without applications of  $ (EC^{\ast})$ and denote it by  $\tau^{2}$.

\begin{remark}
(i)  $ \tau^{2}  $  is constructed by converting  $ (EC) $  into
 $ (ID_\Omega ) $; (ii) Each node of  $ \tau^{2}  $  has the form  $ H_0 \vert H_0^ *
 $,   where  $ H_0\in \tau^{1}  $  and  $ H_0^ *  $  is a (possibly empty) subset of  $ G_0^
*$.
\end{remark}

We  need the following construction to eliminate applications of $(EW)$
in $\tau^{2}$.

\begin{construction}
Let  $ H \in \tau^{2}  $,   $ H' \subseteq H $  and  $ Th_{\tau^{2} } (H) = ( {H_0,
\cdots,H_n } ) $,   where  $ H_0 = H $,   $ H_n =G_0\vert G_0^ *$.  Hypersequents  $ \left\langle {H_k }
\right\rangle _{H:H'}$
and trees
$ \tau _{H:H'}^{2}  ( {\left\langle {H_k }
\right\rangle _{H:H'} } )$  for all $0\leqslant  k\leqslant  n$ are constructed inductively as follows.

(i) $ \left\langle {H_0 } \right\rangle _{H:H'} \coloneq H' $  and  $ \tau _{H:H'}^{2} ( {\left\langle {H_0 } \right\rangle _{H:H'} } ) $  consists of a single node  $ H' $;

(ii) Let  $ \cfrac{G'\vert S' \quad G''\vert S''}{G'\vert
G''\vert H'' }(II) $  (or  $ \cfrac{G'\vert S'}{G'\vert S''   }(I)) $  be in  $ \tau^{2}
 $,    $ H_k = G'\vert S'  $  and  $ H_{k + 1} = G'\vert G''\vert H'' $
(accordingly $ H_{k + 1} = G'\vert S''   $  for  $ (I) $)
for some  $ 0 \leqslant k \leqslant n - 1$.

If  $ S' \in \left\langle
{H_k } \right\rangle _{H:H'}  $ $$ \left\langle {H_{k + 1} } \right\rangle
_{H:H'} \coloneq {\left\langle {H_k } \right\rangle _{H:H'} \backslash
\{S'\}}\vert G''\vert H'' $$
$$(\mathrm{accordingly} \,\,\,  \left\langle {H_{k + 1} }
\right\rangle _{H:H'} \coloneq {\left\langle {H_k } \right\rangle _{H:H'}
\backslash \{S'\}}\vert S'' \,\,\,  \mathrm{ for}\,\,\, (I))$$
and  $ \tau _{H:H'}^{2} ( {\left\langle {H_{k + 1} } \right\rangle _{H:H'} }
) $  is constructed by combining trees  $$ \tau _{H:H'}^{2}
( {\left\langle {H_k } \right\rangle _{H:H'} } ),    \tau^{2}(G''\vert S'' )\,\,\,  \mathrm{with}\,\,\,
 \cfrac{\left\langle {H_k } \right\rangle _{H:H'}\,\,\,  G''\vert S''
}{\left\langle {H_{k + 1} } \right\rangle _{H:H'}
}(II) $$
$$ (\mathrm{accordingly} \,\,\, \tau _{H:H'}^{2}
( {\left\langle {H_k } \right\rangle _{H:H'} } )\,\,\,  \mathrm{with}\,\,\,
 \cfrac{\left\langle {H_k } \right\rangle _{H:H'} }{\left\langle {H_{k + 1} }
\right\rangle _{H:H'}   }(I)  \,\,\,  \mathrm{for}  \,\,\, (I)) $$
otherwise
 $ \left\langle {H_{k + 1} } \right\rangle _{H:H'}
\coloneq \left\langle {H_k } \right\rangle _{H:H'}$  and
$ \tau _{H:H'}^{2}  ( {\left\langle {H_{k + 1} }
\right\rangle _{H:H'} } ) $  is constructed by combining  $$ \tau _{H:H'}^{2}  (
{\left\langle {H_k } \right\rangle _{H:H'} } )\,\,\, \mathrm{with}\,\,\,
 \cfrac{\left\langle {H_k } \right\rangle _{H:H'}
  }{\left\langle {H_{k + 1} } \right\rangle _{H:H'}
  }(ID_\Omega ).$$

(iii) Let  $ \cfrac{G'}{G'\vert S'}(EW) \in \tau^{2}  $,   $ H_k = G'  $  and  $ H_{k + 1} = G'\vert S' $
then
$ \left\langle {H_{k + 1} } \right\rangle _{H:H'} \coloneq \left\langle {H_k } \right\rangle _{H:H'}    $
and  $ \tau _{H:H'}^{2}  ( {\left\langle {H_{k + 1} } \right\rangle
_{H:H'} } ) $  is constructed by combining  $ \tau _{H:H'}^{2}  ( {\left\langle
{H_k } \right\rangle _{H:H'} } ) \,\,\, \mathrm{with}\,\,\,\cfrac{\left\langle {H_k }
\right\rangle _{H:H'}
}{\left\langle {H_{k + 1} } \right\rangle _{H:H'}
}(ID_\Omega ).$
\end{construction}

\begin{lemma}
(i) $ \left\langle {H_k } \right\rangle _{H:H'} \subseteq H_k  $  for all  $ 0
\leqslant k \leqslant n $;

(ii)  $ \tau _{H:H'}^{2}  ( {\left\langle {H_k } \right\rangle _{H:H'}
} ) $  is a derivation of  $ \left\langle {H_k } \right\rangle _{H:H'}  $
from  $ H' $  without  $ (EC)$.
\end{lemma}

\begin{proof}
 The proof is by induction on  $ k$. For the base step,  $ \left\langle {H_0 } \right\rangle _{H:H'} = H'$  and  $ \tau _{H:H'}^{2}  ( {\left\langle {H_0 }
\right\rangle _{H:H'} } ) $  consists of a single node  $ H'$.  Then  $ \left\langle
{H_0 } \right\rangle _{H:H'} \subseteq H_0 = H $,
 $ \tau _{H:H'}^{2}
( {\left\langle {H_0 } \right\rangle _{H:H'} } ) $  is a derivation
of  $ \left\langle {H_0 } \right\rangle _{H:H'}  $  from  $ H' $  without  $ (EC)$.

For the induction step,  suppose that $ \left\langle {H_k } \right\rangle _{H:H'}  $
and $ \tau _{H:H'}^{2}  ( {\left\langle {H_k } \right\rangle _{H:H'}
} ) $  be constructed such that  (i) and (ii) hold for some  $ 0 \leqslant
k \leqslant n - 1$.  There are two cases to be considered.

\textbf{Case 1} Let  $ \cfrac{G'\vert S' }{G'\vert S''
}(I) \in \tau^{2}  $,   $ H_k = G'\vert S'  $  and $ H_{k + 1} = G'\vert
S''$.  If  $ S' \in \left\langle {H_k }
\right\rangle _{H:H'}  $  then  \\
$ \left\langle {H_k } \right\rangle _{H:H'}
\backslash \{S'\} \subseteq G' $  by  $ \left\langle {H_k } \right\rangle
_{H:H'} \subseteq H_k = G'\vert S' $.  Thus  $ \left\langle {H_{k + 1} }
\right\rangle _{H:H'} = ( {\left\langle {H_k } \right\rangle _{H:H'}
\backslash \{S'\}} )\vert S'' \subseteq G'\vert S'' =
H_{k + 1}.$
Otherwise  $ S' \notin \left\langle {H_k }
\right\rangle _{H:H'}  $  then  $ \left\langle {H_k } \right\rangle _{H:H'}\subseteq G'$ by $ \left\langle {H_k } \right\rangle _{H:H'}\subseteq H_k=G'\vert S'$. Thus
$ \left\langle {H_{k + 1} }
\right\rangle _{H:H'} \subseteq H_{k + 1}$ by  $ \left\langle {H_{k + 1} }
\right\rangle _{H:H'} =  \left\langle {H_k } \right\rangle _{H:H'}\subseteq G'\subseteq H_{k+1}$.
$ \tau _{H:H'}^{2}  ( {\left\langle {H_{k + 1} }
\right\rangle _{H:H'} } ) $  is a derivation of  $ \left\langle {H_{k + 1}
} \right\rangle _{H:H'}  $  from  $ H' $  without  $ (EC) $  since  $ \tau _{H:H'}^{2}  ( {\left\langle {H_k } \right\rangle _{H:H'} } ) $  is  such one  and
$ \cfrac{\left\langle {H_k } \right\rangle _{H:H'}}{\left\langle {H_{k + 1} } \right\rangle _{H:H'} }(I)$  is a valid  instance of  a rule $(I)$ of $ {\rm {\bf GL}}$.
The case of  applications of two-premise rule  is proved by a similar procedure  and omitted.

\textbf{Case 2} Let $ \cfrac{G'}{G'\vert S'}(EW) \in \tau^{2}  $,   $ H_k = G'$  and $ H_{k + 1} = G'\vert S'$. Then
$ \left\langle {H_{k + 1} }
\right\rangle _{H:H'} \subseteq H_{k + 1}$ by  $ \left\langle {H_{k + 1} }
\right\rangle _{H:H'} =  \left\langle {H_k } \right\rangle _{H:H'}\subseteq H_k\subseteq H_{k+1}$.
$ \tau _{H:H'}^{2}  ( {\left\langle {H_{k + 1} }
\right\rangle _{H:H'} } ) $  is a derivation of  $ \left\langle {H_{k + 1}
} \right\rangle _{H:H'}  $  from  $ H' $  without  $ (EC) $  since  $ \tau _{H:H'}^{2}  ( {\left\langle {H_k } \right\rangle _{H:H'} } ) $  is  such one  and
$ \cfrac{\left\langle {H_k } \right\rangle _{H:H'}}{\left\langle {H_{k + 1} } \right\rangle _{H:H'} }(ID_{\Omega})$  is valid by Definition 4.3.

\end{proof}

\begin{definition}
 The manipulation described in Construction 4.7  is called derivation-pruning operation.
\end{definition}

\begin{notation}
We denote  $ \left\langle {H_n } \right\rangle _{H:H'}  $  by
 $G_{H:H'}^{2}  $,    $ \tau _{H:H'}^{2}  (
{\left\langle {H_n } \right\rangle _{H:H'} } ) $  by  $ \tau _{H:H'}^{2}   $  and say that  $ H' $  is  transformed into  $ G_{H:H'}^{2}  $  in  $ \tau^{2}$.
\end{notation}

Then Lemma 4.8 shows that   $ \cfrac{\underline{\quad H' \quad}}{ G_{H:H'}^{2} }\left\langle {\tau _{H:H'}^{2} } \right\rangle  $, $G_{H:H'}^{2}\subseteq G_{0}\vert G_{0}^{ *} $.  Now, we continue to process  $ \tau  $  as follows.

\textbf{Step 3} Let  $ \cfrac{G'  }{G'\vert S'}(EW)\in\tau^{2}  $  then  $ \tau _{G'\vert S' :G'}^{2} ( {\left\langle {H_n } \right\rangle _{G'\vert S'
:G'} } ) $  is a derivation of  $ \left\langle {H_n } \right\rangle
_{G'\vert S' :G'}  $  from  $ G' $  thus a proof of  $ \left\langle {H_n
} \right\rangle _{G'\vert S' :G'}  $  is constructed by combining
 $ \tau^{2}( {G'} ) $  and $ \tau _{G'\vert S' :G'}^{2} ( {\left\langle {H_n } \right\rangle _{G'\vert S' :G'}
} ) $  with  $ \cfrac{G'
}{G'    }(ID_\Omega )
$.  By repeatedly applying the procedure above, we construct a proof $ \tau^{3} $ of $ G_{1}\vert G_{1}^ * $  without
 $ (EW)    $  in  $ {\rm {\bf GL}} $,   where $ G_{1}\subseteq G_0, G_{1}^ * \subseteq G_0^ *$ by Lemma 4.8 (i).

\textbf{Step 4 }Let  $ \Gamma,p, \bot \Rightarrow \Delta \in \tau^{3} $  (or
 $ \Gamma,p \Rightarrow \top,\Delta  $,   $ \cfrac{G\vert \Gamma \Rightarrow
\Delta }{G\vert \Gamma,p \Rightarrow \Delta }(WL)) $  then there exists
 $ \Gamma ' \Rightarrow \Delta ' \in H $  such that  $ p \in \Gamma ' $  for all  $ H
\in Th_{\tau^{3}} (\Gamma,p, \bot \Rightarrow \Delta ) $  (accordingly  $ H \in
Th_{\tau^{3}} (\Gamma,p \Rightarrow \top,\Delta ) $,   $ H \in Th_{\tau^{3}}
(\Gamma,p \Rightarrow \Delta )) $  thus a proof is constructed by replacing
top-down  $ p $  in each  $ \Gamma ' $  with  $  \top$.

Let  $ \Gamma, \bot
\Rightarrow p,\Delta  $  (or  $ \Gamma \Rightarrow \top,p,\Delta  $,
 $ \cfrac{G\vert \Gamma \Rightarrow \Delta }{G\vert \Gamma \Rightarrow p,\Delta
}(WR)) $  is a leaf of  $ \tau^{3} $  then there exists  $ \Gamma ' \Rightarrow
\Delta ' \in H $  such that  $ p \in \Delta ' $  for all  $ H \in Th_{\tau^{3}}
(\Gamma, \bot \Rightarrow p,\Delta ) $  (accordingly  $ H \in Th_{\tau^{3}}
(\Gamma \Rightarrow \top,p,\Delta ) $  or  $ H \in Th_{\tau^{3}} (\Gamma
\Rightarrow p,\Delta )) $  thus a proof is constructed by replacing top-down
 $ p $  in each  $ \Gamma ' $  with  $  \bot$.

 Repeatedly applying the procedure
above, we construct a proof $ \tau^{4}  $ of $ G_{2}\vert G_{2}^ * $  in  $ {\rm {\bf GL}} $ such that there doesn't exist occurrence of  $ p $
in  $ \Gamma  $  or  $ \Delta  $  at each leaf labeled by  $ \Gamma, \bot \Rightarrow
\Delta  $  or  $ \Gamma \Rightarrow \top,\Delta  $,   or  $ p $  is not the weakening
formula  $ A $  in  $ \cfrac{G\vert \Gamma \Rightarrow \Delta }{G\vert \Gamma
\Rightarrow A,\Delta }(WR) $  or  $ \cfrac{G\vert \Gamma \Rightarrow \Delta
}{G\vert \Gamma,A \Rightarrow \Delta }(WL) $  when  $ (WR) $  or  $ (WL) $  is
available. Define two operations $\sigma _l$ and $\sigma _r$  on sequents by  $\sigma _l(\Gamma,p\Rightarrow \Delta)\coloneq\Gamma,\top\Rightarrow \Delta$
and $\sigma _r(\Gamma\Rightarrow p,\Delta)\coloneq\Gamma \Rightarrow \bot,\Delta$. Then $ G_{2}\vert G_{2}^ * $ is obtained by applying $\sigma _l$ and $\sigma _r$ to some designated sequents in $ G_{1}\vert G_{1}^ * $.

\begin{definition}
The manipulation described in Step 4  is called  eigenvariable-replacing operation.
\end{definition}

\textbf{Step 5 }A proof  $ \tau ^ *  $  is constructed from  $ \tau^{4}  $  by
assigning inductively one unique identification number to each occurrence of
 $ p $  in  $ \tau^{4}  $  as follows.

One unique identification number, which is a positive integer, is assigned
to each leaf of the form  $ p \Rightarrow p $  in  $ \tau^{4}  $  which corresponds
to  $ p_k \Rightarrow p_k  $  in  $ \tau ^ *$. Other nodes of $\tau^{4}$  are processed as follows.

$\bullet$ Let  $ \cfrac{G_{1}\vert \Gamma,\lambda p \Rightarrow \mu p,\Delta }{G_{1}\vert
\Gamma',\lambda p \Rightarrow \mu p,\Delta'}(I) \in \tau^{4}  $.  Suppose
that all occurrences of  $ p $  in  $ G_{1}\vert \Gamma, \lambda p \Rightarrow \mu
p,\Delta  $  are assigned identification numbers and have the form
$ G_{1}'\vert \Gamma,p_{i_1 }, \cdots
,p_{i_\lambda } \Rightarrow p_{j_1 }, \cdots,p_{j_\mu },\Delta $
in  $ \tau ^ *$,  which we often write as  $ G_{1}'\vert\Gamma,\{p_{i_k }\}_{k=1}^{\lambda}\Rightarrow
\{p_{j_k }\}_{k=1}^{\mu},\Delta.$
Then  $ G_{1}\vert \Gamma ',\lambda p \Rightarrow \mu
p,\Delta ' $  has the form
$ G_{1}'\vert \Gamma ',\{p_{i_{k} }\}_{k=1}^{ \lambda}
\Rightarrow \{p_{j_{k}}\}_{k=1}^{\mu},\Delta'.$

$\bullet$  Let $ \cfrac{G' \quad  G'' }{G'''}( \wedge _{rw} ) \in \tau^{4}$,  where
$G' \equiv G_1 \vert \Gamma,\lambda p \Rightarrow \mu p,A,\Delta, $
 $ G'' \equiv G_2 \vert \Gamma,\lambda p \Rightarrow \mu p,B,\Delta,$
$G''' \equiv G_1 \vert G_2 \vert \Gamma
,\lambda p \Rightarrow \mu p,A \wedge B,\Delta
 \vert \Gamma,\lambda p \Rightarrow \mu p,A \wedge B,\Delta.$
  Suppose that  $ G' $  and
 $ G'' $  have the forms $G_1' \vert \Gamma,\{p_{i_{1k} }\}_{k=1}^{\lambda}\Rightarrow \{p_{j_{1k} }\}_{k=1}^{\mu},A, \Delta
$  and
$ G_2' \vert \Gamma,\{p_{i_{2k} }\}_{k=1}^{\lambda}\Rightarrow \{p_{j_{2k} }\}_{k=1}^{\mu},B, \Delta$  in  $ \tau ^ * $, respectively.
Then  $ G''' $  has the form
$G_1 '\vert G_2' \vert \Gamma,
 \{p_{i_{1k} }\}_{k=1}^{\lambda}\Rightarrow \{p_{j_{1k} }\}_{k=1}^{\mu},A \wedge B, \Delta
\vert \Gamma,\{p_{i_{2k} }\}_{k=1}^{\lambda}\Rightarrow \{p_{j_{2k} }\}_{k=1}^{\mu},A \wedge B, \Delta. $ All applications of  $  (\vee _{lw}) $  are processed by the procedure similar to that of  $(\wedge _{rw})$.

$\bullet$  Let$\cfrac{G' \quad  G''}{G'''}(\odot  _r ) \in
\tau^{4}$,  where
 $ G' \equiv G_1 \vert \Gamma _1,\lambda _1 p \Rightarrow \mu _1
p,A,\Delta _1, $\\
$G'' \equiv G_2 \vert \Gamma _2,\lambda _2 p
\Rightarrow \mu _2 p,B,\Delta _2, $
$ G''' \equiv G_1 \vert G_2 \vert
\Gamma _1,\Gamma _2,(\lambda _1 + \lambda _2 )p \Rightarrow (\mu _1 + \mu
_2 )p,A\odot  B,\Delta _1,\Delta _2. $
  Suppose that  $ G' $  and  $ G'' $  have the forms
$ G_1'\vert \Gamma _1,\{p_{i_{1k} }\}_{k=1}^{\lambda_{1}}\Rightarrow \{p_{j_{1k} }\}_{k=1}^{\mu_{1}},
A,\Delta _1 $ and
$G_2' \vert \Gamma_2,\{p_{i_{2k} }\}_{k=1}^{\lambda_{2}}\Rightarrow \{p_{j_{2k} }\}_{k=1}^{\mu_{2}},B,\Delta _2 $  in  $ \tau ^ *$,   respectively.
 Then  $ G'''  $  has the
form  $ G_1 '\vert G_2 '\vert \Gamma _1,\Gamma _2,\{p_{i_{1k} }\}_{k=1}^{\lambda_{1}},\{p_{i_{2k} }\}_{k=1}^{\lambda_{2}}\Rightarrow \{p_{j_{1k} }\}_{k=1}^{\mu_{1}}, \{p_{j_{2k} }\}_{k=1}^{\mu_{2}},
 A\odot  B,\Delta _1,\Delta _2.$
All  applications of  $  (\to_l)  $  are processed by the procedure similar to that of  $ (\odot _r)$.

$\bullet$  Let $ \cfrac{G'  \quad G''}{G'''}(COM) \in \tau^{4}$,  where
$ G' \equiv G_1 \vert \Gamma _1,\Pi _1,\lambda _1 p \Rightarrow
\mu _1 p,\Sigma _1,\Delta _1, $\\
$ G'' \equiv G_2 \vert \Gamma _2,\Pi
_2,\lambda _2 p \Rightarrow \mu _2 p,\Sigma _2,\Delta _2,$
$G''' \equiv  G_1 \vert G_2 \vert \Gamma _1,\Gamma _2,(\lambda _{11} + \lambda
_{21} )p \Rightarrow (\mu _{11} + \mu _{21} )p,\Delta _1,\Delta _2\vert$\\
$\Pi _1,\Pi _2,(\lambda _{12} + \lambda _{22} )p \Rightarrow (\mu _{12} +
\mu _{22} )p,\Sigma _1,\Sigma _2,  $
 where  $ \lambda _{11} + \lambda _{12} = \lambda _1,\lambda _{21} +
\lambda _{22} = \lambda _2  $,   $ \mu _{11} + \mu _{12} = \mu _1,\mu _{21} +
\mu _{22} = \mu _2$.

Suppose that  $ G' $  and  $ G'' $  have the forms
$ G_1' \vert \Gamma _1,\Pi _1, \{p_{i_{k}^{1} }\}_{k=1}^{ \lambda_{1}}
\Rightarrow \{p_{j_{k}^{1}}\}_{k=1}^{\mu_{1}},\Sigma _1,\Delta _1$ and\\
$ G_2' \vert \Gamma _2,\Pi _2,\{p_{i_{k}^{2} }\}_{k=1}^{ \lambda_{2}}
\Rightarrow \{p_{j_{k}^{2}}\}_{k=1}^{\mu_{2}},\Sigma _2,\Delta _2$  in  $ \tau ^ *  $,  respectively.   Then  $ G'''  $  has the form\\
$ G_1' \vert G_2 '\vert \Gamma _1,\Gamma _2,\{p_{i_{1k}^{1} }\}_{k=1}^{ \lambda_{11}},\{p_{i_{1k}^{2} }\}_{k=1}^{ \lambda_{21}} \Rightarrow  \{p_{j_{1k}^{1}}\}_{k=1}^{\mu_{11}}, \{p_{j_{1k}^{2}}\}_{k=1}^{\mu_{21}},\Delta _1, \Delta _2 \vert$\\
$\Pi _1,\Pi _2,\{p_{i_{2k}^{1} }\}_{k=1}^{ \lambda_{12}},\{p_{i_{2k}^{2} }\}_{k=1}^{ \lambda_{22}} \Rightarrow
\{p_{j_{2k}^{1}}\}_{k=1}^{\mu_{12}}, \{p_{j_{2k}^{2}}\}_{k=1}^{\mu_{22}},\Sigma _1,\Sigma _2$,  where\\
$\{p_{i_{k}^{w} }\}_{k=1}^{ \lambda_{w}}=\{p_{i_{1k}^{w} }\}_{k=1}^{ \lambda_{w1}}\bigcup\{p_{i_{2k}^{w} }\}_{k=1}^{ \lambda_{w2}},\{p_{j_{k}^{w}}\}_{k=1}^{\mu_{w}}
 =\{p_{j_{1k}^{w}}\}_{k=1}^{\mu_{w1}}\bigcup\{p_{j_{2k}^{w}}\}_{k=1}^{\mu_{w2}}  $
for $w=1,2$.

\begin{definition}
The manipulation described in Step 5  is  called  eigenvariable-labeling operation.
\end{definition}

\begin{notation}
Let $G_{2}$  and $G_{2}^ * $  be converted to $G$  and $G^*$  in $ \tau ^ * $, respectively. Then $ \tau ^ *  $ is a proof  of $ G\vert G^*$.
\end{notation}

  In preprocessing of $\tau$,  each $ \cfrac{G_i'''\vert \{S_i^c \}^{m_i' }}{G_i'''\vert S_i^c}(EC^{\ast})_{i} $ is converted  into
  $ \cfrac{G_i''\vert \{S_i^c\}^{m_i' }}{G_i ''\vert \{S_i^c \}^{m_i' }}(ID_\Omega )_{i} $ in Step 2,  where $ G_i''' \subseteq G_i''$ by Lemma 4.4.
   $ \cfrac{G'  }{G'\vert S'}(EW)\in\tau^{2}  $ is converted
  into  $ \cfrac{G''  }{G''}(ID_\Omega )$ in Step 3,  where $ G'' \subseteq G'$ by Lemma 4.8(i).   Some  $G' \vert \Gamma ',p \Rightarrow\Delta ' \in\tau^{3}$  (or $G' \vert \Gamma ' \Rightarrow p,\Delta '$ ) is revised as  $G' \vert \Gamma ',\top \Rightarrow \Delta '$  (or  $G' \vert \Gamma ' \Rightarrow\bot,\Delta '$) in Step 4.   Each occurrence of $p$ in $\tau^{4}$ is  assigned the unique identification number in Step 5.    The whole preprocessing above is depicted by Figure 13.
\large
\[
\xrightarrow{\tau}{} G_0 \xrightarrow[{ \wedge _{rw}, \vee _{lw} }]{{Step\,\,\,   1:\,\,\tau^{1} }}G_0 \xrightarrow[{ EC}]{{Step   \,\,\,2:\,\,\tau^{2}}}G_0 |G_0^ *
\xrightarrow[{EW }]{{Step\,\,\,   3:\,\,\tau^{3}}}G_1 |G_1^* \] \[\xrightarrow[{\top, \bot,W }]{{Step \,\,\,4:\,\,\tau^{4}}}G_2 |G_2^*  \xrightarrow[{ID \,\,\,numbers}]{{Step\,\,\,   5:\,\,\tau^*}}G|G^*\qquad\qquad\qquad\qquad
\]
\normalsize
\begin{center}
 \footnotesize  FIGURE 13\normalsize\quad Preprocessing of  $\tau$
\end{center}

\begin{notation}
Let  $ \cfrac{G_i '''\vert \{S_i^c\}^{m_i' }}{G_i'''\vert S_i^c}(EC^{\ast})_{i}, 1\leqslant  i \leqslant  N $  be
all  $ (EC^*)$-nodes of  $ \tau^{1}$ and $G_i'''\vert \{S_i^c \}^{m_i' }$ be converted
 to $G_i''\vert \{S_i^c\}^{m_i}$ in  $ \tau ^*$.  Note that there are no identification numbers for
  occurrences of variable $p$  in  $S_i^c\in G_i'''\vert \{S_i^c\}^{m_i' }$ meanwhile they  are assigned to
$p$   in  $S_i^c\in G_i ''\vert \{S_i^c\}^{m_i }$.  But we use the same notations
for  $S_i^c\in G_i '''\vert \{S_i^c \}^{m_i' }$ and  $S_i^c\in G_i ''\vert \{S_i^c\}^{m_i }$
 for simplicity.

In the whole paper,  let  $ H_i^c=G_{i}' \vert \{S_i^c\}^{m_i } $  denote the unique node of  $ \tau^*  $  such that  $ H_i^c \leqslant G_{i}'' \vert \{S_i^c\}^{m_i} $  and  $ S_i^c  $  is the focus sequent of  $ H_i^c  $  in  $ \tau ^ *  $,   in which case we denote the focus one
 $ S_{i1}^c  $  and others  $ S_{i2}^c \vert \cdots \vert S_{im_i }^c  $  among
 $ \{S_i^c \}^{m_i }$.  We sometimes denote  $ H_i^c  $  also by  $ G_{i }'\vert \{S_{iv}^c \}_{v =
1}^{m_i }  $  or  $ G_{i }'\vert S_{i1}^c \vert \{S_{iv}^c \}_{v = 2}^{m_i }  $.   We then write $G^ *  $  as  $ \{S_{iv}^c \}_{i = 1 \cdots N}^{v = 2 \cdots m_i }$.

We call $ H_i^c  $, $ S_{iu}^c  $ the $i$-th  pseudo-$(EC)$ node of  $ \tau^*  $ and pseudo-$(EC)$ sequent, respectively.    We abbreviate  pseudo-$EC$ as $pEC. $ Let  $ H\in \tau^ *  $, by  $ S_i^c  \in  H$ we mean that  $ S_{iu}^c  \in  H$ for some
$1\leqslant  u \leqslant  m_{i}$.

It is possible that  there doesn't  exist  $ H_i^c \leqslant G_{i}'' \vert \{S_i^c \}^{m_i} $
such that  $ S_i^c  $  is the focus sequent of  $ H_i^c  $  in  $ \tau ^ *  $,
  in which case $\{S_i^c \}^{m_i}\subseteq G\vert G^{\ast} $,  then it hasn't any
  effect on  our argument  to treat all such  $S_i^c$ as members of $G$.
  So we assume that  all $ H_i^c$ are always defined for all $G_i''\vert \{S_i^c\}^{m_i}$ in  $ \tau ^ *$, i.e.,  $ H_i^c>G\vert G^{\ast}  $.
\end{notation}

\begin{proposition}
(i)  $ \{S_{iv}^c \}_{v = 2 \cdots m_i }\subseteq H $ for all $H\leqslant  H_i^c$;
(ii) $ G^ *= \{S_{iv}^c \}_{i = 1 \cdots N}^{v = 2 \cdots m_i }$.

\end{proposition}

Now,  we replace locally each $ \cfrac{G'  }{G'}(ID_\Omega )$ in $\tau^{\ast}$ with $ G' $ and denote the resulting proof also by   $\tau^{\ast}$, which has no essential difference with the original one but could simplify subsequent arguments.
We introduce the system  $ {\rm {\bf GL}}_\Omega$ as follows.

\begin{definition}
 $ {\rm {\bf GL}}_\Omega  $  is a restricted subsystem of
 $ {\rm {\bf GL}} $  such that

(i)  $ p $  is designated as the unique eigenvariable
by which we mean that it is not used to built up any formula
containing logical connectives and only used as a sequent-formula.

(ii) Each occurrence of  $ p $ on each side of  every component  of  a hypersequent in $ {\rm {\bf GL}}$ is assigned
one unique identification number  $ i $   and written as  $ p_i$  in  $ {\rm {\bf GL}}_\Omega$.
Initial sequent  $ p\Rightarrow p $  of  $ {\rm {\bf GL}} $ has the form
 $ p_{i} \Rightarrow p_{i}$  in  $ {\rm {\bf GL}}_\Omega$.

 (iii) Each sequent  $ S $  of  $ {\rm {\bf GL}} $  in the form  $ \Gamma,\lambda p
\Rightarrow \mu p,\Delta  $  has the form  \[ \Gamma,\{p_{i_k }\}_{k=1}^{\lambda}\Rightarrow
\{p_{j_k }\}_{k=1}^{\mu},\Delta\]  in  $ {\rm
{\bf GL}}_\Omega  $,  where  $ p $  does not occur in  $ \Gamma  $  or  $ \Delta  $,    $ i_k
\ne i_l $  for all  $ 1 \leqslant k < l \leqslant \lambda  $,    $ j_k\ne j_l  $  for
all  $ 1 \leqslant k < l \leqslant \mu$.
 Define  $ v_l (S) = \{i_1, \cdots
,i_\lambda \} $  and  $ v_r (S) = \{j_1, \cdots,j_\mu \}$.  Let  $ G $  be a hypersequent of  $ {\rm {\bf
GL}}_\Omega  $ in the form  $ S_1 \vert \cdots \vert S_n  $  then  $ v_l (S_k) \bigcap v_l (S_l) =
\emptyset  $  and  $ v_r (S_k) \bigcap v_r (S_l) = \emptyset  $  for all  $ 1
\leqslant k < l \leqslant n$.  Define  $ v_l (G) = \bigcup _{k = 1}^n v_l (S_k
) $,    $ v_r (G) = \bigcup _{k = 1}^n v_r (S_k )$.  Here, $l$  and  $r$   in $ v_l$ and $v_r$  indicate the left side and right side of a sequent, respectively.

(iv) A hypersequent $ G $  of  $ {\rm {\bf
GL}}_\Omega  $ is called closed if  $ v_l (G) = v_r (G)$.  Two hypersequents $ G' $  and  $ G'' $  of  $ {\rm {\bf GL}}_\Omega  $ are called disjoint if  $ v_l (G') \bigcap v_l (G'') = \emptyset
 $,   $ v_l (G') \bigcap v_r (G'') = \emptyset  $,   $ v_r (G') \bigcap v_l (G'') =
\emptyset  $  and  $ v_r (G') \bigcap v_r (G'') = \emptyset$.   $ G'' $  is a copy of
 $ G' $  if they are disjoint and there exist two bijections  $ \sigma _l :v_l
(G') \to v_l (G'') $  and  $ \sigma _r :v_r (G') \to v_r (G'') $  such that  $ G'' $
can be obtained by applying  $ \sigma _l  $  to antecedents of sequents in  $ G' $
and  $ \sigma _r  $  to succedents of sequents in  $ G'$,  i.e., $ G''=  \sigma _r \circ\sigma _l(G')$.

(v) A closed hypersequent  $ G'\vert G''\vert G'''  $  can be contracted as  $ G'\vert
G''  $  in  $ {\rm {\bf GL}}_\Omega  $  under the condition that  $ G''  $  and  $ G'''  $
are closed and  $ G'''  $  is a copy of  $ G''  $.   We call it the constraint
external contraction rule and denote by
 \[ \cfrac{
 G'\vert G''\vert G''' }{G'\vert G'' }(EC_\Omega ). \]
 Furthermore, if there doesn't exist
two closed hypersequents $H',H''\subseteq G'\vert G'' $ such that $ H''  $  is a copy of  $ H'  $ then we
call it the fully constraint contraction rule and denote by  $ \cfrac{\underline{G'\vert G'' \vert G'''}}{G'\vert G'' }\left\langle {EC_\Omega ^ *} \right\rangle $.

 (vi)  $ (EW)$  and $(CUT)$ of  $ {\rm {\bf GL}} $
are forbidden. $ (EC)$,  $ ( \wedge _r ) $  and  $ ( \vee _l ) $  of  $ {\rm {\bf GL}} $  are
replaced with $ (EC_\Omega) $,  $ ( \wedge _{rw} ) $  and  $ ( \vee _{lw} ) $  in  $ {\rm {\bf
GL}}_\Omega  $,   respectively.

(vii) $ G_1 \vert S_1   $  and  $ G_2 \vert S_2  $  are closed
and disjoint  for each two-premise rule \\
$ \cfrac{  G_1 \vert S_1 \quad G_2 \vert S_2 }{G_1 \vert G_2 \vert H'}(II) $  of  $ {\rm {\bf GL}}_\Omega$ and, $ G' \vert S' $  is closed
for each one-premise rule  $ \cfrac{  G' \vert S' }{G' \vert S''}(I) $.

(viii)  $ p $  doesn't occur in  $ \Gamma  $  or  $ \Delta  $  for each initial sequent
 $ \Gamma, \bot \Rightarrow \Delta  $  or  $ \Gamma \Rightarrow \top,\Delta
 $ and,  $ p $  doesn't act as the weakening formula  $ A $  in  $ \cfrac{G\vert \Gamma
\Rightarrow \Delta }{G\vert \Gamma \Rightarrow A,\Delta }(WR) $  or
 $ \cfrac{G\vert \Gamma \Rightarrow \Delta }{G\vert \Gamma,A \Rightarrow
\Delta }(WL)$  when  $(WR)$  or  $(WL)$  is available.

\end{definition}

\begin{lemma}
Let  $ \tau  $  be a cut-free proof of  $ G_0  $  in  $ {\rm {\bf
L}} $  and  $ \tau^ *  $  be the tree  resulting from preprocessing of  $ \tau$.

(i) If $ \cfrac{G'\vert S' }{G'\vert S''}(I)\in\tau^ * $ then $ v_l (G'\vert
S'') = v_r (G'\vert S'') = v_r (G'\vert S') = v_l (G'\vert S') $;

 (ii) If $ \cfrac{
 G' \vert S' \quad G'' \vert S'' }{G' \vert G'' \vert
H'}(II)\in\tau^ * $  then
 $ v_l
(G'\vert G''\vert H') = $
$v_l (G' \vert S' )
\bigcup v_l (G'' \vert S'' )=
v_r (G'\vert G''\vert H') = v_r (G' \vert S' ) \bigcup v_r (G'' \vert
S'' )$;

(iii)  If  $H\in\tau^ *$ and   $k\in v_l(H)$ then $k\in v_r(H)$;

(iv) If  $H\in\tau^ *$ and  $k\in v_l(H)$ (or $k\in v_r(H)$) then $H\leqslant  p_{k}\Rightarrow p_{k}$;

 (v)  $ \tau ^ *  $  is a proof of  $ G\vert G^ *  $  in  $ {\rm
{\bf GL}}_\Omega$ without $ (EC_\Omega) $;

 (vi) If  $H',H''\in\tau^ *$  and  $H'\| H''$ then  $v_l(H')\bigcap v_l(H'')=\emptyset$,
$v_r(H')\bigcap v_r(H'')=\emptyset$.
\end{lemma}

\begin{proof}
Claims from (i) to (iv) are immediately  from Step 5 in preprocessing of  $ \tau$ and Definition 4.16.
 (v) is from Notation 4.13 and Definition 4.16.
 Only (vi) is proved as follows.

 Suppose that $k\in v_l(H')\bigcap v_l(H'')$.  Then $H'\leqslant  p_{k}\Rightarrow p_{k}$,  $H''\leqslant  p_{k}\Rightarrow p_{k}$ by Claim (iv).  Thus $H'\leqslant  H''$ or $H''\leqslant  H'$,  a contradiction with $H'\| H''$ hence $v_l(H')\bigcap v_l(H'')=\emptyset$.\\  $v_r(H')\bigcap v_r(H'')=\emptyset$ is proved by a similar procedure and omitted.
\end{proof}

\section{The generalized density rule $(\mathcal{D})$  for $ {\bf GL}_{\Omega}$}
In this section, we define the generalized density rule $(\mathcal{D})$   for  $ {\bf GL}_{\Omega}$ and prove that it is admissible in ${\rm {\bf GL}}_\Omega$.
\begin{definition}
Let  $ G $  be a closed hypersequent of  $ {\rm {\bf GL}}_\Omega  $  and  $ S \in G$.
Define  $ \left[ S \right]_G=\bigcap\{H:S\in H\subseteq G,v_l (H ) = v_r (H) \}$, i.e., $\left[ S \right]_G$ is the
minimal closed unit of $G$ containing  $S.$  In general,
for  $G'\subseteq G$,   define $ \left[ G'\right]_G=\bigcap\{H:G'\subseteq H\subseteq G,v_l (H ) = v_r (H) \}$.

\end{definition}

Clearly,  $ \left[ S \right]_G = S$  if  $ v_l ( S ) = v_r ( S) $  or  $ p $  does not occur in  $S$.  The following construction  gives a procedure  to construct  $\left[ S \right]_G$ for any given $S\in G$.

\begin{construction}
Let  $ G$ and $S $ be as above.  A sequence  $ G_1,G_2, \cdots,G_n  $  of hypersequents is constructed recursively as follows.
(i) $ G_1 = \{S\} $;
(ii) Suppose that  $ G_k  $  is constructed for  $ k \geqslant 1$.  If  $ v_l ( {G_k } ) \ne v_r ( {G_k } ) $  then there exists $ i_{k+1} \in v_l ( {G_k } )\backslash v_r ( {G_k } ) $
(or  $ i_{k+1} \in v_r ( {G_k } )\backslash v_l ( {G_k } )) $
thus there exists the unique $ S_{k + 1} \in G\backslash G_k  $  such that  $ i_{k+1} \in v_r
( {S_{k + 1} } )\backslash v_l( {S_{k + 1} } ) $ (or  $ i_{k+1} \in v_l ( {S_{k + 1} } )\backslash v_r( {S_{k + 1} } )) $  by  $ v_l (G ) = v_r (G) $
 and Definition 4.16  then  let  $ G_{k + 1} = G_k \vert S_{k + 1}  $
otherwise the procedure terminates and $ n\coloneq k $.
\end{construction}

\begin{lemma}
(i)  $ G_n =\left[ S \right]_G$;\\
(ii)Let $S'\in \left[ S \right]_G$ then  $\left[ S' \right]_G=\left[ S \right]_G$;\\
(iii)Let $G'\equiv G\vert H'$,    $G''\equiv G\vert H''$,$v_{l}(G')=v_{r}( G')$,
$v_{l}(G'')=v_{r}( G'')$,  $v_l (H')\ominus v_r ( H' )=
v_l (H'')\ominus v_r ( H'' ) $  then  $\left[ H' \right]_{G'}\backslash H'=\left[ H'' \right]_{G''}\backslash H''$, where $A\ominus B$ is the symmetric difference of two multisets $A, B$;\\
(iv)Let  $ v_{lr} ( {G_k } ) = v_l ( {G_k } ) \bigcap v_r
( {G_k } )$ then $ \left| {v_{lr} ( {G_k } )} \right| +
1 \geqslant \left| {G_k } \right| $  for all  $ 1 \leqslant k \leqslant n$;\\
(v)  $ \left| {v_l ( {\left[ S\right]_G } )} \right| + 1
\geqslant \left| {\left[ S\right]_G } \right|. $
\end{lemma}

\begin{proof}
(i) Since $G_{k}\subseteq G_{k+1}$  for $1\leqslant  k\leqslant  n-1$ and $S\in G_{1}$ then $S\in G_{n}\subseteq G$
thus $\left[ S \right]_G\subseteq G_{n}$ by $v_l (G_{n} ) = v_r (G_{n}) $.   We prove  $ G_{k}\subseteq\left[ S \right]_G $ for $1\leqslant  k\leqslant  n$  by induction on  $ k$ in the  following.  Clearly, $G_1\subseteq \left[ S \right]_G$. Suppose that $G_k\subseteq \left[ S \right]_G$ for some $1\leqslant  k\leqslant  n-1$.
Since  $ i_{k+1} \in v_l (G_k)\backslash v_r ( {G_k } ) $ (or  $ i_{k+1} \in v_r ( {G_k } )\backslash v_l ( {G_k } )) $  and $ i_{k+1} \in v_r ( {S_{k + 1} } ) $ (or  $ i_{k+1} \in v_l ( {S_{k + 1} } )) $ then $S_{k+1}\in\left[ S \right]_G$ by $G_k\subseteq \left[ S \right]_G$ and $v_l (\left[ S \right]_G ) = v_r (\left[ S \right]_G)$ thus $G_{k+1}\subseteq \left[ S \right]_G$. Then $ G_{n}\subseteq\left[ S \right]_G $ thus
$ G_n =\left[ S \right]_G$.

(ii) By (i),  $ \left[ S \right]_G=S_{1}\vert S_{2}\vert \cdots\vert S_{n}$, where $S_{1}=S$.  Then $S'=S_{k}$ for some $1\leqslant  k\leqslant  n$ thus $ i_{k} \in v_r( {S_{k} } )\backslash v_l ( {S_{k} } ) $ (or  $ i_{k} \in v_l ( {S_{k } } )\backslash v_r( {S_{k} } )) $  hence there exists the unique $k'<k$ such that $ i_{k} \in v_l ( {S_{k' } } )\backslash v_r( {S_{k'} } )$ (or  $ i_{k} \in v_r( {S_{k'} } )\backslash v_l ( {S_{k'} }) ) $ if $k\geq2$ hence $S_{k'}\in \left[ S_{k} \right]_G$. Repeatedly,  $S_{1}\in \left[ S_{k} \right]_G$, i.e.,  $S\in \left[ S' \right]_G$ then  $\left[ S\right]_G\subseteq\left[ S' \right]_G$.  $\left[ S'\right]_G\subseteq\left[ S \right]_G$ by   $S'\in \left[ S \right]_G$ then $\left[ S'\right]_G=\left[ S \right]_G$.

(iii)  It holds immediately from Construction 5.2 and (i).

(iv) The proof is by induction on  $ k$.  For the base step, let  $ k
= 1 $  then  $ \left| {G_k } \right| = 1 $  thus  $ \left| {v_{lr} ( {G_k }
)} \right| + 1 \geqslant \left| {G_k } \right| $  by  $ \left|
{v_{lr} ( {G_k } )} \right| \geqslant 0$.  For the induction step,
suppose that  $ \left| {v_{lr} ( {G_k } )} \right| + 1 \geqslant
\left| {G_k } \right| $  for some  $ 1 \leqslant k < n$.  Then  $ \left| {v_{lr}
( {G_{k + 1} } )} \right| \geqslant \left| {v_{lr} ( {G_k }
)} \right| + 1 $  by  $i_{k+1}  \in v_{lr} ( {G_{k + 1} } )\backslash
v_{lr} ( {G_k } ) $  and  $ v_{lr} ( {G_k } ) \subseteq
v_{lr} ( {G_{k + 1} } ) $.  Then  $ \left| {v_{lr} ( {G_{k + 1} } )} \right|
+ 1 \geqslant \left| {G_{k + 1} } \right| $  by  $ \left| {G_{k + 1} }
\right| = \left| {G_k } \right| + 1 = k + 1$.

 (v) It holds by (iv)  and
$ v_{lr} ( {\left[ S\right]_G } ) = v_l ( {\left[ S \right]_G } )$.
\end{proof}

\begin{definition}
Let  $ G = S_1 \vert \cdots \vert S_r  $  and  $ S_l  $  be in the
form  $ \Gamma _l, \{p_{i_{k}^l }\}_{k=1}^{ \lambda _{l} }\Rightarrow\{p_{j_{k}^l }\}_{k=1}^{\mu _l},\Delta _l  $  for  $ 1 \leqslant l\leqslant r$.

(i) If  $ S \in G $  and  $ \left[ S \right]_G  $  be  $ S_{k_1 } \vert \cdots \vert
S_{k_u }  $  then  $ \mathcal{D}_{G}( S ) $  is defined as\\
 $ \Gamma _{k_1 }, \cdots,\Gamma _{k_u } \Rightarrow ( {\left| {v_l
(\left[ S \right]_G )} \right| -\left|  \left[ S \right]_G\right| + 1} )t,\Delta
_{k_1 }, \cdots,\Delta _{k_u };
 $

(ii) Let  $  \bigcup _{k = 1}^v \left[ {S_{q_k } } \right]_G = G $  and $ \left[
{S_{q_k } } \right]_G \bigcap \left[ {S_{q_l } } \right]_G = \emptyset  $  for
all $ 1 \leqslant k < l \leqslant v $ then  $ \mathcal{D}( G ) $  is
defined as  $ \mathcal{D}_{G}( {S_{q_1 } } )\vert \cdots \vert
\mathcal{D}_{G}( {S_{q_v } } )$.

(iii) We call  $(\mathcal{D})$ the generalized density rule  of  $ {\bf GL}_{\Omega}$,  whose conclusion  $ \mathcal{D}( G ) $  is defined by (ii)  if its premise is $G$.
\end{definition}

Clearly,  $ \mathcal{D}( {p_k \Rightarrow p_k } ) $  is  $  \Rightarrow
t $  and  $ \mathcal{D}(S) = S $  if  $ p $  does not occur in  $S$.

\begin{lemma}
Let $G'\equiv G\vert S$ and $ G''\equiv G\vert S_{1}\vert S_{2}$ be closed and $ \left[ S_{1} \right]_{G''}\bigcap\left[ S_{2} \right]_{G''}=\emptyset$,  where
$S_{1}=\Gamma _1,\{p_{i_{k}^1 }\}_{k=1}^{ \lambda _{1} } \Rightarrow \{p_{j_{k}^1 }\}_{k=1}^{\mu _{1} },\Delta _1;$\,\,\,
$ S_{2}=\Gamma _2,\{p_{i_{k}^2 }\}_{k=1}^{ \lambda _{2} }\Rightarrow
\{p_{j_{k}^2 }\}_{k=1}^{ \mu _{2} },\Delta _2 ;$\\
$S =\Gamma _1,\Gamma _2,\{p_{i_{k}^1 }\}_{k=1}^{ \lambda _{1} },
\{p_{i_{k}^2 }\}_{k=1}^{ \lambda _{2} }\Rightarrow \{p_{j_{k}^1 }\}_{k=1}^{\mu _{1} }, \{p_{j_{k}^2 }\}_{k=1}^{\mu _{2} }, \Delta _1,\Delta _2;$
$ \mathcal{D}_{G''}(S_{1} )=\Gamma_{1},\Sigma_{1}\Rightarrow \Pi_{1},\Delta_{1}$ and
$ \mathcal{D}_{G''}(S_{2} )=\Gamma_{2},\Sigma_{2}\Rightarrow \Pi_{2}, \Delta_{2}.$ Then $ \mathcal{D}_{G'}(S)=\Gamma_{1},\Sigma_{1},\Gamma_{2},\Sigma_{2}\Rightarrow \Pi_{1},\Delta_{1}, \Pi_{2}, \Delta_{2}.$
\end{lemma}

\begin{proof}
Since  $ \left[ S_{1} \right]_{G''}\bigcap\left[ S_{2} \right]_{G''}=\emptyset $ then $ \left[ S \right]_{G'}= \left[ S_{1} \right]_{G''}\backslash \{S_{1}\}\bigcup\left[ S_{2} \right]_{G''}\backslash \{S_{2}\}\bigcup\\
 \{S\}$ by $v_{l}(S)=v_{l}( S_{1}\vert S_{2})$,  $v_{r}(S)=v_{r}( S_{1}\vert S_{2})$ and Lemma 5.3 (iii).  Thus\\
$\left| {v_l(\left[ S \right]_{G'})} \right| =\left| {v_l(\left[S_{1} \right]_{G''})} \right|+ \left| {v_l(\left[S_{2} \right]_{G''})} \right|$,  $\left| {\left[ S \right]_{G'}} \right| =\left| {\left[S_{1} \right]_{G''}} \right|+ \left| \left[S_{2} \right]_{G''} \right|-1$.  Hence  \[\left| {v_l(\left[ S \right]_{G'})} \right|-\left| {\left[ S \right]_{G'}} \right| +1=\left| {v_l(\left[S_{1} \right]_{G''})} \right|-\left| {\left[S_{1} \right]_{G''}} \right| +1+\left| {v_l(\left[S_{2} \right]_{G''})} \right|- \left| \left[S_{2} \right]_{G''} \right|+1.\]
 Therefore
$ \mathcal{D}_{G'}(S)=\Gamma_{1},\Sigma_{1},\Gamma_{2},\Sigma_{2}\Rightarrow \Pi_{1},\Delta_{1}, \Pi_{2}, \Delta_{2}  $ by
\[\Pi_{1}=(\left| {v_l(\left[S_{1} \right]_{G''})} \right|-\left| {\left[S_{1} \right]_{G''}} \right| +1)t,\Pi_{1}\backslash (\left| {v_l(\left[S_{1} \right]_{G''})} \right|-\left| {\left[S_{1} \right]_{G''}} \right| +1)t\]
\[\Pi_{2}=(\left| {v_l(\left[S_{2} \right]_{G''})} \right|-\left| {\left[S_{2} \right]_{G''}} \right| +1)t,\Pi_{2}\backslash (\left| {v_l(\left[S_{2} \right]_{G''})} \right|-\left| {\left[S_{2} \right]_{G''}} \right| +1)t\]
\[ \mathcal{D}_{G'}(S)=\Gamma_{1},\Sigma_{1},\Gamma_{2},\Sigma_{2}\Rightarrow(\left| {v_l(\left[ S \right]_{G'})} \right|-\left| {\left[ S \right]_{G'}} \right| +1)t,\quad\qquad\qquad\qquad\qquad\]
\[ \Pi_{1}\backslash (\left| {v_l(\left[S_{1} \right]_{G''})} \right|-\left| {\left[S_{1} \right]_{G''}} \right| +1)t,\Delta_{1}, \Pi_{2} \backslash(\left| {v_l(\left[S_{2} \right]_{G''})} \right|-\left| {\left[S_{2} \right]_{G''}} \right| +1)t, \Delta_{2}\] where $\lambda t=\{ \underbrace {t, \cdots, t}_\lambda \}$.
\end{proof}

\begin{lemma}$\mathrm{([A.5.1])}$
If there exists a proof $ \tau  $ of  $ G  $  in  $ {\rm {\bf GL}}_\Omega  $  then there exists a proof of  $ \mathcal{D}( {G} ) $  in  $ {\rm {\bf GL}}$,  i.e.,  $(\mathcal{D})$ is admissible in ${\rm {\bf GL}}_\Omega$.
\end{lemma}

\begin{proof}
 We proceed by induction on the height of $ \tau  $.  For the base step, if  $ G $  is  $ p_k
\Rightarrow p_k  $  then  $ \mathcal{D}( G ) $  is  $  \Rightarrow t $
otherwise  $ \mathcal{D}( G ) $  is  $ G $  then  $  \vdash _{{\rm {\bf
GL}}} \mathcal{D}( G ) $  holds. For the induction step,
the following cases are considered.

 $ \bullet $   Let \[ \cfrac{G'\vert S'}{G'\vert S''}( \to _r
) \in \tau\] where \[ S' \equiv A,\Gamma,\{p_{i_{k} }\}_{k=1}^{ \lambda}
\Rightarrow \{p_{j_{k}}\}_{k=1}^{\mu},\Delta, B,  \]
\[S'' \equiv
\Gamma,\{p_{i_{k} }\}_{k=1}^{ \lambda}
\Rightarrow \{p_{j_{k}}\}_{k=1}^{\mu},\Delta, A \to B. \]  Then $ \left[ S'' \right]_{G'\vert S''}= \left[ S' \right]_{G'\vert S'}\backslash \{S'\}\vert S''$
by $v_l(S')=v_l(S'')$,  $v_r(S')=v_r(S'')$ and Lemma 5.3 (iii).
Let   $ \mathcal{D}_{G'\vert S'} ( {S'} ) = A,\Gamma,\Gamma '' \Rightarrow \Delta '',\Delta,B $  then
 $ \mathcal{D}_{G'\vert S''} ( {S''} )= \Gamma,\Gamma ''
\Rightarrow \Delta '',\Delta, A \to B $  thus a proof of  $ \mathcal{D}(
{G'\vert S''} ) $  is constructed by combining the proof of
 $ \mathcal{D}( {G'\vert S'} ) $  and  $ \cfrac{\mathcal{D}_{G'\vert S'}
( {S'} )}{\mathcal{D}_{G'\vert S''} ( {S''} )}( \to _r
)$.  Other rules of type  $ (I) $  are processed by a procedure similar to
above.

 $ \bullet $   Let  \[ \cfrac{G_1 \vert S_1\quad  G_2 \vert S_2   }{G_1 \vert G_2 \vert S_3
}(\odot  _r ) \in \tau\] where
\[  S_1 \equiv \Gamma _1,\{p_{i_{k}^{1} }\}_{k=1}^{ \lambda_{1}}
\Rightarrow \{p_{j_{k}^{1}}\}_{k=1}^{\mu_{1}},A,\Delta _1\]
\[ S_2 \equiv
 \Gamma _2,\{p_{i_{k}^{2} }\}_{k=1}^{ \lambda_{2}}
\Rightarrow \{p_{j_{k}^{2}}\}_{k=1}^{\mu_{2}},B,\Delta _2 \]
\[ S_3 \equiv
\Gamma _1,\Gamma _2,\{p_{i_{k}^{1} }\}_{k=1}^{ \lambda_{1}}
,\{p_{i_{k}^{2} }\}_{k=1}^{ \lambda_{2}}
\Rightarrow \{p_{j_{k}^{2}}\}_{k=1}^{\mu_{2}},\{p_{j_{k}^{1}}\}_{k=1}^{\mu_{1}},  A\odot  B,\Delta
_1,\Delta _2.\]

 Let
\[\mathcal{D}_{G_1 \vert S_1 } (
{S_1 } ) = \Gamma _1,\Gamma _{11} \Rightarrow \Delta _{11},( {\left| {v_l
(\left[ {S_1 } \right]_{G_1 \vert S_1 } )} \right| - \left| {\left[ {S_1 }
\right]_{G_1 \vert S_1 } } \right| + 1} )t,A,\Delta _1,\]
\[ \mathcal{D}_{G_2 \vert S_2 } ( {S_2 } ) =\Gamma _2,\Gamma
_{21} \Rightarrow \Delta _{21},( {\left| {v_l (\left[ {S_2 }
\right]_{G_2 \vert S_2 } )} \right| - \left| {\left[ {S_2 } \right]_{G_2
\vert S_2 } } \right| + 1} )t,B,\Delta _2. \] Then
 $ \mathcal{D}_{G_1 \vert G_2 \vert S_3 } ( {S_3  } ) $  is
\[ \Gamma _1,\Gamma _2,\Gamma _{11},\Gamma _{21} \Rightarrow \Delta _{11},\Delta
_{21},A\odot  B,\Delta _1,\Delta _2,\]
\[({\left| {v_l (\left[ {S_1 } \right]_{G_1 \vert S_1 } )}
\right| + \left| {v_l (\left[ {S_2 } \right]_{G_2 \vert S_2 } )} \right| -
\left| {\left[ {S_1 } \right]_{G_1 \vert S_1 } } \right| - \left| {\left[
{S_2 } \right]_{G_2 \vert S_2 } } \right| + 2} )t \]  by
 $ \left[ {S_3 } \right]_{G_1 \vert G_2 \vert S_3 } = (
{\left[ {S_1 } \right]_{G_1 \vert S_1 } \backslash \{S_1 \}} ) \bigcup
( {\left[ {S_2 } \right]_{G_2 \vert S_2 } \backslash \{S_2 \}} )
\bigcup \{S_3 \}$.  Then the proof of  $ \mathcal{D}( {G_1 \vert G_2 \vert
S_3 } ) $  is constructed by combining  $  \vdash _{{\rm {\bf GL}}}
\mathcal{D}( {G_1 \vert S_1 } ) $  and  \\
$  \vdash _{{\rm {\bf GL}}}
\mathcal{D}( {G_2 \vert S_2 } ) $  with  $ \cfrac{\mathcal{D}_{G_1 \vert S_1 } (S_1 )\quad \mathcal{D}_{G_2\vert S_2}(S_2) }{\mathcal{D}_{G_1 \vert G_2 \vert S_3 } ( {S_3  } )}(\odot  _r )$.  All
applications of  $ (\to _l) $  are processed by a procedure similar to that of
 $ \odot  _r  $  and omitted.

 $ \bullet $  Let \[ \cfrac{G' \quad G''  }{G'''}( \wedge _{rw} )\in\tau  \]
 where
 \[ G' \equiv G_1 \vert S_1, \quad G'' \equiv G_2 \vert S_2,\quad G''' \equiv G_1 \vert G_2 \vert S_1' \vert S_2',\]
 \[ S_w \equiv \Gamma _w,\{p_{i_{k}^{w} }\}_{k=1}^{ \lambda_{w}}
\Rightarrow \{p_{j_{k}^{w}}\}_{k=1}^{\mu_{w}},A_{w},\Delta _w,  \]
\[ S_w' \equiv \Gamma _w,\{p_{i_{k}^{w} }\}_{k=1}^{ \lambda_{w}}
\Rightarrow \{p_{j_{k}^{w}}\}_{k=1}^{\mu_{w}},A_{1} \wedge A_{2},\Delta _w   \]
for $w=1,2.$
Then
 $ \left[ {S_1' } \right]_{G'''} =  {\left[ {S_1 } \right]_{G'}
\backslash \{S_1\} }  \vert S_1' $,
 $ \left[ {S_2' } \right]_{G'''} = {\left[ {S_2 } \right]_{G'' } \backslash  \{S_2\} } \vert  S_2' $  by Lemma 5.3 (iii).
Let  \[ \mathcal{D}_{G_w \vert S_w } ( {S_w} ) =
\Gamma _w,\Gamma _{w1} \Rightarrow \Delta _{w1}, ( {\left| {v_l (\left[ {S_w } \right]_{G_w \vert
S_w } )} \right| - \left| {\left[ {S_w } \right]_{G_w \vert S_w } } \right|
+ 1} )t,A_{w},\Delta _1 \]
 for $w=1,2.$  Then
\[\mathcal{D}_{G'''}( {S_w'  }) = \Gamma _w,\Gamma _{w1} \Rightarrow \Delta _{w1}, ( {\left| {v_l (\left[ {S_w }
\right]_{G_w \vert S_w} )} \right| - \left| {\left[ {S_w } \right]_{G_w
\vert S_w } } \right| + 1} )t, A_{1} \wedge A_{2},\Delta _w  \]
for $w=1,2.$  Then the proof of
 $ \mathcal{D}( {G''' } ) $  is
constructed by combining  $  \vdash _{{\rm {\bf GL}}} \mathcal{D}( {G'} ) $  and
$  \vdash _{{\rm {\bf GL}}} \mathcal{D}( {G'' } ) $  with
$ \cfrac{\mathcal{D}_{G'} ( {S_1 }) \quad
\mathcal{D}_{G''} ( {S_2 } )  }{\mathcal{D}_{G'''} (S_1'\vert S_2') }( \wedge _{rw} )$.  All applications of  $ (\vee_{lw}) $  are processed by a procedure similar to that of  $ (\wedge _{rw})  $  and
omitted.

 $ \bullet $   Let  \[ \cfrac{G' \quad G''  }{G'''}(COM)\in\tau  \] where
 \[ G' \equiv G_1 \vert S_1, \quad G'' \equiv G_2 \vert S_2,\quad G''' \equiv G_1 \vert G_2 \vert S_3 \vert S_4\]
 \[ S_1 \equiv \Gamma _1,\Pi _1,\{p_{i_{k}^{1} }\}_{k=1}^{ \lambda_{1}}
\Rightarrow \{p_{j_{k}^{1}}\}_{k=1}^{\mu_{1}}, \Sigma _1,\Delta _1,\]
\[ S_2 \equiv \Gamma _2,\Pi _2,\{p_{i_{k}^{2} }\}_{k=1}^{ \lambda_{2}}
\Rightarrow \{p_{j_{k}^{2}}\}_{k=1}^{\mu_{2}},\Sigma _2,\Delta _2,\]
\[ S_3 \equiv \Gamma _1,\Gamma _2,\{p_{i_{1k}^{1} }\}_{k=1}^{ \lambda_{11}},\{p_{i_{1k}^{2} }\}_{k=1}^{ \lambda_{21}} \Rightarrow  \{p_{j_{1k}^{1}}\}_{k=1}^{\mu_{11}}, \{p_{j_{1k}^{2}}\}_{k=1}^{\mu_{21}},\Delta _1,\Delta _2, \]
\[
S_4 \equiv \Pi _1,\Pi _2,\{p_{i_{2k}^{1} }\}_{k=1}^{ \lambda_{12}},
\{p_{i_{2k}^{2} }\}_{k=1}^{ \lambda_{22}} \Rightarrow  \{p_{j_{2k}^{1}}\}_{k=1}^{\mu_{12}}, \{p_{j_{2k}^{2}}\}_{k=1}^{\mu_{22}},\Sigma _1,\Sigma _2\]
where
$\{p_{i_{k}^{w} }\}_{k=1}^{ \lambda_{w}}=\{p_{i_{1k}^{w} }\}_{k=1}^{ \lambda_{w1}}\bigcup\{p_{i_{2k}^{w} }\}_{k=1}^{ \lambda_{w2}}, \{p_{j_{k}^{w}}\}_{k=1}^{\mu_{w}}
 =\{p_{j_{1k}^{w}}\}_{k=1}^{\mu_{w1}}\bigcup\{p_{j_{2k}^{w}}\}_{k=1}^{\mu_{w2}}$
 for $w=1,2.$

\textbf{Case 1}  $S_3\in \left[ {S_4 } \right]_{G'''}$.  Then
$ \left[ {S_3 } \right]_{G'''} = \left[ {S_4 } \right]_{G'''}
$ by Lemma 5.3 (ii) and \\
$ \left[
{S_3 } \right]_{G'''} = \left[ {S_1 } \right]_{G'} \vert \left[ {S_2 }
\right]_{G''} \vert S_3 \vert S_4 \backslash \{S_1,S_2 \}$ by Lemma 5.3 (iii).  Then
 \[ \left| {v_l (\left[ {S_3 } \right]_{G'''} )} \right| - \left| {\left[
{S_3 } \right]_{G'''} } \right| + 1 =\left| {v_l (\left[ {S_1 }
\right]_{G'} )} \right| + \left| {v_l (\left[ {S_2 } \right]_{G''} )}
\right| - \left| {\left[ {S_1 } \right]_{G'} } \right| - \left| {\left[ {S_2
} \right]_{G''} } \right| + 1 \geqslant 0. \]
Thus  $ \left| {v_l (\left[ {S_1 }
\right]_{G'} )} \right| - \left| {\left[ {S_1 } \right]_{G'} } \right| + 1
\geqslant 1 $  or  $ \left| {v_l (\left[ {S_2 } \right]_{G''} )} \right| -
\left| {\left[ {S_2 } \right]_{G''} } \right| + 1 \geqslant 1 $.
  Hence we
assume that, without loss of generality,\[\mathcal{D}_{G'} ( {S_1 })= \Gamma _1,\Pi _1,\Gamma ' \Rightarrow \Delta ',t,\Sigma _1,\Delta _1, \]
\[ \mathcal{D}_{G''} ( {S_2 } )= \Gamma _2,\Pi _2,\Gamma '' \Rightarrow \Delta '',
\Sigma _2,\Delta _2. \]   Then
 \[ \mathcal{D}_{G'''} ( {S_3 \vert S_4  } ) =
\Gamma _1,\Pi _1,\Gamma ',\Gamma _2,\Pi _2,\Gamma '' \Rightarrow
\Delta ',\Sigma _1,\Delta _1,\Delta '',\Sigma _2,\Delta _2. \]
    Thus the proof
of  $ \cfrac{\underline{\mathcal{D}_{G'}
( {S_1 })\quad  \mathcal{D}_{G''} ( {S_2 } )}}{\mathcal{D}_{G'''} ( {S_3
\vert S_4  } )} $  is constructed by
\[ \cfrac{\Gamma _1,\Pi _1,\Gamma ' \Rightarrow \Delta ',t,\Sigma _1,\Delta
_1\quad \cfrac{\Gamma _2,\Pi _2,\Gamma '' \Rightarrow \Delta
'',\Sigma _2,\Delta _2 }{\Gamma _2,\Pi _2,\Gamma '',t \Rightarrow
\Delta '',\Sigma _2,\Delta _2 }(t_l )}{\Gamma _1,\Pi _1,\Gamma ',\Gamma
_2,\Pi _2,\Gamma '' \Rightarrow \Delta ',\Sigma _1,\Delta _1,\Delta
'',\Sigma _2,\Delta _2 }(CUT).
\]

\textbf{Case 2}  $S_3\notin \left[ {S_4 } \right]_{G'''}$. Then
$ \left[ {S_3 } \right]_{G'''} \bigcap \left[ {S_4 }
\right]_{G'''} = \emptyset$  by Lemma 5.3 (ii).   Let
\[ S_{3w} \equiv \Gamma _w,
\{p_{i_{1k}^{w} }\}_{k=1}^{ \lambda_{w1}}\Rightarrow  \{p_{j_{1k}^{w}}\}_{k=1}^{\mu_{w1}}, \Delta _w, \]
\[ S_{4w} \equiv  \Pi _w,\{p_{i_{2k}^{w} }\}_{k=1}^{ \lambda_{w2}} \Rightarrow  \{p_{j_{2k}^{w}}\}_{k=1}^{\mu_{w2}},\Sigma _w, \]
 for $w=1,2.$  Then
\[ \left[ {S_3 } \right]_{G'''} = \left[ {S_{31} }
\right]_{G_{1}\vert S_{31} \vert S_{41}} \backslash \{S_{31} \}  \bigcup \left[
{S_{32} } \right]_{G_{2}\vert S_{32} \vert S_{42}} \backslash \{S_{32} \}  \bigcup
\{S_3\},\]
\[ \left[ {S_4 } \right]_{G'''} =
\left[ {S_{41} } \right]_{G_{1}\vert S_{31} \vert S_{41}} \backslash \{S_{41} \}
\bigcup \left[ {S_{42} } \right]_{G_{2}\vert S_{32} \vert S_{42}} \backslash \{S_{42} \} \bigcup \{S_4\} \]  by $v_l (S_{3} )=v_l (S_{31}\vert S_{32} )$,
$v_l (S_{1} )=v_l (S_{31}\vert S_{41} )$,  $v_l (S_{2} )=v_l (S_{32}\vert S_{42} )$  and\\
$v_l (S_{4} )=v_l (S_{41}\vert S_{42} )$. Let
\[ \mathcal{D}_{G_{w}\vert S_{3w} \vert S_{4w} } (
{S_{3w} } ) = \Gamma _w,{\rm X}_{3w} \Rightarrow \Psi _{3w}
,\Delta _w,\]
\[ \mathcal{D}_{G_{w}\vert S_{3w} \vert S_{4w} } ( {S_{4w}
} ) = \Pi _w,{\rm X}_{4w} \Rightarrow \Psi _{4w},\Sigma _w\]
 for  $w=1,2.$
 Then
\[ \mathcal{D}_{G'} ( {S_1 } ) = \Gamma _1,\Pi _1
,{\rm X}_{31},{\rm X}_{41} \Rightarrow \Psi _{31},\Psi _{41},\Sigma _1
,\Delta _1,\]
\[ \mathcal{D}_{G''} ( {S_2 } )= \Gamma _2,\Pi _2,{\rm
X}_{32},{\rm X}_{42} \Rightarrow \Psi _{32},\Psi _{42},\Sigma _2,\Delta
_2,\]
\[ \mathcal{D}_{G'''} ( {S_3 } ) = \Gamma _1,{\rm X}_{31}
,\Gamma _2,{\rm X}_{32} \Rightarrow \Psi _{31},\Delta _1,\Psi _{32}
,\Delta _2,\]
\[ \mathcal{D}_{G'''} ( {S_4 } ) = \Pi _1,{\rm X}_{41},\Pi _2
,{\rm X}_{42} \Rightarrow \Psi _{41},\Sigma _1,\Psi _{42},\Sigma _2
\] by Lemma 5.5, $ \left[ {S_3 } \right]_{G'''} \bigcap \left[ {S_4 }
\right]_{G'''} = \emptyset,   \left[ {S_{31} }
\right]_{G_{1}\vert S_{31} \vert S_{41}} \bigcap \left[ {S_{41} } \right]_{G_{1}\vert S_{31} \vert S_{41}}=\emptyset,\\
 \left[
{S_{32} } \right]_{G_{2}\vert S_{32} \vert S_{42}} \bigcap \left[ {S_{42} } \right]_{G_{2}\vert S_{32} \vert S_{42}}=\emptyset$.
Then
the proof of  $ \mathcal{D}_{G'''} ( {S_3 \vert S_4  }
) $  is constructed by combing the proofs of
$ \mathcal{D}_{G'} (
{S_1 } ) $  and  $ \mathcal{D}_{G''} ( {S_2 } ) $  with
$ \cfrac{   \mathcal{D}_{G'} ( {S_1 }) \quad \mathcal{D}_{G''} ( {S_2 } )}{\mathcal{D}_{G'''} ( {S_3
\vert S_4  } )}(COM).$

 $ \bullet $   $ \cfrac{G'\vert G'' \vert G''' }{G'\vert
G'' }(EC_\Omega ) \in \tau$.  Then $G', G'' $  and  $ G'''  $  are closed and  $ G''' $  is a copy of  $ G''  $ thus $ \mathcal{D}_{G'\vert G'' \vert G'''}({ G'' })= \mathcal{D}_{G'\vert G'' \vert G'''}({ G''' })$ hence a proof of  $ \mathcal{D}({G'\vert G''} ) $  is constructed by combining the proof of
 $\mathcal{D}( G'\vert G'' \vert G''') $  and  $ \cfrac{\mathcal{D}(G'\vert G'' \vert G''') }{\mathcal{D}({G'\vert G'' } )}(EC^* )$.
\end{proof}

The following two lemmas are corollaries  of Lemma 5.6.

\begin{lemma}
If there exists a derivation of  $ G_0  $  from  $ G_1,
\cdots,G_r  $  in  $ {\rm {\bf GL}}_\Omega  $  then there exists a derivation of
 $ \mathcal{D}( {G_0 } ) $  from  $ \mathcal{D}( {G_1 } ),
\cdots,\mathcal{D}( {G_r } ) $  in  $ {\rm {\bf GL}}$.
\end{lemma}

\begin{lemma}
Let  $ \tau  $  be a cut-free proof of  $ G_0  $  in  $ {\rm {\bf
GL}} $  and  $ \tau ^ *  $  be the proof of  $ G\vert G^ *  $  in  $ {\rm {\bf GL}}_\Omega  $
resulting from preprocessing of  $ \tau$. Then  $  \vdash _{\rm {\bf
GL}} \mathcal{D}( {G\vert G^ * } )$.
\end{lemma}

\section{Extraction of Elimination Rules}
In this section, we will investigate Construction 4.7
further to extract more derivations from $\tau ^ * $.

Any two sequents in a hypersequent seem independent of one another in the sense that they can only be contracted into one by $(EC)$  when it is applicable.  Note that one-premise logical rules just modify one sequent of a hypersequent and two-premise rules associate a sequent in a hypersequent with one in a different hypersequent.

$\tau^{\ast}$ (or any proof without $(EC_{\Omega})$ in $\mathbf{GL}_{\Omega}$)  has an essential property, which we  call the distinguishability of  $\tau^{\ast}$,  i.e.,  any variables,  formulas, sequents or hypersequents  which occur at the node $H$ of  $\tau^{\ast}$ occur inevitably  at  $H'<H$ in some forms.

Let  $ H\equiv  G'\vert S'\vert S'' \in \tau ^ *$.   If  $  S'$ is equal to   $S''$   as two sequents
then  the case  that $\tau_{H:S'}^{\ast}$  is equal to $\tau_{H:S''}^{\ast}$ as two derivations could possibly  happen.  This means that
both  $  S'$  and   $S''$  are the focus sequent of one node in $\tau^*$ when  $G_{H:S'}^{\ast}\neq S'$ and  $G_{H:S''}^{\ast}\neq S''$,  which contradicts that
each node has the unique focus sequent in any derivation.  Thus we need differentiate  $ S' $  from  $ S'' $  for all  $ G'\vert S'\vert S'' \in \tau ^ *$.

Define  $ \overline {S'} \in \tau ^ *  $  such that  $ G'\vert
S'\vert S'' \leqslant \overline {S'}  $,    $ S' \in \overline {S'}  $  and  $ S' $  is the principal sequent of  $ \overline {S'}  $.   If  $ \overline {S'}  $  has the unique principal sequent,  $ N_{S'} \coloneq 0 $,   otherwise  $ N_{S'} \coloneq 1 $  (or  $ N_{S'} = 2) $  to indicate that  $ S' $  is one designated  principal sequent (or accordingly  $ N_{S'} =2 $  for another) of such  an application as  $ (COM), ( \wedge _{rw} ) $  or  $ ( \vee _{lw} ) $.  Then we may regard  $ S' $  as
 $ (S'; \mathcal{P}(\overline {S'}),N_{S'} ) $.  Thus  $ S' $  is always different from  $ S'' $  by $ \mathcal{P}(\overline {S'})\neq \mathcal{P}(\overline {S''})  $  or,  $\mathcal{P}(\overline {S'}) =\mathcal{P}(\overline {S''})  $  and  $ N_{S'} \ne N_{S''}$.
 We formulate it by the following construction.

 \begin{construction}$\mathrm{([A.5.2])}$
 A labeled tree  $\tau ^{**}$, which has the same tree structure as $\tau^*$,   is constructed as follows.

 (i) If  $S$ is a leaf  $\tau^*$, define $ \overline {S} =S$,  $N_{S}=0$  and    the node $\mathcal{P}(S)$ of $\tau^{**}$ is labeled by  \\
 $(S; \mathcal{P}(\overline {S}),N_{S})$;

 (ii) If $\cfrac{G'\vert S'}{H\equiv G'\vert S''}(I)\in \tau ^*$ and  $\mathcal{P}(G'\vert S')$ be labeled by
 $\mathcal{G'}\vert (S'; \mathcal{P}(\overline {S'}),N_{S'})$ in $\tau^{**}$. Then define $\overline {S''}=H$, $N_{S''}=0$ and  the node $\mathcal{P}(H)$  of $\tau^{**}$ is labeled by $\mathcal{G'}\vert (S''; \mathcal{P}(\overline {S''}),N_{S''})$;

  (iii)  If  $ \cfrac{G'\vert S' \quad G''\vert S''}{H\equiv G'\vert G''\vert H'}(II) \in \tau ^ *  $,  $\mathcal{P}(G'\vert S')$ and   $\mathcal{P}(G''\vert S'')$ be  labeled by
  $\mathcal{G'}\vert (S'; \mathcal{P}(\overline {S'}),N_{S'})$ and   $\mathcal{G''}\vert (S''; \mathcal{P}(\overline {S''}),N_{S''})$ in $\tau^{**}$, respectively. If  $H'=S_{1}\vert  S_{2}$  then define $\overline {S_{1}}=\overline {S_{2}}=H$,   $N_{S_{1}}=1$,  $N_{S_{2}}=2$ and  the node $\mathcal{P}(H)$ of $\tau^{**}$ is  labeled by $\mathcal{G'}\vert \mathcal{G''}\vert  (S_{1}; \mathcal{P}(\overline {S_{1}}),N_{S_{1}})\vert (S_{2}; \mathcal{P}(\overline {S_{2}}),N_{S_{2}})$.    If  $H'=S_{1}$  then define $\overline {S_{1}}=H$,   $N_{S_{1}}=0$ and $\mathcal{P}(H)$  is  labeled by $\mathcal{G'}\vert\mathcal{ G''}\vert  (S_{1}; \mathcal{P}(\overline {S_{1}}),N_{S_{1}})$.
\end{construction}

In the whole paper,  we treat  $\tau^{*}$ as  $\tau^{**}$ without mention of  $\tau^{**}$.  Note that
in preprocessing of  $ \tau$, some logical applications could also be converted to $(ID_\Omega )$ in Step 3 and
we need fix the focus sequent  at  each node  $H$ and subsequently assign valid identification numbers to each $H'<H$ by eigenvariable-labeling operation.

\begin{proposition}
(i)   $ G'\vert S'\vert S'' \in \tau ^ *$ implies $ \{S'\} \bigcap \{S''\} = \emptyset$;
(ii) $H\in \tau ^ *  $ and $ H'\vert H'' \subseteq H  $  imply  $ H' \bigcap H'' = \emptyset$;
 (iii) Let  $ H\in \tau^ *  $ and  $ S_i^c  \in  H$  then $H\leqslant  H_i^c$ or $H_i^c\leqslant  H $.
\end{proposition}

\begin{proof}
(iii)  Let  $ S_i^c  \in  H$  then   $ S_i^c =S_{iu}^c $ for some $1\leqslant  u \leqslant  m_{i}$ by Notation 4.14.
Thus  $ S_i^c  \in H_i^c$ also by Notation 4.14.  Hence $H\leqslant  \overline {S_i^c }$ and $H_i^c\leqslant  \overline {S_i^c }$ by Construction 6.1.  Therefore $H\leqslant  H_i^c$ or $H_i^c \leqslant  H $.
\end{proof}

\begin{lemma}
Let  $ H \in \tau ^ *  $  and  $ Th(H) = ( {H_0, \cdots,H_n }
) $,   where  $ H_0 =H $,   $ H_n = G\vert G^*$,   $ G_k \subseteq H $  for  $ 1 \leqslant k \leqslant 3 $.

  (i)  If  $ G_3 = G_1
\bigcap G_2$ then $ \left\langle {H_i } \right\rangle _{H:G_3 } = \left\langle
{H_i } \right\rangle _{H:G_1 }  \bigcap \left\langle {
H_i } \right\rangle _{H:G_2 }  $  for all  $ 0 \leqslant i \leqslant n$;

  (ii)  If  $ G_3 = G_1\vert G_2$ then $ \left\langle {H_i } \right\rangle _{H:G_3 } = \left\langle
{H_i } \right\rangle _{H:G_1 }  \vert  \left\langle {
H_i } \right\rangle _{H:G_2 }  $  for all  $ 0 \leqslant i \leqslant n$.
\end{lemma}

\begin{proof}
The proof is by induction on  $ i $  for  $ 0 \leqslant i < n$.
Only (i) is proved as follows and (ii) by a similar procedure and omitted.

For the base step,  $ \left\langle {H_0 } \right\rangle _{H:G_3 } =
\left\langle {H_0 } \right\rangle _{H:G_1 }  \bigcap \left\langle
{ H_0 } \right\rangle _{H:G_2 }  $ holds by $ \left\langle {H_0 }
\right\rangle _{H:G_1 } = G_1   $,    $ \left\langle { H_0
} \right\rangle _{H:G_2 } = G_2  $,   $ \left\langle {H_0 } \right\rangle _{H:G_3
} = G_3  $  and $ G_3 = G_1 \bigcap G_2$.

For the induction step,  suppose that
 $ \left\langle {H_i } \right\rangle _{H:G_3 } = \left\langle {H_i }
\right\rangle _{H:G_1 }  \bigcap \left\langle { H_i }
\right\rangle _{H:G_2 }  $  for some  $ 0 \leqslant i < n$.  Only is the  case of one-premise rule  given in the following and other cases are omitted.

 Let  $ \cfrac{G'\vert S'}{G'\vert S''
}(I) \in \tau ^ *  $,   $ H_i = G'\vert S'  $  and $ H_{i + 1} = G'\vert
S''$.

 Let  $ S' \in \left\langle {H_i } \right\rangle
_{H:G_3 }  $.   Then  $ \left\langle {H_{i + 1} } \right\rangle _{H:G_3 } = (
{\left\langle {H_i } \right\rangle _{H:G_3 } \backslash \{S'\}} )\vert
S'' $,   \\
 $ \left\langle {H_{i + 1} } \right\rangle _{H:G_1 } = (
{\left\langle {H_i } \right\rangle _{H:G_1 } \backslash \{S'\}} )\vert
S'' $  by  $ S' \in \left\langle {H_i } \right\rangle _{H:G_1 }
 $  and  \\
 $ \left\langle {H_{i + 1} } \right\rangle _{H:G_2 } = (
{\left\langle {H_i } \right\rangle _{H:G_2 } \backslash \{S'\}} )\vert
S'' $  by  $ S' \in \left\langle {H_i } \right\rangle _{H:G_2 }  $.   Thus\\
 $ \left\langle {H_{i + 1} } \right\rangle _{H:G_3 } = \left\langle {H_{i + 1}
} \right\rangle _{H:G_1 }  \bigcap \left\langle { H_{i +
1} } \right\rangle _{H:G_2 }  $  by  $ \left\langle {H_i } \right\rangle _{H:G_3
} = \left\langle {H_i } \right\rangle _{H:G_1 }  \bigcap
\left\langle { H_i } \right\rangle _{H:G_2 }$.

Let  $ S' \notin \left\langle {H_i } \right\rangle _{H:G_1 }  $  and
 $ S' \notin \left\langle
{H_i } \right\rangle _{H:G_2 }  $.   Then  $ \left\langle {H_{i + 1} }
\right\rangle _{H:G_1 } = \left\langle {H_i } \right\rangle _{H:G_1 }
 $,  \\
 $ \left\langle {H_{i + 1} } \right\rangle _{H:G_2 } = \left\langle {H_i }
\right\rangle _{H:G_2 }  $  and  $ \left\langle {H_{i + 1} } \right\rangle
_{H:G_3 } = \left\langle {H_i } \right\rangle _{H:G_3 }  $.   Thus\\
  $ \left\langle
{H_{i + 1} } \right\rangle _{H:G_3 } = \left\langle {H_{i + 1} }
\right\rangle _{H:G_1 }  \bigcap \left\langle { H_{i + 1}
} \right\rangle _{H:G_2 }  $  by  $ \left\langle {H_i } \right\rangle _{H:G_3 } =
\left\langle {H_i } \right\rangle _{H:G_1 }  \bigcap \left\langle
{ H_i } \right\rangle _{H:G_2 }$.

Let  $ S' \notin \left\langle
{H_i } \right\rangle _{H:G_1 },  S' \in \left\langle {H_i }
\right\rangle _{H:G_2 }$.   Then  $ \left\langle {H_{i + 1} } \right\rangle
_{H:G_1 } = \left\langle {H_i } \right\rangle _{H:G_1 }  $, \\
  $ \left\langle
{H_{i + 1} } \right\rangle _{H:G_3 } = \left\langle {H_i } \right\rangle
_{H:G_3 }  $  and  $ \left\langle {H_{i + 1} } \right\rangle _{H:G_2 } =
({\left\langle {H_i } \right\rangle _{H:G_2 } \backslash \{S'\}} )\vert
S'' $.  Thus  \\
$ \left\langle {H_{i + 1} } \right\rangle _{H:G_3 } = \left\langle
{H_{i + 1} } \right\rangle _{H:G_1 }  \bigcap \left\langle {{\kern
1pt} H_{i + 1} } \right\rangle _{H:G_2 }  $ by  $ \left\langle {H_i }
\right\rangle _{H:G_3 } = \left\langle {H_i } \right\rangle _{H:G_1 } {\kern
1pt} \bigcap \left\langle { H_i } \right\rangle _{H:G_2 }$,
 $ S'' \notin \left\langle {H_{i + 1} } \right\rangle _{H:G_1 }$.

 The case of  $ S' \notin \left\langle {H_i } \right\rangle _{H:G_2 },  S' \in \left\langle {H_i } \right\rangle _{H:G_1 }  $   is proved by a  similar procedure  and omitted.
\end{proof}

\begin{lemma}
(i)   Let  $ G'\vert S' \in \tau ^ *  $  then  $G_{G'\vert S':S'}^{ *}   \bigcap G_{G'\vert S':G'}^{ *}  = \emptyset, G_{G'\vert S':G'}^{ *}\vert G_{G'\vert S':S'}^{ *}= G\vert G^ *;$

(ii) $H\in \tau^{\ast}$,  $ H'\vert H'' \subseteq H $  then $ G_{H:H'\vert H''}^{\ast} =
G_{H:H'}^{\ast} \vert G_{H:H''}^{\ast}$.
\end{lemma}
\begin{proof}
(i) and (ii) are immediately from Lemma 6.3.
\end{proof}

\begin{notation}
We write $ \tau _{H_i^c:S_{i1}^c }^{ * }$, $ G _{H_i^c:S_{i1}^c }^{ * }$  as
$ \tau _{S_{i1}^c }^{ * }$, $ G _{S_{i1}^c }^{ * }$, respectively, for the sake of simplicity.
\end{notation}

\begin{lemma}
(i) $G _{S_{i1}^c }^{ * } \subseteq G\vert G^{\ast}$;

(ii) $ \tau _{S_{i1}^c }^{ * }  $  is a
derivation of  $G_{S_{i1}^c }^{ *
}  $ from  $ S_{i1}^c  $,  which we denote by
 $ \cfrac{\underline{\,\,\,\, S_{i1}^c\,\,\,\,}}{ G_{S_{i1}^c }^{ * } }\left\langle {\tau _{S_{i1}^c }^{ *
} } \right\rangle $;

(iii)   $G_{S_{iu}^c }^{ * } = {S_{iu}^c }$  and  $ \tau _{S_{iu}^c }^{ * }  $
consists of a single node  $  {S_{iu}^c } $ for all $2\leqslant  u\leqslant  m_{i}$;

(iv) $ v_l (
{G_{S_{i1}^c }^{ * } }
)\backslash v_l ( {S_{i1}^c } ) = v_r ( {G_{S_{i1}^c }^{ * } } )\backslash
v_r ( {S_{i1}^c } ) $;

(v) $\left\langle {H} \right\rangle _{S_{i1}^c }\in\tau _{S_{i1}^c }^{*} $  implies
$H\leqslant  H_{i}^c$.  Note that $\left\langle {H} \right\rangle _{S_{i1}^c }$ is undefined for
any $H> H_{i}^c$ or $H\| H_{i}^c$.

(vi) $ S_{j}^c \in G_{S_{i1}^c }^{ * }  $  implies  $H_i^c\nleqslant H_j^c$.
\end{lemma}

\begin{proof}
Claims from (i) to (v) are immediately from Construction 4.7 and Lemma 4.8.

 (vi) Since $ S_j^c  \in   G_{S_{i1}^c }^{ * }\subseteq G\vert G^{\ast}$ then $S_j^c\ $ has the form $S_{ju}^c $ for some $u\geq2$ by Notation 4.14.  Then $G_{S_{j}^c }^{ * } = {S_{j}^c }$ by (iii).   Suppose that  $ H_i^c \leqslant  H_j^c$.  Then $ S_{j}^c$ is transferred  from $ H_j^c$ downward to $H_i^c$ and in side-hypersequent of  $ H_i^c$  by Notation 4.14 and $G\vert G^{\ast}<H_i^c \leqslant  H_j^c$.  Thus   $ \{ {S_{i1}^c } \}  \bigcap\{ {S_{j}^c } \}  = \emptyset  $ at $H_i^c $ since $S_{i1}^c $ is the unique focus sequent of  $ H_i^c$.  Hence  $ S_{j}^c\notin G_{S_{i1}^c }^{ * }  $ by Lemma 6.3 and (iii),   a contradiction therefore  $ H_i^c\nleqslant H_j^c$.
\end{proof}

\begin{lemma}
Let  $ \cfrac{G'\vert S' \quad G''\vert S''}{H\equiv  G'\vert G''\vert H'}(II) \in \tau ^ *$.
 (i) If $ S_j^c  \in  G_{H:H' }^{ * }$  then $H_j^c\leqslant  H $ or $H_j^c\| H $;
 (ii) If $ S_j^c  \in  G_{H:G'' }^{ * }$  then $H_j^c\leqslant  H $ or $H_j^c\| G'\vert S'  $.
\end{lemma}
\begin{proof}
(i) We impose a restriction on $(II)$ such that each sequent in $H'$  is  different from $S'$ or $S''$ otherwise we treat it as an $(EW)$-application.  Since $ S_j^c  \in  G_{H:H' }^{ *}\subseteq G\vert G^{\ast}$ then $S_j^c\ $ has the form $S_{ju}^c $ for some $u\geq2$ by Notation 4.14.  Thus $G_{S_{j}^c }^{ * } = {S_{j}^c }$.  Suppose that  $ H_j^c > H$.  Then  $S_j^c$  is transferred  from $ H_j^c$ downward to $H$. Thus $S_j^c\in H'$ by  $G_{S_{j}^c }^{ * } =S_{j}^c\in  G_{H:H' }^{ * }$ and Lemma 6.3.  Hence $S_j^c=S'$ or $S_j^c=S''$,  a contradiction with the restriction above.  Therefore $H_j^c\leqslant  H $ or $H_j^c\| H $.

(ii) Let $S_j^c  \in G_{H:G''}^{ * }$.  If  $H_{j}^{c}> H$ then $S_{j}^{c}\in H$ by Proposition 4.15(i) and  thus  $S_{j}^{c}\in G''$ by Lemma 6.3 and,   hence  $H_{j}^{c}\parallel G'\vert S'$ by $H_{j}^{c}\geqslant G''\vert S''$, $G'\vert S'\| G''\vert S''$.  If  $H_{j}^{c}\| H$ then  $H_{j}^{c}\parallel G'\vert S'$ by $H\leqslant  G'\vert S'$. Thus
 $ H_j^c \leqslant H$ or  $ H_j^c  \| G'\vert S' $.
\end{proof}

\begin{definition}
(i) By  $ H_i^c \rightsquigarrow H_j^c  $  we mean that  $ S_{ju}^c \in G_{S_{i1}^c }^{ *}$ for some $2\leqslant  u \leqslant  m_{j}$;
(ii) By  $ H_i^c  \leftrightsquigarrow H_j^c $  we mean that  $ H_i^c \rightsquigarrow H_j^c$  and  $ H_j^c\rightsquigarrow H_i^c $ ;
(iii) $H_i^c \nrightsquigarrow H_j^c $ means that  $ S_{ju}^c \notin G_{S_{i1}^c }^{ *}$ for all $2\leqslant  u \leqslant  m_{j}$.
\end{definition}

Then Lemma 6.6 (vi) shows that  $ H_i^c \rightsquigarrow H_j^c  $  implies  $ H_i^c
\nleqslant H_j^c$.
\begin{lemma}
Let   $ H_i^c
\| H_j^c$,  $ H_i^c\rightsquigarrow H_j^c$,   $ \cfrac{G'\vert S'  \,\,\,G''\vert S''}{G'\vert G''\vert H'}(II) \in \tau ^ *  $ such that  $G'\vert S' \leqslant  H_{i }^c$, $ G''\vert S'' \leqslant  H_{j}^c$.   Then  $ S' \in \left\langle {G'\vert S'  }
\right\rangle _{S_{i1}^c }  $.
\end{lemma}
\begin{proof}
Suppose that  $ S' \notin \left\langle {G'\vert S'} \right\rangle _{S_{i1}^c }  $.   Then $\left\langle {G'\vert S'} \right\rangle _{S_{i1}^c }\subseteq G'$ by $\left\langle {G'\vert S'} \right\rangle _{S_{i1}^c }\subseteq G'\vert S' $,
 $ \left\langle {G'\vert G''\vert H'} \right\rangle _{S_{i1}^c } = \left\langle {G'\vert S' } \right\rangle _{S_{i1}^c }$ by Construction 4.7.   Thus  $ \left\langle {G'\vert G''\vert H'} \right\rangle _{S_{i1}^c} \subseteq G'$.  Hence
$  G''\vert H' \bigcap \left\langle {G'\vert G''\vert H'} \right\rangle
_{S_{i1}^c } = \emptyset  $ by Proposition 6.2 (ii).   Therefore
 $ S_{ju}^c \notin G_{S_{i1}^c }^{ *}$ for all $1\leqslant  u \leqslant  m_{j}$ by Lemma 6.3,   i.e., $ H_i^c\nrightsquigarrow H_j^c$,  a contradiction and hence
 $ S' \in \left\langle {G'\vert S'} \right\rangle _{S_{i1}^c }$.
\end{proof}

Lemma 6.6 (ii) shows that $ \tau _{S_{i1}^c }^{ * }  $  is a
derivation of  $G_{S_{i1}^c }^{ *
}  $ from  one premise $ S_{i1}^c  $.  We generalize it by  introducing derivations from multiple  premises in the following.  In the remainder of  this section,  let  $ I = \{ {H_{i_1 }^c, \cdots
,H_{i_m }^c } \}  \subseteq \{ {H_{1 }^c, \cdots
,H_{N }^c } \}  $,   $ H_{i_k }^c \leftrightsquigarrow H_{i_q }^c  $  for all  $ 1
\leqslant k < q \leqslant m$.   Then  $  H_{i_k }^c \nleqslant H_{i_q }^c  $  and
 $ H_{i_q }^c \nleqslant  H_{i_k }^c$  by Lemma 6.6 (vi) thus  $ H_{i_k }^c \| H_{i_q }^c  $ for all  $ 1 \leqslant k < q \leqslant m$.

\begin{notation}
  $ H_{I}^V$ denotes  the intersection node of  $ H_{i_1 }^c, \cdots,H_{i_m}^c  $. We sometimes write the intersection node of  $ H_{i }^c$ and $H_{j}^c$  as $ H_{ij}^V$.  If $I=\{H_i^c \}$,   $H_I^V \coloneq H_i^c $, i.e.,  the intersection node of a single node is itself.
\end{notation}

Let   $ \cfrac{G'\vert S'  \,\,\,G''\vert S''}{G'\vert G''\vert H'}(II) \in \tau ^ *  $ such that  $ G'\vert G''\vert H' = H_{I }^V $.  Then  $ I$  is divided into two subsets
 $ I_l = \{ {H_{l_1}^c, \cdots, H_{l_{m(l)}}^c }
\}  $  and  $ I_r= \{ {H_{r_1 }^c, \cdots,H_{r_{m(r)} }^c } \}  $,   which occur in the left subtree  $ \tau^ * (G'\vert S') $  and right subtree  $ \tau ^ * (G''\vert S'') $  of  $ \tau ^ * (G'\vert G''\vert H') $,   respectively.

Let  $ \mathcal{I}= \{ {S_{i_1 1}^c, \cdots,S_{i_m 1}^c } \}  $,   $ \mathcal{I}_l= \{ {S_{l_{1}1}^c, \cdots,S_{l_{m(l)}1}^c } \}$,
$ \mathcal{I}_r= \{ {S_{r_{1}1}^c, \cdots,S_{r_{m(r)}1}^c} \}$ such that
$ \mathcal{I}=\mathcal{I}_l\bigcup\mathcal{I}_r$.  A derivation $ \tau
_{\mathcal{I}}^{ * }  $  of  $ \left\langle {G\vert G^ * }
\right\rangle _{\mathcal{I}}  $ from  $ S_{i_1 1}^c, \cdots
,S_{i_m 1}^c  $  is constructed by induction on $ |I|$.
 The base case of  $ \left| I \right| = 1 $  has been done by Construction 4.7.  For the induction case,   suppose that derivations  $ \tau _{\mathcal{I}_l}^{ * }  $  of
 $ \left\langle {G\vert G^ * } \right\rangle _{\mathcal{I}_l}  $
from  $ S_{l_{1}1}^c, \cdots, S_{l_{m(l)}1}^c$  and
 $ \tau _{\mathcal{I}_r }^{ * }  $  of  $ \left\langle {G\vert G^ *
} \right\rangle _{\mathcal{I}_r }  $  from  $ S_{r_{1}1}^c, \cdots, S_{r_{m(r)}1}^c$  are constructed.
Then  $ \tau
_{\mathcal{I}}^{ * }  $  of  $ \left\langle {G\vert G^ * }
\right\rangle _{\mathcal{I}}  $ from  $ S_{i_1 1}^c, \cdots
,S_{i_m 1}^c  $  is constructed as follows.

\begin{construction} $\mathrm{([A.5.2])}$
(i)\[ \left\langle {H} \right\rangle _{\mathcal{I}}
\coloneq \left\langle {H} \right\rangle
_{\mathcal{I}_l}\,\, \mathrm{for} \,\,\mathrm{all }\,\,G'\vert S'\leqslant  H\leqslant  H_{i}^{c} \,\,\mathrm{for} \,\,\mathrm{some}\,\,\, H_{i}^{c}\in I_{l},\]
\[ \left\langle {H} \right\rangle _{\mathcal{I}}
\coloneq \left\langle {H} \right\rangle
_{\mathcal{I}_r}\,\, \mathrm{for} \,\,\mathrm{all }\,\,G''\vert S''\leqslant  H\leqslant  H_{i}^{c} \,\,\mathrm{for} \,\,\mathrm{some}\,\,\, H_{i}^{c}\in I_{r},\]
\[
\tau
_{\mathcal{I}}^{ * } ( {\left\langle {G'\vert S'}
\right\rangle _{\mathcal{I}} } ) \coloneq  \tau
_{\mathcal{I}_l}^{ * } ( {\left\langle {G'\vert S'}
\right\rangle _{\mathcal{I}_l} }),\,\,\,
\tau_{\mathcal{I}}^{ * } ( {\left\langle {G''\vert S''}
\right\rangle _{\mathcal{I}} } ) \coloneq \tau
_{\mathcal{I}_r }^{ * } ( {\left\langle {G''\vert S''}
\right\rangle _{\mathcal{I}_r } } );
\]
(ii)\[\left\langle G'\vert G''\vert H'\right\rangle
_{\mathcal{I}} \coloneq {\left\langle G' \right\rangle _{\mathcal{I}_l
}}\vert \left\langle G''\right\rangle _{\mathcal{I}_r } \vert
H'\] and
\[\cfrac{\left\langle G'\vert S'
\right\rangle _{\mathcal{I}_l}\,\,\, \left\langle G''\vert S'' \right\rangle _{\mathcal{I}_r }
}{\left\langle G'\vert G''\vert H' \right\rangle
_{\mathcal{I}} }(II) \in \tau _{\mathcal{I}}^{ *
};
\]
(iii) Other nodes of  $ \tau _{\mathcal{I}}^{ * }  $  are
built up by Construction 4.7  (ii).
\end{construction}

The following lemma is a generalization of Lemma 6.6.

\begin{lemma} Let  $ Th (H_{i_k}^c ) = ( {H_{{i_k}0}^c, \cdots,H_{{i_k}n_{i_k}
}^c } ) $, where $1 \leqslant k \leqslant m,  H_{{i_k}0}^c =H_{i_k}^c$ and
$ H_{{i_k}n_{i_k} }^c = G\vert G^ *$.
   Then,  for all $0 \leqslant u \leqslant n_{i_k }$,

(i) $$ \left\langle {H_{i_k u}^c } \right\rangle _{\mathcal{I}} =
\bigcap\{\left\langle {H_{i_k u}^c }
\right\rangle_{S_{j1}^c }:H_j^c \in I,H_{i_k u}^c \leqslant H_j^c \}; $$

(ii)
$$ \cfrac{\underline{\{ {S_{j1}^c:H_j^c \in
I,H_{i_k u}^c \leqslant H_j^c } \}}}{
\left\langle
{H_{i_k u}^c } \right\rangle _\mathcal{I} }\left\langle { \tau _\mathcal{I}^{ *} ( {\left\langle {H_{i_k u}^c }
\right\rangle _\mathcal{I} } )} \right\rangle ;$$

(iii)   $$ v_l ( {\left\langle {H_{i_k u}^c } \right\rangle _{_\mathcal{I} }
} )\backslash\bigcup \{ v_l(S_{j1}^c ):H_j^c \in I,H_{i_k u}^c
\leqslant H_j^c  \}  =$$$$
 v_r ( {\left\langle {H_{i_k u}^c
} \right\rangle _{_\mathcal{I} } } )\backslash\bigcup \{
v_r (S_{j1}^c ):H_j^c \in I,H_{i_k u}^c \leqslant H_j^c \};$$

(iv) $\left\langle {H} \right\rangle _\mathcal{I}\in\tau _{\mathcal{I}}^{*} $  if and only if
$H\leqslant  H_{i}^c$ for some $ H_{i}^c\in I$. Note that $\left\langle {H} \right\rangle _{\mathcal{I} }$ is undefined  if $H> H_{i}^c$ or $H\| H_{i}^c$ for all $ H_{i}^c\in I$.
\end{lemma}

\textbf{Proof:} (i) is proved by induction on  $ \left| I \right|$.
For the base step, let  $ \left| {I} \right| = 1 $  then the claim
holds clearly. For the induction step, let  $ \left| {I}
\right| \geqslant 2 $  then  $ \left| {I_l} \right| \geqslant 1 $  and
 $ \left| {I_r} \right| \geqslant 1$.  Then  $ S' \in \left\langle
{G'\vert S'} \right\rangle _{S_{i1}^c }  $  for all  $ H_i^c \in I_l $  by Lemma 6.9  and
 $ H_i^c \rightsquigarrow H_j^c  $  for all  $ H_j^c \in I_r$.   $ \left\langle
{G'\vert S'} \right\rangle _{\mathcal{I}_l} =\bigcap _{H_i^c \in I_l }
\left\langle {G'\vert S'} \right\rangle _{S_{i1}^c }  $  by the
induction hypothesis then  $ S' \in \left\langle {G'\vert S'} \right\rangle
_{\mathcal{I}_l}  $  thus  $ \left\langle {G'\vert G''\vert H'}
\right\rangle _{\mathcal{I}_l} = \left\langle {G'}
\right\rangle _{\mathcal{I}_l} \vert G''\vert H' $  by  $ G'\vert
S' \leqslant H_{I_l }^V$.

$ \left\langle {G'\vert G''\vert H'}
\right\rangle _{\mathcal{I}_r } = \left\langle {G''}
\right\rangle _{\mathcal{I}_r } \vert G'\vert H' $  holds by a
procedure similar to above then  $$ \left\langle {G'\vert G''\vert H'}
\right\rangle _{\mathcal{I}} = \left\langle {G'} \right\rangle
_{\mathcal{I}_l} \vert \left\langle {G''} \right\rangle
_{\mathcal{I}_r } \vert H'$$$$\qquad\qquad\qquad\qquad\qquad\qquad= ( {\left\langle {G'}
\right\rangle _{\mathcal{I}_l} \vert G''\vert H'} ) \bigcap
( {\left\langle {G''} \right\rangle _{\mathcal{I}_r } \vert
G'\vert H'} )$$$$
\qquad\qquad\qquad\qquad\qquad=\left\langle {G'\vert G''\vert H'} \right\rangle
_{\mathcal{I}_l} \bigcap \left\langle {G'\vert G''\vert H'}
\right\rangle _{\mathcal{I}_r }  $$  by  $ \left\langle {G'}
\right\rangle _{\mathcal{I}_l} \subseteq G' $  and  $ \left\langle
{G''} \right\rangle _{\mathcal{I}_r } \subseteq G''$.  Other
claims hold immediately from Construction 6.11.

\begin{lemma}
(i) Let  $G_{\mathcal{I}}^{ * }$ denote $\left\langle {G\vert G^* } \right\rangle _{\mathcal{I}}$ then
$G_{\mathcal{I}}^{ * }= \bigcap_{H_i^c \in I} G_{S_{i1}^c }^{ * }$;

(ii) $ \cfrac{\underline{S_{i_1 1}^c \,\,\,\cdots\,\,\,S_{i_m 1}^c}}{
G_{\mathcal{I}}^{ * } }\left\langle {\tau
_{\mathcal{I}}^{ * } } \right\rangle; $

(iii)   $ v_l ( { G_{\mathcal{I}}^{ * }}
)\backslash \bigcup_{H_j^c \in I}v_l(S_{j1}^c )= v_r ( {G_{\mathcal{I}}^{ * } }
)\backslash\bigcup_{H_j^c \in I}v_r(S_{j1}^c )$;

(iv)  $ S_j^c \in G_{\mathcal{I}}^{ * }  $  implies  $ H_i^c \nleqslant H_j^c  $  for all $ H_i^c \in I$.
\end{lemma}

\begin{proof}
(i), (ii) and  (iii) are immediately from Lemma 6.12.  (iv) holds by (i) and Lemma 6.6 (vi).
\end{proof}

Lemma 6.13 (iv) shows that there exists no copy of $S_{i_k}^{c}$ in
$G_{\mathcal{I}}^{ * }$ for any $1\leqslant  k\leqslant  m$.  Then we may regard them to be eliminated in $ \tau_{\mathcal{I}}^{ * } $.  We then call $ \tau_{\mathcal{I}}^{ * } $  an elimination derivation.

Let $ \mathcal{I'}= \{ {S_{i_1 u_1}^c, \cdots,S_{i_m u_m}^c } \}$  be another set of sequents to $I$
such that   $ G'\equiv S_{i_1 u_1}^{c}\vert\cdots\vert S_{i_m u_m}^{c}$  is a copy of  $G''\equiv  S_{i_1 1}^{c}\vert\cdots\vert S_{i_m 1}^{c}$. Then $G'$ and $G''$  are disjoint and there exist two bijections  $ \sigma _l :v_l
(G') \to v_l (G'') $  and  $ \sigma _r :v_r (G') \to v_r (G'')$ such that $\sigma _r \circ\sigma _l(G')=G''$.
By applying $\sigma _r \circ\sigma _l$  to $ \tau_\mathcal{I}^{ * }$, we construct a derivation from  $S_{i_1 u_1}^{c},\cdots, S_{i_m u_m}^{c}$ and  denote it by $ \tau_\mathcal{I'}^{ * }$ and its root by $ G_\mathcal{I'}^{ * }$.

Let $ \mathbf{I'}= \{ {G_{b_1} \vert S_{i_1 u_1}^c, \cdots,G_{b_m} \vert S_{i_m u_m}^c } \} $ be a set of hypersequents to $I$, where  $G_{b_k} \vert S_{i_k u_k}^c$ be closed for all $1\leqslant  k\leqslant  m$.  By applying $ \tau_\mathcal{I'}^{ * }$ to $S_{i_1 u_1}^c, \cdots, S_{i_m u_m}^c $  in  $G_{b_1} \vert S_{i_1 u_1}^{c},\cdots,
G_{b_m} \vert S_{i_m u_m}^{c}$,  we construct  a derivation from
 $$G_{b_1} \vert S_{i_1 u_1}^{c},\cdots,
G_{b_m} \vert S_{i_m u_m}^{c}$$ and  denote it by  $ \tau_{\mathbf{I'}}^{ * }$  and its root by $ G_\mathbf{I'}^{ * }$.  Then $ G_\mathbf{I'}^{ * }=\{G_{b_k} \}_{k=1}^{m}\vert G_\mathcal{I'}^{ * }$.
\begin{definition}
We will use all $ \tau_{\mathbf{I'}}^{ * }$ as  rules of $\mathbf{GL_{\Omega}}$ and call them   elimination rules.  Further, we call
$S_{i_1 u_1}^{c},\cdots,S_{i_m u_m}^{c}$  focus sequents and,  all sequents in $G_{\mathcal{I'}}^{ *}$ principal sequents and,  $G_{b_1},\cdots, G_{b_m}$  side-hypersequents  of  $ \tau_{\mathbf{I'}}^{ * }$.
\end{definition}

\begin{remark}
We regard  Construction 4.7 as a procedure  $\mathcal{F}$, whose inputs are $\tau^2, H, H'$  and output $\tau_{H:H'}^2$.   With such a viewpoint, we write $\tau_{H:H'}^2$  as $\mathcal{F}_{H:H'}(\tau^2)$. Then $\tau
_{\mathcal{I}}^{ * }  $  can be constructed by iteratively  applying $\mathcal{F}$ to
$\tau^{\ast}$, i.e., $\tau_{\mathcal{I}}^{ * } =\mathcal{F}_{H_{i_m}^c:S_{i_m 1}^c}(\cdots \mathcal{F}_{H_{i_1}^c:S_{i_11}^c}(\tau^{\ast})\cdots)$.

We replace locally each $ \cfrac{G'  }{G'}(ID_\Omega )$ in  $\tau_{\mathcal{I}}^{* }$  with $ G' $ and denote the resulting derivation also by  $\tau_{\mathcal{I}}^{* }$.
Then each non-root node in  $\tau_{\mathcal{I}}^{* }$ has the focus sequent.

Let $H\in\tau_{\mathcal{I}}^{ * }$.  Then there exists a unique  node in $\tau^{*}$,  which we denote by $\mathcal{O}(H)$ such that  $H$  comes from $\mathcal{O}(H)$  by Construction 4.7 and 6.11.  Then the  focus sequent  of $\mathcal{O}(H)$  in $\tau^{*}$ is  the focus  of $H$ in  $\tau_{\mathcal{I}}^{ * }$  if $H$ is a non-root node and,
$\mathcal{O}(H)=H$ or $H\subseteq\mathcal{O}(H)$ as two hypersequents.
Since the relative position of any two nodes in $\tau^{ * }$  keep unchanged  in constructing  $\tau_{\mathcal{I}}^{ * }$,  $H_{1}\leqslant_{\tau_{\mathcal{I}}^{ * }}H_{2}$  if and only if
$\mathcal{O}(H_{1})\leqslant_{\tau^{ * }}\mathcal{O}(H_{2})$ for any $H_{1}, H_{2}\in\tau_{\mathcal{I}}^{ * }$.   Especially,   $\mathcal{O}(S_{i_k1}^c)=H_{i_k}^c$ for  $S_{i_k1}^c\in\tau_{\mathcal{I}}^{* }$.

Let $H\in\tau_{\mathcal{I}}^{ * }$.  Then  $H'\equiv \sigma _r \circ\sigma _l(H)\in \tau_{\mathcal{I'}}^{* }$ and  $H''\equiv\{G_{b_k}:H\leqslant_{\tau_{\mathcal{I}}^{ * }}S_{i_k1}^c\mathrm{ and }1\leqslant  k\leqslant  m\} \mid H'\in \tau_{\mathbf{I'}}^{* }$.  Define $\mathcal{O}(H')=\mathcal{O}(H'')=
\mathcal{O}(H)$.  Then $\mathcal{O}(G_{\mathbf{I'}}^{\ast})=G\vert G^{\ast}$ and $\mathcal{O}(G_{b_k}\vert S_{i_k u_k}^c)=H_{i_k}^c$ for all  $G_{b_k}\vert S_{i_k u_k}^c\in\tau_{\mathbf{I'}}^{* }$.
\end{remark}

Since $G_\mathcal{I}^\ast =\left\langle {G\vert G^\ast } \right\rangle
_\mathcal{I} \subseteq G\vert G^\ast $, then each $(pEC)$-sequent in
$G_\mathcal{I}^\ast $ has the form $S_{jv}^c $ for some $1\leqslant
j\leqslant N$, $2\leqslant v\leqslant m_j $ by Proposition 4.15(ii).  Then we introduce the
following definition.

\begin{definition}
(i) By $S_j^c \in
G_\mathcal{I}^\ast $  we means that there exists $H\in \tau
_\mathcal{I}^\ast $ such that $S_j^c \in H$, $\mathcal{O}(H)=H_j^c $.   So is $S_j^c \in G_{\mathcal{I}'}^\ast $.

(ii) Let $S_j^c \in
G_\mathcal{I}^\ast $. By $H_j^c \leqslant_{ \tau
_\mathcal{I}^\ast}  H_i^c$  we means that there exist $H, H'\in \tau
_\mathcal{I}^\ast $ such that $S_j^c \in H$, $\mathcal{O}(H)=H_j^c, \mathcal{O}(H')=H_i^c $ and
$H_j^c \leqslant_{\tau^{\ast}}  H_i^c$.   We usually  write  $  \leqslant_{ \tau
_\mathcal{I}^\ast}  $   as  $\leqslant$.
\end{definition}

\section{Separation of one branch}
In the remainder of this paper,  we assume that  $ p $
occurs at most one time for each sequent in  $ G_0$  as the one in Main  theorem,  $ \tau$  be a cut-free proof of  $ G_0  $  in  $ {\rm {\bf GL}} $  and  $\tau ^ * $  the proof of  $ G\vert G^ *  $  in  $ {\rm {\bf GL}}_\Omega  $
resulting from preprocessing of  $ \tau$. Then
$\left| v_{l}(S)\right|+\left| v_{r}(S) \right|\leq1$ for all
 $S\in G$,  which plays a key role in discussing the separation of branches.

\begin{definition}
By  $ S' \in _c G' $  we mean that there exists some copy of  $ S' $  in  $ G'$.
 $ G' \subseteq _c G'' $  if  $ S' \in _c G'' $  for all  $ S' \in G'$.
  $ G' = _c G'' $  if  $ G' \subseteq _c G'' $ and  $ G''\subseteq _c G' $. Let  $ G_{11},
\cdots,G_{1m}  $  be  $ m $  copies of  $ G_1  $  then we denote  $ G'\vert G_{11} \vert
\cdots \vert G_{1m}  $  by  $ G'\vert \{G_{1u} \}_{u = 1}^{m}$ or $ G'\vert \{G_{1} \}^{m}$.
\end{definition}

\begin{definition}
Let  $ I =  \{ {H_{i_1 }^c, \cdots,H_{i_m }^c } \}  \subseteq
\{ {H_1^c, \cdots,H_N^c } \}  $,   $ H_{i_k }^c  \| H_{i_l }^c  $  for all
 $ 1 \leqslant k < l \leqslant m$.   $ \lceil {S_{i_k }^c } \rceil_{I}  $
is called a branch of  $ H_{i_k }^c  $  to $I$ if it is a closed hypersequent such that
\begin{flushleft}
$(i)\,\,  \lceil {S_{i_k }^c } \rceil_{I}\subseteq_{c}G\vert G^{\ast},$\\
$(ii)\,\,  S_{i_k }^c \in \lceil {S_{i_k }^c } \rceil_{I} ,$ \\
$(iii)\,\,  S_j^c \in \lceil {S_{i_k }^c } \rceil_{I} $  implies   $H_j^c \leqslant
H_{i_k }^c$    or   $H_j^c  \| H_i^c $   for all   $H_i^c \in I.$
\end{flushleft}
\end{definition}
Then (i) $ S_{i_l }^c \notin _c \lceil {S_{i_k }^c } \rceil_{I}  $  for all
 $ 1 \leqslant k, l \leqslant m$,  $k\ne l$; (ii) $ S_j^c \in \lceil {S_{i_k }^c } \rceil_{I}  $ and
$ H_j^c \nleqslant H_{i_k }^c  $   imply  $ H_j^c \notin I$.

In this section, let  $ I = \{
{H_{i }^c } \}  $,    $ {\rm {\bf I}} = \{\lceil {S_{i }^c }\rceil_{I} \} $,  we will give an algorithm to eliminate all $S_{j }^c\in \lceil {S_{i }^c } \rceil_{I}$
satisfying  $H_{j }^c\leqslant  H_{i }^c$.

\begin{construction}$\mathrm{([A.3])}$
 A sequence of hypersequents  $G_{\rm {\bf I}}^{ \medstar(q)}  $  and their derivations  $ \tau _{\rm {\bf I}}^{\medstar(q)}  $  from $\lceil {S_{i }^c }\rceil_{I}$  for all  $ q
\geqslant 0 $  are constructed inductively as follows.

For the base case, define  $ G_{\rm {\bf I}}^{\medstar(0)} $ to be $\lceil {S_{i }^c }\rceil_{I}$ and, $\tau _{\rm {\bf I}}^{ \medstar(0)}$   be $ \cfrac{\underline{\quad\qquad}}{G_{\rm {\bf I}}^{\medstar(0 )}}$.  For the induction case, suppose that  $ \tau _{\rm {\bf I}}^{\medstar(q)}  $  and
$G_{\rm {\bf I}}^{ \medstar(q)}  $  are constructed for some  $ 0 \leqslant q$.  If there
exists no $ S_j^c \in G_{\rm {\bf I}}^{\medstar(q)}  $ such that  $ H_j^c \leqslant H_{i }^c $,  then  the procedure terminates and define  $J_{\rm {\bf I}} $ to be $ q$;   otherwise define $ H_{i_{q} }^c$ such that $ S_{i_{q} }^c \in G_{\rm {\bf I}}^{\medstar (q)}  $,  $ H_{i_{q} }^c \leqslant H_{i }^c $   and $ H_j^c\leqslant  H_{i_{q} }^c$ for all
 $ S_j^c \in G_{\rm {\bf I}}^{\medstar(q)}, H_j^c\leqslant  H_{i }^c $.
 Let  $ S_{i_{q} 1}^c
, \cdots,S_{i_{q} m_{q} }^c  $  be all copies of  $ S_{i_{q} }^c  $  in
$ G_{\rm {\bf I}}^{\medstar(q)}  $ then define
$G_{\rm {\bf I}}^{\medstar(q+1)} =G_{\rm {\bf I}}^{\medstar(q)} \backslash
\{S_{i_{q} u}^c \}_{u = 1}^{m_{q} }  \vert
\{G_{S_{i_{q} u}^c}^{ *} \}_{u = 1}^{ m_{q} }  $
and its derivation  $ \tau _{\rm {\bf I}}^{ \medstar(q+1)}  $  is constructed by sequentially
applying $ \tau _{S_{i_{q} 1}^c}^{ *}, \cdots,
\tau _{S_{i_{q} m_{q} }^c}^{ *}  $ to
 $ S_{i_{q} 1}^c, \cdots,S_{i_{q} m_{q} }^c  $  in
 $ G_{\rm {\bf I}}^{\medstar(q)}$,  respectively.   Notice that we assign new identification numbers to new
occurrences of  $ p $ in  $\tau _{S_{i_{q} u }^c}^{ *}$ for all
$0\leqslant  q\leqslant  J_{\rm {\bf I}}-1$, $1\leqslant  u\leqslant  m_{q}$.
\end{construction}

\begin{lemma}
(i) $ H_{i_{0}}^{c}=  H_{i}^{c}$ and $ H_{i_{q+1}}^{c}< H_{i_{q}}^{c} $
for all $0\leqslant  q\leqslant  J_{\rm {\bf I}}-2$;

(ii)   $ { G_{\rm {\bf I}}^{\medstar(q)} }\subseteq_{c} G\vert G^{\ast}   $  is closed for all $0\leqslant  q\leqslant  J_{\rm {\bf I}}$;

 (iii) $ \cfrac{\underline{\lceil {S_{i}^c } \rceil_{I}}}{G_{\rm {\bf I}}^{\medstar(q )} }
 \left\langle {\tau _{\rm {\bf I}}^{ \medstar(q)} } \right\rangle  $  for all $0\leqslant  q\leqslant  J_{\rm {\bf I}}$,  especially, $ \cfrac{\underline{\lceil {S_{i}^c } \rceil_{I}}}{G_{\rm {\bf I}}^{\medstar(J_{\rm {\bf I}} )} }
 \left\langle {\tau _{\rm {\bf I}}^{ \medstar(J_{\rm {\bf I}} )} } \right\rangle  $;

 (iv) $ S_j^c \in G_{\rm {\bf I}}^{\medstar(J_{\rm {\bf I}})}$ implies  $ H_j^c  \| H_{i }^c  $ and,  $ S_j^c \in G_{S_{i_{q} u}^c}^{ *}$ for some  $\tau
_{G_b \vert S_{i_{q} u}^c}^{ *}\in \tau _{\rm {\bf I}}^{\medstar(J_{\rm {\bf I}} )}$ or $S_j^c \in\lceil {S_{i}^c } \rceil_{I}$,  $ H_j^c \nleqslant H_{i }^c  $, where $G_b  = G_I^{\medstar(q)} \backslash \{ S_{i_{q} v}^c \} _{v = 1}^{u} |\{ G_{S_{i_{q} v}^c }^* \} _{v = 1}^{u-1 }$, $ G_b \vert S_{i_{q} u}^c  $  is closed and $ 0 \leqslant
q \leqslant J_{\rm {\bf I}}-1$, $1 \leqslant u \leqslant m_{q}$.
\end{lemma}

\begin{proof}
(i) Since $ S_{i }^{c} \in G_{\rm {\bf I}}^{\medstar(0)}  $ by $ S_{i }^{c} \in  \lceil {S_{i }^c }
\rceil_{I}= G_{\rm {\bf I}}^{\medstar(0)} $ and,  $ H_j^c\leqslant  H_{i}^c$ for all $ S_j^c \in G_{\rm {\bf I}}^{ \medstar (0)}, H_j^c\leqslant  H_{i }^c $  then $ H_{i_{0}}^{c}=  H_{i}^{c}$.
If $ S_{i_{q+1} }^c \in G_{\rm {\bf I}}^{\medstar(q)} \backslash \{S_{i_{q} u}^c \}_{u = 1}^{m_{q} }$
then $ H_{i_{q+1}}^{c}\leqslant H_{i_{q}}^{c} $  by  $S_{i_{q+1}}^{c} \in G_{\rm {\bf I}}^{\medstar(q)}$, $ H_{i_{q+1} }^c\leqslant H_{i }^c $ thus  $ H_{i_{q+1}}^{c}< H_{i_{q}}^{c} $
by  all copies of  $ S_{i_{q} }^c  $  in  $ G_{\rm {\bf I}}^{\medstar(q)} $ being collected in  $\{S_{i_{q} u}^c \}_{u = 1}^{m_{q} }$.  If  $S_{i_{q+1} }^c \in\{G_{S_{i_{q} u}^c}^{ *} \}_{u = 1}^{ m_{q} }$ then $H_{i_{q}}^{c}\nleqslant H_{i_{q+1}}^{c}$ by Lemma 6.6 (vi) thus $ H_{i_{q+1}}^{c}< H_{i_{q}}^{c} $ by $ H_{i_{q} }^c\leqslant H_{i }^c $,  $ H_{i_{q+1} }^c\leqslant H_{i }^c $.
Then $ H_{i_{q+1}}^{c}< H_{i_{q}}^{c} $ by  $G_{\rm {\bf I}}^{\medstar(q+1)} =G_{\rm {\bf I}}^{\medstar(q)} \backslash
\{S_{i_{q} u}^c \}_{u = 1}^{m_{q} }  \vert
\{G_{S_{i_{q} u}^c}^{ *} \}_{u = 1}^{ m_{q} }  $.  Note that  $H_{i_{J_{\rm {\bf I}}} }^c$ is undefined in Construction 7.3.

(ii) $v_l(G_{\rm {\bf I}}^{\medstar(0)})=v_r({ G_{\rm {\bf I}}^{\medstar(0)} }), { G_{\rm {\bf I}}^{\medstar(0)} }\subseteq_{c} G\vert G^{\ast}$ by $ G_{\rm {\bf I}}^{\medstar(0)} =\lceil {S_{i }^c }\rceil_{I}$.
Suppose that $v_l(G_{\rm {\bf I}}^{\medstar(q)})=v_r({ G_{\rm {\bf I}}^{\medstar(q)} }),  { G_{\rm {\bf I}}^{\medstar(q)} }\subseteq_{c} G\vert G^{\ast}$ then $v_l(G_{\rm {\bf I}}^{\medstar(q+1)})=v_r({ G_{\rm {\bf I}}^{\medstar(q+1)} }),
{ G_{\rm {\bf I}}^{\medstar(q+1 )} }\subseteq_{c} G\vert G^{\ast}$ by $G_{\rm {\bf I}}^{\medstar(q+1)} =G_{\rm {\bf I}}^{\medstar(q)} \backslash
\{S_{i_{q} u}^c \}_{u = 1}^{m_{q} }  \vert
\{G_{S_{i_{q} u}^c}^{ *} \}_{u = 1}^{ m_{q} }  $,
$v_l(G_{S_{i_{q} u}^c}^{ *} \backslash
\{S_{i_{q} u}^c \})=
v_r(G_{S_{i_{q} u}^c}^{ *} \backslash
\{S_{i_{q} u}^c \}) $ and
$G_{S_{i_{q} u}^c}^{ *}\subseteq_{c} G\vert G^{\ast}$ for all $1\leqslant  u\leqslant  m_{q}$.

(iii) $\tau _{\rm {\bf I}}^{ \medstar(0)}$ is $ \cfrac{\underline{\quad\qquad}}{G_{\rm {\bf I}}^{\medstar(0 )}} \left\langle {\tau _{\rm {\bf I}}^{ \medstar(0)} } \right\rangle  $.
Given $ \cfrac{\underline{\,\,\lceil {S_{i}^c } \rceil_{I}\,\,}}{G_{\rm {\bf I}}^{\medstar(q )} }
 \left\langle {\tau _{\rm {\bf I}}^{ \medstar(q)} } \right\rangle  $  then
 $ \cfrac{\underline{\,\,\lceil {S_{i}^c } \rceil_{I}\quad}}{G_{\rm {\bf I}}^{\medstar(q+1 )} }
 \left\langle {\tau _{\rm {\bf I}}^{ \medstar(q+1)} } \right\rangle  $ is constructed by
linking up  the conclusion of
previous derivation to the premise of  its successor
 in the sequence of derivations
 $$ \cfrac{\underline{\lceil {S_{i}^c } \rceil_{I}}}{G_{\rm {\bf I}}^{\medstar(q )} }
 \left\langle {\tau _{\rm {\bf I}}^{ \medstar(q)} } \right\rangle,  \cfrac{\underline{\,\,G_{\rm {\bf I}}^{\medstar(q)} \backslash
\{S_{i_{q} 1}^c \} \vert S_{i_{q} 1}^c\,\,}}{G_{\rm {\bf I}}^{\medstar(q)} \backslash
\{S_{i_{q} 1}^c \}  \vert G_{S_{i_{q} 1}^c}^{ *}}
 \left\langle {\tau _{S_{i_{q} 1}^c}^{ *}} \right\rangle,\cdots,
\cfrac{\underline{\,\,\,G_{\rm {\bf I}}^{\medstar(q)} \backslash
\{S_{i_{q} u}^c \}_{u = 1}^{m_{q}-1 }
  \vert S_{i_{q}m_{q}}^c\,\,\, \vert
\{G_{S_{i_{q} u}^c}^{ *}\}_{u = 1}^{m_{q}-1}}}{G_{\mathbf{I}}^{\medstar(q + 1)}  =G_{\rm {\bf I}}^{\medstar(q)} \backslash
\{S_{i_{q} u}^c \}_{u = 1}^{m_{q}}  \vert
\{G_{S_{i_{q} u}^c}^{ *}\}_{u = 1}^{m_{q}} }
\left\langle {\tau _{S_{i_{q} m_{q}}^c}^{ *}} \right\rangle,$$ as shown in the following figure.

\[
\cfrac{{\underline {{\kern 1pt} {\kern 1pt} {\kern 1pt} {\kern 1pt} {\kern 1pt} \cfrac{{\underline {\cfrac{{\underline {{\kern 1pt} {\kern 1pt} {\kern 1pt} {\kern 1pt} {\kern 1pt} {\kern 1pt} \cfrac{{\underline {{\kern 1pt} {\kern 1pt} {\kern 1pt} {\kern 1pt} {\kern 1pt} {\kern 1pt} {\kern 1pt} {\kern 1pt} {\kern 1pt} {\kern 1pt} {\kern 1pt} {\kern 1pt} {\kern 1pt} {\kern 1pt} {\kern 1pt} {\kern 1pt} {\kern 1pt} {\kern 1pt} {\kern 1pt} {\kern 1pt} {\kern 1pt} {\kern 1pt} {\kern 1pt} {\kern 1pt} {\kern 1pt} {\kern 1pt} {\kern 1pt} {\kern 1pt} {\kern 1pt} {\kern 1pt} {\kern 1pt} {\kern 1pt} {\kern 1pt} {\kern 1pt} {\kern 1pt} {\kern 1pt} {\kern 1pt} {\kern 1pt} {\kern 1pt} {\kern 1pt} {\kern 1pt} {\kern 1pt} {\kern 1pt} {\kern 1pt} {\kern 1pt} {\kern 1pt} {\kern 1pt} {\kern 1pt} {\kern 1pt} {\kern 1pt} {\kern 1pt} {\kern 1pt} {\kern 1pt} {\kern 1pt} {\kern 1pt} {\kern 1pt} {\kern 1pt} {\kern 1pt} {\kern 1pt} {\kern 1pt} {\kern 1pt} {\kern 1pt} {\kern 1pt} {\kern 1pt} {\kern 1pt} {\kern 1pt} {\kern 1pt} {\kern 1pt} {\kern 1pt} {\kern 1pt} {\kern 1pt} {\kern 1pt} {\kern 1pt} {\kern 1pt} {\kern 1pt} {\kern 1pt} {\kern 1pt} {\kern 1pt} {\kern 1pt} {\kern 1pt} {\kern 1pt} {\kern 1pt} {\kern 1pt} {\kern 1pt} {\kern 1pt} {\kern 1pt} {\kern 1pt} {\kern 1pt} {\kern 1pt} {\kern 1pt} {\kern 1pt}\,\,\,\,
 [S_{i}^c {\kern 1pt} ]_{I}\,\,\,\, {\kern 1pt} {\kern 1pt} {\kern 1pt} {\kern 1pt} {\kern 1pt} {\kern 1pt} {\kern 1pt} {\kern 1pt} {\kern 1pt} {\kern 1pt} {\kern 1pt} {\kern 1pt} {\kern 1pt} {\kern 1pt} {\kern 1pt} {\kern 1pt} {\kern 1pt} {\kern 1pt} {\kern 1pt} {\kern 1pt} {\kern 1pt} {\kern 1pt} {\kern 1pt} {\kern 1pt} {\kern 1pt} {\kern 1pt} {\kern 1pt} {\kern 1pt} {\kern 1pt} {\kern 1pt} {\kern 1pt} {\kern 1pt} {\kern 1pt} {\kern 1pt} {\kern 1pt} {\kern 1pt} {\kern 1pt} {\kern 1pt} {\kern 1pt} {\kern 1pt} {\kern 1pt} {\kern 1pt} {\kern 1pt} {\kern 1pt} {\kern 1pt} {\kern 1pt} {\kern 1pt} {\kern 1pt} {\kern 1pt} {\kern 1pt} {\kern 1pt} {\kern 1pt} {\kern 1pt} {\kern 1pt} {\kern 1pt} {\kern 1pt} {\kern 1pt} {\kern 1pt} {\kern 1pt} {\kern 1pt} {\kern 1pt} {\kern 1pt} } }}{{\cfrac{{\underline {G_{\mathbf{I}}^{\medstar(q)}  = G_{\mathbf{I}}^{\medstar(q)} \backslash \{ S_{i_{q} u}^c \} _{u = 1}^{m_{q}  } |\{ S_{i_{q} u}^c \} _{u = 2}^{m_{q}  } |S_{i_{q} 1}^c {\kern 1pt} {\kern 1pt} } }}{{G_{\mathbf{I}}^{\medstar(q)} \backslash \{ S_{i_{q} u}^c \} _{u = 1}^{m_{q}  } |\{ S_{i_{q} u}^c \} _{u = 3}^{m_{q}  } |S_{i_{q} 2}^c |G_{S_{i_{q} 1}^c }^* }}\left\langle {\tau _{S_{i_{q} 1}^c }^* } \right\rangle }}\left\langle {\tau _{\bf{I}}^{\medstar(q)} } \right\rangle {\kern 1pt} } }}{ \vdots }\left\langle {\tau _{S_{i_{q} 2}^c }^* } \right\rangle } }}{{G_{\mathbf{I}}^{\medstar(q)} \backslash \{ S_{i_{q} u}^c \} _{u = 1}^{m_{q}  } |S_{i_{q} m_{q}  }^c |\{ G_{S_{i_{q} u}^c }^* \} _{u = 1}^{m_{q}   - 1} }}{\kern 1pt} {\kern 1pt} } }}{{G_{\mathbf{I}}^{\medstar(q + 1)}  = G_{\mathbf{I}}^{\medstar(q)} \backslash \{ S_{i_{q} u}^c \} _{u = 1}^{m_{q}  } |\{ G_{S_{i_{q} u}^c }^* \} _{u = 1}^{m_{q}  } }}\left\langle {\tau _{S_{i_{q} m_{q}  }^{c}}^{*} } \right\rangle
\]

\begin{center}
 \normalsize\quad A  derivation of  $G_{\mathbf{I}}^{\medstar(q + 1)}$  from $G_{\mathbf{I}}^{\medstar(q)}$
\end{center}

(iv) Let  $ S_j^c \in G_{\rm {\bf I}}^{\medstar (J_{\rm {\bf I}} )}  $.  Then $ H_j^c \nleqslant H_{i }^c  $
by  the definition of  $J_{\rm {\bf I}}$.  If  $S_j^c \in\lceil {S_{i}^c } \rceil_{I}$,  then
$ H_j^c  \| H_{i }^c  $  by $ H_j^c \nleqslant H_{i }^c  $ and the definition of  $\lceil {S_{i}^c } \rceil_{I}$.  Otherwise,  by Construction 7.3,  there exists some  $\tau
_{G_b \vert S_{i_{q} u}^c}^{ *}$ in $\tau _{\rm {\bf I}}^{ \medstar (J_{\rm {\bf I}} )}$
such that  $ S_j^c \in G_{S_{i_{q} u}^c}^{ *}.$  Then $ H_{i_{q} }^c \nleqslant H_j^c  $ by Lemma 6.6 (vi).
Thus  $ H_{i }^c \nleqslant H_j^c  $   by  $ H_{i_{q} }^c \leqslant
H_{i }^c  $.   Hence $ H_j^c  \| H_{i }^c  $.
\end{proof}

Lemma 7.4   shows that  Construction 7.3 presents a derivation $ \tau _{\rm {\bf I}}^{\medstar(J_{\rm {\bf I}})}  $ of $G_{\rm {\bf I}}^{\medstar(J_{\rm {\bf I}})} $  from
 $\lceil {S_{i }^c }\rceil_{I}$  such that there doesn't exist  $ S_j^c \in G_{\rm {\bf I}}^{ \medstar (J_{\rm {\bf I}})}  $  satisfying  $ H_j^c \leqslant H_{i }^c $, i.e., all  $ S_j^c \in \lceil {S_{i }^c }\rceil_{I}$ satisfying  $ H_j^c \leqslant H_{i }^c $  are eliminated by Construction 7.3.  We generalize this procedure as follows.

\begin{construction}
Let  $ H \in \tau ^ * $,   $ H_{1}\subseteq H$ and $ H_{2}\subseteq_{c} G\vert G^{\ast} $.  Then
 $ G_{H:H_{l}}^{\medstar(J_{H:H_{l}})}  $ and its derivation  $ \tau _{H:H_{l}}^{\medstar (J_{H:H_{l}})}  $ for $ l=1,2  $ are constructed by  procedures similar to that of Construction 7.3 such that  $ H_j^c \nleqslant H $ for all  $ S_j^c \in G_{H:H_{l}}^{\medstar(J_{H:H_{l}})}$, where
$ G_{H:H_{1}}^{\medstar(0)}\coloneq G_{H:H_{1}}^{\ast} $,
$ \tau _{H:H_{1}}^{\medstar(0)}\coloneq \tau _{H:H_{1}}^{\ast}$, which are defined by Construction 4.7.
\end{construction}

We sometimes write $J_{\rm {\bf I}}$,  $J_{H:H_{l}}$ as  $J$ for simplicity.  Then the following lemma holds clearly.
\begin{lemma}
(i) $ \cfrac{\underline{\,\,\,\, H_{l}\,\,\,\,}}{G_{H:H_{l}}^{ \medstar(J)} }
 \left\langle {\tau _{H:H_{l}}^{ \medstar(J)} } \right\rangle  $,  $ H_j^c \nleqslant H$  for all $ S_j^c \in G_{H:H_{l}}^{ \medstar(J)}$.

(ii)  If  $S_i^c\in H $ and $ H_i^c > H $  then  $ G_{H:S_i^c }^{ \medstar (J)} = S_i^c$.

(iii) If $S\in_c G$  or,   $S\in_c G^{\ast}$  is a copy of  $S_{i1}^{c}$ and $ H_i^c \nleqslant H $  then $ G_{H:S}^{ \medstar (J)} = S$.

(iv) Let  $ H'\vert H'' \subseteq H\in \tau^{\ast}$. Then  $ G_{H:H'\vert H''}^{ \medstar (J)} =
G_{H:H'}^{\medstar(J)} \vert G_{H:H''}^{\medstar (J)}$ by suitable assignments of
 identification numbers to new occurrences of  $p$ in constructing
 $ \tau_{H:H'\vert H''}^{ \medstar (J)} $,
$\tau_{H:H'}^{\medstar(J)}$ and  $ \tau_{H:H''}^{\medstar (J)}$.

(v) $ G_{\rm {\bf I}}^{\medstar(J)}  =
\bigcup \{G_{H_{i}^{c}:S_{j}^{c}}^{\medstar(J)}:S_j^c \in\lceil {S_{i}^c } \rceil_{I},H_j^c\leqslant  H_{i}^{c} \} \vert \bigcup \{S_{j}^{c}:S_j^c \in\lceil {S_{i}^c } \rceil_{I},H_j^c\nleqslant H_{i}^{c}\}\vert \bigcup\{S:S\in\lceil {S_{i}^c } \rceil_{I},S\in_c G\}$.

\end{lemma}
\begin{proof}
(i) is proved by a procedure similar to that of  Lemma 7.4 (iii),  (iv)  and omitted.

(ii) Since $S_{i1}^{c}$ is the focus sequent of  $H_{i}^{c}$ then it is revised by some rule at
the node lower than $H_{i}^{c}$.  Thus $S_i^c\  \in H$ is some copy of  $S_{i1}^{c}$ by  $ H_i^c > H $.
Hence $S_i^c\ $ has the form $S_{iu}^c $ for some $u\geq2$.  Therefore it is transferred  downward to
$G\vert G^{\ast}$, i.e.,   $S_i^c\  \in G\vert G^{\ast}$. Then
$ G_{H:S_i^c }^{ \medstar (0)} =G_{H:S_i^c }^{ \ast} = S_i^c$.
Since there exists  no  $S_j^c\in  G_{H:S_j^c }^{ \medstar (0)}, H_j^c\leqslant  H$ then $J=0$.  Thus
$ G_{H:S_i^c }^{ \medstar (J)} = S_i^c$.

(iii)  is proved by a procedure similar to that of (ii) and omitted.

(iv) Since  $ H'\vert H'' \subseteq H\in \tau^{\ast}$,  then $ H'\bigcap H'' = \emptyset $
by Proposition 6.2.  Thus $ G_{H:H'\vert H''}^{\medstar (0)} = G_{H:H'\vert H''}^{\ast} =
G_{H:H'}^{\ast} \vert G_{H:H''}^{\ast}=G_{H:H'}^{\medstar(0)} \vert G_{H:H''}^{\medstar (0)}$.
Suppose that  $ G_{H:H'\vert H''}^{ \medstar (q)} =G_{H:H'}^{\medstar(q)} \vert G_{H:H''}^{\medstar q)}$ for
some $q\geq0$.  Then all copies $\{S_{i_{q} u}^c \}_{u = 1}^{m_{q}}$
of  $ S_{i_{q} }^c  $  in $G_{H:H'\vert H''}^{ \medstar (q)}$  are  divided two
subsets $\{S_{i_{q} u}^c \}_{u = 1}^{m_{q}}\bigcap G_{H:H'}^{\medstar(q)} $  and
$\{S_{i_{q} u}^c \}_{u = 1}^{m_{q}}\bigcap G_{H:H''}^{\medstar(q)} $.  Thus we can construct
$ G_{H:H'\vert H''}^{ \medstar (q+1)}, G_{H:H'}^{\medstar(q+1)} $ and $ G_{H:H''}^{\medstar (q+1)}$
simultaneously and assign  the same identification numbers to  new occurrences of  $p$  in
$G_{H:H'}^{\medstar(q+1)} $ and $ G_{H:H''}^{\medstar (q+1)}$ as the corresponding one in
$ G_{H:H'\vert H''}^{ \medstar (q+1)}$.  Hence
$ G_{H:H'\vert H''}^{ \medstar (q+1)} =
G_{H:H'}^{\medstar(q+1)} \vert G_{H:H''}^{\medstar (q+1)}$. Then
  $ G_{H:H'\vert H''}^{ \medstar (J)} =
G_{H:H'}^{\medstar(J)} \vert G_{H:H''}^{\medstar (J)}$.

Note that the requirement is
imposed only on one derivation that  distinct occurrence of $p$ has different identification number.
We permit  $ G_{H:H''}^{\medstar (q+1)}= G_{H:H''}^{\medstar (q)}$ or   $ G_{H:H'}^{\medstar (q+1)}= G_{H:H'}^{\medstar (q)}$ in the proof above,  which has no essential  effect on the proof of the claim.

(v) is immediately from (iv).
\end{proof}

Lemma 7.6(v) shows that   $G_{\rm {\bf I}}^{\medstar(J)}$  could be
constructed by  applying $ \tau_{H_{i}^{c}:S_j^c }^{ \medstar (J)}$ sequentially to each
 $S_j^c \in\lceil {S_{i}^c } \rceil_{I}$ satisfying $H_j^c\leqslant  H_{i}^{c}$. Thus the requirement   $ H_{i_{q+1}}^{c}< H_{i_{q}}^{c} $ in Construction 7.3  is not necessary, but which make the termination of the procedure obvious.

\begin{construction}  Apply  $ (EC_\Omega^{\ast} ) $  to  $ G_{\rm {\bf I}}^{ \medstar (J)} $  and denote the resulting hypersequent by  $G_{\rm {\bf I}}^\medstar$ and its derivation by  $ \tau _{\rm{\bf I}}^\medstar  $. It is possible that  $ (EC_\Omega^{\ast} ) $  is not applicable to  $ G_{\rm {\bf I}}^{\medstar (J )}  $  in which case we apply  $ \left\langle { ID_\Omega}
\right\rangle  $  to it for the regularity of the derivation.
\end{construction}

\begin{lemma}

(i) $ \cfrac{\underline{\lceil {S_{i}^c } \rceil_{I}
}}{G_{\rm {\bf I}}^\medstar}
\left\langle {\tau _{\rm {\bf I}}^\medstar } \right\rangle  $,
 $G_{\rm {\bf I}}^\medstar $  is
closed  and  $ H_j^c  \| H_{i }^c  $  for all  $ S_j^c \in G_{\rm {\bf I}}^\medstar$;

(ii) $ \tau _{\rm {\bf I}}^\medstar  $  is constructed by applying elimination
rules, say,
  $ \cfrac{\underline{\,\, G_b \vert S_{i_{q} u}^c\,\,}}
  {G_b \vert G_{S_{i_{q} u}^c}^{ *} }\left\langle {\tau
_{G_b \vert S_{i_{q} u}^c}^{*} } \right\rangle  $,   and the fully
constraint contraction rules, say,  $ \cfrac{\underline{G_{2}}}{G_{1} }\left\langle {EC_\Omega ^ * } \right\rangle  $,   where  $ H_{i_{q} }^c
\leqslant H_{i }^c  $,    $ G_b \vert S_{i_{q} u}^c  $  is closed for  $ 0 \leqslant
q \leqslant J-1$, $1 \leqslant u \leqslant m_{q}$.
\end{lemma}

\begin{proof}
Immediately from Lemma 7.4.
\end{proof}

\begin{definition}
Let  $ G' \in\tau _{\rm {\bf I}}^{ \medstar (J )}, H'\subseteq G' $ and $\left| v_{l}(H')\right|+\left| v_{r}(H')\right|\geq1$.  $ H' $  is called separable in  $ \tau _{\rm {\bf I}}^{\medstar (J )}  $ if  there exists $ H\subseteq G_{\rm {\bf I}}^{\medstar(J )}  $  such that
$\left| v_{l}(S) \right|+\left| v_{r}(S) \right|=1$ for all $S\in H$,  $ v_l (H) = v_l (H') $ and  $ v_r (H) = v_r(H') $,  and  $H'$  is  called  to be separated into  $H$ and $\widehat{H'}\coloneq H$.
\end{definition}

Note that  $ \tau _{\rm {\bf I}}^{\medstar(J )}  $  is a derivation
without $(EC_{\Omega})$  in  $ {\rm {\bf GL}}_\Omega  $.  Then we can extract elimination derivations from it by  Construction 4.7.

\begin{notation}
Let  $H'\subseteq G'\in \tau _{\rm {\bf I}}^{\medstar (J )}$.
 $ \tau _{\rm {\bf I}\{G':H'\}}^{\medstar(J )}  $  denotes the derivation
from $H'$,   which extracts from $ \tau _{\rm {\bf I}}^{\medstar (J )}$ by
 Construction 4.7,  and  denote  its  root  by
$G_{{\rm {\bf I}}\{G':H'\}}^{\medstar(J )}$.
\end{notation}

The following two lemmas  show that Construction 7.3 and 7.5 force some sequents in  $ \lceil {S_{i}^c } \rceil_{I}$ or $H'$ to be separable.

\begin{lemma}
Let $ \cfrac{G'\vert S' \quad G''\vert S''}{H\equiv G'\vert G''\vert H'}(II) \in \tau ^ *  $,
$\tau_{G_{b } \vert S_{i_{q}u }^c}^{ *}\in\tau _{\rm {\bf I}}^\medstar $,
 $ \cfrac{G_{b } \vert \left\langle {G'\vert S'
     } \right\rangle_{S_{i_{q}u }^c }\,\,\, G''\vert S'' }{H_{1}\equiv G_{b} \vert \left\langle {G'}
\right\rangle _{S_{i_{q}u }^c} \vert G''\vert H'}(II)\in\tau
_{G_{b } \vert S_{i_{q}u }^c}^{ *} $.  Then $H'$  is  separable in  $ \tau _{\rm {\bf I}}^{\medstar (J)}$  and
there is a unique copy of  $ \widehat{S''}\vert G_{{\rm {\bf I}}\{H_{1}:G''\}}^{ \medstar(J)}  $
 in  $G_{\rm {\bf I}}^\medstar$.
\end{lemma}
\begin{proof}
We write  $\leqslant_{ \tau _{\rm {\bf I}}^\medstar}$ as  $\leqslant_{\medstar}$ for  simplicity.
Clearly,   $G _{\rm {\bf I}\{H_1:G''\mid H'\}}^{\medstar(J )}$ is  a copy of $G _{H:G''\mid H'}^{\medstar(J)}  $ and,\\
$ \tau _{\rm {\bf I}\{H_1:G''\mid H'\}}^{\medstar(J)}  $ has no difference with $\tau _{H:G''\mid H'}^{\medstar(J)}  $ except some applications of $(ID_{\Omega})$ and identification numbers of some $p's$.

By Construction 4.7,  $G'\vert S' \in Th_{\tau^{\ast}}(H_{i_{q} }^c)$, $ S'\in\left\langle {G'\vert S'} \right\rangle_{S_{i_{q}u }^c }$ by $G''\vert H'\subseteq H_{1}\in \tau _{G_{b } \vert S_{i_{q}u}^c }^{*}$. Then $G
_{H_{1}:H'}^{ *}\subseteq G_{S_{i_{q}u }^c}^{ *} $,
$G'\vert S' \leqslant_{\tau^{\ast}} H_{i_{q} }^c\leqslant_{\tau^{\ast}} H_{i }^c$ by Lemma 6.6(v)  and 6.4(i).  $G_{{\rm {\bf I}}\{H_{1}:H'\}}^{\medstar(J)}\subseteq
G_{{\rm {\bf I}}}^{\medstar(J)}$ by  Lemma 4.8 (i).

Let  $ S_j^c  \in G_{{\rm {\bf I}}\{H_{1}:H'\}}^{\medstar(J)}$.  Then
 $ H_j^c \nleqslant_{\tau^{\ast}} H_{i} ^c$
by $G_{{\rm {\bf I}}\{H_{1}:H'\}}^{\medstar(J)}\subseteq
G_{{\rm {\bf I}}}^{\medstar(J)}$.
Suppose that   $S_j^c \in\lceil {S_{i}^c } \rceil_{I}$.     Then  $ H_j^c \|_{\tau^{\ast}} H_{i} ^c$  by $ H_j^c \nleqslant_{\tau^{\ast}} H_{i} ^c$ and the definition of  $\lceil {S_{i}^c } \rceil_{I}$.  Thus $S_j^c \in H'$  by $S_j^c \in\lceil {S_{i}^c } \rceil_{I}$, $ S_j^c  \in G_{{\rm {\bf I}}\{H_{1}:H'\}}^{\medstar(J)}$ and $G_{{\rm {\bf I}}}^{\medstar(J)}\leqslant_{\medstar} H_1\leqslant_{\medstar}\lceil {S_{i}^c } \rceil_{I}$.  Hence $S_j^c =S'$,  a contradiction with the restriction of $(II)$ (See the proof of Lemma 6.7). Therefore $S_j^c \notin\lceil {S_{i}^c }\rceil_I $.
Then,  by Lemma  7.4(iv),  there exists some  $\tau
_{G_{b'}\vert S_{i_{k} w}^c}^{ *}$ in $\tau _{\rm {\bf I}}^{ \medstar (J)}$
such that  $ S_j^c \in G_{S_{i_{k} w}^c}^{ *}$ then $ H_{i_{k}}^c \nleqslant_{\tau^{\ast}} H_j^c  $ by Lemma 6.6 (vi).
$  H_j^c\nleqslant_{\tau^{\ast}}  H_{i_{k}}^c$   by  $ H_{i_{k}}^c \leqslant_{\tau^{\ast}}
H_{i }^c  $, $ H_j^c \nleqslant_{\tau^{\ast}} H_{i} ^c$.   Thus  $ H_{i_{k}}^c \|_{\tau^{\ast}} H_j^c$.
 Let  $ \cfrac{G_{1}\vert S_{1} \quad G_{2}\vert S_{2}}{G_{1}\vert G_{2}\vert H_{2}}(II) \in \tau ^ *  $, where $G_{1}\vert G_{2}\vert H_{2}=H_{i_{k}j}^{V}$, $ G_{1}\vert S_{1}\leqslant_{\tau^{\ast}} H_{i_{k} }^c $, $ G_{2}\vert S_{2}\leqslant_{\tau^{\ast}} H_j^c $.  Then $S_{1}\in\left\langle {G_{1}\vert S_{1}} \right\rangle_{S_{i_{k}w }^c}$,  $G_{\rm {\bf I}}^{\medstar(k)} \backslash
\{S_{i_{k} v}^c \}_{v = 1}^{w} \vert\{ G_{S_{i_{k} v}^c }^* \} _{v = 1}^{w-1 } \vert
\left\langle {G_{1}\vert S_{1}} \right\rangle_{S_{i_{k} w}^c}\in \tau _{\rm {\bf I}}^{\medstar
(J)}$, $ G_{2}\vert S_{2} \in \tau _{\rm {\bf I}}^{\medstar
(J)}$. Thus
$ H_j^c  \|_{\medstar}  \lceil {S_{i}^c } \rceil_{I}$  by $G_{\rm {\bf I}}^{\medstar(k)} \backslash \{S_{i_{k} v}^c \}_{v = 1}^{w} \vert\{ G_{S_{i_{k} v}^c }^* \} _{v = 1}^{w-1 }  \vert
\left\langle {G_{1}\vert S_{1}} \right\rangle_{S_{i_{k} w}^c}\leqslant_{\medstar} \lceil {S_{i}^c } \rceil_{I}$,
$ G_{2}\vert S_{2}\leqslant_{\medstar}  H_j^c $ and $G_{1}\vert S_{1}\|_{\tau^{\ast}}G_{2}\vert S_{2}$.

Suppose that  $ H_{1}<_{\medstar} H_j^c$.  Then  $S_j^c \in H'$  by $S_j^c \in H_j^c$ and  $ S_j^c  \in G_{{\rm {\bf I}}\{H_{1}:H'\}}^{\medstar(J)}$.  Thus $S_j^c =S'$ or  $S_j^c =S''$,  a contradiction with the restriction of $(II)$ hence $ H_j^c\leqslant_{\medstar} H_{1}$  or $ H_j^c \|_{\medstar} H_{1}$.   Therefore $ H_j^c \|_{\medstar} H_{1}$  by $ H_{1}\leqslant_{\medstar} \lceil {S_{i}^c } \rceil_{I}$, $ H_j^c  \|_{\medstar} \lceil {S_{i}^c } \rceil_{I}$.  Thus $ v_l(S_j^c) \bigcap v_l (H' ) = \emptyset  $,   $ v_r (S_j^c) \bigcap v_r (H')
 = \emptyset  $.

 Since $ G_{{\rm {\bf I}}\{H_{1}:H'\}}^{\medstar(J)}\subseteq_{c} G\vert G^{\ast} $,
we divide it into two hypersequents  $G_{{\rm {\bf I}}\{H_{1}:H'\}}^{0(J)}$ and
 $G_{{\rm {\bf I}}\{H_{1}:H'\}}^{\ast(J)}$ such that
 $G_{{\rm {\bf I}}\{H_{1}:H'\}}^{\medstar(J)}=
 G_{{\rm {\bf I}}\{H_{1}:H'\}}^{0(J)}\vert
G_{{\rm {\bf I}}\{H_{1}:H'\}}^{\ast(J)},
G_{{\rm {\bf I}}\{H_{1}:H'\}}^{0(J)}\subseteq_{c} G,
G_{{\rm {\bf I}}\{H_{1}:H'\}}^{\ast(J)}\subseteq_{c} G^{\ast} $.
   Hence  $ v_l
(G_{{\rm {\bf I}}\{H_{1}:H'\}}^{\ast(J)}) \bigcap v_l (H' ) = \emptyset  $  and  $ v_r (G_{{\rm {\bf I}}\{H_{1}:H'\}}^{\ast(J)}) \bigcap v_r (H')
 = \emptyset  $.   Thus  $v_l (H')\subseteq  v_l(G_{{\rm {\bf I}}\{H_{1}:H'\}}^{0(J)}) $  and  $v_r (H')\subseteq v_r (G_{{\rm {\bf I}}\{H_{1}:H'\}}^{0(J)}) $ by  $v_l (H' )\subseteq  v_l(G_{{\rm {\bf I}}\{H_{1}:H'\}}^{\medstar(J)} ) $  and  $v_r (H' )\subseteq v_r (G_{{\rm {\bf I}}\{H_{1}:H'\}}^{\medstar(J)}) $.
Since $\left| {v_{l}(S)}\right|+\left| {v_{r}(S)} \right|\leq1$ for all
$S\in G_{{\rm {\bf I}}\{H_{1}:H'\}}^{0(J)}\subseteq_{c} G $,   then there exists one and only one sequent
 $ S \in G_{{\rm {\bf I}}\{H_{1}:H'\}}^{0(J)} $ for every
occurrence of  $ p_{i} $  in  $ H'$  such that  $ p_{i} $  occurs in  $ S$ thus $ H'$  is  separable in  $G_{{\rm {\bf I}}\{H_{1}:H'\}}^{\medstar(J)}$  and
let it be separated to  $\widehat{H'} $.
Then  $S',  S'' $  are separable in  $ \tau _{\rm {\bf I}}^{ \medstar (J)}  $  by  $v_l(S'\vert S'' )=v_l(H' ),  v_r(S'\vert S'' )=v_r(H' )  $ and separated to   $\widehat{S'}$ and $\widehat{S'' }$, respectively.
 Then  $ \widehat{S''}
 \vert G_{{\rm {\bf I}}\{H_1:G''\}}^{\medstar (J)} \subseteq G_{\rm {\bf I}}^{0(J)}
\vert G_{\rm {\bf I}}^{ * (J)}  $  is closed since
 $ G''\vert S'' $  is closed.  Thus all copies of   $ \widehat{S''}\vert G_{{\rm {\bf I}}\{H_1:G''\}}^{ \medstar(J)}  $  in  $ \tau _{\rm {\bf I}}^{ \medstar (J)}  $ are contracted into one by  $ (EC_\Omega^{\ast})$ in  $G_{\rm {\bf I}}^\medstar$.
\end{proof}

\begin{lemma}
(i) All  copies of  $S_{i}^c$ in $\lceil {S_{i}^c } \rceil_{I}$  are  separable in  $ \tau _{\rm {\bf I}}^{\medstar (J)}$;

(ii) Let $H \in \tau ^ *$, $H' \subseteq H$,
 $H_{j}^{c}\leqslant  H$ or   $H_{j}^{c}\| H$ for all $S_{j}^{c} \in G_{H:H' } ^{\ast}$.  Then $H'$  is  separable in  $ \tau _{H:H'}^{\medstar (J)}$.
\end{lemma}
\begin{proof}
(i) and (ii) are proved by a procedure similar to that of Lemma 7.11 and omitted.
\end{proof}

\begin{definition}
The skeleton of  $\tau _{\rm {\bf I}}^\medstar  $,  which we denote by
${\bar \tau } _{\rm {\bf I}}^\medstar  $, is constructed by replacing  all
$ \cfrac{\underline{ \,\,G_b \vert S_{i_{q} u}^c\,\, }}{G_b \vert
G_{S_{i_{q} u}^c}^{ *} }\left\langle {\tau_{G_b \vert S_{i_{q} u}^c}^{*} }
 \right\rangle\in\tau _{\rm {\bf I}}^\medstar $  with
 $ \cfrac{{G_b \vert S_{i_{q} u}^c}}{G_b \vert G_{S_{i_{q} u}^c}^{ *} }(\tau
_{G_b \vert S_{i_{q} u}^c}^{ *})$, i.e,  $G_b \vert S_{i_{q} u}^c$ is the parent node of
$G_b  \vert G_{S_{i_{q} u}^c}^{ *}$ in ${\bar \tau } _{\rm {\bf I}}^\medstar  $.
\end{definition}

\begin{lemma}
 ${\bar \tau } _{\rm {\bf I}}^\medstar$  is a linear structure with the lowest node $G_{\rm {\bf I}}^\medstar$ and the highest $\lceil {S_{i}^c } \rceil_{I}$.
\end{lemma}
\begin{proof}
It holds by all  $\tau_{G_b \vert S_{i_{q} u}^c}^{*}$ and
$EC_{\Omega}^{\ast}$   in $ \tau _{\rm {\bf I}}^\medstar $ being one-premise rules.
\end{proof}

\begin{definition}
We call Construction 7.3 together with 7.7   the separation algorithm of one branch and,  Construction 7.5  the separation algorithm along $H$.
\end{definition}

\section{Separation algorithm of multiple branches }

In this section,  let  $ I =  \{ {H_{i_1 }^c, \cdots,H_{i_m }^c } \}  \subseteq
\{ {H_1^c, \cdots,H_N^c } \}  $ such that  $ H_{i_k }^c  \| H_{i_l }^c  $  for all
 $ 1 \leqslant k < l \leqslant m$. We will generalize the separation algorithm of one branch
to that of multiple branches. Roughly speaking,  we give an algorithm to eliminate all  $S_{j }^c\in G\vert G^{\ast}$ satisfying  $H_{j }^c\leqslant H_{i_k }^c$ for some $H_{i_k }^c\in I$.

\begin{definition}
 $\overline I \coloneq \{H_j^c: H_j^c
\leqslant H_i^c \, \, for\, \, some\, \, \, H_i^c \in I\}$.
\end{definition}

\begin{theorem}$\mathrm{([A.4,  A.5.4])}$
Let $ {\rm {\bf I}}
= \{ \lceil {S_{i_1 }^c } \rceil_{I}, \cdots,  \lceil {S_{i_m }^c } \rceil_{I} \}$.
Then there exist one closed  hypersequent  $ G_{\rm {\bf I}}^\medstar\subseteq _c G\vert G^ *$  and its derivation  $ \tau _{\rm {\bf I}}^\medstar$  from  $ \lceil {S_{i_1
}^c } \rceil_{I}  $,  \ldots,  $ \lceil {S_{i_m }^c }
\rceil_{I}  $  in  $ {\rm {\bf GL}}_{\rm {\bf \Omega }}  $  such that

(i)   $ \tau _{\rm {\bf I}}^\medstar$  is constructed by applying elimination
rules, say,
\[\cfrac{\underline{G_{b_1} \vert S_{j_1}^c \,\,\,
  G_{b_2} \vert S_{j_2}^c \,\,\,\cdots\,\,\,
 G_{b_w}\vert S_{j_w}^c}}{G_{\mathbf{I}_{\rm {\bf j}}}^{\ast}=\{
{G_{b_k }} \} _{k = 1}^w \vert  G_{\mathcal{I}_{\rm {\bf j}}}^{\ast} }\left\langle {\tau
_{\mathbf{I}_{\rm {\bf j}}}^{ * } }
\right\rangle,\]
 and the fully constraint contraction rules,
say  $ \cfrac{\underline{G_{2}}}{G_{1} }\left\langle {EC_\Omega ^ *} \right\rangle $,   where  $ 1 \leqslant w \leqslant m $,   $ H_{j_k}^c \leftrightsquigarrow  H_{j_{l} }^c  $   for all  $1\leqslant k <l\leqslant
w $,   ${I}_{\rm {\bf j}}  = \{ H_{j_1}^c, \cdots,
H_{j_w}^c\} \subseteq \overline I $,   $ \mathcal{I}_{\rm {\bf j}} = \{ {S_{j_1}^c, \cdots, S_{j_w}^c } \}$,  $ {\rm {\bf I}}_{\rm {\bf j}} =
\{ {G_{b_1 } \vert S_{j_1}^c, \cdots,
G_{b_w }\vert S_{j_w}^c } \}$  and  $G_{b_k }\vert S_{j_k}^c$  is closed for
all  $ 1 \leqslant k \leqslant w $. Then  $ H_i^c \nleqslant H_j^c  $  for all  $ S_j^c \in G_{\mathcal{I}_{\rm {\bf j}}}^{\ast}  $  and  $ H_i^c \in I$.

(ii) For all  $ H \in {\bar \tau } _{\rm {\bf I}}^\medstar$,

\[
 \partial _{\tau
_{\rm {\bf I}}^\medstar}(H)  \coloneq \left\{ \begin{array}{l}
 G\vert G^\ast \quad H   \,\,is\,\, the\,\, root \,\,of \,\,  {\bar \tau } _{\rm {\bf
I}}^\medstar\,\,or \,\,   G_{2}   \,\,in \,\,  \cfrac{\underline{G_{2}}}{G_{1}}\left\langle {EC_\Omega ^ * \,\,\,\mathrm{or}\,\,\, ID_\Omega  } \right\rangle \in {\bar \tau } _{\rm
{\bf I}}^\medstar,  \\
 H_{j_k }^c \,\,\,\quad \,\,\,H \,\,  is \,\,  G_{b_k }\vert S_{j_k }^c    \,\,in  \,\, \tau _{\mathbf{I}_{\rm {\bf j}} }^{ * } \in {\bar \tau }_{\rm {\bf I}}^\medstar\,\,   for\,\, some\,\,
 1 \leqslant k \leqslant w,\\
 \end{array} \right.
\]
where,  ${\bar \tau } _{\rm {\bf I}}^\medstar$ is the skeleton  of $\tau _{\rm {\bf I}}^\medstar$ which
is defined as Definition 7.13. Then \\
$ \partial _{\tau _{\rm {\bf I}}^\medstar}
(G_{\mathbf{I}_{\rm {\bf j}}}^{\ast})
\leqslant \partial _{\tau _{\rm {\bf I}}^\medstar} ( {G_{{b_{ k} } }
\vert S_{j_{k } }^c } ) $  for some  $ 1
\leqslant k \leqslant w $  in  $ \tau _{{\rm {\bf I}}_{\rm {\bf j}}}^{ * }  $.

(iii) Let  $H \in {\bar \tau } _{\rm {\bf I}}^\medstar$,   $ G\vert G^{*} < \partial _{\tau_{\rm {\bf I}}^\medstar} (H) \leqslant H_I^V  $,    then $ G_{H_I^V :H}^{ \medstar (J)}\in \tau _{\rm {\bf I}}^\medstar$   and  it is constructed  by applying the  separation algorithm along $H_I^V$  to  $H$ and,   is an upper hypersequent of
 either $ \left\langle {EC_\Omega^ * } \right\rangle $ if it is applicable,   or $ \left\langle {ID_\Omega  } \right\rangle$ otherwise.

(iv)  $ S_j^c \in G_{\rm {\bf I}}^\medstar  $  implies  $ H_j^c  \| H_i^c  $  for all  $H_i^c \in I $
and, $ S_j^c \in G_{\mathcal{I}_{\rm {\bf j}} }^{\ast} $ for some
$\tau
_{\mathbf{I}_{\rm {\bf j}}}^{ * }\in \tau _{\rm {\bf I}}^\medstar$  or $ S_j^c \in \lceil {S_{i_k }^c } \rceil_{I}$ for some $H_{i_k}^c\in I$ satisfying $ H_j^c \nleqslant H_{i_k}^c$.
\end{theorem}

Note that in Claim (i),  bold  $\bf j$  in ${I}_{\rm {\bf j}}, \mathcal{I}_{\rm {\bf j}}$ or $ {\rm {\bf I}}_{\rm {\bf j}}$ indicates the
 $w$-tuple  $(j_1,\cdots,j_w)$ in  $S_{j_1}^c,\cdots,S_{j_w}^c$.  Claim (iv) shows the final aim of Theorem 8.2,  i.e., there exists no  $S_j^c \in G_{\rm {\bf I}}^\medstar  $  such  that $ H_j^c \leqslant H_i^c  $  for some  $H_i^c \in I $.  It is almost impossible to construct
 $ \tau _{\rm {\bf I}}^\medstar$  in a non-recursive way.   Thus
 we use Claims (i), (ii) and (iii) in Theorem 8.2  to characterize the structure of  $\tau _{\bf{I}}^\medstar$  in order to construct it recursively.

\begin{proof}
$ \tau _{\rm {\bf I}}^\medstar$  is constructed by induction on  $ \left| I
\right|$.
For the base case,  let $ \left| I \right| = 1 $.  Then $ \tau _{{\rm {\bf I}}}^\medstar$ is constructed  by Construction 7.3 and 7.7.  Here, Claim (i) holds by Lemma 7.8(ii), Lemma 7.4(i) and Lemma 6.6 (vi),  Claim (ii) by Lemma 7.4(i), (iii) is clear and (iv) by Lemma 7.4(iv).

For the induction case,   let  $\left| I \right| \geqslant 2 $.
Let  $ \cfrac{G'\vert S' \quad G''\vert S''}{G'\vert G''\vert H'}(II) \in \tau ^ *  $,   where
 $ G'\vert G''\vert H' = H_{I}^V$.  Then  $ \{ {H_{i_1 }^c, \cdots
,H_{i_m }^c } \}  $  is divided into two subsets  $I_l =
\{ {H_{l_{1} }^c, \cdots,H_{l_{m(l) } }^c } \},   I_r = \{ {H_{r_{1} }^c, \cdots,H_{r_{m(r) } }^c }
\},$  which occur in the left subtree  $ \tau ^ * (G'\vert S') $ and right
subtree  $ \tau ^ * (G''\vert S'') $  of  $ \tau ^ * (H_{I}^V ) $,
respectively. Then $m(l)+m(r)=m$. Let  $ {\rm {\bf I}}_l =
\{\lceil {S_{l_{1} }^c } \rceil_{I}, \cdots,
 \lceil {S_{l_{m(l) } }^c }
\rceil_{I} \},  {\rm {\bf I}}_r = \{\lceil
{S_{r_{1} }^c } \rceil_{I}, \cdots, \lceil {S_{r_{m(r) } }^c } \rceil_{I}
\}.$
Suppose that  derivations $ \tau _{{\rm {\bf I}}_l }^\medstar$ of  $G _{{\rm {\bf I}}_l }^\medstar$  and  $ \tau _{{\rm {\bf I}}_r }^\medstar$ of  $G _{{\rm {\bf I}}_r }^\medstar$  are constructed such that Claims from (i) to (iv) hold.  There are three cases to be considered in the following.

\textbf{Case 1}  $ S' \notin \left\langle {G'\vert
S'} \right\rangle _{\mathcal{I}_{{\rm {\bf j}}_l } }  $ for all  $ \tau _{{\rm {\bf I}}_{{\rm {\bf j}}_l
} }^{ * } \in \tau _{{\rm {\bf I}}_l
}^\medstar$. Then $ \tau _{\rm {\bf I}}^\medstar\coloneq \tau _{{\rm {\bf I}}_l }^\medstar$ and
$G_{\rm {\bf I}}^{\medstar}\coloneq G_{{\rm {\bf I}}_l }^{\medstar}$.

$\bullet$  For Claim (i), let $ \tau _{{\rm {\bf{I}}}_{{\rm {\bf j}}_{l}} }^{ * }\in \tau _{{\rm {\bf I}}_{l}}^\medstar$ and
$S_j^c \in G_{\mathcal{I}_{{\rm {\bf j}}_{l}}}^{ \ast}$.
By the induction hypothesis,   $ H_i^c \nleqslant H_j^c  $  for all  $ H_i^c \in I_{l}$.
 Since $ S' \notin \left\langle {G'\vert S'}
 \right\rangle _{\mathcal{I}_{{\rm {\bf j}}_l }}  $ then $ G'' \vert H'\bigcap \left\langle {G'\vert G''\vert H'}
 \right\rangle _{\mathcal{I}_{{\rm {\bf j}}_l } }=\emptyset  $.  Thus $G_{H_{I}^{V}:G'' \vert H'}^{ \ast} \bigcap G_{\mathcal{I}_{{\rm {\bf j}}_{l}}}^{ \ast}=\emptyset$ by Lemma 6.3 and 6.4. Then  $ S_j^c \notin G_{H_{I}^{V}:G'' \vert H'}^{ \ast}$. Thus $G'' \vert S'' \nleqslant H_j^c  $ by Proposition 4.15(i).  Hence,  for all $ H_i^c \in I_{r}$,  $ H_i^c \nleqslant H_j^c  $ by $ G'' \vert S'' \leqslant H_i^c $.  Then   $ H_i^c \nleqslant H_j^c  $  for all $ H_i^c \in I$.  Claims (ii) and (iii) follow directly from  the induction hypothesis.

$\bullet$  For Claim (iv), let $ S_j^c \in G_{\rm {\bf I}}^\medstar  $.    It follows from the induction hypothesis that   $ H_j^c  \| H_i^c  $  for all  $H_i^c \in I_{l} $ and, $ S_j^c \in G_{\mathcal{I}_{{\rm {\bf j}}_{l}} }^{\ast} $ for some $ \tau _{{\rm {\bf{I}}}_{{\rm {\bf j}}_{l}} }^{ * }\in \tau _{{\rm {\bf I}}_{l}}^\medstar$  or $ S_j^c \in \lceil {S_{l_{k} }^c } \rceil_{I}$ for some $H_{l_{k}}^c\in I_{l}, H_j^c \nleqslant H_{l_{k}}^c$.  Then  $H_j^c\nleqslant  H_I^V$ by $ H_j^c  \| H_{l_{1}}^c, H_I^V<H_{l_{1}}^c$.

If  $ S_j^c \in \lceil {S_{l_{k} }^c } \rceil_{I}$ for some $H_{l_{k}}^c\in I_{l}, H_j^c \nleqslant H_{l_{k}}^c$
then $ H_j^c  \| H_i^c  $  for all  $H_i^c \in I$ by  the definition of branches to  $I$. Thus we assume that $ S_j^c \in G_{\mathcal{I}_{{\rm {\bf j}}_{l}} }^{\ast} $ for some
$ \tau _{{\rm {\bf{I}}}_{{\rm {\bf j}}_{l}} }^{ * }\in \tau _{{\rm {\bf I}}_{l}}^\medstar$ in the following.  If  $G'\vert S'\leqslant H_j^c$ then $ H_j^c  \| H_i^c  $  for all  $H_i^c \in I_{r} $ thus  $ H_j^c  \| H_i^c  $  for all  $H_i^c \in I$. Thus let  $G'\vert S'\nleqslant H_j^c$  in the following.
By  the proof of Claim (i) above,  $G''\vert S''\nleqslant H_j^c$.  Then
   $H_I^V \nless  H_j^c $ by  $G'\vert S'\nleqslant H_j^c$  and  $G''\vert S''\nleqslant H_j^c$.  Thus  $ H_j^c  \| H_I^V $.  Hence   $ H_j^c  \| H_i^c  $  for all  $H_i^c \in I$.

\textbf{ Case 2} $ S'' \notin \left\langle {G''\vert
S''} \right\rangle _{\mathcal{I}_{{\rm {\bf j}}_r } }  $  for all  $ \tau _{{\rm {\bf I}}_{{\rm {\bf j}}_r
}}^{ * } \in \tau _{{\rm {\bf I}}_r
}^\medstar$.  Then $ \tau _{\rm {\bf I}}^\medstar\coloneq \tau _{{\rm {\bf I}}_r }^\medstar$ and
$G_{\rm {\bf I}}^{\medstar}\coloneq G_{{\rm {\bf I}}_r }^{\medstar}$.  This case is proved by a procedure similar to that of Case 1 and omitted.

\textbf{ Case 3} $ S' \in \left\langle {G'\vert
S'} \right\rangle _{\mathcal{I}_{{\rm {\bf j}}_l }}  $ for some  $ \tau _{{\rm {\bf I}}_{{\rm {\bf j}}_l
} }^{ * } \in \tau _{{\rm {\bf I}}_l
}^\medstar$ and  $ S'' \in \left\langle {G''\vert
S''} \right\rangle _{\mathcal{I}_{{\rm {\bf j}}_r }^{\medstar} }  $  for some  $ \tau _{{\rm {\bf I}}_{{\rm {\bf j}}_r
}}^{ * } \in \tau _{{\rm {\bf I}}_r
}^\medstar$.

Given\[ \cfrac{\underline{
  G_{b_{r1} } \vert S_{j_{r1}}^c\,\,\,   G_{b_{r2} } \vert S_{j_{r2}}^c  \,\,\,\cdots \,\,\,G_{b_{rv} } \vert
S_{j_{rv}}^c
 }}{G _r  \equiv \{ {G_{b_{rk} } } \} _{k = 1}^ v\vert G_{\mathcal{I}_{{\rm {\bf j}}_r }}^{\ast} }\left\langle {\tau _{{\rm {\bf I}}_{{\rm {\bf j}}_r
  }}^{ * } } \right\rangle \in
\tau _{{\rm {\bf I}}_r }^\medstar\]
 such that  $ S'' \in \left\langle {G''\vert
S''} \right\rangle _{\mathcal{I}_{{\rm {\bf j}}_r }}$ and $ H_{j_{rk} }^c > H_I^V  $  for all  $ 1 \leqslant k \leqslant v$, where,
$ 1 \leqslant v \leqslant m(r) $,   $ G_{b_{rk} } \vert
S_{j_{rk} }^c  $  is closed for all  $ 1 \leqslant k
\leqslant v $,
   $ I_{{{\mathbf{j}}_r}} =\{ {H_{j_{r1}}^{c},
 H_{j_{r2}}^{c}, \cdots, H_{j_{rv}}^{c}} \}\subseteq \overline{I_{r}}, $
$ \mathcal{I}_{{{\mathbf{j}}_r}} =
\{ {S_{j_{r1}}^{c},
 S_{j_{r2}}^{c}, \cdots, S_{j_{rv}}^{c}} \},$
$ {\mathbf{I}}_{{{\mathbf{j}}_r}} = \{ {{G_{{b_{r1}}}}|S_{j_{r1}}^c, \cdots,{G_{{b_{rv}}}}|S_{j_{rv}}^c} \}.$
 Then  $H_{I_{{{\mathbf{j}}_r}} }^V\geqslant G''\vert S''  $
  by    $ I_{{{\mathbf{j}}_r}}\subseteq\overline{I_{r}}$  and  $H_{j_{rk}}^{c} > H_I^V  $  for all  $ 1 \leqslant k \leqslant v$.  Thus
  $ H_j^c \rightsquigarrow H_i^c  $  for all  $H_j^c \in I_{{\rm {\bf j}}_r
}  $  and  $ H_i^c \in I_l$ by $ S'' \in \left\langle {G''\vert
S''} \right\rangle _{\mathcal{I}_{{\rm {\bf j}}_r } }$ and Construction 6.11.

For each $\tau _{{\rm {\bf I}}_{{\rm {\bf j}}_r
  }}^{ * } \in
\tau _{{\rm {\bf I}}_r }^\medstar$  above,   we construct a derivation $\tau _{{\rm {\bf I}}_l }^\medstar(\tau_{{\rm {\bf I}}_{{\rm {\bf j}}_r } }^\ast )$ in which you may regard $\tau _{{\rm {\bf I}}_l }^\medstar$ as a subroutine,  and $\tau_{{\rm {\bf I}}_{{\rm {\bf j}}_r } }^\ast $ as its input in the following stage 1.  Then a derivation  $ \tau _{{\rm {\bf I}}_r }^\medstar(\tau _{{\rm {\bf I}}_l }^\medstar(\tau _{{\rm {\bf I}}_{{\rm {\bf j}}_r } }^{ * } ))$ is constructed by calling $\tau _{{\rm {\bf I}}_l }^\medstar(\tau_{{\rm {\bf I}}_{{\rm {\bf j}}_r } }^\ast )$ in Stage 2,   in which you may regard $ \tau _{{\rm {\bf I}}_r }^\medstar(\tau _{{\rm {\bf I}}_l }^\medstar(\tau _{{\rm {\bf I}}_{{\rm {\bf j}}_r } }^{ * } ))$ as a routine and $\tau _{{\rm {\bf I}}_l }^\medstar(\tau_{{\rm {\bf I}}_{{\rm {\bf j}}_r } }^\ast )$  as its subroutine.

Firstly,  we present some properties  of $ \tau _{\rm {\bf I}}^\medstar$   which are derived from Claims (i) $\sim$ (iv)  and  applicable  to  $\tau _{{\rm {\bf I}}_l }^\medstar$ or $ \tau _{{\rm {\bf I}}_r }^\medstar$  under the  induction hypothesis.

\begin{notation}
Let
 $$G_{\dagger}\coloneq  \widehat{S''}\vert G_{H_I^V:G''}^{ \medstar(J)} \vert G_{H_I^V:H'}^{\medstar (J)}\backslash\{ \widehat{S'} \vert  \widehat{S''} \}\,\,\,and$$
 $$G_{\ddagger}\coloneq\{ {G_{b_{rk} } } \} _{k = 1}^{v} \vert  \widehat{S''}  \vert
G_{^{H_I^V:\left\langle {G''} \right\rangle _{\mathcal{I}_{{\rm
{\bf j}}_r }} }}^{ \medstar (J)} \vert G_{H_I^V:H'}^{\medstar (J)}\backslash
\{ \widehat{S'} \vert  \widehat{S''} \}$$ be two close hypersequents,  $G_{\dagger}\subseteq H$ for some $H\in  \tau _{{\rm {\bf I}}_l }^\medstar$  and
$G_{\ddagger}\backslash\{G_{b_{rk} } \} _{k = 1}^{v}\subseteq H$ for some $H\in  \tau _{{\rm {\bf I}}_r }^\medstar$.
\end{notation}
Generally,  $ \widehat{S''}\subseteq G_{\dagger}$  is  a copy of  $ \widehat{S''}\subseteq G_{\ddagger}$, i.e., eigenvariables in $ \widehat{S''}\subseteq G_{\dagger}$ have different identification numbers with those in $ \widehat{S''}\subseteq G_{\ddagger}$,  so are $H', G'', S'$.

\begin{lemma}
$S_{j}^{c}\in G_{\dagger}  $  implies  $H_{j}^{c}\parallel G'\vert S'$.
\end{lemma}

\begin{proof}
Let $S_{j}^{c}\in G_{\dagger}\subseteq G_{H_I^V:G''\vert H'}^{ \medstar(J)} $.   Then $H_{j}^{c}\nleqslant H_I^V$ by Lemma 7.6(i). Thus
$H_{j}^{c}> H_I^V$ or $H_{j}^{c}\| H_I^V$.  If  $H_{j}^{c}\| H_I^V$ then $H_{j}^{c}\parallel G'\vert S'$ by $H_I^V< G'\vert S'$ and Proposition 2.12(ii).    If  $H_{j}^{c}> H_I^V$ then $S_{j}^{c}\in H_I^V$ by Proposition 4.15(i).  Thus  $S_{j}^{c}\in G''$ by Lemma 6.3, Lemma 6.7(i).   Hence  $H_{j}^{c}\parallel G'\vert S'$ by $H_{j}^{c}\geqslant G''\vert S''$, $G'\vert S'\| G''\vert S''$.
\end{proof}

\begin{lemma}
(1)  ${\bar \tau } _{\rm {\bf I}}^\medstar$ is an $m$-ary tree and,   $ \tau _{\rm {\bf I}}^\medstar$ is a binary tree;

(2)  Let  $ H \in {\bar \tau } _{\rm {\bf I}}^\medstar$ then $ \partial _{\tau _{\rm {\bf I}}^\medstar}( H) \leqslant H_{i_{k}}^c $  for some  $ 1 \leqslant k \leqslant m$;

(3)  Let  $ H \in {\bar \tau } _{\rm {\bf I}}^\medstar$ then  $ H_I^V  \nparallel \partial _{\tau _{\rm {\bf I}}^\medstar}( H) $;

(4)  Let  $w>1$ in $ \tau _{{\rm {\bf I}}_{\rm {\bf j}} }^{
* } \in \tau _{\rm {\bf I}}^\medstar$ then $H_I^V<H_{j_k}^c$ for all $ 1 \leqslant k \leqslant w$.

(5)  Let  $ \tau _{{\rm {\bf I}}_{\rm {\bf j}} }^{
* } \in \tau _{\rm {\bf I}}^\medstar$,  $ \partial _{\tau _{\rm {\bf I}}^\medstar}
( {G_{b_{k} } \vert S_{j_{k }}^c }
) \leqslant H_I^V  $  for some  $ 1 \leqslant k \leqslant w$.  Then  $ w = 1 $.
\end{lemma}

\begin{proof}
(1)  is immediately from Claim (i).  (2) holds by $G\vert G^{\ast}\leqslant H_{j_k }^c$ and $H_{j_k }^c\leqslant H_{i_k}^c$ for some $H_{i_k}^c\in I$ by ${I}_{\rm {\bf j}}\subseteq \overline I $.  (3) holds by Proposition 2.12(iii),  (2) and $H_I^V\leqslant H_{i_{k}}^c $.

For (4), let  $ w > 1 $.   Then  $ H_{{j}_1}^c  \|H_{{j}_k}^c  $ for
each  $ 2 \leqslant k \leqslant w $,   $ H_{{j }_1}^c \leqslant H_{{i_g } }^c  $ and $H_{{j }_k}^c \leqslant H_{{i_h } }^c $ for some $H_{{i_g } }^c, H_{{i_h } }^c \in I$ by (2).  Thus
$ H_{{j}_1}^c  \|H_{{i}_h}^c  $  and $ H_{{j}_k}^c  \|H_{{i}_g}^c  $  by Proposition 2.12(ii).  Hence $ H_{{j }_1}^c \nleqslant H_{{i_g } {i_h } }^V  $ by  $H_{{i_g } {i_h } }^V < H_{{i_h } }^c $,  and $ H_{{j}_k}^c \nleqslant H_{{i_g } {i_h }}^V  $ by $H_{{i_g } {i_h } }^V < H_{{ i_g } }^c $.
Thus  $ H_I^V < H_{{j }_1 }^c  $  and
$ H_I^V < H_{{j }_k }^c  $  by (3),   $ H_I^V \leqslant H_{{j_1 } {j_k } }^V  $.   Hence  $ H_I^V < H_{{j }_k}^c  $  for all  $ 1 \leqslant k \leqslant w$. (5) is from (4).
\end{proof}

\begin{lemma}
Let $ \cfrac{\underline{H_{i,1}\,\,\, \cdots\,\,\,H_{i,w_{i}}}}{H_{i-1,1}}\left\langle {
\tau _{{\rm {\bf I}}_{\rm {\bf j}}(i) }^{* } } \right\rangle \in
 \tau _{\rm {\bf I}}^\medstar$  for all $1\leqslant i\leqslant n$ such that $ \partial _{\tau
_{\rm {\bf I}}^\medstar}(H_{0,1}) = G\vert G^ *  $ and $ \partial _{\tau
_{\rm {\bf I}}^\medstar}(H_{n,1}) \leqslant H_{I}^{V} $.
Then $ \partial _{\tau
_{\rm {\bf I}}^\medstar}(H_{i,1}) \leqslant H_{I}^{V} $  and $w_{i}=1$ for all $1\leqslant i\leqslant n$.
\end{lemma}

\begin{proof}
The proof is by induction on  $n$.  Let $n=1$ then
$w_1=1$  by Lemma 8.5(5) and $ \partial _{\tau
_{\rm {\bf I}}^\medstar}(H_{1,1}) \leqslant H_{I}^{V} $.  For the induction step, let $ \partial _{\tau
_{\rm {\bf I}}^\medstar}(H_{i,1}) \leqslant H_{I}^{V} $ for some $1<i\leqslant n$ then $w_i=1$ by Lemma 8.5(5).  Since $ \cfrac{\underline{H_{i,1}\,\,\,\cdots\,\,\,H_{i,w_{i}}}}{H_{i-1,1}}\left\langle {
\tau _{{\rm {\bf I}}_{\rm {\bf j}}(i) }^{* } } \right\rangle \in
 \tau _{\rm {\bf I}}^\medstar$ then $ \partial _{\tau
_{\rm {\bf I}}^\medstar}(H_{i-1,1}) \leqslant \partial _{\tau
_{\rm {\bf I}}^\medstar}(H_{i,k}) $ for some $1\leqslant k\leqslant w_i$ by Claim (ii).  Then
$ \partial _{\tau_{\rm {\bf I}}^\medstar}(H_{i-1,1}) \leqslant \partial _{\tau
_{\rm {\bf I}}^\medstar}(H_{i,1})\leqslant H_{I}^{V} $ by $w_{i}=1$.  Thus $w_{i-1}=1$  by Lemma 8.5(5).
\end{proof}

\begin{definition}
Let $ \cfrac{\underline{G_{2}}}{G_{1}}\left\langle {EC_\Omega ^ * } \right\rangle\in  \tau _{{\rm {\bf I}} }^{\medstar} $.  The module of  $ \tau _{{\rm {\bf I}} }^{\medstar} $ at  $G_{2}$, which we denote by $ \tau _{{\rm {\bf I}} :G_{2} }^{\medstar}  $,  is defined as follows: (1) $  G_{2} \in \tau _{{\rm {\bf I}} :G_{2} }^{\medstar}$; (2) $ \cfrac{\underline{H_1     \cdots    H_u  }}{H_0 }\left\langle {\tau _{{\rm {\bf I}}_{{\rm{\bf j}} } }^{ * } } \right\rangle \in
\tau _{{\rm {\bf I}} :G_{2} }^{\medstar} $  if  $ H_0 \in \tau _{{\rm {\bf I}} :G_{2} }^{\medstar} $; (3)  $ H_1 \notin \tau _{{\rm {\bf I}}:G_{2} }^{\medstar}  $  if
 $ \cfrac{\underline{H_1 }}{H_0 }\left\langle {EC_\Omega ^ * } \right\rangle \in \tau _{{\rm {\bf I}} }^{\medstar}$,    $ H_0 \in \tau _{{\rm {\bf I}} :G_{2} }^{\medstar}$.
\end{definition}

Each node of  $ \tau _{{\rm {\bf I}} :G_{2} }^{\medstar}  $ is determined bottom-up, starting with $G_{2}$,  whose root is $G_{2}$ and leaves may be branches, leaves of $ \tau^{\ast}$ or lower hypersequents of  $\left\langle {EC_\Omega ^ * } \right\rangle$-applications. While each node of  $ \tau _{H :H'}^{\ast}  $ is determined top-down, starting with $H'$,  whose root is a subset of $G\vert G^{\ast}$ and leaves  contain  $H'$  and some leaves  of $ \tau^{\ast}$.

 \begin{lemma}
(1)   $ \tau _{{\rm {\bf I}}:G_{2} }^{\medstar}  $ is a derivation without  $\langle EC_\Omega ^*\rangle$ in $ {\rm {\bf GL}}_{\rm {\bf \Omega }}  $.

(2) Let   $ H' \in {\bar \tau }_{{\rm {\bf I}}:G_{2} }^{\medstar}$ and $\partial _{\tau
_{\rm {\bf I}}^\medstar}(H') > H_{I}^{V} $. Then $ \partial _{\tau
_{\rm {\bf I}}^\medstar}(H) > H_{I}^{V} $ for all $H \in {\bar \tau } _{{\rm {\bf I}}:G_{2} }^{\medstar}$ and $ H\geqslant H'$.
\end{lemma}

\begin{proof}
(1) is clear and (2) immediately from Lemma 8.6.
\end{proof}

{\bf  Stage 1  Construction of Subroutine $\tau _{{\rm {\bf I}}_l }^\medstar(\tau
_{{\rm {\bf I}}_{{\rm {\bf j}}_r } }^\ast ).$}
Roughly speaking, $\tau _{{\rm {\bf I}}_l }^\medstar(\tau _{{\rm {\bf
I}}_{{\rm {\bf j}}_r } }^\ast )$ is constructed by
replacing some nodes $\tau _{{\rm {\bf I}}_{{\rm {\bf j}}_l }}^\ast \in \tau _{{\rm {\bf I}}_l }^\medstar$ with $\tau _{{\rm {\bf
I}}_{{\rm {\bf j}}_l }\bigcup {\rm {\bf I}}_{{\rm {\bf
j}}_r } }^\ast $ in post-order. However, the ordinal
postorder-traversal algorithm cannot be used to construct $\tau _{{\rm {\bf
I}}_l }^\medstar(\tau _{{\rm {\bf I}}_{{\rm {\bf j}}_r }
}^\ast )$ because the tree structure of $\tau _{{\rm {\bf I}}_l }^\medstar
(\tau _{{\rm {\bf I}}_{{\rm {\bf j}}_r } }^\ast )$ is
generally different from that of $\tau _{{\rm {\bf I}}_l }^\medstar$ at some
nodes $H\in \tau _{{\rm {\bf I}}_l }^\medstar$ satisfying  $\partial _{\tau _{{\rm{\bf I}}_l }^\medstar} (H)<H_{{\rm {\bf I}}_l }^V $.   Thus we construct
a sequence  $ \tau _{{\rm {\bf I}}_l }^{\medstar (q)}  $  of
trees for all  $ q \geqslant 0 $  inductively as follows.

For the
base case, we mark all  $ \left\langle {EC_\Omega ^ * } \right\rangle
$-applications in  $ \tau _{{\rm {\bf I}}_l }^\medstar$  as unprocessed and define
such marked derivation to be  $ \tau _{{\rm {\bf I}}_l }^{\medstar (0)}$.  For the
induction case, let  $ \tau _{{\rm {\bf I}}_l }^{\medstar (q)}  $  be constructed.
If all applications of  $ \left\langle {EC_\Omega ^ * } \right\rangle  $  in
 $ \tau _{{\rm {\bf I}}_l }^{\medstar (q)}  $  are marked as processed,   we firstly delete the root of the tree  resulting from the  procedure  and then,  apply $ \langle EC_\Omega ^*\rangle$ to the root of the resulting derivation  if it is applicable otherwise add an $\langle ID_\Omega\rangle$-application to it  and finally,   terminate the procedure.   Otherwise we select one of the outermost
unprocessed  $ \left\langle {EC_\Omega ^ * } \right\rangle
$-applications in  $ \tau _{{\rm {\bf I}}_l }^{\medstar (q)}$,   say,
 $ \cfrac{\underline{    {\kern
1pt}   G_{q + 1}^{ \circ \circ }  {\kern
1pt}      }}{G_{q +1}^ \circ }\left\langle {EC_\Omega ^ * } \right\rangle _{q + 1}^ \circ  $,  and  perform the following steps to construct  $ \tau _{{\rm {\bf I}}_l
}^{\medstar (q + 1)}  $  in which
$ \cfrac{\underline{   G_{q + 1}^{
\circ \circ }  }}{G_{q + 1}^ \circ }\left\langle {EC_\Omega ^ * }
\right\rangle _{q + 1}^ \circ  $  be revised as  $ \cfrac{\underline{
      G_{q+ 1}^{ \cdot \cdot }     {\kern1pt}   }}{G_{q + 1}^{\circ}}\left\langle {EC_\Omega ^
* } \right\rangle _{q + 1}^ \cdot  $   such that

  (a) $ \tau _{{\rm {\bf I}}_l
}^{\medstar (q + 1)}  $  is constructed by locally revising  $ \tau _{{\rm {\bf I}}_l :
     G_{q + 1}^{\circ \circ } }^{\medstar (q)}$ and leaving other nodes of  $ \tau _{{\rm {\bf I}}_l }^{\medstar (q)}  $  unchanged, particularly  including  $G_{q + 1}^{\circ}$;

         (b)  $ \tau _{{\rm {\bf I}}_l }^{\medstar (q + 1)} (G_{q + 1}^{ \cdot\cdot} ) $  is a derivation in  $ {\rm {\bf GL}}_{\rm {\bf \Omega }}  $;

   (c)  $G_{q + 1}^{ \cdot\cdot} = G_{q + 1}^{ \circ\circ}$  if  $ S' \notin \left\langle {G'\vert S'}
\right\rangle _{\mathcal{I}_{{\rm {\bf j}}_{_l } }
}  $  for all
$ \tau _{\mathbf{I}_{{\rm {\bf j}}_{_l } }}^{ *} \in \tau _{{\rm {\bf I}}_l }^\medstar(
     G_{q + 1}^{
\circ \circ }
  ) $  otherwise\\
   $ G_{q + 1}^{ \cdot\cdot} = G_{q + 1}^{ \circ\circ}
\backslash G_{\dagger}^{m_{q+1}}
\vert G_{\ddagger}^{m_{q+1}}$  for some $m_{q+1}\geqslant 1$.

\begin{remark}
 By two superscripts  $\circ$ and $\cdot$ in $\left\langle {EC_\Omega ^* } \right\rangle _{q + 1}^ \circ $ or $\left\langle {EC_\Omega ^* } \right\rangle _{q + 1}^\cdot$,  we indicate the unprocessed state and processed state, respectively.   This procedure determines an ordering for all  $ \left\langle {EC_\Omega ^ * } \right\rangle$-applications in  $ \tau _{{\rm {\bf I}}_l }^{\medstar}$ and  the subscript $q+1$  indicates that it is the $q+1$-th application of  $ \left\langle {EC_\Omega ^ * } \right\rangle$  in a post-order transversal of $ \tau _{{\rm {\bf I}}_l }^{\medstar}$.  $G_{q + 1}^{\circ \circ }$ and $G_{q + 1}^ \circ$ ($G_{q+ 1}^{ \cdot \cdot } $ and $G_{q + 1}^{\cdot}$) are the premise and conclusion of $\left\langle {EC_\Omega ^ * }\right\rangle _{q + 1}^ \circ$ ($\left\langle {EC_\Omega ^ * }\right\rangle _{q + 1}^\cdot  $),  respectively.
 \end{remark}

\textbf{Step 1 (Delete)}.  Take the module   $ \tau _{{\rm {\bf I}}_l :
     G_{q + 1}^{\circ \circ } }^{\medstar (q)}$  out of  $ \tau _{{\rm {\bf I}}_l }^{\medstar (q)}  $.
 Since $\left\langle {EC_\Omega ^ * } \right\rangle _{q + 1}^ \circ$  is the unique unprocessed   $ \left\langle {EC_\Omega ^ * } \right\rangle
$-applications in  $ \tau _{{\rm {\bf I}}_l }^{\medstar(q)}(G_{q +1}^ \circ) $ by its choice criteria,   $ \tau _{{\rm {\bf I}}_l :
     G_{q + 1}^{\circ \circ } }^{\medstar (q)}$ is the same as $ \tau _{{\rm {\bf I}}_l :
     G_{q + 1}^{\circ \circ } }^{\medstar}$ by Claim (a).
      Thus it is a derivation.
If    $ \partial _{\tau _{{\rm {\bf I}}_l}^\medstar} ( {H } ) \leqslant H_I^V  $   for all $H\in \tau _{{\rm {\bf I}}_l :G_{q + 1}^{ \circ \circ } }^{\medstar (q)}$,  delete all internal nodes of  $ \tau _{{\rm {\bf I}}_l :
     G_{q + 1}^{\circ \circ } }^{\medstar (q)}$.
Otherwise there exists
\[\cfrac{\underline{G_{b_{l'1} } \vert S_{j_{l'1}}^c\,\,\,
    G_{b_{l'2} } \vert
S_{j_{l'2}}^c \,\,\,     \cdots \,\,\,    G_{b_{l'u'} }
\vert S_{j_{l'u'}}^c}}{G _{l'} \equiv
\{ {G_{b_{l'k} } } \} _{k = 1}^{u' }\vert
G_{\mathcal{I}_{\textbf{j}_{l'} } }^{\ast} }\left\langle
{\tau _{\mathbf{I}_{\textbf{j}_{l'} } }^{ *} } \right\rangle \\
\in \tau _{{\rm {\bf I}}_l :
  G_{q + 1}^{ \circ \circ } }^{\medstar (q)}  \]
   such that
$ \partial _{\tau _{{\rm {\bf I}}_l }^\medstar} ( {G_{b_{l'k} } \vert
S_{j_{l'k} }^c } ) > H_I^V  $  for
all  $ 1 \leqslant k \leqslant u' $  and  $ \partial _{\tau _{{\rm {\bf I}}_l
}^\medstar} ( {G_{l'} } ) \leqslant H_I^V  $  by Lemma 8.8(2) and   $ \partial _{\tau _{{\rm {\bf I}}_l}^\medstar} (G_{q + 1}^{\circ \circ } )=G\vert G^{*}\leqslant H_I^V  $,   then delete all  $ H \in \tau _{{\rm {\bf I}}_l :  G_{q+ 1}^{ \circ \circ } }^{\medstar (q)}  $,    $ G_{q + 1}^ { \circ \circ } \leqslant H <G_{l'}  $.  We denote  the structure resulting from the deletion operation above by $ \tau _{{\rm {\bf I}}_l :  G_{q + 1}^{ \circ \circ } (1)}^{\medstar (q)}  $.   Since  $ \partial _{\tau _{{\rm {\bf I}}_l}^\medstar} ( {G_{l'} } ) \leqslant H_I^V  $  then $ \tau _{{\rm {\bf I}}_l :  G_{q + 1}^{ \circ \circ } (1)}^{\medstar (q)}  $ is a tree by Lemma 8.6.
Thus it is also a derivation.

\textbf{Step 2 (Update)}.    For each  $ G_{q'}^\circ \in \tau _{{\rm {\bf I}}_l :    G_{q+ 1}^{ \circ \circ }(1) }^{\medstar (q)}  $  which satisfies
 $ \cfrac{\underline{ G_{q'}^{ \cdot \cdot }}}{G_{q'}^ \circ
}\left\langle {EC_\Omega ^ * } \right\rangle _{q'}^ \cdot \in \tau _{I_l}^{\medstar (q)}  $  and
 $ S' \in \left\langle {G'\vert S'} \right\rangle
_{\mathcal{I}_{{\rm {\bf j}}_{_l } } }  $
for some  $ \tau _{{\rm {\bf I}}_{{\rm {\bf j}}_{_l } } }^{ * } \in \tau _{{\rm {\bf I}}_l }^\medstar(
    G_{q'}^{ \circ \circ } ) $,
we replace  $ H $  with  $ H\backslash G_{\dagger}
\vert G_{\ddagger}  $  for each  $ H \in \tau _{{\rm {\bf I}}_l :   G_{q+ 1}^{ \circ \circ }(1)}^{\medstar (q)}  $,   $ G_{l'} \leqslant  H \leqslant G_{q'}^\circ$.

Since
 $ \cfrac{\underline{
    G_{q'}^{ \cdot \cdot }}}{G_{q'}^\circ }\left\langle {EC_\Omega ^ * } \right\rangle
_{q'}^ \cdot \in \tau _{I_l }^{\medstar (q)}({G_{q + 1}^{ \circ \circ}}) $  and  $ \left\langle {EC_\Omega ^ * } \right\rangle _{q + 1}^\circ  $  is the outermost unprocessed \\
 $  \left\langle {EC_\Omega ^* } \right\rangle$-application in  $ \tau _{{\rm {\bf I}}_l }^{\medstar (q)}  $
then  $ q' \leqslant q $  and  $ \left\langle {EC_\Omega^ * } \right\rangle _{q'}^ \cdot
 $  has been processed.  Thus Claims (b) and (c) hold for  $ \tau _{\mathbf{I}_l }^{\medstar (q)}   ( G_{q'}^ \cdot ) $  by the induction hypothesis.  Then $ \cfrac{\underline{
    G_{q'}^{ \cdot \cdot }}}{G_{q'}^\cdot} $ is a valid $ \left\langle {EC_\Omega^ * } \right\rangle$-application since  $ \cfrac{\underline{
    G_{q'}^{\circ\circ}}}{G_{q'}^\circ}$,  $ \cfrac{\underline{
    G_{\dagger}^{m_{q'}}}}{G_{\dagger}}$ and $ \cfrac{\underline{
    G_{\ddagger}^{m_{q'}}}}{G_{\ddagger}}$ are valid, where $ G_{q'}^{ \cdot\cdot} = G_{q'}^{ \circ\circ}\backslash G_{\dagger}^{m_{q'}}\vert G_{\ddagger}^{m_{q'}}$,  $G_{q'}^\cdot= G_{q'}^\circ\backslash G_{\dagger}\vert G_{\ddagger} $.

\begin{lemma}
Let  $ G_{l'} <  H \leqslant G_{q'}^\circ$.  Then $ \partial _{\tau _{{\rm {\bf I}}_l}^\medstar} (H)\geqslant G' \vert S'$.
\end{lemma}
 \begin{proof}
Since  $ G_{l'} <  H$ then  $G_{b_{l'k} } \vert S_{j_{l'k} }^c \leqslant H $  for some $1\leqslant k \leqslant u'$.  If  $ \partial _{\tau _{{\rm {\bf I}}_l}^\medstar} (H)\geqslant H_{I_{l}}^{V}$ then $ \partial _{\tau _{{\rm {\bf I}}_l}^\medstar} (H)\geqslant G' \vert S'$.   Otherwise all applications between $G_{l'}$ and $H$ are one-premise rules by Lemma 8.6. Then $H_{j_{l'k} }^c \leqslant \partial _{\tau _{{\rm {\bf I}}_l}^\medstar} (H) $ by Claim (ii).  Thus  $ \partial _{\tau _{{\rm {\bf I}}_l}^\medstar} (H)\geqslant G' \vert S'$ by  $H_{I}^{V}<H_{j_{l'k} }^c$, $\partial _{\tau _{{\rm {\bf I}}_l}^\medstar} (H) \leqslant H_{l_{k'}}^{c}  $ for some $1\leqslant k' \leqslant m(l)$ by Claim (i).
\end{proof}

Since $ \partial _{\tau _{{\rm {\bf I}}_l}^\medstar} (H)\geqslant G' \vert S'$ by  Lemma 8.10  and  $H_{j}^{c}\| G' \vert S'$   for each  $ S_{j}^{c}\in G_{\dagger}$ by Lemma 8.4,  then $ G_{\dagger}\subseteq H$ as side-hypersequent of $H$.  Thus this step
updates the revision of  $ G_{q'}^{\cdot\cdot} $  downward to  $ G_{l'}$.

Let  $m'$ be the number of  $G_{q'}^{\circ}$ satisfying  the above conditions,  $ \tau _{{\rm {\bf I}}_l : G_{q + 1}^{\circ \circ }(1)}^{\medstar (q)}  $, $G_{l'}$  and $G_{b_{l'k} } \vert S_{j_{l'k} }^c $  for all  $ 1 \leqslant k \leqslant u' $  be  updated as $ \tau _{{\rm {\bf I}}_l : G_{q + 1}^{\circ \circ } (2)}^{\medstar (q)}  $, $G_{l''}$,  $ G'_{b_{l'k} } \vert S_{j_{l'k} }^c $, respectively.  Then  $ \tau _{{\rm {\bf I}}_l :  G_{q + 1}^{ \circ \circ } (2)}^{\medstar (q)}  $  is a derivation and $G_{l''}=G_{l'} \backslash
 G_{\dagger}^{m'}\vert G_{\ddagger}^{m'} $.

 \textbf{Step 3 (Replace).}  All
 $ \tau _{{\rm {\bf I}}_{{\rm {\bf j}}_{_l } }}^{ * } \in \tau _{{\rm {\bf I}}_l : G_{q + 1}^{
\circ \circ } (2)}^{\medstar (q)}  $ are processed  in post-order. If
 $ H_i^c \rightsquigarrow H_j^c  $  for all  $ H_i^c
 \in I_{{\rm {\bf j}}_{_l } }  $  and  $H_j^c \in I_{{\rm {\bf j}}_r }  $  it
proceeds by the following procedure otherwise it remains unchanged.
Let  $\tau _{{\rm {\bf I}}_{{\rm {\bf j}}_l } }^{ * } $ be in the form   $$\cfrac{\underline{G_{b_{l1} } \vert S_{j_{l1}}^c\,\,\,    G_{b_{l2} } \vert
S_{j_{l2}}^c \,\,\,     \cdots\,\,\,     G_{b_{lu} }
\vert S_{j_{lu}}^c}}{G_{l} \equiv
\{ {G_{b_{lk} } } \} _{k = 1}^{u} \vert
G_{\mathcal{I}_{{\rm {\bf j}}_{_l } } }^{\ast} }.$$

 Then  $H_{j_{lk}}^c\geqslant G'\vert S'$ for all $1\leqslant k\leqslant u$ by Lemma 8.10,
$G_{b_{lk} } \vert S_{j_{lk}}^c>G_{l''}$.

  Firstly,  replace
  $ \tau _{{\rm {\bf I}}_{{\rm {\bf j}}_{_l } } }^{ * }  $  with
$ \tau _{{\rm {\bf I}}_{{\rm {\bf j}}_{_l } } \cup {\rm {\bf I}}_{{\rm {\bf j}}_r } }^{
* }$.  We may  rewrite  the roots of
$ \tau _{{\rm {\bf I}}_{{\rm {\bf j}}_{_l } }}^{ * }  $  and  $ \tau _{{\rm {\bf I}}_{{\rm {\bf j}}_{_l }
} \cup {\rm {\bf I}}_{{\rm {\bf j}}_r } }^{ * }  $   as
 \nonumber
\begin{flalign}
G_l& = \{ {G_{b_{lk} } }
\} _{k = 1}^{u} \vert  G_{H_I^V:\left\langle {G'} \right\rangle
_{\mathcal{I}_{{\rm {\bf j}}_{_l } } } }^{ \ast} \vert
 G_{H_I^V:G''\vert H' }^{ \ast}\,\,\,  \mathrm{and} \\
 G_{l,r} &\equiv \{ {G_{b_{lk} } }\} _{k = 1}^{u }\vert  G_{H_I^V:\left\langle {G'} \right\rangle
_{\mathcal{I}_{{\rm {\bf j}}_{_l } } } }^{\ast} \vert
\{ {G_{b_{rk} } }\} _{k = 1}^{v}
\vert  G_{H_I^V:\left\langle {G''} \right\rangle
_{\mathcal{I}_{{\rm {\bf j}}_r }}\vert H'  }^{ \ast},
\end{flalign}
 respectively.

 Let $ G_{l''} < H \leqslant G_l $.  By Lemma 8.10, $ \partial
_{\tau _{{\rm {\bf I}}_l }^\medstar} ( H ) \geqslant G'\vert S'$.  By Lemma 6.7,  $ H_j^c \leqslant H_I^V<G'\vert S' $ or  $ H_j^c  \| G'\vert S' $  for all  $ S_j^c \in G_{H_I^V:G'' \vert H'}^{ *}  $.
 Thus $ G_{H_I^V:G''\vert H'}^{ \ast} \subseteq H $.
Secondly,   we replace  $ H $  with
$ H\backslash G_{H_I^V:G''\vert H' }^{\ast}\vert \{ {G_{b_{rk} } } \} _{k = 1}^{v} \vert G_{H_I^V:\left\langle {G''} \right\rangle _{\mathcal{I}_{{\rm
{\bf j}}_r }}\vert H'  }^{ \ast}  $  for all  $ G_{l'} \leqslant H \leqslant G_l$.
Let  $m''$  be the number of  $\tau _{{\rm {\bf I}}_{{\rm {\bf j}}_l } }^{ * }  \in
 \tau _{{\rm {\bf I}}_l : G_{q + 1}^{ \circ \circ } (2)}^{\Omega
(q)}$ satisfying  the replacement conditions above,  $ \tau _{{\rm {\bf I}}_l : G_{q + 1}^{\circ \circ } (2)}^{\medstar (q)}  $,  $G_{l''}$  and  $G'_{b_{l'k} } \vert S_{j_{l'k} }^c $  for all  $ 1 \leqslant k \leqslant u' $ be updated as $ \tau _{{\rm {\bf I}}_l : G_{q + 1}^{\circ \circ } (3)}^{\medstar (q)}  $, $G_{l'''}$,  $ G''_{b_{l'k} } \vert S_{j_{l'k} }^c $, respectively.  Then  $ \tau _{{\rm {\bf I}}_l :  G_{q + 1}^{ \circ \circ } (3)}^{\medstar (q)}  $  is a derivation of $G_{l'''}$ and
$G_{l'''}=G_{l''} \backslash
 \{G_{H_I^V:G''\vert H'}^{\ast} \}^{m''}\vert \{\{ {G_{b_{rk} } } \} _{k = 1}^{v} \vert
G_{H_I^V:\left\langle {G''} \right\rangle _{\mathcal{I}_{{\rm
{\bf j}}_r }}\vert H'}^{ \ast}\}^{m''}.$

\textbf{Step 4  (Separation along $H_I^V$).}  Apply the  separation algorithm along $H_I^V$ to  $ G_{l'''}  $  and denote the resulting derivation by  $ \tau _{{\rm {\bf I}}_l: G_{q + 1}^{ \circ \circ } (4)}^{\medstar (q)}  $  whose root is
labeled by  $ G_{q + 1}^{ \cdot \cdot }$.
Then all  $ G_{H_I^V:\left\langle {G''}
 \right\rangle _{\mathcal{I}_{{\rm {\bf j}}_r }} \vert H'}^{\ast} $  in  $ G_{l'''}  $  are transformed into  $ G_{H_I^V:\left\langle {G''} \right\rangle _{\mathcal{I}_{{\rm
{\bf j}}_r }} \vert H'}^{ \medstar(J)}$ in  $ \tau _{{\rm {\bf I}}_l : G_{q + 1}^{ \circ \circ } (4)}^{\Omega
(q)}$.  Since $ \cfrac{G'\vert S' \quad G''\vert S''}{H_I^V=G'\vert G''\vert H'}(II)\in\tau^{\ast},$
$$  \cfrac{\{ {G_{b_{lk} } } \} _{k = 1}^{u} \vert \left\langle {G'\vert S'}
\right\rangle _{\mathcal{I}_{{\rm {\bf j}}_{_l } }
}\quad \{ {G_{b_{rk} } } \} _{k = 1}^{v} \vert \left\langle {G''\vert S''}
\right\rangle _{\mathcal{I}_{{\rm {\bf j}}_{_r} }
}}{\{ {G_{b_{lk} } } \} _{k = 1}^{u} \vert\{ {G_{b_{rk} } } \} _{k = 1}^{v}\vert  \left\langle {G'}
\right\rangle _{\mathcal{I}_{{\rm {\bf j}}_{_l } }
}\vert \left\langle {G''}
\right\rangle _{\mathcal{I}_{{\rm {\bf j}}_{_r} }
}\vert H'}(II)\in
\tau_{\mathbf{I}_{{\rm {\bf j}}_{_l } }\cup
\mathbf{I}_{{\rm {\bf j}}_{_r} }}^{\ast}\in  \tau _{{\rm {\bf I}}_l : G_{q + 1}^{ \circ \circ } (3)}^{\medstar(q)},$$
$H'$,  $ S' $  and  $ S'' $  are  separable in  $ \tau _{{\rm {\bf I}}_l :
      G_{q+ 1}^{ \circ \circ } (4)}^{\medstar (q)}  $  by  a procedure similar to that of Lemma 7.11.
Let  $ S' $  and  $ S'' $  be separated
into  $  \widehat{S'}  $  and  $  \widehat{S''}$, respectively.
By Claim (iii), $G_{H_I^V:G_{l'}}^{\medstar(J)}=G_{q + 1}^{ \circ\circ}.$

\begin{flalign}
G_{H_I^V:G_{l''}}^{\medstar(J)}&=G_{q + 1}^{ \circ\circ}\backslash
 G_{\dagger}^{m'}\vert G_{\ddagger}^{m'}\,\,\, by \,\,\,Lemma \,\,\,7.6(iv), \\
G_{q + 1}^{ \cdot \cdot}&=G_{H_I^V:G_{l'''}}^{\medstar(J)}\\
&=
G_{H_I^V:G_{l''}}^{\medstar(J)}\backslash
\{G_{H_I^V:G''\vert H'}^{\medstar(J)} \}^{m''}\vert \{\{ {G_{b_{rk} } } \} _{k = 1}^{v} \vert
G_{H_I^V:\left\langle {G''} \right\rangle _{\mathcal{I}_{{\rm
{\bf j}}_r }}\vert H'}^{ \medstar(J)} \}^{m''}\\
&=G_{H_I^V:G_{l''}}^{\medstar(J)}\backslash
\{  \widehat{S'} \vert  \widehat{S''}\vert G_{H_I^V:G''}^{ \medstar(J)} \vert G_{H_I^V:H'}^{\medstar (J)}\backslash
\{ \widehat{S'} \vert  \widehat{S''} \} \}^{m''}\vert \\
&\{\{ {G_{b_{rk} } } \} _{k = 1}^{v} \vert
 \widehat{S'} \vert \widehat{S''}  \vert G_{^{H_I^V:\left\langle {G''} \right\rangle _{\mathcal{I}_{{\rm
{\bf j}}_r }} }}^{ \medstar (J)} \vert
G_{H_I^V:H'}^{\medstar (J)}\backslash
\{ \widehat{S'} \vert  \widehat{S''} \} \}^{m''} \\
&=G_{H_I^V:G_{l''}}^{\medstar(J)}\backslash
 G_{\dagger}^{m''}\vert G_{\ddagger}^{m''}\\
&=\{G_{q + 1}^{ \circ\circ}\backslash G_{\dagger}^{m'}\vert G_{\ddagger}^{m'}\}\backslash G_{\dagger}^{m''}\vert G_{\ddagger}^{m''}\\
&=G_{q + 1}^{ \circ\circ}\backslash G_{\dagger}^{m'+m''}\vert G_{\ddagger}^{m'+m''}\\
&=G_{q + 1}^{ \circ\circ}\backslash G_{\dagger}^{m_{q+1}}\vert G_{\ddagger}^{m_{q+1}}
   \end{flalign}
 where $m_{q+1}\coloneq m'+m''$.

\textbf{Step 5 (Put back). }  Replace  $ \tau _{{\rm {\bf I}}_l :
     G_{q + 1}^{\circ \circ } }^{\medstar (q)}  $  in  $ \tau _{{\rm {\bf I}}_l }^{\medstar (q)}  $
with  $ \tau _{{\rm {\bf I}}_l: G_{q + 1}^{ \circ \circ }
(4)}^{\medstar (q)}$ and mark  $ \cfrac{\underline{G_{q + 1}^{ \cdot \cdot }
 }}{G_{q + 1}^\circ }\left\langle {EC_\Omega ^ * } \right\rangle
_{q + 1}^{\circ}  $  as processed, i.e.,  revise $ \left\langle {EC_\Omega ^ * } \right\rangle_{q + 1}^{\circ}$ as $ \left\langle {EC_\Omega ^ * } \right\rangle_{q + 1}^{\cdot}$.
Among leaves of $ \tau _{{\rm {\bf I}}_l :
     G_{q + 1}^{\circ \circ } }^{\medstar (q)}$,  all $G_{q'}^{\circ}$  are updated as $G_{q'}^{\cdot}$  and others  keep unchanged in $ \tau _{{\rm {\bf I}}_l: G_{q + 1}^{ \circ \circ }(4)}^{\medstar (q)}$.  Then this replacement is feasible,  especially,   $G_{q + 1}^{\circ \circ } $ be replaced with $G_{q + 1}^{\cdot\cdot }$. Define the tree resulting from Step 5  to be
 $ \tau _{{\rm {\bf I}}_l }^{\medstar (q + 1)}$. Then Claims (a), (b) and (c) hold
for  $ q + 1 $  by the above construction.

Finally,  we construct a derivation of
$G_{{\rm {\bf I}}_l }^{\medstar}\backslash G_{\dagger}\vert G_{\ddagger}$
from  $ \lceil {S_{l_{1} }^c }
\rceil_{I },  \cdots,
 \lceil {S_{l_{m(l) } }^c }
\rceil_{I } $,
  $ G_{b_{r1} } \vert S_{j_{r1}}^c, \cdots,  $
$G_{b_{rv} } \vert S_{j_{rv}}^c  $ in  $ {\rm {\bf GL}}_{\rm {\bf \Omega }}  $,   which we denote by  $ \tau _{{\rm {\bf I}}_l }^\medstar(\tau _{{\rm
{\bf I}}_{{\rm {\bf j}}_r }}^{ * }) $.

\begin{remark} All elimination rules used in constructing  $ \tau _{{\rm {\bf I}}_l }^\medstar $  are extracted from  $ \tau ^ *$.  Since  $ \tau _{{\rm
{\bf I}}_{{\rm {\bf j}}_r } }^{ * }  $  is a
derivation in  $ {\rm {\bf GL}}_{\rm {\bf \Omega }}  $  without  $ (EC_{\Omega}) $,   we may extract elimination rules from  $ \tau _{{\rm
{\bf I}}_{{\rm {\bf j}}_r } }^{ * }  $  which we may
use to construct  $\tau _{{\rm {\bf I}}_l }^\medstar(\tau _{{\rm
{\bf I}}_{{\rm {\bf j}}_r }}^{ * })$  by a procedure similar to that of
constructing  $ \tau _{{\rm {\bf I}}_l }^\medstar$  with minor revision at every
node  $ H $  that  $ \partial _{\tau _{{\rm {\bf I}}_l }^\medstar} (H) \leqslant
H_I^V$.  Note that updates and replacements in  Steps 2 and 3 are essentially inductive operations but we neglect it for simplicity.

We may also think of constructing $ \tau _{{\rm {\bf I}}_l }^\medstar(\tau _{{\rm
{\bf I}}_{{\rm {\bf j}}_r }}^{ * }) $
as grafting  $ \tau _{{\rm {\bf I}}_{{\rm {\bf j}}_r }
}^{ * }  $  in  $ \tau _{{\rm {\bf I}}_l }^\medstar$  by adding   $ \tau _{{\rm {\bf I}}_{{\rm {\bf j}}_r } }^{ *}
 $  to some $\tau_{{\rm {\bf I}}_{{\rm {\bf j}}_{l}}}^{ *}\in \tau _{{\rm {\bf I}}_l }^\medstar$.  Since the rootstock  $ \tau_{{\rm {\bf I}}_l }^\medstar$  of the
grafting process is invariant in Stage 2, we encapsulate  $\tau _{{\rm {\bf I}}_l }^\medstar (\tau _{{\rm
{\bf I}}_{{\rm {\bf j}}_r }}^{ * }
)$  as an rule  in  $ {\rm {\bf GL}}_{\rm {\bf \Omega }}  $  whose
premises are  $ G_{b_{r1} } \vert S_{j_{r1} }^c
,     G_{b_{r2} }\vert S_{j_{r2} }^c,
\cdots,  G_{b_{rv} } \vert S_{j_{rv} }^c  $  and conclusion is  $   \widehat{S''}  {\kern
1pt} \vert \{ {G_{b_{rk} } } \} _{k = 1}^{v} \vert
G_{^{H_I^V:\left\langle {G''} \right\rangle _{\mathcal{I}_{{\rm
{\bf j}}_r }} }}^{ \medstar (J)} \vert G_{H_I^V:H'}^{\medstar (J)}\backslash
\{ \widehat{S'} \vert  \widehat{S''} \}\vert G_{\mathbf{I}_{l\backslash r} }^\medstar $,   i.e.,

\[\cfrac{\underline{\qquad  G_{b_{r1} } \vert S_{j_{r1}}^c\quad    G_{b_{r2} } \vert
S_{j_{r2}}^c \quad\cdots  \qquad  G_{b_{rv} } \vert
S_{j_{rv} }^c \qquad }}{ \widehat{S''}  {\kern
1pt} \vert \{ {G_{b_{rk} } } \} _{k = 1}^{v} \vert
G_{^{H_I^V:\left\langle {G''} \right\rangle _{\mathcal{I}_{{\rm
{\bf j}}_r }} }}^{ \medstar (J)} \vert G_{H_I^V:H'}^{\medstar (J)}\backslash
\{ \widehat{S'} \vert  \widehat{S''} \}\vert G_{\mathbf{I}_{l\backslash r }}^\medstar}\left\langle
 \tau _{{\rm {\bf I}}_l }^\medstar(\tau _{{\rm{\bf I}}_{{\rm {\bf j}}_r }}^{ * }
) \right\rangle,
\]
 where,   $ G_{{\rm {\bf I}}_{l \backslash r} }^ \medstar= G_{{\rm {\bf I}}_l }^{\medstar}\backslash G_{\dagger} $  is closed.
\end{remark}

 {\bf  Stage 2  Construction of Routine $ \tau _{{\rm {\bf I}}_r }^\medstar(\tau _{{\rm {\bf
I}}_l }^\medstar(\tau _{{\rm {\bf I}}_{{\rm {\bf j}}_r } }^{ * } ))$.}  A sequence  $ \tau _{{\rm {\bf I}}_r }^{\medstar (q)}  $  of
trees for all  $ q \geqslant 0 $  is constructed inductively as follows.   $ \tau _{{\rm {\bf I}}_r }^{\medstar (0)}$,  $ \tau _{{\rm {\bf I}}_r }^{\medstar (q)}  $,
 $ \cfrac{\underline{    {\kern
1pt}   G_{q + 1}^{ \circ \circ }  {\kern
1pt}      }}{G_{q +
1}^ \circ }\left\langle {EC_\Omega ^ * } \right\rangle _{q + 1}^ \circ  $ are defined as those of Stage 1.
Then we perform the following steps  to construct  $ \tau _{{\rm {\bf I}}_r }^{\medstar (q + 1)}  $  in which  $ \cfrac{\underline{ {\kern
1pt}      G_{q + 1}^{
\circ \circ }
  }}{G_{q + 1}^ \circ }\left\langle {EC_\Omega ^ * }
\right\rangle _{q + 1}^ \circ  $  be revised as  $ \cfrac{\underline{{\kern
1pt}
G_{q + 1}^{ \cdot \cdot }
   }}{G_{q + 1}^\circ }\left\langle
{EC_\Omega ^ * } \right\rangle _{q + 1}^ \cdot  $  such that
Claims (a)  and  (b)  are same as those of Stage 1 and
 (c)  $ G_{q + 1}^{ \cdot \cdot}= G_{q + 1}^{ \circ\circ } $  if  $ S'' \notin \left\langle
{G''\vert S''} \right\rangle _{\mathcal{I}_{{\rm {\bf j}}_{_r } }^{({\rm
{\bf t}}_{_r } )} }  $  for all  $ \tau _{\bf{I}_{{\rm {\bf j}}_{_r }
} }^{ *} \in \tau _{{\rm {\bf I}}_r }^\medstar
(G_{q + 1}^{ \circ \circ }
    ) $  otherwise\\
 $G_{q + 1}^{\cdot\cdot}= G_{q + 1}^{ \circ\circ} \backslash
     \{ \widehat{S'}\vert  G_{H_I^V:G'}^{\medstar(J)}  \}^{m_{q+1}}
\vert \{G_{{\rm {\bf I}}_{l\backslash r}  }^\medstar\}^{m_{q+1}}$
  for some $m_{q+1}\geqslant 1$.

\textbf{Step 1 (Delete)}.      $ \tau _{{\rm {\bf I}}_r :
     G_{q + 1}^{\circ \circ } }^{\medstar (q)}  $  and $ \tau _{{\rm {\bf I}}_r:  G_{q + 1}^{ \circ \circ } (1)}^{\medstar (q)}  $ are defined as before.

 \[ \cfrac{\underline{
  G_{b_{r'1} }\vert S_{j_{r'1}}^c\,\,\,   G_{b_{r'2} } \vert S_{j_{r'2}}^c  \,\,\,\cdots \,\,\,G_{b_{r'v'} } \vert S_{j_{r'v'}}^c
 }}{G_{r'} \equiv \{ {G_{b_{r'k} } } \} _{k = 1}^{v'} \vert  G_{\mathcal{I}_{{\rm {\bf j}}_{r'} }}^{\ast} }\left\langle {\tau _{{\rm {\bf I}}_{{\rm {\bf j}}_{r'} }}^{ *} } \right\rangle\in \tau _{{\rm {\bf I}}_r :
     G_{q + 1}^{\circ \circ } }^{\medstar (q)}\]
satisfies
$ \partial _{\tau _{{\rm {\bf I}}_r }^\medstar}
( {G_{b_{r'k} } \vert S_{j_{r'k}}^c }
) > H_I^V  $  for all  $ 1 \leqslant k \leqslant v' $  and
 $ \partial _{\tau _{{\rm {\bf I}}_r }^\medstar} ( {G_{r'} } )
\leqslant H_I^V  $.

\textbf{Step 2 (Update)}.   For all  $ G_{q'}^ \circ \in
\tau _{{\rm {\bf I}}_r :
   G_{q + 1}^{ \circ \circ } (1) }^{\medstar (q)}
 $  which satisfy  $ \cfrac{\underline{
    G_{q'}^{ \cdot \cdot }}}{G_{q'}^\circ}\left\langle {EC_\Omega ^ * } \right\rangle
_{q'}^ \cdot \in \tau _{{\rm {\bf I}}_r}^{\medstar (q)}  $ and  $ S'' \in \left\langle
{G''\vert S''} \right\rangle _{\mathcal{I}_{{\rm {\bf j}}_{_r } }}  $  for some  $ \tau _{{\rm {\bf I}}_{{\rm
{\bf j}}_{_r} } }^{ *} \in \tau
_{{\rm {\bf I}}_r }^\medstar(G_{q'}^{ \circ \circ } ) $,   we replace  $ H $  with  $ H\backslash
\{\widehat{S'} \vert G_{H_I^V:G'}^{\medstar (J)}  \}\vert  G_{{\rm {\bf I}}_{l\backslash r} }^\medstar $  for all  $ H \in\tau _{{\rm {\bf I}}_r :
   G_{q + 1}^{ \circ \circ } (1)}^{\medstar(q)}  $,   $ G_{r'} \leqslant  H \leqslant G_{q'}^{\circ }$.   Then Claims (a) and (b) are proved  by a procedure as before.   Let $m'$ be the number of  $G_{q'}^{\cdot}$ satisfying the  above conditions.   $ \tau _{{\rm {\bf I}}_r : G_{q + 1}^{\circ \circ } (1)}^{\medstar (q)}  $, $G_{r'}$  and $G_{b_{r'k} } \vert
    S_{j_{r'k}}^c $  for all  $ 1 \leqslant k \leqslant v' $ be updated as $ \tau _{{\rm {\bf I}}_r : G_{q + 1}^{\circ \circ } (2)}^{\medstar (q)}  $, $G_{r''}$,  $ G'_{b_{r'k} } \vert S_{j_{r'k}}^c $, respectively.  Then  $ \tau _{{\rm {\bf I}}_r :  G_{q + 1}^{ \circ \circ } (2)}^{\medstar (q)}  $  is a derivation and $G_{r''}=G_{r'} \backslash
\{ \widehat{S'} \vert G_{H_I^V:G'}^{\medstar(J)}\}^{m'}\vert \{G_{{\rm {\bf I}}_{l\backslash r} }^\medstar\}^{m'}$.

\textbf{Step 3 (Replace).} All  $ \tau _{{\rm {\bf I}}_{{\rm {\bf j}}_r } }^{ * }  \in
 \tau _{{\rm {\bf I}}_r : G_{q + 1}^{ \circ \circ } (2)}^{\Omega
(q)}$ are processed  in post-order. If  $ H_i^c \rightsquigarrow H_j^c  $  for all  $ H_i^c  \in
I_{{\rm {\bf j}}_r }  $  and $H_j^c  \in I_l  $  it proceeds by
the following procedure otherwise it remains unchanged.
Let  $\tau _{{\rm {\bf I}}_{{\rm {\bf j}}_r }}^{ * } $ be in the form   \[\cfrac{\underline{
  G_{b_{r1} } \vert S_{j_{r1}}^c\,\,\,   G_{b_{r2} } \vert S_{j_{r2}}^c  \,\,\,\cdots \,\,\,G_{b_{rv} } \vert
S_{j_{rv}}^c}}{G_r \equiv \{ {G_{b_{rk} } } \} _{k = 1}^{v}  \vert G_{\mathcal{I}_{{\rm {\bf j}}_r }}^{ \ast} }.\]
Then
there exists the unique $1\leqslant k'\leqslant v'$ such that $G_{r''}<G_{b_{r'k'} } \vert S_{j_{r'k'}}^c\leqslant G_r$.

Firstly, we replace
 $ \tau _{{\rm {\bf I}}_{{\rm {\bf j}}_r } }^{ *}  $
with  $\tau _{{\rm {\bf I}}_l }^\medstar(\tau _{{\rm
{\bf I}}_{{\rm {\bf j}}_r }}^{ *}
)$.  We may rewrite the roots of  $ \tau _{{\rm {\bf I}}_{{\rm
{\bf j}}_r } }^{ *}  $,   $\tau _{{\rm {\bf I}}_l }^\medstar(\tau _{{\rm
{\bf I}}_{{\rm {\bf j}}_r }}^{ *}
)$  as
$G_r =\{G_{b_{rk} }\}_{k = 1
}^{ v}  \vert
G_{H_I^V:\left\langle {G''} \right\rangle _{\mathcal{I}_{j_r
} } }^{\ast} \vert
 G_{H_I^V:G'\vert H'}^{\ast},$
$G_{l\backslash r} \equiv  \{G_{b_{rk} } \}_{k = 1}^{v} \vert  \widehat{S''}   \vert
G_{^{H_I^V:\left\langle {G''} \right\rangle _{\mathcal{I}_{{\rm
{\bf j}}_r }} }}^{ \medstar (J)} \vert
 G_{H_I^V:H'}^{ \medstar (J)}\backslash\{ \widehat{S'} \vert  \widehat{S''} \}
\vert G_{{\rm {\bf I}}_{l\backslash r} }^\medstar, $  respectively.

Let  $ G_{r''}
< H \leqslant G_r  $.  Then  $ \partial _{\tau _{{\rm {\bf I}}_r }^\medstar}
( H ) \geqslant G''\vert S'' $ by Lemma 8.10.   Thus  $ G_{H_I^V:G'\vert H'}^{\medstar (0)} \subseteq H,  \{S_j^c:S_j^c \in G_{H_I^V:\left\langle {G''}
\right\rangle _{\mathcal{I}_{j_r } } }^{\ast}, H_j^c
\geqslant G''\vert S'' \}
=\{S_j^c:S_j^c \in G_{H_I^V:\left\langle {G''}
 \right\rangle _{\mathcal{I}_{j_r }} }^{ \medstar (J)}, H_j^c \geqslant G''\vert S'' \}.$  Define
  $ G_H^{ ** }= \{S_j^c:S_j^c \in
G_{H_I^V:\left\langle {G''} \right\rangle _{\mathcal{I}_{j_r
} } }^{ * }, S_j^c\mathrm{\,\,\,be \,\,\,the\,\,\, focus \,\,\,sequent\,\,\, of \,\,\,some\,\,\, } H' \in \tau _{{\rm {\bf I}}_r :   G_{q + 1}^{ \circ \circ } (2)}^{\medstar(q)}, H \leqslant H' \leqslant G_r\}$.

Then we replace  $ H $  with
\[ H\backslash \{G_{H_I^V:\left\langle {G''} \right\rangle
_{\mathcal{I}_{j_r } } }^{ \ast}\backslash G_H^{ ** } \vert
G_{H_I^V:G'\vert H'}^{ \ast} \}
\vert  \widehat{S''}   \vert\{G_{^{H_I^V:\left\langle {G''} \right\rangle
_{\mathcal{I}_{{\rm {\bf j}}_r }} }}^{
\medstar (J)} \backslash G_H^{ ** } \}\vert G_{H_I^V:H'}^{ \medstar(J)}\backslash
\{ \widehat{S'} \vert  \widehat{S''}\}
 \vert G_{{\rm {\bf I}}_{l\backslash r} }^\medstar \]  for all  $G_{b_{r'k'} } \vert S_{j_{r'k'}}^c\leqslant H \leqslant G_r$.

Let  $m''$  be the number of  $ \tau _{{\rm {\bf I}}_{{\rm {\bf j}}_r } }^{ *} \in
 \tau _{{\rm {\bf I}}_r : G_{q + 1}^{ \circ \circ } (2)}^{\Omega
(q)}$ satisfying  the replacement conditions as above,  $ \tau _{{\rm {\bf I}}_r : G_{q + 1}^{\circ \circ } (2)}^{\medstar (q)}  $,  $G_{r''}$  and  $G'_{b_{r'k} } \vert S_{j_{r'k}}^c $  for all  $ 1 \leqslant k \leqslant v' $ be updated as $ \tau _{{\rm {\bf I}}_r : G_{q + 1}^{\circ \circ } (3)}^{\medstar (q)}  $, $G_{r'''}$,  $ G''_{b_{r'k} } \vert S_{j_{r'k} }^c $, respectively.  Then  $ \tau _{{\rm {\bf I}}_r :  G_{q + 1}^{ \circ \circ } (3)}^{\medstar (q)}  $  is a derivation and
 $G_{r'''}=G_{r''} \backslash H_{1}^{m''}\vert H_{2}^{m''} $, where
 \begin{flalign}
 H_{3}&=G_{G'_{b_{r'k} } \vert S_{j_{r'k}}^c}^{\ast\ast},\\
H_{1}&=G_{H_I^V:\left\langle {G''} \right\rangle
_{\mathcal{I}_{j_r } } }^{\ast} \backslash H_{3} \vert
G_{H_I^V:G'\vert H'}^{\ast},\\
H_{2}&= \widehat{S''}   \vert
G_{^{H_I^V:\left\langle {G''} \right\rangle
_{\mathcal{I}_{{\rm {\bf j}}_r }} }}^{\medstar(J)} \backslash H_{3}\vert G_{H_I^V:H'}^{ \medstar (J)}\backslash \{ \widehat{S'} \vert  \widehat{S''}\} \vert
  G_{{\rm {\bf I}}_{l\backslash r} }^\medstar.
 \end{flalign}

\textbf{Step 4 (Separation along $H_I^V$).} Apply  the  separation algorithm along $H_I^V$  to
 $ G_{r'''}  $  and denote the resulting derivation by  $ \tau _{{\rm {\bf I}}_r
: G_{q + 1}^{ \circ \circ } (4)}^{\medstar (q)}  $  whose root is
labeled by  $ G_{q + 1}^{ \cdot \cdot }$.

By Claim (iii),  $G_{H_I^V:G_{r'}}^{\medstar(J)}=G_{q + 1}^ {\circ\circ}$.
\begin{flalign}
G_{H_I^V:G_{r''}}^{\medstar(J)}&=G_{q + 1}^ {\circ\circ}\backslash
\{ G_{H_I^V:G'}^{\medstar (J)} \vert  \widehat{S'}\}^{m'}\vert
\{G_{{\rm {\bf I}}_{l\backslash r} }^\medstar\}^{m'},\\
G_{H_I^V:H_1}^{\medstar(J)}&=G_{H_I^V:\left\langle {G''} \right\rangle
_{\mathcal{I}_{j_r } } }^{ \medstar(J)} \backslash  G_{H_I^V:H_3 }^{\medstar(J)} \vert
 G_{H_I^V:G'\vert H'}^{ \medstar(J)},\\
 G_{H_I^V:H_2}^{ \medstar(J)}&= \widehat{S''}   \vert
G_{^{H_I^V:\left\langle {G''} \right\rangle
_{\mathcal{I}_{{\rm {\bf j}}_r }} }}^{\medstar(J)} \backslash G_{H_I^V:H_3 }^{ \medstar (J)}\vert G_{H_I^V:H'}^{\medstar (J)}\backslash
\{ \widehat{S'} \vert  \widehat{S''}\} \vert G_{{\rm {\bf I}}_{l\backslash r} }^ \medstar.
 \end{flalign}Then
 \begin{flalign}
 G_{H_I^V:G_{r'''}}^{ \medstar(J)}&=G_{H_I^V:G_{r''}}^{\medstar(J)}\backslash
\{G_{H_I^V:G'\vert H'}^{\medstar(J)}\}^{m''}\vert\{ \widehat{S''}\vert G_{H_I^V:H'}^{\medstar(J)}\backslash
\{ \widehat{S'} \vert  \widehat{S''}\} \vert G_{{\rm {\bf I}}_{l\backslash r} }^\medstar\}^{m''}\\
&=G_{H_I^V:G_{r''}}^{\medstar(J)}\backslash
\{G_{H_I^V:G'}^{ \medstar (J)}\vert  \widehat{S'}\}^{m''}
\vert \{G_{{\rm {\bf I}}_{l\backslash r} }^\medstar\}^{m''}.
 \end{flalign}
Then
$$G_{q + 1}^{ \cdot \cdot }=G_{H_I^V:G_{r'''}}^{\medstar(J)}=
G_{q + 1}^ {\circ\circ}\backslash
\{\widehat{S'} \vert G_{H_I^V:G'}^{\medstar (J)}  \}^{m_{q+1}}\vert \{G_{\mathbf{I}_{l \backslash r}}^ \medstar\}^{m_{q+1}}$$
where $m_{q+1} \coloneq m'+m''$.

\textbf{Step 5 (Put back).} Replace  $ \tau _{{\rm {\bf I}}_r:  G_{q + 1}^{\circ \circ } }^{\medstar (q)}  $  in  $ \tau _{{\rm {\bf I}}_r }^{\medstar (q)}  $  with  $ \tau _{{\rm {\bf I}}_r :  G_{q + 1}^{ \circ \circ }
(4)}^{\medstar (q)}  $  and    revise  $ \cfrac{\underline{G_{q + 1}^{ \cdot \cdot }
 }}{G_{q + 1}^\circ }\left\langle {EC_\Omega ^ * } \right\rangle
_{q + 1}^{\circ}  $  as  $ \cfrac{\underline{G_{q + 1}^{ \cdot \cdot }
 }}{G_{q + 1}^\circ }\left\langle {EC_\Omega ^ * } \right\rangle
_{q + 1}^{\cdot} $.
Define the resulting tree from Step 5 to be
 $ \tau _{{\rm {\bf I}}_r }^{\medstar (q + 1)}$ then Claims (a), (b) and (c) hold
for  $ q + 1 $  by the above construction.

Finally,  we construct a derivation of
$G_{{\rm {\bf I}}_r }^{\medstar}\backslash \{ \widehat{S'}\vert  G_{H_I^V:G'}^{\medstar(J)}  \}
\vert G_{{\rm {\bf I}}_{l\backslash r}  }^\medstar$
from  $ \lceil {S_{i_1
}^c } \rceil_{I}  $,  \ldots,  $ \lceil {S_{i_m }^c }
\rceil_{I}  $  in  $ {\rm {\bf GL}}_{\rm {\bf \Omega }}  $.
Since the major operation of Stage 2 is to replace  $ \tau _{{\rm {\bf I}}_{{\rm
{\bf j}}_r }}^{ * }  $  with  $ \tau
_{{\rm {\bf I}}_l }^\medstar(\tau _{{\rm {\bf I}}_{{\rm {\bf j}}_r }^{(
{{\rm {\bf t}}_r } )} }^{ * } ) $  for all  $ \tau _{{\rm {\bf
I}}_{{\rm {\bf j}}_r }}^{ * } \in
\tau _{{\rm {\bf I}}_r }^\medstar$ satisfying  $ S'' \in  \left\langle
{G''\vert S''} \right\rangle _{\mathcal{I}_{{\rm {\bf j}}_r }^{( {{\rm
{\bf t}}_r } )} }  $,   then we denote the resulting derivation from Stage 2 by
 $ \tau _{{\rm {\bf I}}_r }^\medstar(\tau _{{\rm {\bf
I}}_l }^\medstar(\tau _{{\rm {\bf I}}_{{\rm {\bf j}}_r } }^{ * } ))$.

In the following, we prove that the claims from (i) to (iv) hold if $ \tau _{\rm {\bf I}}^\medstar\coloneq\tau _{{\rm {\bf I}}_r }^\medstar(\tau _{{\rm {\bf
I}}_l }^\medstar(\tau _{{\rm {\bf I}}_{{\rm {\bf j}}_r } }^{ * } ))$ and
$G_{\rm {\bf I}}^{\medstar}\coloneq G_{{\rm {\bf I}}_r }^{\medstar}\backslash \{ \widehat{S'}\vert  G_{H_I^V:G'}^{\medstar(J)}  \}\vert G_{{\rm {\bf I}}_{l\backslash r}  }^\medstar$.

$\bullet$ For Claim (i), (ii): Let $ \cfrac{\underline{H_1 \,\,\,    \cdots  \,\,\,  H_w }}{H_0 }\left\langle {\tau _{{\rm {\bf I}}_{{\rm{\bf j}} } }^{ * } } \right\rangle \in
\tau _{{\rm {\bf I}}}^{\medstar} $ and  $ S_j^c \in G_{\mathcal{I}_{{\rm {\bf j}}}}^{ \ast}$.
Then  $\partial _{\tau _{\rm {\bf I}}^\medstar}( H_{k}) \nleqslant H_j^c  $  for all $1\leqslant k\leqslant w$ by Lemma 6.13(iv).

If  $ \partial _{\tau _{\rm {\bf I}}^\medstar}
( H_{k'}) \leqslant H_I^V $ for some $1\leqslant k'\leqslant w$,
then $ H_i^c \nleqslant H_j^c  $  for all $ H_i^c \in I$  by $ \partial _{\tau _{\rm {\bf I}}^\medstar}( H_{k'}) \leqslant H_I^V\leqslant H_i^c$.  Thus Claim (i) holds and Claim (ii) holds by Lemma 8.5(5)  and Lemma 7.6(i).  Note that Lemma 8.5(5)  is independent of Claims from  (ii) to (iv).

Otherwise  $\tau _{{\rm {\bf I}}_{{\rm{\bf j}} } }^{ * }$  is
built up  from  $ \tau _{{\rm {\bf I}}_{{\rm {\bf j}}_r }}^{ * } \in \tau _{{\rm {\bf I}}_r }^\medstar$, $ \tau _{{\rm {\bf I}}_{{\rm {\bf j}}_l }}^{ * } $ or $ \tau _{{\rm {\bf I}}_{{\rm {\bf j}}_{_l } } \cup {\rm {\bf I}}_{{\rm {\bf j}}_{_r }
} }^{ * }\in \tau _{{\rm {\bf I}}_l
}^\medstar(\tau _{{\rm {\bf I}}_{{\rm {\bf j}}_r }}^{ * } ) $ by keeping their focus and principal sequents unchanged and making their side-hypersequents possibly  to be modified,  but which has no effect on discussing Claim (ii) and then Claim (ii) holds for $\tau _{{\rm {\bf I}}}^\medstar$ by the induction hypothesis on Claim (ii) of  $\tau _{{\rm {\bf I}}_l
}^\medstar$ or $\tau _{{\rm {\bf I}}_r
}^\medstar$.

If $\tau _{{\rm {\bf I}}_{{\rm{\bf j}} } }^{ * }$   is from  $ \tau _{{\rm {\bf I}}_{{\rm {\bf
j}}_{_l } } \cup {\rm {\bf I}}_{{\rm {\bf j}}_{_r }
} }^{ * }$   then $ S' \in
\left\langle {G'\vert S'} \right\rangle _{\mathcal{I}_{{\rm {\bf j}}_l
} }  $  and
$ S'' \in \left\langle {G''\vert
S''} \right\rangle _{\mathcal{I}_{{\rm {\bf j}}_r }}  $ by the choice of  $ \tau _{{\rm {\bf I}}_{{\rm {\bf j}}_l
} }^{ * }  $  and  $ \tau _{{\rm {\bf I}}_{{\rm {\bf j}}_r
}}^{ * }  $  at Stage 1. By the induction hypothesis,
 $ H_i^c \nleqslant H_j^c  $  for all  $ S_j^c \in G_{{\rm
{\mathcal{I}}}_{{\rm {\bf j}}_l } }^{\ast }, $
$ H_i^c \in I_l  $ and  $ H_i^c \nleqslant H_j^c  $  for all  $ S_j^c \in G_{{\rm
{\mathcal{I}}}_{{\rm {\bf j}}_r}}^{\ast }, $
$ H_i^c \in I_r  $.  Then $ H_i^c \nleqslant H_j^c  $  for all  $ S_j^c
\in G_{\mathcal{I}_{{\rm {\bf j}}}}^{ \ast}=G_{{\rm {\mathcal{I}}}_{{\rm {\bf j}}_{_l } } \cup{\rm {\mathcal{I}}}_{{\rm {\bf j}}_{_r } } }^{ \ast }$,
 $ H_i^c \in I$ by  $ G_{{\rm {\mathcal{I}}}_{{\rm {\bf j}}_{_l }
} \cup {\rm {\mathcal{I}}}_{{\rm {\bf j}}_{_r } }}^{\ast} = G_{{\rm {\mathcal{I}}}_{{\rm {\bf j}}_{_l } }}^{\ast } \bigcap G_{{\rm {\mathcal{I}}}_{{\rm {\bf j}}_{_r }} }^{\ast}  $,    $ I = I_l \cup I_r  $.

If $\tau _{{\rm {\bf I}}_{{\rm{\bf j}} } }^{ * }$  is from  $ \tau _{{\rm {\bf I}}_{{\rm {\bf j}}_l } }^{ * }$   then  $ S'
\notin \left\langle {G'\vert S'} \right\rangle _{\mathcal{I}_{{\rm {\bf
j}}_l } }  $  by  Step 3 at Stage 1.  Then
$ \left\langle {G'\vert
G''\vert H'} \right\rangle _{\mathcal{I}_{{\rm {\bf j}}_l } } \bigcap ( {G''\vert H'} ) = \emptyset  $.  Thus  $ S_j^c \notin G_{H_{I}^{V}:G'' \vert H'}^{ \ast}$.  Hence $ G''\vert S''\nleqslant H_j^c  $.  Therefore  $H_i^c \nleqslant H_j^c$ for all  $ H_i^c \in I_r$ by $ G''\vert S''\leqslant H_i^c  $.  Thus  $H_i^c \nleqslant H_j^c$ for all  $ H_i^c \in I$ by $ S_j^c \in G_{\mathcal{I}_{{\rm {\bf j}}}}^{ \ast}=G_{{\rm {\mathcal{I}}}_{{\rm {\bf j}}_{_l } } }^{\ast } $ and the induction hypothesis from  $ \tau _{{\rm {\bf I}}_{{\rm {\bf j}}_l } }^{ * }\in \tau _{{\rm {\bf I}}_l}^\medstar$.
The case of $\tau _{{\rm {\bf I}}_{{\rm{\bf j}} } }^{ * }$   built up from  $ \tau _{{\rm {\bf I}}_{{\rm {\bf j}}_r } }^{ * }$ is proved by a procedure similar to above and omitted.

$\bullet$  Claim (iii) holds by Step 4  at Stage 1 and 2.  Note that  in the whole of Stage 1,  we treat
 $\{ {G_{b_{rk} } } \} _{k = 1}^{v}$ as a side-hypersequent. But it is possible that there exists  $S_{j}^{c}\in\{ {G_{b_{rk} } } \} _{k = 1}^{v}$ such that  $H_{j}^{c}\leqslant H_{I}^{V}$.  Since
 we haven't  applied  the  separation algorithm to  $\{ {G_{b_{rk} } } \} _{k = 1}^{v}$  in Step 4  at Stage 1, then  it  could make Claim (iii) invalid.  But it is not difficult to find that we
just move  the separation of such $S_{j}^{c}$ to Step 4  at Stage 2. Of course, we can move it to
Step 4  at Stage 1, but which make the discussion complicated.

$\bullet$  For Claim (iv),  we prove (1) $ H_i^c  \| H_j^c  $  for all  $ S_j^c \in G_{{\rm {\bf I}}_{l \backslash r} }^ \medstar$ and $H_i^c \in I $,   (2) $ H_i^c  \| H_j^c  $  for all  $ S_j^c \in G_{{\rm {\bf I}}_r }^{\medstar}\backslash \{ \widehat{S'}\vert  G_{H_I^V:G'}^{\medstar(J)}  \}$ and $ H_i^c \in I $.  Only (1)  is proved as follows and (2) by a similar procedure and omitted.

Let  $S_j^c \in G_{{\rm {\bf I}}_{l \backslash r} }^ \medstar$.   Then  $S_j^c \in G_{{\rm {\bf I}}_{l } }^ \medstar$
and $S_j^c \notin \widehat{S''} \vert G_{H_I^V:G''}^{ \medstar(J)}\vert G_{H_I^V:H'}^{\medstar (J)}\backslash\{ \widehat{S'} \vert  \widehat{S''} \}$ by the definition of $ G_{{\rm {\bf I}}_{l \backslash r} }^ \medstar$.
By a procedure similar to that of  Claim (iv) in Case 1,  we get  $H_j^c\nleqslant  H_I^V$  and  assume that $ S_j^c \in G_{\mathcal{I}_{{\rm {\bf j}}_{l}} }^{\ast} $ for some
$ \tau _{{\rm {\bf{I}}}_{{\rm {\bf j}}_{l}} }^{ * }\in \tau _{{\rm {\bf I}}_{l}}^\medstar$ and  let  $G'\vert S'\nleqslant H_j^c$  in the following.

Suppose that  $G''\vert S''\leqslant H_j^c$.  Then $ S_j^c \in   G_{H_I^V:G''}^{ \ast}$ and $ S' \in \left\langle {G'\vert S' } \right\rangle _{\mathcal{I}_{{\rm {\bf j}}_{l}} }  $ by  $ S_j^c \in G_{\mathcal{I}_{{\rm {\bf j}}_{l}} }^{\ast} $.   Hence $ S_j^c \in  G_{H_I^V:G''}^{ \medstar(J)}$ by $H_j^c \geqslant G''\vert S''>H_I^V $. Therefore
$S_j^c \in \widehat{S''} \vert G_{H_I^V:G''}^{ \medstar(J)}\vert G_{H_I^V:H'}^{\medstar (J)}\backslash\{ \widehat{S'} \vert  \widehat{S''} \}$, a contradiction thus $G''\vert S''\nleqslant H_j^c$.  Then
   $H_I^V \nless  H_j^c $ by  $G'\vert S'\nleqslant H_j^c$  and  $G''\vert S''\nleqslant H_j^c$.  Thus  $ H_j^c  \| H_I^V $.  Hence  $ H_j^c  \| H_i^c  $  for all  $H_i^c \in I$. This completes  the proof of Theorem 8.2.

\end{proof}
\begin{definition}
The manipulation described in Theorem 8.2 is called  derivation-grafting operation.
\end{definition}

\section{The proof of Main  theorem }
Recall that in Main  theorem ${G_0 \equiv G'\vert \{ {\Gamma _i,p
\Rightarrow \Delta _i } \} _{i = 1 \cdots n} \vert \{ {\Pi _j
\Rightarrow p,\Sigma _j  } \} _{j = 1 \cdots m}}.$

\begin{lemma}
(i) If  $ G_2= G_0 \backslash \{\Gamma _1,p
\Rightarrow \Delta _1 \} $
 and $  \vdash _{{\rm {\bf GL}}} \mathcal{D}_0( {G_2 } ) $
 then  $  \vdash _{{\rm {\bf GL}}} \mathcal{D}_0 (
{G_0 })$;\\
$(i') $ If  $ G_2= G_0 \backslash \{\Pi _1 \Rightarrow p,\Sigma _1  \} $  and
 $\vdash _{{\rm {\bf GL}}} \mathcal{D}_0( {G_2} ) $
 then  $  \vdash _{{\rm {\bf GL}}} \mathcal{D}_0 ({G_0 })$;\\
(ii) If  $ G_2 = G_0 \vert \{\Gamma _1,p \Rightarrow \Delta _1 \} $  and  $
\vdash _{{\rm {\bf GL}}} \mathcal{D}_0 ( {G_2} ) $  then  $  \vdash
_{{\rm {\bf GL}}} \mathcal{D}_0 ( {G_0 } ) $;\\
$(ii')$  If  $ G_2 = G_0 \vert \{\Pi _1 \Rightarrow p,\Sigma _1 \} $  and  $  \vdash
_{{\rm {\bf GL}}} \mathcal{D}_0 ( {G_2} ) $
then  $  \vdash _{{\rm{\bf GL}}} \mathcal{D}_0 ( {G_0 } ) $;\\
(iii) If  $ G_2 = G_0 \backslash \{\Gamma _1,p \Rightarrow \Delta _1 \}\vert
\{\Gamma _1, \top \Rightarrow \Delta _1 \} $
and  $  \vdash _{{\rm {\bf GL}}}
\mathcal{D}_0 ( {G_2 } ) $
 then $  \vdash _{{\rm {\bf GL}}}
\mathcal{D}_0 ( {G_0 } ) $;\\
$(iii') $  If  $ G_2 = G_0 \backslash \Pi _1 \Rightarrow p,\Sigma _1\vert
\Pi _1 \Rightarrow \bot,\Sigma _1 $
and  $  \vdash _{{\rm {\bf GL}}}\mathcal{D}_0 ( {G_2} ) $
then $  \vdash _{{\rm {\bf GL}}}\mathcal{D}_0 ( {G_0 } )$.
\end{lemma}

\begin{proof}
(i) Since  $ \mathcal{D}_0 ( {G_2} ) = G'\vert
\{\Gamma _i,\Pi _j \Rightarrow \Delta _i,\Sigma _j \} _{i = 2
\cdots n;j = 1 \cdots m}\,\subseteq\, G'\vert
\{\Gamma _1,\Pi _j \Rightarrow
\Delta _1,\Sigma _j \}_{j = 1 \cdots m} \vert\\
 \{\Gamma _i,\Pi
_j \Rightarrow \Delta _i,\Sigma _j \} _{i = 2 \cdots n;j = 1
\cdots m} = \mathcal{D}_0 ( {G_0} ) $  then
 $  \vdash _{{\rm {\bf
GL}}} \mathcal{D}_0 ( {G_0 } ) $  holds.
If $n=1$,   we replace
all $p$  in $\Pi _j \Rightarrow p,\Sigma _j $  with $\perp$.
Then
 $  \vdash _{{\rm {\bf
GL}}} \mathcal{D}_0 ( {G_0 } ) $  holds by applying  $(CUT)$ to
 $\Gamma _1,\perp \Rightarrow \Delta _1 $  and
 $ G' \vert \{ {\Pi _j\Rightarrow \perp,\Sigma _j  } \} _{j = 1 \cdots m}.$

(ii) Since  $ \mathcal{D}_0 ( {G_2 }
) = G'\vert \{\Gamma _1,\Pi _j \Rightarrow \Delta _1,\Sigma _j
\} _{j = 1 \cdots m} \vert
\{\Gamma _i,\Pi _j \Rightarrow \Delta
_i,\Sigma _j \} _{i = 1 \cdots n;j = 1 \cdots m}  $  then\\
  $  \vdash
_{{\rm {\bf GL}}} \mathcal{D}_0 ( {G_0 } ) $  holds by applying
 $ (EC^*) $  to  $ \mathcal{D}_0 ( {G_2 } )$.

(iii) Since  $ \mathcal{D}_0
( {G_2 } ) = G'\vert \Gamma _1, \top \Rightarrow \Delta _1 \vert
\{\Gamma _i,\Pi _j \Rightarrow \Delta _i,\Sigma _j \} _{i = 2
\cdots n;j = 1 \cdots m}  $  then  $  \vdash _{{\rm {\bf GL}}} G'' \equiv
G'\vert \Gamma _1,\Pi _1 \Rightarrow \Delta _1,\Sigma _1 \vert \{\Gamma _i
,\Pi _j \Rightarrow \Delta _i,\Sigma _j \} _{i = 2 \cdots n;j =
1 \cdots m}  $  holds by applying  $ (CUT) $  to  $ \Gamma _1, \top \Rightarrow
\Delta _1  $  in  $ \mathcal{D}_0 ( {G_2 } ) $  and  $ \Pi _1 \Rightarrow
\top,\Sigma _1  $.   Thus  $  \vdash _{{\rm {\bf GL}}} \mathcal{D}_0 ( {G_0
} ) $  holds by applying  $ (EW) $  to  $ G''$.

$ (\mathrm{i}'), (\mathrm{ii}') $ and $(\mathrm{iii}') $  are proved by a
procedure respectively  similar to those of (i), (ii) and (iii) and omitted.
\end{proof}

Let  $ I =  \{ {H_{i_1 }^c, \cdots,H_{i_m }^c } \}  \subseteq
\{ {H_1^c, \cdots,H_N^c } \} $,    $ G_I  $  denote a closed hypersequent
such that  $ G_I \subseteq _c G\vert G^ *  $  and  $ H_j^c  \| H_i^c  $  for all  $ S_j^c
\in G_I  $  and  $ H_{i }^c \in I$.

\begin{lemma}
There exists  $ G_I  $  such that  $  \vdash _{{\rm {\bf
GL_{\Omega}}}} G_I $  for all
$ I \subseteq \{ {H_1^c, \cdots,H_N^c }\}$.
\end{lemma}

\begin{proof}
 The proof is by induction on  $m$.   For the base step, let  $m=0$, then $I = \emptyset  $  and  $ G_I \coloneq G\vert G^ *  $  and  $  \vdash_{{\rm {\bf GL_{\Omega}}}} G_I  $  by Lemma 4.17(v).

For the induction step, suppose that  $m \geqslant 1 $  and there exists  $ G_I  $
such that  $  \vdash _{{\rm {\bf GL_{\Omega}}}} G_I $
 for all  $ \left| I \right| \leqslant m - 1$.
 Then there exist
 $ G_{I\backslash \{H_{i_k }^c \}}  $  for all  $ 1 \leqslant k \leqslant m$ such that  $  \vdash _{{\rm {\bf GL_{\Omega}}} }G_{I\backslash \{H_{i_k }^c \}}  $  and  $ H_j^c  \| H_i^c  $  for
all  $ S_j^c \in G_{I\backslash \{H_{i_k }^c \}}  $  and  $ H_{i}^c\in I\backslash \{H_{i_k }^c \}$.

If  $ H_j^c  \| H_{i_k }^c  $  for all  $ S_j^c \in G_{I\backslash \{H_{i_k }^c \}}  $  then $ G_I \coloneq G_{I\backslash \{H_{i_k }^c \}}$  and the claim holds clearly.
Otherwise there
exists  $ S_j^c \in G_{I\backslash \{H_{i_k }^c \}}  $  such that  $ H_j^c \leqslant
H_{i_k }^c  $  or  $ H_j^c > H_{i_k }^c  $  then we  rewrite  $ G_{I\backslash \{H_{i_k }^c \}}  $  as  $ \lceil {S_{i_k '}^c } \rceil_{\{H_{i_k' }^c \}\cup I\backslash \{H_{i_k }^c \}} $,   where we define  $ H_{i_k '}^c$ such that $ S_{i_k '}^c \in
G_{I\backslash \{H_{i_k }^c \}}  $ and,  $ S_j^c\in G_{I\backslash \{H_{i_k}^c \}}$ implies  $ H_j^c \leqslant H_{i_k '}^c  $ or $ H_j^c  \| H_{i }^c  $ for all  $ H_{i}^c\in \{H_{i_k' }^c \}\cup I\backslash \{H_{i_k }^c \}$.   If we can't define $ G_I $ to be $ G_{I\backslash \{H_{i_k }^c \}}  $  for  each  $ 1 \leqslant k \leqslant m$,
let $ I '\coloneq  \{ {H_{i_1' }^c, \cdots,H_{i_m' }^c } \} $.
Then  $ G_{I'} $  is
constructed by applying the separation algorithm of multiple branches (or one  branch if $m=1$) to
 $ \lceil {S_{i_1 '}^c } \rceil_{I'}, \cdots,\lceil {S_{i_m '}^c } \rceil_{I'}  $
 Then  $  \vdash _{{\rm {\bf GL}_{\Omega}}}G_{I'} $ by $ \vdash _{{\rm {\bf GL}_{\Omega}}} \lceil {S_{i_1 '}^c } \rceil_{I'}, \cdots, \vdash _{{\rm {\bf GL}_{\Omega}}}\lceil {S_{i_m '}^c } \rceil_{I'}  $,  Theorem 8.2 (or Lemma 7.8 (i) for one branch).  Let  $ G_I \coloneq G_{I'}$ then $  \vdash _{{\rm {\bf GL}_{\Omega}}}G_{I} $ clearly.
\end{proof}

\textbf{The proof of Main  theorem:} Let  $ I = \{ {H_1^c, \cdots,H_N^c }\} $  in Lemma 9.2.  Then there exists  $ G_I  $  such that  $  \vdash _{{\rm {\bf GL_{\Omega}}} } G_I $,    $ G_I \subseteq _c G\vert G^ *  $  and  $ H_j^c  \| H_i^c  $  for all  $ S_j^c \in G_I  $  and  $H_i^c\in I$.  Then $  \vdash _{{\rm {\bf
GL}}} \mathcal{D}(G_I ) $ by Lemma 5.6.

Suppose that  $ S_j^c \in G_I  $.  Then
$ H_j^c  \| H_i^c  $  for all  $H_i^c\in I$.  Thus $ H_j^c  \| H_j^c  $ by $H_j^c\in I$,  a contradiction with  $ H_j^c \leqslant H_j^c  $ and hence there doesn't exist  $ S_j^c \in G_I  $.  Therefore $ G_I \subseteq _c G$ by $ G_I \subseteq _c G\vert G^ *  $.

By removing
the identification number of each occurrence of  $ p $  in  $G$,   we obtain the
sub-hypersequent $ G_2$ of $G_2\vert G_2^ * $, which is the root of $\tau^4$ resulting from Step 4 in Section 4.  Then $  \vdash _{{\rm {\bf GL}}}
\mathcal{D}_0 (G_2) $ by  $  \vdash _{{\rm {\bf GL}} }
\mathcal{D}(G_I ) $ and $ G_I \subseteq _c G$. Since $G_2$  is constructed by adding or removing some $\Gamma _i,p\Rightarrow \Delta _i $ or $\Pi _j\Rightarrow p,\Sigma _j $ from $G_0$,  or  replacing  $\Gamma _i,p\Rightarrow \Delta _i $ with $\Gamma _i,\top\Rightarrow \Delta _i $, or   $\Pi _j\Rightarrow p,\Sigma _j $ with $\Pi _j\Rightarrow \perp,\Sigma _j $,  then  $  \vdash _{{\rm {\bf GL}}} \mathcal{D}_0 ({G_0 } ) $  by Lemma 9.1.  This completes the proof of Main  theorem.$
\largesquare
$

\begin{theorem}
Density elimination holds for all ${\rm {\bf GL}}$ in
 $\{{\rm {\bf GUL}},{\rm {\bf
GIUL}},{\rm {\bf GMTL}},{\rm {\bf GIMTL}}\}.$
\end{theorem}
\begin{proof}
It follows immediately from Main  theorem.
\end{proof}

\section {Final remarks and open problems}
Recently, we  have generalized our method described in this paper to the non-commutative substructural logic $\mathbf{GpsUL}^{\ast}$  in [24]. This result shows that $\mathbf{GpsUL}^{\ast}$  is the logic of pseudo-uninorms and their residua and answered the question posed by Prof. Metcalfe, Olivetti, Gabbay and Tsinakis in [17, 18].

It has often been the case in the past that metamathematical proofs of the standard completeness have the corresponding algebraic ones, and vise verse. In particular,  Baldi and Terui [3] had given an algebraic proof of the standard completeness of {\bf UL} and their method had also been extended by Galatos  and Horcik [11].
A natural problem is whether there is an algebraic proof corresponding to
 our proof-theoretic one. It seems difficult to obtain it  by using the insights gained from the approach described in this paper  because ideas and syntactic manipulations introduced here are complicated
 and specialized.  In addition, Baldi and Terui [3] also mentioned some open problems.  Whether our method could be applied to their problems is another research direction.

\section*{Acknowledgements}
I am  grateful to Prof. Lluis Godo,   Prof. Arnon Avron,  Prof. Jean-Yves Girard and  Prof.  George Metcalfe for valuable discussions.  I would like to thank two anonymous reviewers for carefully reading the first version of this article and many instructive suggestions.
\addcontentsline{toc}{section}{References}
\section*{References}
\bibliographystyle{elsarticle-harv}

\addcontentsline{toc}{section}{Appendices}

\begin{appendices}
 \section*{Appendices}
 \renewcommand{\appendixname}{Appendix~\Alph{section}}
\normalsize\rm

\section* {A.1 Why do we adopt Avron-style hypersequent calculi?}

A hypersequent calculus is called Pottinger-style if its two-premise
rules are in the form of $\cfrac{G\vert S'{\kern 1pt}{\kern
1pt}{\kern 1pt}{\kern 1pt}{\kern 1pt}{\kern 1pt}{\kern 1pt}G\vert
S''}{G\vert H'}(II)$ and, Avron-style if in the form of $\cfrac{G'\vert
S'{\kern 1pt}{\kern 1pt}{\kern 1pt}{\kern 1pt}{\kern 1pt}{\kern 1pt}{\kern
1pt}{\kern 1pt}G''\vert S''}{G'\vert G''\vert H'}(II)$.
In the viewpoint of  Avron-style systems,  each application of  two-premise
rules contains implicitly applications of $(EC)$ in Pottinger-style systems, as shown in the following.

\[
\boxed{\cfrac{{G|S'{\kern 1pt}  {\kern 1pt} {\kern 1pt} {\kern 1pt} {\kern 1pt} {\kern 1pt} {\kern 1pt} G|S''}}
{{G|H'}}(II)\xrightarrow[{in{\kern 1pt} {\kern 1pt} Avron-style {\kern 1pt} {\kern 1pt} system}]{{corresponds{\kern 1pt} {\kern 1pt} to{\kern 1pt} }}\cfrac{{\cfrac{{G|S' {\kern 1pt} {\kern 1pt} {\kern 1pt} {\kern 1pt} {\kern 1pt} G|S''}}
{{G|G|H'}}(II)}}
{{G|H'}}(EC^ *  )}
\]

 The choice of the underlying
system of hypersequent calculus is vital to our purpose and it gives the
background or arena.  In Pottinger-style system, $G_0$ in Section 3  is proved without application of $(EC)$ as follows.
But it seems helpless to prove that
$H_0$ is a theorem of \textbf{IUL}.\\
\scriptsize
$\boxed {\cfrac{\cfrac{C\Rightarrow C}{C\Rightarrow
C\vert \Rightarrow p,B\vert B\Rightarrow p,\neg A\odot \neg
A}\cfrac{\cfrac{B\Rightarrow B}{B\Rightarrow B{\kern 1pt}\vert
p,p\Rightarrow A\odot A}\cfrac{{\kern 1pt}{\kern 1pt}{\kern 1pt}{\kern
1pt}\cfrac{\cfrac{\cfrac{p\Rightarrow p{\kern 1pt}{\kern 1pt}{\kern 1pt}{\kern
1pt}{\kern 1pt}{\kern 1pt}A\Rightarrow A}{A\Rightarrow
p{\kern 1pt}\vert p\Rightarrow A}{\kern 1pt}{\kern 1pt}{\kern
1pt}\cfrac{p\Rightarrow p{\kern 1pt}\,\,\,A\Rightarrow A}{A\Rightarrow p{\kern
1pt}\vert p\Rightarrow A}}{A\Rightarrow p{\kern 1pt}\vert p,p\Rightarrow
A\odot A}}{\Rightarrow p,\neg A{\kern 1pt}\vert p,p\Rightarrow A\odot
A}\cfrac{\cfrac{\cfrac{p\Rightarrow p{\kern 1pt}{\kern 1pt}{\kern 1pt}{\kern
1pt}{\kern 1pt}{\kern 1pt}A\Rightarrow A}{A\Rightarrow
p{\kern 1pt}\vert p\Rightarrow A}{\kern 1pt}{\kern 1pt}{\kern
1pt}\cfrac{p\Rightarrow p{\kern 1pt}{\kern 1pt}{\kern 1pt}{\kern 1pt}{\kern
1pt}{\kern 1pt}A\Rightarrow A}{A\Rightarrow p{\kern
1pt}\vert p\Rightarrow A}}{A\Rightarrow p{\kern 1pt}\vert p,p\Rightarrow
A\odot A}}{\Rightarrow p,\neg A{\kern 1pt}\vert p,p\Rightarrow A\odot
A}{\kern 1pt}{\kern 1pt}{\kern 1pt}{\kern 1pt}}{{\kern 1pt}\Rightarrow
p,p,\neg A\odot \neg A{\kern 1pt}\vert p,p\Rightarrow A\odot A}}{\Rightarrow
p,B\vert B\Rightarrow p,\neg A\odot \neg A{\kern 1pt}\vert p,p\Rightarrow
A\odot A}{\kern 1pt}{\kern 1pt}{\kern 1pt}}{\Rightarrow p,B\vert
B\Rightarrow p,\neg A\odot \neg A{\kern 1pt}\vert p\Rightarrow C\vert
C,p\Rightarrow A\odot A}}$
\normalsize

The peculiarity of our method is
not only to focus on controlling the role of the external contraction rule
in the hypersequent calculus but also  introduce other  syntactic manipulations. For example, we label occurrences of the eigenvariable $p$  introduced by an
application of the density rule in order to be able to trace these
occurrences from the leaves (axioms) of the derivation to the root (the
derived hypersequent).

\section *{A.2 Why do we need the constrained external contraction rule?}

We use the example in Section 3 to answer this question.   Firstly, we illustrate Notation 4.14 as follows. In Figure 4,  let
$S_{11}^c =A\Rightarrow p_2 ;
S_{12}^c =A\Rightarrow p_1 ;
S_{21}^c =A\Rightarrow p_4 ;
S_{22}^c =A\Rightarrow p_3 ;
S_{31}^c =p_1 ,p_2 \Rightarrow A\odot A;
$
$
S_{32}^c =p_3 ,p_4 \Rightarrow A\odot A;
G_1' =p_1 ,p_2 \Rightarrow A\odot A;
G_2' =p_3 ,p_4 \Rightarrow A\odot A;
$
$
G_3' =A\Rightarrow p_1 {\kern 1pt}\vert \Rightarrow p_2 ,B\vert
B\Rightarrow p_4 ,\neg A\odot \neg A{\kern 1pt}\vert A\Rightarrow p_3 .
$
Then $H_i^c =G_i' \vert S_{i1}^c \vert S_{i2}^c $ for $i=1,2,3$. $H_i^c $
are   $(pEC)$-nodes and, $S_{i1}^c $ and $S_{i2}^c $ are
$(pEC)$-sequents.

Let $G_{H_1^c :A\Rightarrow p_2}^\ast =\Rightarrow p_2 ,B\vert B\Rightarrow
p_4 ,\neg A\odot \neg A\vert A\Rightarrow p_3\vert p_3 ,p_4 \Rightarrow A\odot A$.
We denote the derivation $\tau _{H_1^c :A\Rightarrow p_2 }^\ast $ of
$G_{H_1^c :A\Rightarrow p_2 }^\ast $ from $A\Rightarrow p_2 $ by
$\cfrac{\underline{{\kern 1pt}{\kern 1pt}{\kern 1pt}{\kern
1pt}A\Rightarrow p_2 {\kern 1pt}{\kern 1pt}{\kern 1pt}{\kern 1pt}{\kern
1pt}{\kern 1pt}{\kern 1pt}{\kern 1pt}{\kern 1pt}}}{G_{H_1^c :A\Rightarrow p_2 }^\ast }\left\langle {\tau
_{H_1^c :A\Rightarrow p_2}^\ast } \right\rangle $.
Since we focus on sequents in $G^\ast $ in the separation algorithm, we
abbreviate  $\cfrac{\underline{{\kern 1pt}{\kern 1pt}{\kern 1pt}{\kern
1pt}A\Rightarrow p_2 {\kern 1pt}{\kern 1pt}{\kern 1pt}{\kern
1pt}{\kern 1pt}}}{G_{H_1^c :A\Rightarrow p_2}^\ast }\left\langle {\tau
_{H_1^c :A\Rightarrow p_2 }^\ast } \right\rangle $ to
$\cfrac{\underline{{\kern 1pt}{\kern 1pt}{\kern 1pt}{\kern 1pt}{\kern
1pt}{\kern 1pt}{\kern 1pt}{\kern 1pt}\,S_{11}^c \,{\kern 1pt}{\kern 1pt}{\kern 1pt}{\kern1pt}{\kern 1pt}{\kern 1pt}{\kern 1pt}}}{S_{22}^c \vert S_{32}^c
}\left\langle {\tau _{S_{11}^c }^\ast } \right\rangle $ and further to
$\cfrac{1}{2\vert 3}\left\langle {\tau _1^\ast } \right\rangle $. Then the
separation algorithm $\tau_{H_1^c :G|G^* }^{\medstar}$ is abbreviated as
$$\boxed {\cfrac{\cfrac{\cfrac{1\vert 2\vert 3}{2'\vert 3'\vert 2\vert
3}\left\langle {\tau _1^\ast } \right\rangle }{2'\vert 2}\left\langle {\tau
_3^\ast ,\tau _{3'}^\ast } \right\rangle }{2}\left\langle {EC_\Omega }
\right\rangle }$$
where $2'$ and $3'$ are abbreviations of $A\Rightarrow p_5
$ and $p_5 ,p_6 \Rightarrow A\odot A$, respectively. We also write $2'$ and
$3'$ respectively as $2$ and $3$ for simplicity. Then the whole separation
derivation is given as follows.
\[
\boxed {\cfrac{\cfrac{\cfrac{\cfrac{1\vert 2\vert 3}{2\vert 3\vert 2\vert
3}\left\langle {\tau _1^\ast } \right\rangle }{2\vert 2}\left\langle {\tau
_3^\ast ,\tau _3^\ast } \right\rangle }{2}\left\langle {EC_\Omega }
\right\rangle {\kern 1pt}{\kern 1pt}{\kern 1pt}{\kern 1pt}{\kern 1pt}{\kern
1pt}{\kern 1pt}{\kern 1pt}{\kern 1pt}{\kern 1pt}{\kern 1pt}{\kern 1pt}{\kern
1pt}{\kern 1pt}{\kern 1pt}{\kern 1pt}{\kern 1pt}{\kern 1pt}{\kern 1pt}{\kern
1pt}{\kern 1pt}{\kern 1pt}{\kern 1pt}{\kern 1pt}{\kern 1pt}{\kern 1pt}{\kern
1pt}{\kern 1pt}{\kern 1pt}{\kern 1pt}{\kern 1pt}{\kern 1pt}{\kern 1pt}{\kern
1pt}{\kern 1pt}\cfrac{\cfrac{\cfrac{1\vert 2\vert 3}{1\vert 1\vert
3}\left\langle {\tau _2^\ast } \right\rangle {\kern 1pt}}{1\vert
1}\left\langle {\tau _3^\ast } \right\rangle {\kern 1pt}}{1}\left\langle
{EC_\Omega } \right\rangle }{\emptyset }\left\langle {\tau _{\{1,2\}}^\ast }
\right\rangle {\kern 1pt}}
\]
 where $\emptyset $  is  an abbreviation of    $G''$  in Page 14  and means that all sequents in it are copies of sequents in $G_0$.  Note that the simplified notations become intractable when we  decide  whether  $\left\langle {EC_\Omega } \right\rangle $  is applicable  to resulting hypersequents.
If no application of $\left\langle {EC_\Omega } \right\rangle $ is used in it,
all resulting hypersequents fall into the set $\{1\vert 2\vert \underbrace
{3\vert \cdots \vert 3}_l,\,\,2\vert 2\vert \underbrace {3\vert \cdots \vert
3}_m,\,\,1\vert 1\vert \underbrace {3\vert \cdots \vert 3}_n: l\geq0, m\geq0,n\geq0\}$ and $\emptyset $
is never obtained.

\section*{ A.3 Why do we need the separation of branches?}

In Figure 11,   $p_1 $ and $p_2 $ in the premise of $\cfrac{\underline{\quad{\kern 1pt}{\kern 1pt}{\kern 1pt}{\kern 1pt}{\kern
1pt}{\kern 1pt}{\kern 1pt}{\kern 1pt}{\kern 1pt}{\kern 1pt}{\kern 1pt}p_1
,p_2 \Rightarrow A\odot A\quad{\kern 1pt}{\kern 1pt}{\kern 1pt}}}{p_1
\Rightarrow C\vert C,p_2 \Rightarrow A\odot A}\left\langle {\tau _{S_{31}^c
}^\ast } \right\rangle $  could be viewed as being tangled
in one sequent $p_1 ,p_2 \Rightarrow A\odot A$ but in the conclusion of
$\left\langle {\tau _{S_{31}^c }^\ast } \right\rangle $ they are separated
into two sequents $p_1 \Rightarrow C$ and $C,p_2 \Rightarrow A\odot A$,
which are copies of sequents in $G_0$.    In Figure 5,  $p_2 $ in $A\Rightarrow p_2 $ falls into $\Rightarrow p_2 ,B$ in the root of $\tau _{H_1^c :A\Rightarrow p_2 }^\ast $
and $\Rightarrow p_2 ,B$ is a copy of a sequent in $G_0$. The same is true for
 $p_4 $ in $A\Rightarrow p_4 $ in Figure 8. But it's not the case.

Lemma 6.6 (vi) shows that in the elimination rule $\cfrac{\underline{{\kern
1pt}{\kern 1pt}{\kern 1pt}{\kern 1pt}S_{11}^c {\kern 1pt}{\kern 1pt}{\kern
1pt}}}{G_{S_{11}^c }^\ast }\left\langle {\tau _{S_{11}^c }^\ast }
\right\rangle $, $S_j^c\in G_{S_{11}^c }^\ast $ implies $H_j^c <H_i^c $ or $H_j^c \parallel H_i^c $.  If there exists no
$S_j^c\in G_{S_{11}^c }^\ast $ such that $H_j^c <H_i^c $,  then $S_j^c\in G_{S_{11}^c }^\ast
$ implies $H_j^c \parallel H_i^c $ and,  thus each occurrence of $p's$ in
$S_{11}^c$ is fell into a unique sequent which is a copy of a sequent in
$G_0$.  Otherwise there exists $S_j^c\in G_{S_{11}^c }^\ast $ such that $H_j^c <H_i^c $, then
we apply $\left\langle {\tau _{S_j^c }^\ast } \right\rangle $ to $S_j^c$ in $G_{S_{11}^c
}^\ast $ and the whole operations can be written as
$$\cfrac{\underline{\cfrac{\underline{\quad\qquad{\kern 1pt}{\kern 1pt}{\kern 1pt}{\kern
1pt}{\kern 1pt}{\kern 1pt}{\kern 1pt}{\kern 1pt}{\kern 1pt}{\kern 1pt}{\kern
1pt}{\kern 1pt}{\kern 1pt}{\kern 1pt}S_{11}^c \qquad\quad{\kern 1pt}{\kern 1pt}{\kern
1pt}{\kern 1pt}{\kern 1pt}{\kern 1pt}{\kern 1pt}{\kern 1pt}{\kern 1pt}{\kern
1pt}{\kern 1pt}}}{G_{S_{11}^c }^{\medstar(0)} \equiv G_{S_{11}^c }^\ast \backslash
\{S_j^c \}\vert S_j^c }\left\langle {\tau _{S_{11}^c }^\ast } \right\rangle
}}{G_{S_{11}^c }^{\medstar(1)} \equiv G_{S_{11}^c }^\ast \backslash \{S_j^c \}\vert
G_{S_j^c }^\ast }\left\langle {\tau _{S_j^c }^\ast } \right\rangle .$$
Repeatedly we can get $G_{S_{11}^c }^{\medstar(J)} $ such that $S_j^c\in
G_{S_{11}^c }^{\medstar(J)} $ implies $H_j^c \parallel H_1^c $.  Then each occurrence
of $p's$ in $S_{11}^c$ is fell into a unique sequent in $G_{S_{11}^c }^{\medstar(J)} $ which is a copy of
a sequent in $G_0$.  In such case, we call occurrences of $p's$ in $S_{11}^c$ are separated in $ G_{S_{11}^c }^{\medstar(J)} $ and call such a procedure the separation algorithm. It is the starting point of the separation algorithm. We
introduce branches in order to tackle the case of multiple-premise separation derivations for which it is necessary to apply $(EC_\Omega )$ to the resulting hypersequents.

\section* {A.4 Some questions about Theorem 8.2}

In Theorem 8.2, $\tau _{\rm {\bf I}}^\medstar$  is constructed by induction
on the number $\left| I \right|$ of branches.  As usual, we take the algorithm
of $\left| I \right|-1$ branches as the induction hypothesis. Why do we take
$\tau _{{\rm {\bf I}}_l }^\medstar$ and $\tau _{{\rm {\bf I}}_r }^\medstar$ as
the induction hypothesises?

Roughly speaking, it degenerates the case of $\left| I \right|$ branches
into the case of two branches in the following sense. The subtree $\tau
^\ast (G''\vert S'')$ of  $\tau ^\ast $ is as a whole contained in $\tau
_{\mathcal{I}_{{\rm {\bf j}}_l }}^\ast $ or not in it.
Similarly, $\tau ^\ast (G'\vert S')$ of $\tau ^\ast $ is as a whole
contained in $\tau _{\mathcal{I}_{{\rm {\bf j}}_r }
}^\ast $ or not in it.  It is such a division of $I$ into $I_l $ and $I_r $
that makes the whole algorithm possible.

Claim (i) of Theorem 8.2 asserts that $H_i^c \not {\leqslant }H_j^c $ for
all $S_j^c \in G_{\mathcal{I}_{\rm {\bf j}} }^\ast $ and
$H_i^c \in I$. It guarantees that $\tau _{\mathcal{I}_{\rm {\bf j}}}^\ast $ is not far from the final aim of Theorem 8.2 but
roughly close to it if we define some complexity to calculate it. If $H_i^c
\leqslant H_j^c $, the complexity of $G_{{\rm {\bf I}}_{\rm {\bf j}}}^\ast $ is more than or equal to that of $\left\lceil {S_i^c }
\right\rceil _I $ under such a definition of complexity and thus such an
application of $\tau _{\mathcal{I}_{\rm {\bf j}} }^\ast $
is redundant at least.  Claim (iii) of Theorem 8.2 guarantees the validity of the step 4 of Stage 1 and 2.

The tree structure of the skeleton of $\tau _{{\rm {\bf I}}_l }^\medstar(\tau
_{{\rm {\bf I}}_{{\rm {\bf j}}_r } }^\ast )$  can be
obtained by deleting some node $H\in \bar {\tau }_{{\rm {\bf I}}_l }^\medstar
$ satisfying  $\partial _{\tau _{{\rm {\bf I}}_l }^\medstar} (H)\leqslant H_{\rm {\bf I}}^V $.
The same is true for $\tau _{{\rm {\bf I}} }^\medstar$ if $\tau
_{{\rm {\bf I}}_l }^\medstar(\tau _{{\rm {\bf I}}_{{\rm {\bf j}}_r }}^\ast )$ is treated as a rule or a subroutine whose premises
are same as ones of $\tau _{{\rm {\bf I}}_{{\rm {\bf j}}_r }}^\ast $.  However,  it is incredibly difficult to imagine
or describe the structure of  $\tau _{{\rm {\bf I}} }^\medstar$  if you want to expand it as a normal derivation,  a binary tree.

All syntactic manipulations  in constructing
$\tau _{{\rm {\bf I}} }^\medstar$  are performed on the skeletons
of  $\tau _{{\rm {\bf I_l}} }^\medstar$ or $\tau _{{\rm {\bf I_r}} }^\medstar$.
  The structure of the proof of Theorem 8.2  is depicted in the following figure.

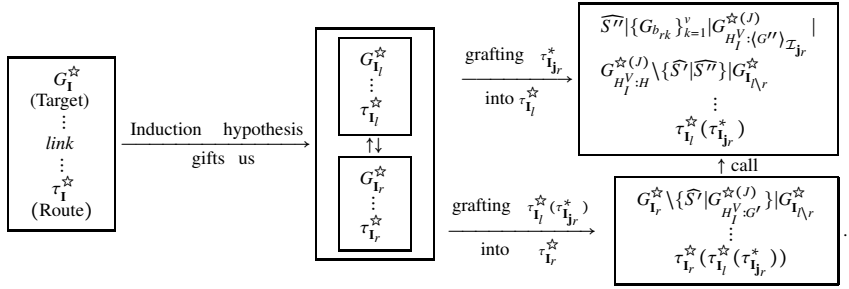
\scriptsize
\[
\boxed {\begin{array}{l}
 {\kern 1pt}{\kern 1pt}{\kern 1pt}{\kern 1pt}{\kern
1pt}{\kern 1pt}{\kern 1pt}{\kern 1pt}{\kern 1pt}G_{\rm {\bf I}}^{\medstar}\\
 \mbox{(Target)} \\
 {\kern 1pt}{\kern 1pt}{\kern 1pt}{\kern 1pt}{\kern 1pt}{\kern 1pt}{\kern 1pt}{\kern 1pt}{\kern
1pt}{\kern 1pt} {\kern 1pt}{\kern 1pt}\vdots \\
 {\kern 1pt}{\kern 1pt}{\kern 1pt}{\kern 1pt}{\kern 1pt}{\kern 1pt}link \\
 {\kern 1pt}{\kern 1pt}{\kern 1pt}{\kern 1pt}{\kern 1pt}{\kern 1pt}{\kern
1pt}{\kern 1pt}{\kern 1pt}\,\,\vdots \\
 {\kern 1pt}{\kern 1pt}{\kern 1pt}{\kern 1pt}{\kern 1pt}{\kern
1pt}\,\,\tau _{\rm {\bf I}}^\medstar\\
 (\mbox{Route}) \\
 \end{array}}\xrightarrow[{{\kern 1pt}{\kern 1pt}{\kern 1pt}{\kern
1pt}{\kern 1pt}{\kern 1pt}{\kern 1pt}\mbox{gifts}{\kern 1pt}{\kern
1pt}{\kern 1pt}{\kern 1pt}{\kern 1pt}\mbox{us}}]{{\kern 1pt}{\kern
1pt}\mbox{Induction}{\kern 1pt}{\kern 1pt}{\kern 1pt}{\kern 1pt}{\kern
1pt}{\kern 1pt}{\kern 1pt}{\kern 1pt}\mbox{hypothesis}}
\boxed {\begin{array}{l}
 \boxed {\begin{array}{l}
 G_{{\rm {\bf I}}_l }^{\medstar}\\
 \,\,\,\vdots \\
\tau _{{\rm {\bf
I}}_l }^\medstar\\
 \end{array}} \\
 {\kern 1pt}{\kern 1pt}{\kern 1pt}{\kern 1pt}{\kern 1pt}{\kern 1pt}\,\,\uparrow
\downarrow {\kern 1pt} \\
 \boxed {\begin{array}{l}
 G_{{\rm {\bf I}}_r }^{\medstar}\\
 {\kern 1pt}{\kern 1pt}{\kern 1pt}{\kern 1pt}\vdots \\
\tau _{{\rm {\bf
I}}_r }^\medstar\\
 \end{array}} \\
 \end{array}}{\begin{array}{*{20}c}
 {\begin{array}{l}
 \xrightarrow[{\mbox{into}\,\, \tau _{{\rm {\bf I}}_l }^\medstar{\kern 1pt}{\kern
1pt}}]{\mbox{grafting}{\kern 1pt}{\kern 1pt}{\kern 1pt}{\kern 1pt}{\kern
1pt}{\kern 1pt}\tau _{{\rm {\bf I}}_{{\rm {\bf j}}_r }
}^\ast }
\boxed {\begin{array}{l}
 \widehat{S''}\vert \{G_{b_{rk} } \}_{k=1}^v \vert G_{H_I^V :\left\langle
{G''} \right\rangle _{\mathcal{I}_{{\rm {\bf j}}_r } }
}^{\medstar(J)} \vert \\
 G_{H_I^V :H}^{\medstar(J)} \backslash \{\widehat{S'}\vert \widehat{S''}\}\vert
G_{{\rm {\bf I}}_{l\backslash r} }^{\medstar}\\
 \qquad\qquad\qquad\vdots \\
 \qquad\qquad\tau _{{\rm {\bf I}}_l }^\medstar(\tau _{{\rm {\bf
I}}_{{\rm {\bf j}}_r } }^\ast ) \\
 \end{array}} \\
 {\kern 1pt}{\kern 1pt}{\kern 1pt}{\kern 1pt}{\kern 1pt}{\kern 1pt}{\kern
1pt}{\kern 1pt}{\kern 1pt}{\kern 1pt}{\kern 1pt}{\kern 1pt}{\kern 1pt}{\kern
1pt}{\kern 1pt}{\kern 1pt}{\kern 1pt}{\kern 1pt}{\kern 1pt}{\kern 1pt}{\kern
1pt}{\kern 1pt}{\kern 1pt}{\kern 1pt}{\kern 1pt}{\kern 1pt}{\kern 1pt}{\kern
1pt}{\kern 1pt}{\kern 1pt}{\kern 1pt}\qquad\qquad\qquad\qquad\quad\uparrow \mbox{call}{\kern
1pt} \\
 \end{array}} \dotfill \\
 {\xrightarrow[{\mbox{into}{\kern 1pt}{\kern 1pt}{\kern 1pt}{\kern
1pt}{\kern 1pt}{\kern 1pt}{\kern 1pt}{\kern 1pt}{\kern 1pt}{\kern 1pt}{\kern
1pt}\tau _{{\rm {\bf I}}_r }^\medstar}]{\mbox{grafting}{\kern 1pt}{\kern
1pt}{\kern 1pt}{\kern 1pt}{\kern 1pt}{\kern 1pt}\tau _{{\rm {\bf I}}_l
}^\medstar(\tau _{{\rm {\bf I}}_{{\rm {\bf j}}_r }
}^\ast )}{\kern 1pt}{\kern 1pt}{\kern 1pt}{\kern 1pt}{\kern 1pt}
\boxed
{\begin{array}{l}
 G_{{\rm {\bf I}}_r }^{\medstar}\backslash
\{\widehat{S'}\vert G_{H_I^V :G'}^{\medstar(J)} \}\vert G_{{\rm {\bf
I}}_{l\backslash r} }^{\medstar}\\
 {\kern 1pt}{\kern 1pt}{\kern 1pt}{\kern 1pt}{\kern 1pt}{\kern 1pt}{\kern
1pt}{\kern 1pt}{\kern 1pt}{\kern 1pt}{\kern 1pt}{\kern 1pt}{\kern 1pt}{\kern
1pt}{\kern 1pt}{\kern 1pt}{\kern 1pt}{\kern 1pt}\,\,\qquad\vdots \\
\qquad\tau _{{\rm {\bf
I}}_r }^\medstar(\tau _{{\rm {\bf I}}_l }^\medstar(\tau _{{\rm {\bf I}}_{{\rm{\bf j}}_r }}^\ast )) \\
 \end{array}}} \dotfill \\
\end{array} }
\]
\normalsize
\section *{A.5 Illustrations  of   notations and algorithms}

We use the example in Section 3  to illustrate some notations and algorithms in this paper.

\subsection* { A.5.1 Illustration of  two cases of  (COM) in  the proof of  Lemma 5.6}

 Let $\cfrac{G'\quad G''}{G'''}(COM)$ be $\cfrac{p_1 \Rightarrow p_1 \quad A\Rightarrow A}{A\Rightarrow p_1 \vert p_1 \Rightarrow A}(COM)$, where $G'=S_1 =p_1 \Rightarrow p_1 $; $G''=S_2 =A\Rightarrow A$;  $S_3 =A\Rightarrow p_1 $; $S_4 =p_1 \Rightarrow A$ and $G'''=S_3 \vert S_4 $. Then $\left[ {S_3 } \right]_{G'''} =\left[ {S_4 } \right]_{G'''} $; $\mathcal{D}_{G'} (S_1 )=\Rightarrow t$; $\mathcal{D}_{G''} (S_2 )=A\Rightarrow A$;  $\mathcal{D}_{G'''} (S_3 \vert S_4 )=A\Rightarrow A$. Thus the proof of $\cfrac{\underline{\mathcal{D}_{G'} (S_1 )\quad\mathcal{D}_{G'} (S_2 )}}{\mathcal{D}_{G'} (S_3 \vert S_4 )}$\quad  is constructed by $\cfrac{\Rightarrow t\quad\cfrac{A\Rightarrow A}{A,t\Rightarrow A}(t_l )}{A\Rightarrow A}(CUT)$.

 Let $\cfrac{G'\quad G''}{G'''}(COM)$ be $
\cfrac{{B \Rightarrow B                      \left( \begin{gathered}
  \Rightarrow p_2 ,p_4 ,\neg A \odot \neg A|p_1 ,p_2  \Rightarrow
  A \odot A  |\\
    A \Rightarrow p_1  | A \Rightarrow p_3  |p_3 ,p_4  \Rightarrow A \odot A\\
\end{gathered}  \right)}}
{\left(\begin{gathered}
  \Rightarrow p_2 ,B|B \Rightarrow p_4 ,\neg A \odot \neg A |  A \Rightarrow p_1  | \\
  p_1 ,p_2  \Rightarrow A \odot A |A \Rightarrow p_3  |p_3 ,p_4  \Rightarrow A \odot A \\
\end{gathered} \right) }(COM),\\
$
 where $G'=S_1 =B\Rightarrow B$;  $G_2 = p_1 ,p_2 \Rightarrow A\odot A\vert A\Rightarrow p_1 \vert A\Rightarrow p_3 \vert p_3 ,p_4 \Rightarrow A\odot A$;\\
 $S_2 =\Rightarrow p_2 ,p_4 ,\neg A\odot \neg A$;  $G''=G_2 \vert S_2 $;  $S_3 =\Rightarrow p_2 ,B$;  $S_4 =B\Rightarrow p_4 ,\neg A\odot \neg A$ and $G'''=G_2 \vert S_3 \vert S_4 $. Then $\mathcal{D}_{G'} (S_1 )=B\Rightarrow B$; $\mathcal{D}_{G''} (S_2 )=A,A\Rightarrow A\odot A,\neg A\odot \neg A,A\odot A$;  $\mathcal{D}_{G'''} (S_3 ) =A\Rightarrow B,A\odot A$; $\mathcal{D}_{G'''} (S_4 ) =A,B\Rightarrow A\odot A,\neg A\odot \neg A; \mathcal{D}_{G'''} (S_3 \vert S_4 )=\mathcal{D}_{G'''} (S_4 )\vert \mathcal{D}_{G'''} (S_4 ) $.

 Thus the proof of $\cfrac{\underline{\mathcal{D}_{G'} (S_1 )\quad \mathcal{D}_{G'} (S_2 )}}{\mathcal{D}_{G'} (S_3 \vert S_4 )}$ is constructed by
 $$\cfrac{B\Rightarrow B\quad A,A\Rightarrow A\odot A,\neg A\odot \neg A,A\odot A}{A\Rightarrow B,A\odot A\vert A,B\Rightarrow A\odot A,\neg A\odot \neg A}(COM).$$

\subsection* { A.5.2 Illustration of  Construction 6.1}
Let $\tau ^\ast $ be
\[
\cfrac{\cfrac{\cfrac{H_8 \equiv B \Rightarrow B \,\, H_9 \equiv A\Rightarrow A}{H_4 \equiv A\Rightarrow B \vert B \Rightarrow A}(COM)\cfrac{H_{10} \equiv B \Rightarrow B \,\, H_{11} \equiv A\Rightarrow A}{H_5 \equiv
A\Rightarrow B \vert B \Rightarrow A}(COM)}{H_2 \equiv A\Rightarrow B \vert A\Rightarrow B \vert B ,B \Rightarrow A\odot A}(\odot _r )}{H_1 \equiv
A\Rightarrow B \vert \Rightarrow B ,\neg A\vert
B ,B \Rightarrow A\odot A}(\neg _r ).
\]
By Construction 6.1, $\tau ^{\ast \ast }$ is then given as follows.
$$\cfrac{\cfrac{\cfrac{(B \Rightarrow B; 8, 0)\,\, (A\Rightarrow A; 9, 0)}{(A\Rightarrow B; 4, 1)\vert (B \Rightarrow A; 4, 2)}(COM)\,\,\cfrac{(B \Rightarrow B; 10, 0)\,\,(A\Rightarrow A; 11, 0)}{(A\Rightarrow B;  5, 1)\vert (B
\Rightarrow A; 5, 2)}(COM)}{(A\Rightarrow B; 4, 1)\vert (A\Rightarrow B; 5, 1)\vert (B ,B\Rightarrow A\odot A; 2, 0)}(\odot _r )}{(A\Rightarrow
B; 4, 1)\vert (\Rightarrow B ,\neg A; 1, 0)\vert (B ,B \Rightarrow A\odot A; 2, 0)}(\neg _r ).$$

As an example,  we calculate $\wp (H_8 )$.  Since $Th(H_8 )=(H_8 ,H_4 ,H_2 ,H_1 )$,  then $b_3 =1$, $b_2 =b_1 =b_0 =0$ by Definition 2.13.  Thus $
\wp (H_8 )=b_0 2^0+b_1 2^1+b_2 2^2+b_3 2^3=8.
$

Note that we can't distinguish the one from the other for  two $A\Rightarrow B's$ in $H_2\in\tau$. If we divide $H_2$ into $H'\vert H''$, where $H'\equiv A\Rightarrow B$  and   $H''\equiv A\Rightarrow B \vert B ,B \Rightarrow A\odot A$,  then  $H'\bigcap H''=\{A\Rightarrow B\}$ in the conventional meaning of  hypersequents.  Thus only in the sense that we treat  $\tau^{*}$ as  $\tau^{**}$, the assertion that $H'\bigcap H''=\emptyset$ for any $H'\vert H''\subseteq H$  in Proposition 6.2  holds.

\subsection* {A.5.3 Illustration of  Notation 6.10 and Construction 6.11}

Let $I=\{H_1^c ,H_2^c \}, I_l =\{H_1^c \},I_r =\{H_2^c \},  \mathcal{I}= \{ S_{1 1}^c, S_{21}^c  \},   \mathcal{I}_l= \{ S_{11}^c \},
\mathcal{I}_r= \{ S_{21}^c\},$
$$\cfrac{G'\vert S'\,\,\, G''\vert S''}{G'\vert G''\vert H'}(\odot _r )\in
\tau ^\ast, $$ where
 $ G'\vert G''\vert H'=H_I^V;  G'\equiv  A\Rightarrow p_1 \vert p_1 ,p_2
\Rightarrow A\odot A;  S'\equiv \Rightarrow p_2 ,\neg A;\\
G''\equiv A\Rightarrow p_3 \vert p_3 ,p_4 \Rightarrow A\odot A;
S''\equiv \Rightarrow p_4 ,\neg A;  H'\equiv \Rightarrow p_2 ,p_4 ,\neg A\odot \neg A $\\(See Figure 4).

$\left\langle {G'\vert S'} \right\rangle _{ \mathcal{I}_l }
=\Rightarrow p_2 ,\neg A;
\left\langle {G'} \right\rangle _{ \mathcal{I}_l } =\emptyset;
\left\langle {G'\vert G''\vert H'{\kern
1pt}} \right\rangle _{ \mathcal{I}_l } =A\Rightarrow p_3 \vert
\Rightarrow p_2 ,p_4 ,\neg A\odot \neg A\vert {\kern
1pt}p_3 ,p_4 \Rightarrow A\odot A;
\left\langle {G\vert G^\ast } \right\rangle _{ \mathcal{I}_l }
=G_{ \mathcal{I}_l}^\ast =G_{S_{11}^c}^\ast =\Rightarrow p_2
,B\vert B\Rightarrow p_4 ,\neg A\odot \neg A\vert {\kern
1pt}A\Rightarrow p_3 \vert p_3 ,p_4 \Rightarrow A\odot A$ (See
Figure 5).

$\left\langle {G''\vert S''} \right\rangle _{\mathcal{I}_r}
=\Rightarrow p_4 ,\neg A;
\left\langle {G'\vert G''\vert H'} \right\rangle _{\mathcal{I}_r} =A\Rightarrow p_1 \vert
\Rightarrow p_2 ,p_4 ,\neg A\odot \neg A\vert p_1 ,p_2 \Rightarrow A\odot A;$\\
$
\left\langle {G\vert G^\ast } \right\rangle _{\mathcal{I}_r }
=G_{\mathcal{I}_r}^\ast =G_{S_{21}^c}^\ast =A\Rightarrow p_1 \vert
\Rightarrow p_2 ,B\vert B\Rightarrow p_4 ,\neg A\odot \neg A\vert
p_1 \Rightarrow C\vert C,p_2 \Rightarrow A\odot A$ (See
Figure 8).

$\left\langle {G'\vert G''\vert H'} \right\rangle _{\mathcal{I}} =\Rightarrow
p_2 ,p_4 ,\neg A\odot \neg A;
\left\langle {G\vert G^\ast } \right\rangle _{\mathcal{I}} =G_{\mathcal{I}}^\ast
=G_{^{\{S_{11}^c ,S_{21}^c \}}}^\ast =G_{\mathcal{I}_l}^\ast\bigcap G_{\mathcal{I}_r}^\ast=\Rightarrow p_2 ,B\vert B\Rightarrow
p_4 ,\neg A\odot \neg A$ (See Figure 10).

\normalsize

\subsection *{A.5.4 Illustration of  Theorem 8.2}

Note that  sequents in $\left[ { } \right]$ are principal  sequents of elimination rules in the following.
Let $I, I_r, I_l$ be the same as in A.5.3 and,  $
{\rm {\bf I}}=\{\left\lceil {S_{1}^c } \right\rceil _{I } ,\left\lceil
{S_{2}^c } \right\rceil _{I } \},{\rm {\bf I}}_l =\{\left\lceil {S_{1}^c } \right\rceil _{I } \},
{\rm {\bf I}}_r =\{\left\lceil {S_{2}^c } \right\rceil _{I } \},$

$\left\lceil {S_1^c } \right\rceil _I =G_{H_2^c :G|G^* }^{\medstar}
=A\Rightarrow p_5 \vert \Rightarrow p_6 ,B\vert {\kern
1pt}B\Rightarrow p_8 ,\neg A\odot \neg A\vert {\kern
1pt}p_5\Rightarrow C\vert $

$\qquad\qquad C,p_6 \Rightarrow A\odot A\vert {\kern
1pt}B\Rightarrow p_7 ,\neg A\odot \neg A\vert p_7\Rightarrow C{\kern
1pt}\vert C,p_8 \Rightarrow A\odot A,$

$\left\lceil {S_2^c } \right\rceil _I
=G_{H_1^c :G|G^* }^{\medstar}
=\Rightarrow p_2 ,B\vert B\Rightarrow p_4 ,\neg A\odot \neg
A\vert p_1 \Rightarrow C\vert C,p_2 \Rightarrow A\odot A\vert $

$\qquad\qquad A\Rightarrow p_3 \vert \Rightarrow p_1
,B\vert p_3 \Rightarrow C\vert C,p_4 \Rightarrow A\odot A$.

$$\tau _{{\rm {\bf I}}_l }^\medstar=\cfrac{\underline{\cfrac{\underline{\cfrac{\underline{
\quad\left\lceil {S_1^c } \right\rceil _I \quad}}{G_{{{\rm {\bf I}}_l }}^{\medstar(1)} }{\kern
1pt}\left\langle {\tau _{H_1^c :A\Rightarrow p_5}^\ast } \right\rangle}}{G_{{{\rm {\bf I}}_l }}^{\medstar(2)} }{\kern1pt}\left\langle {\tau _{H_3^c :p_{9} ,p_{10} \Rightarrow
A\odot A}^\ast } \right\rangle }}{G_{{{\rm {\bf I}}_l } }^\medstar}\left\langle {EC_\Omega ^\ast } \right\rangle, $$
where
$ G_{{{\rm {\bf I}}_l }}^{\medstar(1)}  =
\left[ { \Rightarrow p_5,B\vert {\kern
1pt}B\Rightarrow p_{10} ,\neg A\odot \neg A\vert A\Rightarrow p_{9} \vert p_{10} ,p_{9} \Rightarrow A\odot A } \right]\vert
\Rightarrow p_6 ,B\vert $

$\qquad \qquad B\Rightarrow p_8 ,\neg A\odot
\neg A\vert   p_5 \Rightarrow C\vert C,p_6 \Rightarrow A\odot A\vert {\kern
1pt}B\Rightarrow p_7 ,\neg A\odot \neg A\vert $

$\qquad \qquad p_7 \Rightarrow C\vert C,p_8 \Rightarrow A\odot A,$

$ G_{{{\rm {\bf I}}_l }}^{\medstar(2)}  =\Rightarrow p_5 ,B\vert {\kern
1pt}B\Rightarrow p_{10} ,\neg A\odot \neg A\vert A\Rightarrow
p_{9} \vert[ p_{9}\Rightarrow C\vert C,p_{10} \Rightarrow
A\odot A]\vert $

$\qquad \qquad\Rightarrow p_6 ,B\vert
B\Rightarrow p_8 ,\neg A\odot \neg A\vert
 p_5\Rightarrow C\vert C,p_6 \Rightarrow A\odot A\vert$

$\qquad \qquad B\Rightarrow p_7 ,\neg A\odot \neg A\vert p_7 \Rightarrow C{\kern
1pt}\vert C,p_8 \Rightarrow A\odot A,
$

$ G_{{{\rm {\bf I}}_l } }^\medstar=\Rightarrow p_5 ,B\vert {\kern
1pt}A\Rightarrow p_{9} \vert p_{9} \Rightarrow C{\kern
1pt}\vert \Rightarrow p_6,B\vert B\Rightarrow p_8
,\neg A\odot \neg A\vert$

 $\qquad \qquad p_5\Rightarrow C\vert C,p_6 \Rightarrow A\odot A\vert {\kern
1pt}B\Rightarrow p_7 ,\neg A\odot \neg A\vert p_7\Rightarrow C{\kern
1pt}\vert C,p_8 \Rightarrow A\odot A,
$

$G_{H_I^V :G''}^{\medstar(J)} = A\Rightarrow p_{9} \vert p_{9} \Rightarrow C\vert C,p_{10} \Rightarrow
A\odot A $;   $\widehat{S''}=B\Rightarrow p_{10},\neg A\odot \neg A$;
$\widehat{S'}=\Rightarrow p_5,B$;
$G_{H_I^V:H'}^{\medstar(J)} =G_{H_I^V :H'}^\ast =\widehat{S'}\vert \widehat{S''}$;
$G_\dag =A\Rightarrow p_{9} \vert p_{9} \Rightarrow C\vert C,p_{10} \Rightarrow
A\odot A\vert B\Rightarrow p_{10} ,\neg A\odot \neg A$.

$$\tau _{{\rm {\bf I}}_r }^\medstar=\cfrac{\underline{\cfrac{\underline{
\quad\left\lceil {S_2^c } \right\rceil _I \quad}}{G_{{\rm {\bf I}}_r}^{\medstar(1)} }{\kern
1pt}\left\langle {\tau _{H_2^c :A\Rightarrow p_3}^\ast } \right\rangle}}{G_{{\rm {\bf I}}_r }^\medstar}\left\langle {EC_\Omega ^\ast } \right\rangle, $$
where   $G_{{\rm {\bf I}}_r }^{\medstar(1)} =\Rightarrow p_2 ,B\vert {\kern
1pt}B\Rightarrow p_4 ,\neg A\odot \neg A\vert p_1
\Rightarrow C\vert C,p_2 \Rightarrow A\odot A\vert $

 $\qquad \qquad\Rightarrow p_1 ,B\vert
 p_3 \Rightarrow C\vert C,p_4 \Rightarrow A\odot A\vert [A\Rightarrow p_{11}
\vert \Rightarrow p_{12} ,B\vert $

$\qquad \qquad B\Rightarrow p_3,\neg A\odot \neg A\vert p_{11} \Rightarrow C\vert C,p_{12} \Rightarrow A\odot A],
$

 $G_{{\rm {\bf I}}_r}^\medstar=\Rightarrow p_2 ,B\vert {\kern
1pt}B\Rightarrow p_4 ,\neg A\odot \neg A\vert p_1
\Rightarrow C\vert C,p_2\Rightarrow A\odot A\vert $

$\qquad \qquad \Rightarrow p_1 ,B\vert
 p_3 \Rightarrow C\vert C,p_4 \Rightarrow A\odot A\vert A\Rightarrow p_{11}\vert
B\Rightarrow p_3 ,\neg A\odot \neg A\vert p_{11} \Rightarrow C.
$

Since there is only one elimination rule in $\tau _{{\rm {\bf I}}_r }^\medstar$,
the case we need to process is $\tau _{H_2^c :A\Rightarrow p_3}^\ast
$, i.e., $$\tau _{{\rm {\bf I}}_{{\rm {\bf j}}_r }}^\ast =\cfrac{\underline{\quad\left\lceil {S_2^c } \right\rceil _I\quad}}{G_{H_2^c :\left\lceil {S_2^c } \right\rceil _I }^{\medstar(1)} }\left\langle {\tau _{H_2^c :A\Rightarrow p_3}^\ast } \right\rangle .$$  Then  $v=1$,
$S_{j_{r1} }^c
=A\Rightarrow p_3 $; $G_{b_{r1} } =\Rightarrow p_2 ,B\vert B\Rightarrow p_4 ,\neg A\odot \neg A\vert p_1
\Rightarrow C\vert \\
C, p_2 \Rightarrow A\odot A\vert \Rightarrow
p_1 ,B\vert p_3 \Rightarrow C\vert C,p_4 \Rightarrow A\odot A$
 in $\tau
_{{\rm {\bf I}}_{{\rm {\bf j}}_r }}^\ast $.

$$\tau _{{\rm {\bf I}}_l }^{\medstar (0)}=\cfrac{\underline{\cfrac{\underline{\cfrac{\underline{
\quad\left\lceil {S_1^c } \right\rceil _I \quad}}{G_{{\rm {\bf I}}_l}^{\medstar(1)} }{\kern
1pt}\left\langle {\tau _{H_1^c :A\Rightarrow p_5}^\ast } \right\rangle}}{G_{{\rm {\bf I}}_l}^{\medstar(2)} }{\kern1pt}\left\langle {\tau _{H_3^c :p_{9} ,p_{10} \Rightarrow
A\odot A}^\ast } \right\rangle }}{G_{{\rm {\bf I}}_l }^\medstar}\left\langle {EC_\Omega ^\ast } \right\rangle^{\circ}_{1}, $$
where   $\partial _{\tau _{{\rm {\bf I}}_l }^\medstar} (\left\lceil {S_1^c } \right\rceil _I)=H_1^c $, $\partial _{\tau _{{\rm
{\bf I}}_l }^\medstar} (G_{{\rm {\bf I}}_l  }^{\medstar(1)} )=H_3^c
<H_I^V $,$\partial _{\tau _{{\rm {\bf I}}_l }^\medstar} (G_{{\rm {\bf I}}_l  }^{\medstar(2)} )=\partial _{\tau _{{\rm {\bf I}}_l }^\medstar}
(G_{{\rm {\bf I}}_l }^\medstar)=G\vert G^\ast $,
$G_1^{\circ \circ } =G_{{\rm {\bf I}}_l }^{\medstar(2)} $,
$G_1^\circ =G_{{\rm {\bf I}}_l }^{\medstar}$.

$$\tau _{{\rm {\bf I}}_l :G_1^{\circ \circ } }^{\medstar (0)}
=\cfrac{\underline{\cfrac{\underline{\quad\left\lceil {S_1^c } \right\rceil _I \quad}}{G_{{\rm {\bf I}}_l }^{\medstar(1)}}\left\langle {\tau _{H_1^c :A\Rightarrow p_5}^\ast } \right\rangle }}{G_{{\rm {\bf I}}_l }^{\medstar(2)} }\left\langle {\tau _{H_3^c :p_{9} ,p_{10} \Rightarrow
A\odot A}^\ast } \right\rangle ,\\
$$$$\tau _{{\rm {\bf I}}_l :G_1^{\circ \circ } (1)}^{\medstar (0)} =\tau _{{\rm{\bf I}}_l :G_1^{\circ \circ } (2)}^{\medstar (0)} =\cfrac{\underline{\,\,\left\lceil {S_1^c } \right\rceil _I\,\,}}{G_{{\rm {\bf I}}_l }^{\medstar(1)} }\left\langle {\tau _{H_1^c :A\Rightarrow p_5}^\ast } \right\rangle.$$

Since there is only one elimination rule in $\tau _{{\rm {\bf I}}_l
:G_1^{\circ \circ } (2)}^{\medstar (0)} $,  the case we need to process is
$\tau _{H_1^c :A\Rightarrow p_5 }^\ast $, i.e.,
$$\tau _{{\rm {\bf
I}}_{{\rm {\bf j}}_l } }^\ast
=\cfrac{\underline{\,\, \left\lceil {S_1^c } \right\rceil _I\,\,}}{G_{{\rm {\bf I}}_l }^{\medstar (1)}}
{\kern1pt}\left\langle {\tau _{H_1^c :A\Rightarrow p_5}^\ast } \right\rangle. $$

Then $u=1$, $S_{j_{l1}^{(t_{l1} )} }^c
=A\Rightarrow p_5 $; $G_{b_{l1} } =\Rightarrow p_6 ,B\vert {\kern
1pt}B\Rightarrow p_8 ,\neg A\odot \neg A\vert\\
p_5\Rightarrow C\vert C,p_6 \Rightarrow A\odot A\vert B\Rightarrow
p_7 ,\neg A\odot \neg A\vert p_7 \Rightarrow C\vert C,p_8
\Rightarrow A\odot A$ in
$\tau _{{\rm {\bf I}}_{{\rm {\bf j}}_l }}^\ast $.

$\tau _{{\rm {\bf I}}_{{\rm {\bf j}}_l }
}^\ast $ is replaced with
$\tau _{{\rm {\bf I}}_{{\rm {\bf j}}_l }\cup
 {\rm {\bf I}}_{{\rm {\bf j}}_r }}^\ast $ in Step 3 of Stage 1, i.e.,

$\cfrac{\underline{{\kern
1pt}{\kern
1pt}\left\lceil {S_1^c } \right\rceil _I {\kern
1pt}{\kern
1pt}\left\lceil {S_2^c }
\right\rceil _I {\kern
1pt}}}{G_{l, r}
}\left\langle {\tau
_{\{H_1^c :A\Rightarrow p_5 ,H_2^c :A\Rightarrow p_3 \}}^\ast }
\right\rangle =\tau _{{\rm {\bf I}}_l :G_1^{\circ \circ } (3)}^{\medstar (0)}
=\tau _{{\rm {\bf I}}_l :G_1^{\circ \circ } (4)}^{\medstar (0)} ,$
where

$G_{l, r}=\Rightarrow p_5 ,B\vert B\Rightarrow p_3 ,\neg
A\odot \neg A\vert G_{b_{r1} } \vert G_{b_{l1}}=$

$\qquad \qquad\Rightarrow p_2 ,B\vert B\Rightarrow p_4 ,\neg A\odot \neg A\vert p_1
\Rightarrow C\vert C,p_2 \Rightarrow A\odot A\vert \Rightarrow p_1 ,B\vert  $

$\qquad \qquad p_3 \Rightarrow C\vert C,p_4
\Rightarrow A\odot A \vert \Rightarrow p_6 ,B\vert B\Rightarrow p_8,\neg A\odot \neg A\vert$

$\qquad \qquad p_5 \Rightarrow C\vert C,p_6 \Rightarrow A\odot A\vert B\Rightarrow p_7 ,\neg A\odot \neg A\vert p_7 \Rightarrow C\vert $

$\qquad \qquad\Rightarrow C,p_8 \Rightarrow A\odot A\vert p_5 ,B\vert B\Rightarrow p_3 ,\neg A\odot \neg A$.

Replacing  $\tau _{{\rm {\bf I}}_l :G_1^{\circ \circ }}^{\medstar (0)}$  in $\tau _{{\rm {\bf I}}_l }^{\medstar (0)} $ with $ \tau _{{\rm {\bf I}}_l :G_1^{\circ \circ } (4)}^{\medstar (0)}$,   then deleting    $G_{{{\rm {\bf I}}_l } }^\medstar$  and after that applying   $\left\langle
{EC_\Omega ^\ast } \right\rangle $  to  $G_{l, r}$  and keeping $G_{b_{r1}}$ unchanged, we get

$$\tau _{{\rm {\bf I}}_l }^\medstar(\tau _{{\rm {\bf I}}_{{\rm {\bf j}}_r
} }^\ast )=\cfrac{\underline{\qquad\quad\cfrac{\underline{\left\lceil {S_1^c }
\right\rceil _I\quad \left\lceil {S_2^c } \right\rceil _I}}{G_{l, r} }\left\langle {\tau _{\{H_1^c :A\Rightarrow p_5 ,H_2^c
:A\Rightarrow p_3\}}^\ast } \right\rangle\qquad}}{\widehat{S''}\vert G_{b_{r1} } \vert G_{H_I^V
:\left\langle {G''}
\right\rangle _{\mathcal{I}_{{\rm {\bf j}}_r }}}^{\medstar(J)} \vert G_{H_I^V :H'}^{\medstar(J)} \backslash
\{\widehat{S'}\vert \widehat{S''}\}\vert G_{{\rm {\bf I}}_{l\backslash r}
}^\medstar}\left\langle
{EC_\Omega ^\ast } \right\rangle, $$

where  $G_{H_I^V :\left\langle {G''}
\right\rangle _{\mathcal{I}_{{\rm {\bf j}}_r }} }^{\medstar(J)} =G_{H_I^V
:\left\langle {G''} \right\rangle _{\mathcal{I}_{{\rm {\bf j}}_r } }}^\ast =\emptyset $;  $\widehat{S'}=\Rightarrow p_5, B;\\  \widehat{S''}=B\Rightarrow p_{3},\neg A\odot \neg A$;   $G_\ddagger =G_{b_{r1} }|\widehat{S''}$;
$G_{H_I^V
:H'}^{\medstar(J)} =G_{H_I^V :H'}^\ast =\widehat{S'}\vert \widehat{S''}$;\\
$G_{{\rm{\bf I}}_{l\backslash r} }^{\medstar}=\Rightarrow p_5 ,B\vert \Rightarrow
p_6 ,B\vert B\Rightarrow p_7 ,\neg A\odot \neg A\vert
p_5 \Rightarrow C\vert C,p_6 \Rightarrow A\odot A\vert \\
p_7 \Rightarrow C\vert C,p_{8} \Rightarrow A\odot A\vert   B\Rightarrow p_8 ,\neg A\odot \neg A$.

\textbf{Stage 2}
$\tau _{{\rm {\bf I}}_r :G_1^{\circ \circ } }^{\medstar (0)} =\tau _{{\rm {\bf
I}}_r :G_1^{\circ \circ } (1)}^{\medstar (0)} =\tau _{{\rm {\bf I}}_r
:G_1^{\circ \circ } (2)}^{\medstar (0)} =\cfrac{\underline{\,\,\left\lceil {S_2^c } \right\rceil _I\,\,}}{G_{{\rm {\bf I}}_r }^{\medstar(1)} }{\kern1pt}\left\langle {\tau _{H_2^c
:A\Rightarrow p_3}^\ast } \right\rangle,$

$\tau _{{\rm {\bf I}}_r :G_1^{\circ \circ } (3)}^{\medstar (0)} =\tau _{{\rm
{\bf I}}_r :G_1^{\circ \circ } (4)}^{\medstar (0)} =\cfrac{\underline{\quad\qquad\qquad\left\lceil {S_1^c }
\right\rceil _I \,\,\,\quad\qquad\qquad\left\lceil {S_2^c }
\right\rceil _I {\quad}}}{\widehat{S''}\vert G_{b_{r1} } \vert G_{H_I^V
:\left\langle {G''}
\right\rangle _{\mathcal{I}_{{\rm {\bf j}}_r }} }^{\medstar(J)} \vert G_{H_I^V :H'}^{\medstar(J)} \backslash
\{\widehat{S'}\vert \widehat{S''}\}\vert G_{{\rm {\bf I}}_{l\backslash r}
}^\medstar}\left\langle
{\tau _{{\rm {\bf I}}_l }^\medstar(\tau _{{\rm {\bf I}}_{{\rm {\bf j}}_r
} }^\ast )} \right\rangle .
$

Replacing  $\tau _{{\rm {\bf I}}_r :G_1^{\circ \circ }}^{\medstar (0)}$  in $\tau _{{\rm {\bf I}}_r }^{\medstar (0)} $ with $ \tau _{{\rm {\bf I}}_r :G_1^{\circ \circ } (4)}^{\medstar (0)}$,   then deleting    $G_{{{\rm {\bf I}}_r} }^\medstar$  and after that applying   $\left\langle
{EC_\Omega ^\ast } \right\rangle $  to
$\widehat{S''}\vert G_{b_{r1} } \vert
G_{H_I^V :\left\langle {G''} \right\rangle _{\mathcal{I}_{{\rm {\bf j}}_r
} } }^{\medstar(J)} \vert G_{H_I^V :H'}^{\medstar(J)}
\backslash \{\widehat{S'}\vert \widehat{S''}\}\vert G_{{\rm {\bf
I}}_{l\backslash r} }^\medstar$,  we get $\tau _{{\rm {\bf I}}}^{\medstar}$.
 \end{appendices}

\end{document}